\documentclass[12pt]{amsart}

\title[Non-invariant radial part formulas]
{Differential-difference operators and radial part formulas for non-invariant elements}
\author{Hiroshi ODA}
\address{Faculty of Engineering, Takushoku University,
815-1, Tatemachi, Hachioji-shi, Tokyo 193-0985, Japan}
\email{hoda@la.takushoku-u.ac.jp}

\usepackage{amsmath,amssymb,amsthm,enumitem,mathrsfs,xypic}%,graphicx

\usepackage[hypertex]{hyperref}

\numberwithin{equation}{section}
\theoremstyle{plain}
 \newtheorem{thm}{Theorem}[section]
 \newtheorem{cor}[thm]{Corollary}
 \newtheorem{lem}[thm]{Lemma}
 \newtheorem{prop}[thm]{Proposition}
\theoremstyle{definition}
 \newtheorem{defn}[thm]{Definition}
 \newtheorem{exmp}[thm]{Example}
 
\theoremstyle{remark}
 \newtheorem{rem}[thm]{Remark}

\allowdisplaybreaks[4]

\DeclareMathOperator{\Ad}{Ad}
\DeclareMathOperator{\ad}{ad}
\DeclareMathOperator{\gr}{gr}
\DeclareMathOperator{\Ind}{Ind}
\DeclareMathOperator{\Mat}{Mat}
\DeclareMathOperator{\Trace}{Trace}
\DeclareMathOperator{\Hom}{Hom}
\DeclareMathOperator{\End}{End}
\DeclareMathOperator{\Ker}{Ker}
\DeclareMathOperator{\Coker}{Coker}
\DeclareMathOperator{\Coim}{Coim}
\DeclareMathOperator{\Image}{Im}
\DeclareMathOperator{\sgn}{sgn}
\DeclareMathOperator{\Lie}{Lie}
\DeclareMathOperator{\Real}{Re}
\DeclareMathOperator{\Imaginary}{Im}
\DeclareMathOperator{\spec}{spec}
\DeclareMathOperator{\id}{id}
\DeclareMathOperator{\ev}{ev}
\DeclareMathOperator{\PW}{PW}
\DeclareMathOperator{\tPW}{\widetilde{\PW}}

\newcommand{\Ksp}{{\widehat K_{\text{\rm sp}}}}
\newcommand{\Kqsp}{{\widehat K_{\text{\rm qsp}}}}
\newcommand{\Km}{{\widehat K_M}}
\newcommand{\Kf}{{\text{\rm $K$-finite}}}
\newcommand{\single}{{\text{\rm single}}}
\newcommand{\double}{{\text{\rm double}}}
\newcommand{\fl}{{\text{\rm fl}}}
\newcommand{\fd}{{\text{\rm fd}}}
\newcommand{\triv}{{\text{\rm triv}}}
\newcommand{\sign}{{\text{\rm sgn}}}
\newcommand{\op}{{\text{\rm op}}}
\newcommand{\CM}{{\text{\rm CM}}}
\newcommand{\cpt}{{\text{\rm c}}}
\newcommand{\reg}{{\text{\rm reg}}}
\newcommand{\der}{\partial}
\newcommand{\Lieg}{{\mathfrak g}}
\newcommand{\Lieh}{{\mathfrak h}}
\newcommand{\Lien}{{\mathfrak n}}
\newcommand{\Liea}{{\mathfrak a}}
\newcommand{\Lieb}{{\mathfrak b}}
\newcommand{\Liek}{{\mathfrak k}}
\newcommand{\Liem}{{\mathfrak m}}
\newcommand{\Lies}{{\mathfrak s}}
\newcommand{\CC}{{\mathbb C}}
\newcommand{\RR}{{\mathbb R}}
\newcommand{\ZZ}{{\mathbb Z}}
\newcommand{\ang}[2]{\langle{#1},{#2}\rangle}
\newcommand{\sang}[2]{({#1},{#2})}
\newcommand{\simarrow}{\xrightarrow{\smash[b]{\lower 0.8ex\hbox{$\sim$}}}}
\newcommand{\ttt}{{2\to2}}
\newcommand{\oto}{{1\to1}}
\newcommand{\trans}{{{}^{\mathrm t}}}
\newcommand{\Ximin}{\Xi^{\text{\normalfont min}}}
\newcommand{\Ximax}{\Xi^{\text{\normalfont max}}}
\newcommand{\Xiind}{\Xi_{\text{\normalfont Ind}}}
\newcommand{\tGammaind}{\tilde\Gamma_{\text{\normalfont Ind}}}

\newcommand{\wrad}{{\text{\normalfont w-rad}}}
\newcommand{\rad}{{\text{\normalfont rad}}}
\newcommand{\gammamin}{\gamma^{\text{\normalfont min}}}
\newcommand{\gammaind}{\gamma_{\text{\normalfont Ind}}}
\newcommand{\Nmin}{\mathscr N^{\text{\normalfont min}}}

\newcommand{\Xiwrad}{\Xi_\wrad}
\newcommand{\Xirad}{\Xi_\rad}

\newcommand{\CCh}{\mathscr C_{\text{\normalfont Ch}}}
\newcommand{\Cwrad}{\mathscr C_\wrad}
\newcommand{\Crad}{\mathscr C_\rad}
\newcommand{\Mod}{\text{-}\mathsf{Mod}}
\newcommand{\HC}{{(\Lieg_\CC,K)\Mod^\fl}}
\newcommand{\FD}{{\mathbf H\Mod^\fd}}

\newcommand{\threeset}[4]{{\vphantom)}_{#1}^{\vphantom{#3}}{(#2)}^{#3}_{#4}}
\newcommand{\bthreeset}[4]{{\vphantom{\bigr)}}_{#1}^{\vphantom{#3}}{\bigl(#2\bigl)}^{#3}_{#4}}

\begin{document}
\maketitle

\begin{abstract}
The classical radial part formula for the invariant differential operators and the
$K$-invariant functions on a Riemannian symmetric space $G/K$
is generalized to some non-invariant cases by use of Cherednik operators and a graded Hecke algebra $\mathbf H$
naturally attached to $G/K$.
We introduce a category $\mathscr C_{\text{\normalfont rad}}$ whose object is
a pair of a $((\operatorname{Lie}G)_{\mathbb C},K)$-module and an $\mathbf H$-module 
satisfying some axioms which are formally the same as
the generalized Chevalley restriction theorem and the generalized radial part formula.
Various pairs of analogous notions in the representation theories for $G$ and $\mathbf H$,
such as the Helgason-Fourier transform and the Opdam-Cherednik transform,
are unified in terms of $\mathscr C_{\text{\normalfont rad}}$.
We construct natural functors which send an $\mathbf H$-module to a $((\operatorname{Lie}G)_{\mathbb C},K)$-module 
and have some universal properties intimately related to $\mathscr C_{\text{\normalfont rad}}$.
%(Preliminary draft: version 0.5 (02/25/2014))
\end{abstract}

\tableofcontents

\section{Introduction}
Let $G$ be a real Lie group with a semisimple  Lie algebra $\Lieg$
and $\Lieg=\Liek+\Liea+\Lien$ a fixed Iwasawa decomposition of $\Lieg$.
We assume that the adjoint action $\Ad(g)$ by any element $g\in G$ is
an inner automorphism of the complexification $\Lieg_\CC$ of $\Lieg$
and that the closed subgroup $K:=N_G(\Liek)=\{g\in G;\,\Ad(g)(\Liek)\subset \Liek\}$
is compact.
(These assumptions are automatically satisfied if $G$ is connected and its center is finite.)
The \emph{Harish-Chandra homomorphism} $\gamma$
is the map of the universal enveloping algebra $U(\Lieg_\CC)$
of $\Lieg_\CC$ into the symmetric algebra $S(\Liea_\CC)$ of $\Liea_\CC$
defined by
\begin{equation}\label{eq:HC-homo}
\begin{aligned}
\gamma: U(\Lieg_\CC)&=\bigl(
\Lien_\CC U(\Lieg_\CC)+U(\Lieg_\CC)\Liek_\CC
\bigr)\oplus U(\Liea_\CC)\\
&\xrightarrow{\text{projection}}
U(\Liea_\CC)=S(\Liea_\CC)
\xrightarrow{\text{shift by }-\rho}
S(\Liea_\CC),
\end{aligned}
\end{equation}
where $S(\Liea_\CC)$ is identified with the algebra of
holomorphic polynomial functions on the dual space $\Liea_\CC^*$
of $\Liea_\CC$
and
$\rho:=\frac12\Trace(\ad_{\Lien}|_\Liea)\in\Liea_\CC^*$.
We have two different $G$-action $\ell(\cdot)$, $r(\cdot)$ on $C^\infty(G)$
defined by
\[
\ell(g)f(x)=f(g^{-1}x),\qquad
r(g)f(x)=f(xg)
\]
for $f \in C^\infty(G)$ and $g\in G$.
Their differential actions are denoted by the same symbols.
Thus, on the $\ell(G)$-module
\[
C^\infty(G/K)\simeq\{f(x) \in C^\infty(G);\,f(xk)=f(x)
\quad\text{for }k\in K\},
\]
$U(\Lieg_\CC)^K$ (the subalgebra of $\Ad(K)$-invariants in $U(\Lieg_\CC)$)
naturally acts by $r(\cdot)$.

Suppose $\lambda\in\Liea_\CC^*$.
A unique function $\phi_\lambda\in
C^\infty(G/K)$ satisfying
\begin{equation}\label{eq:spherical}
\left\{
\begin{aligned}
&r(\Delta)\phi_\lambda=\gamma(\Delta)(\lambda)\phi_\lambda
\quad \text{for }\Delta\in U(\Lieg_\CC)^K, \\
&\ell(K)\text{-invariant},\\
&\phi_\lambda(1 K)=1
\end{aligned}
\right.
\end{equation}
is called a \emph{spherical function}.
Let $A$ and $N$ be respectively the analytic subgroups of $\Liea$ and $\Lien$.
Then from the global Iwasawa decomposition $G=NAK$
we have $G/K\simeq NA$.
Let 
\[\gamma_0 : C^\infty(G/K)\longrightarrow C^\infty(A)\]
be the natural restriction map
and $W$ the Weyl group for $(\Lieg,\Liea)$.
Then the second property in \eqref{eq:spherical}
implies $\gamma_0(\phi_\lambda)\in C^\infty(A)^W$.
By \cite{HO}, Heckman and Opdam started their studies
on the \emph{systems of hypergeometric differential equations},
which are certain modification of the system of differential equations
satisfied by $\gamma_0(\phi_\lambda)$.
At the early stage the existence of such modification had been quite non-trivial.
But after a while
\emph{Cherednik operator}s introduced by \cite{Ch1}
turned out to provide an elegant method to construct the modified systems.
(This idea is due to \cite{He:Heckman}, in which
\emph{Heckman operator}s play the same role as Cherednik operators.)
In this context, a key fact is that
the Cherednik operator $\mathscr T: S(\Liea_\CC)\rightarrow
\End_\CC C^\infty(A)$ with a special parameter (Definition \ref{defn:Chered})
satisfies
for $\Delta\in U(\Lieg_\CC)^K$ and
$f\in C^\infty(G/K)^{\ell(K)}$
a \emph{radial part formula}
\begin{equation}\label{eq:rad1}
\gamma_0\bigl(r(\Delta)f\bigr)
=
\mathscr T\bigl(\gamma(\Delta)\bigr) \gamma_0(f),
\end{equation}
or equivalently,
\begin{equation}\label{eq:rad2}
\gamma_0\bigl(\ell(\Delta)f\bigr)
=
\mathscr T\bigl(\gamma(\theta\Delta)\bigr) \gamma_0(f).
\end{equation}
Here $\theta$ is the Cartan involution of $G$ leaving $K$ invariant.
(The equivalence easily follows from the equality $f(\theta g^{-1})=f(g)$ for $f\in C^\infty(G/K)^{\ell(K)}$.)
In the first part of this paper (\S\S\ref{sec:Ch}, \ref{sec:GHA}--\ref{sec:radII})
we try to generalize
\eqref{eq:rad1} and \eqref{eq:rad2} 
for non-$K$-invariant $\Delta$ and $f$.

First, we consider \eqref{eq:rad1} concerns the radial part of
the action of $r(U(\Lieg_\CC)^K)$ on $C^\infty(G/K)$
and generalize it
to the case where $f\in C^\infty(G/K)$ is no longer $K$-invariant (\S\ref{sec:radI}). 
For example,
\eqref{eq:rad1} holds for
any $\Delta\in U(\Lieg_\CC)^K$ and
any $K$-finite $f \in C^\infty(G/K)$ 
such that all $K$-types in $\ell(U(\Liek_\CC))f$ are
\emph{single-petaled} (Theorem \ref{thm:radC}).
A single-petaled $K$-type is a special kind of $K$-type
introduced by \cite{Oda:HC}.
We denote the set of single-petaled $K$-types by $\Ksp$
(Definition \ref{defn:K-types}).
In \cite{Op:Cherednik} Opdam studies
the following system of differential-difference equations:
\begin{equation}\label{eq:Ch-sys}
\mathscr T(\Delta)\varphi=\Delta(\lambda)\varphi
\quad
\forall\Delta\in S(\Liea_\CC)^W.
\end{equation}
Here $\lambda\in\Liea^*_\CC$ is any fixed spectral parameter and the unknown function $\varphi\in C^\infty(A)$ is not necessarily $W$-invariant.
The \emph{graded Hecke algebra} $\mathbf H$
attached to the Iwasawa decomposition of $G=NAK$ (Definition \ref{defn:H})
contains $S(\Liea_\CC)$
and the group algebra $\CC W$ of $W$ as subalgebras
and $\mathscr T$ naturally extends  to
an algebra homomorphism $\mathbf H\to\End_\CC C^\infty(A)$.
With respect to this $\mathbf H$-module structure
the solution space
$\mathscr A(A, \lambda)$ of \eqref{eq:Ch-sys}
is generically an irreducible submodule of $C^\infty(A)$.
The harmonic analysis
developed by Opdam decomposes $C^\infty_\cpt(A)$
(the space of compactly supported $C^\infty$ functions on $A$)
into the direct integral of $\mathscr A(A, \lambda)$'s (Proposition \ref{prop:FourierH}).
His theory is surprisingly similar to the theory of
the \emph{Helgason-Fourier transform} for $C^\infty_\cpt(G/K)$ (Proposition \ref{prop:FourierG})
where the 
$G$-module\[
\mathscr A(G/K,\lambda)
:=\bigl\{
f \in C^\infty(G/K);\,r(\Delta)f=\gamma(\Delta)(\lambda)f
\quad\text{for }\Delta \in U(\Lieg_\CC)^K
\bigr\}
\]
has a role of constitutional unit of $C^\infty(G/K)$.
This $G$-module is known as
the solution space for a maximal system of invariant differential
operators on $G/K$ (cf.~\cite{Hel4}).
As a direct link between $\mathscr A(A, \lambda)$ and $\mathscr A(G/K,\lambda)$,
we have a linear bijection:
\begin{equation}\label{eq:inviso}
\gamma_0 :
\mathscr A(G/K,\lambda)^{\ell(K)}
\simarrow
{\mathscr A(A, \lambda)}^W;\quad
\phi_\lambda\longmapsto \gamma_0(\phi_\lambda).
\end{equation}
This comes from
the Chevalley restriction theorem,
Harish-Chandra's celebrated exact sequence
\begin{equation}\label{eq:HC-isom}
0\to (U(\Lieg_\CC)\Liek_\CC)^K \to 
U(\Lieg_\CC)^K \xrightarrow{\gamma}
S(\Liea_\CC)^W\to 0,
\end{equation}
and \eqref{eq:rad1}.
Now a generalization of Chevalley restriction theorem
given in \cite{Oda:HC} asserts that
$\gamma_0$ naturally induces for each $V\in\Ksp$ 
a linear bijection
\begin{equation}\label{eq:Ch0}
\Gamma^V_0:
\Hom_K(V, C^\infty(G/K))\simarrow
\Hom_W(V^M, C^\infty(A)).
\end{equation}
Here $M$ is the centralizer of $A$ in $K$
and $V^M$ is the $M$-fixed part of $V$ (see \S\ref{sec:Ch}).
This,
combined with \eqref{eq:HC-isom}
and our generalization of \eqref{eq:rad1},
produces a link stronger than \eqref{eq:inviso}.
That is,
for each $V\in\Ksp$ it holds that
\begin{equation}\label{eq:spec-cor}
\Gamma^V_0:
\Hom_K(V, \mathscr A(G/K,\lambda))\simarrow
\Hom_W(V^M, \mathscr A(A, \lambda))
\end{equation}
(Corollary \ref{cor:spec-cor}).
This means the generalized Chevalley restriction theorem
respects the spectra.

Secondly, we consider \eqref{eq:rad2}
concerns the radial part of the action of $\ell(U(\Lieg_\CC))$ on $C^\infty(G/K)$
and generalize it to the case
where neither operators nor functions are assumed $K$-invariant (\S\ref{sec:radII}). 
By use of a certain involutive automorphism $\theta_{\mathbf H}$
of $\mathbf H$ (Definition \ref{defn:HCartan}), we can rewrite \eqref{eq:rad2} as
\[
\gamma_0\bigl(\ell(\Delta)f\bigr)
=
\mathscr T\bigl(\theta_{\mathbf H}\gamma(\Delta)\bigr) \gamma_0(f).
\]
So it is natural to think of the $\mathbf H$-action $\mathscr T(\theta_{\mathbf H}\cdot)$ on $C^\infty(A)$ as the radial counterpart of the $G$-action $\ell(\cdot)$ on $C^\infty(G/K)$.
With respect to these module structures
of $C^\infty(G/K)$ and $C^\infty(A)$
let us apply the Frobenius reciprocity to
both the sides of \eqref{eq:Ch0}.
We then get a linear bijection
\begin{equation*}
\Gamma_0:
\Hom_{\Lieg_\CC,K}(U(\Lieg_\CC)\otimes_{U(\Liek_\CC)} V,\, C^\infty(G/K)_{\Kf})\simarrow
\Hom_{\mathbf H}(\mathbf H\otimes_{\CC W}V^M,\, C^\infty(A)).
\end{equation*}
Here the subscript ``$\Kf$'' indicates the subspace consisting of $K$-finite vectors.
To formulate our generalization, we need also to generalize the Harish-Chandra homomorphism $\gamma$.
Suppose $E$ is another single-petaled $K$-type.
Then we can define a natural map
\[
\Gamma^E_V:
\Hom_K(E, U(\Lieg_\CC)\otimes_{U(\Liek_\CC)} V)
\rightarrow
\Hom_W(E^M, \mathbf H\otimes_{\CC W} V^M)
\]
(\S\ref{sec:HC}).
The case where $V$ is the trivial $K$-type $\CC_\triv$ is studied in \cite{Oda:HC}
and the case where $E=V=\CC_\triv$ reduces to $\gamma$.
Let us now state our generalization of \eqref{eq:rad2}.
For any
$\Phi\in  
\Hom_{\Lieg_\CC,K}(U(\Lieg_\CC)\otimes_{U(\Liek_\CC)} V,\, C^\infty(G/K)_{\Kf})
\simeq\Hom_K(V, C^\infty(G/K))$
and any $\Psi\in \Hom_K(E, U(\Lieg_\CC)\otimes_{U(\Liek_\CC)} V)$
it holds that
\begin{equation}\label{eq:RPF}
\Gamma^E_0(\Phi\circ\Psi)
=\Gamma_0(\Phi)\circ\Gamma^E_V(\Psi).
\end{equation}
In other words, if we take a basis $\{v_1,\ldots,v_n\}$ of $V$
so that $\{v_1,\ldots,v_m\}$ is a basis of $V^M$ $(m\le n)$
and if we write for any $e\in E^M$ 
\[
\Psi[e]=\sum_{i=1}^n D_i\otimes v_i \text{ with }D_i\in U(\Lieg_\CC),\qquad
\Gamma^E_V(\Psi)[e]=\sum_{i=1}^{m} h_i\otimes v_i \text{ with }h_i\in \mathbf H,
\]
then
\[
\gamma_0\biggl(\sum_{i=1}^n \ell(D_i)\Phi[v_i]\biggr)
=
\sum_{i=1}^{m}\mathscr T(\theta_{\mathbf H}h_i)\gamma_0(\Phi[v_i]).
\]
Actually we have a further generalization of \eqref{eq:RPF} to
the case where $E$ and $V$ are
\emph{quasi-single-petaled} $K$-types (Definition \ref{defn:K-types}).
The complete results will be stated in
Theorem \ref{thm:radD}.
For simplicity, in this introductory section
we shall state all results 
without using the notion of quasi-single-petaled $K$-types.

In section \ref{sec:cor} we define a natural correspondence
\begin{equation}\label{eq:Cinftycor}
\Ximin_0 :
\bigl\{\mathbf H\text{-submodule of } C^\infty(A) \bigr\}
\rightarrow
\bigl\{(\Lieg_\CC,K)\text{-submodule of } C^\infty(G/K)_\Kf \bigr\}
\end{equation}
using \eqref{eq:Ch0}.
For example, $\Ximin_0$ maps
a unique irreducible ${\mathbf H}$-submodule $X_{\mathbf H}(\lambda)$ of $\mathscr A(A, \lambda)$
to a unique irreducible $(\Lieg_\CC,K)$-submodule $X_G(\lambda)$ of $\mathscr A(G/K,\lambda)_\Kf$
(Theorem \ref{thm:Xpair}).
(The module structures of
$\mathscr A(A, \lambda)$ and $\mathscr A(G/K,\lambda)_\Kf$ 
will be studied in \S\ref{sec:modstr}.)
If $\mathscr X$ is an $\mathbf H$-submodule of $C^\infty(A)$
then for each $V\in\Ksp$ the linear map $\Gamma^V_0$
induces a linear bijection 
\begin{equation*}
\Gamma^V_0:
\Hom_K(V, \Ximin_0(\mathscr X))\simarrow
\Hom_W(V^M, \mathscr X)
\end{equation*}
(Theorem \ref{thm:multcor} (iii)).
This property comes from \eqref{eq:RPF}.
Now we can develop
a similar story for the pair
$\bigl(U(\Lieg_\CC)\otimes_{U(\Liek_\CC)}\CC_\triv,\,\mathbf H\otimes_{\CC W}\CC_\triv\bigr)$ instead of 
$\bigl(C^\infty(G/K)_\Kf,\,C^\infty(A)\bigr)$.
Namely, we can define a correspondence
\begin{equation}\label{corr:Ximin}
\Ximin:
\bigl\{\mathbf H\text{-submodule of } \mathbf H\otimes_{\CC W}\CC_\triv \bigr\}
\rightarrow
\bigl\{(\Lieg_\CC,K)\text{-submodule of } U(\Lieg_\CC)\otimes_{U(\Liek_\CC)}\CC_\triv \bigr\}
\end{equation}
for which $\Gamma^V_{\CC_\triv}$ has the same property with $\Gamma^V_0$.
Motivated by this parallelism,
we introduce a new category $\Crad$ (Definitions \ref{defn:CCh}, \ref{defn:Cwrad} and \ref{defn:Crad}).
An object $\mathcal M\in \Crad$, which we call a \emph{radial pair},
is a pair of a $(\Lieg_\CC,K)$-module $\mathcal M_G$
and an $\mathbf H$-module $\mathcal M_{\mathbf H}$ 
satisfying a set of axioms which are
formally the same as the generalized Chevalley restriction theorem 
and the second type of radial part formula.
Some parts of the axioms are as follows:
To each $V\in\Ksp$ there attach a linear map
\[
\tilde\Gamma^V_{\mathcal M}:
\Hom_K(V,\mathcal M_G)\to\Hom_W(V^M,\mathcal M_{\mathbf H})
\]
and a subspace $\Hom_K^\ttt(V,\mathcal M_G)$ of $\Hom_K(V,\mathcal M_G)$ 
such that the restriction of $\tilde\Gamma^V_{\mathcal M}$ to $\Hom_K^\ttt(V,\mathcal M_G)$ gives a bijection
\[
\tilde\Gamma^V_{\mathcal M}:
\Hom_K^\ttt(V,\mathcal M_G)\simarrow\Hom_W(V^M,\mathcal M_{\mathbf H})
\]
(cf.~\eqref{eq:Ch0}) and for any
$\Phi\in \Hom_K^\ttt(V,\mathcal M_G)$,
$E\in\Ksp$ and $\Psi\in \Hom_K(E, U(\Lieg_\CC)\otimes_{U(\Liek_\CC)} V)$
it holds that
\[
\tilde \Gamma^E_{\mathcal M}(\Phi\circ\Psi)
=\tilde \Gamma_{\mathcal M}(\Phi)\circ\Gamma^E_V(\Psi)
\]
(cf.~\eqref{eq:RPF});
Here $\tilde \Gamma_{\mathcal M}(\Phi)$ is a morphism in
$\Hom_{\mathbf H}(\mathbf H\otimes_{\CC W}V^M,\mathcal M_{\mathbf H})$
identified with $\tilde \Gamma_{\mathcal M}^V(\Phi)\in \Hom_W(V^M,\mathcal M_{\mathbf H})$
by the Frobenius reciprocity.
In many important cases
$\Hom_K^\ttt(V,\mathcal M_G)=\Hom_K(V,\mathcal M_G)$ (cf.~Remark \ref{rem:RadR}).
For any $\mathcal M=(\mathcal M_G,\mathcal M_{\mathbf H})\in\Crad$
a correspondence
\[
\Ximin_{\mathcal M}:
\bigl\{\mathbf H\text{-submodule of } \mathcal M_{\mathbf H} \bigr\}
\rightarrow
\bigl\{(\Lieg_\CC,K)\text{-submodule of } \mathcal M_G \bigr\}.
\]
is defined.
If $\mathscr X$ is an $\mathbf H$-submodule of $\mathcal M_{\mathbf H}$
then $(\Ximin_{\mathcal M}(\mathscr X),\mathscr X)$ is again a radial pair
(Theorem \ref{thm:multcor} (ii)).

Besides $\bigl(C^\infty(G/K)_\Kf,\,C^\infty(A)\bigr)$ and $\bigl(U(\Lieg_\CC)\otimes_{U(\Liek_\CC)}\CC_\triv,\,\mathbf H\otimes_{\CC W}\CC_\triv\bigr)$,
there are many natural radial pairs.
For example, \eqref{eq:spec-cor}
implies $\bigl(\mathscr A(G/K,\lambda)_\Kf,\, \mathscr A(A,\lambda)\bigr)\in\Crad$
(Example \ref{exmp:AA}).
Also, $B(\lambda)=\bigl(B_G(\lambda)_\Kf,\,B_{\mathbf H}(\lambda)\bigr)\in\Crad$
(Theorem \ref{thm:R3B})
where $B_G(\lambda)$ and $B_{\mathbf H}$ are
the minimal spherical principal series representations for $G$ and $\mathbf H$
induced from the same character $-\lambda\in\Liea_\CC^*$.
In \S\S\ref{sec:sps}--\ref{sec:F}
we shall see various pairs of analogous notions
appearing in the representation theories for $G$ and $\mathbf H$
can be peacefully packed in the category $\Crad$.
In \S\ref{sec:Poisson} we define
an $\mathbf H$-homomorphism $\mathcal P_{\mathbf H}^\lambda : B_{\mathbf H}(\lambda)\to \mathscr A(A,\lambda)$
analogous to the \emph{Poisson transform} $\mathcal P_G^\lambda : B_G(\lambda)\to\mathscr A(G/K,\lambda)$, and then prove
they constitute a morphism in $\Crad$:
\[
\mathcal P^\lambda=(\mathcal P_{\mathbf H}^\lambda,\mathcal P_{G}^\lambda):\,
\bigl(B_G(\lambda)_\Kf,\,B_{\mathbf H}(\lambda)\bigr)
\to
\bigl(\mathscr A(G/K,\lambda)_\Kf,\, \mathscr A(A,\lambda)\bigr)
\]
(Theorem \ref{thm:Poisson}). As a morphism in $\Crad$ (cf.~Definition \ref{defn:CCh}),
$\mathcal P^\lambda$ satisfies for any $V\in\Ksp$ and $\Phi\in\Hom_K(V, B_G(\lambda))$
\[
\Gamma^V_0\bigl(\mathcal P_{G}^\lambda\circ\Phi\bigr)
=
\mathcal P_{\mathbf H}^\lambda\circ\tilde\Gamma^V_{B(\lambda)}(\Phi).
\]
In \S\ref{sec:intertwining} we study a \emph{Knapp-Stein type intertwining operator}
\[
\Tilde{\mathcal A_G}(w,\lambda):\,
B_G(\lambda)\to B_G(w\lambda)
\quad (w\in W)
\]
and its analogue $\Tilde{\mathcal A_{\mathbf H}}(w,\lambda)$ in the category $\mathbf H\Mod$
of $\mathbf H$-modules.
Of course they constitute a morphism in $\Crad$ (Theorem \ref{thm:it}).
In \S\ref{sec:F} we study the relation between
the Helgason-Fourier transform $\mathcal F_G$
and the \emph{Opdam-Cherednik transform} $\mathcal F_{\mathbf H}$ introduced
respectively in \cite{Hel1} and \cite{Op:Cherednik}.
The Paley-Wiener theorems, the inversion formulas and the Plancherel formulas for both transforms
can be successfully combined in $\Crad$
(Theorem \ref{thm:F}).

In \S\ref{sec:Ch2} we prove 
the generalized Chevalley restriction theorem holds
for the class $\mathscr A$ of analytic functions  (Theorem \ref{thm:Chanal}).
This result implies that
$\bigl(\mathscr A(G/K),\, \mathscr A(A)\bigr)\in\Crad$
(Corollary \ref{cor:AApair}) and that
the correspondence \eqref{eq:Cinftycor}
restricts to
\begin{equation}\label{eq:Analcor}
\Ximin_0 :
\bigl\{\mathbf H\text{-submodule of } \mathscr A(A) \bigr\}
\rightarrow
\bigl\{(\Lieg_\CC,K)\text{-submodule of } \mathscr A(G/K)_\Kf \bigr\}.
\end{equation}
In \S\S\ref{sec:Xirad}--\ref{sec:Xi}
we shall construct three functors $\Xirad$, $\Ximin$ and $\Xi$,
each of which sends an $\mathbf H$-module $\mathscr X$ to a $(\Lieg,K)$-module $\mathscr Y$
such that $(\mathscr Y,\mathscr X)\in\Crad$.
If $\mathscr X\in\mathbf H\Mod$ then there is a sequence of
surjective $(\Lieg,K)$-homomorphisms
\[
\Xirad(\mathscr X)\twoheadrightarrow
\Ximin(\mathscr X)\twoheadrightarrow
\Xi(\mathscr X).
\]
These functors have their own universal properties.
First, the functor $\mathbf H\Mod\ni\mathscr X\mapsto(\Xirad(\mathscr X),\mathscr X)\in\Crad$
is left adjoint to the functor
$\Crad\ni (\mathcal M_G,\mathcal M_{\mathbf H})
\mapsto \mathcal M_{\mathbf H}\in\mathbf H\Mod$ (Theorem \ref{thm:Xirad}).
If $\mathscr X\in\mathbf H\Mod$ has a \emph{central character} (Definition \ref{defn:ccic})
then $\Xirad(\mathscr X)\in(\Lieg_\CC,K)\Mod$
has a corresponding \emph{infinitesimal character} (Theorem \ref{thm:charcorr}).
For a finite-dimensional $\mathscr X\in\mathbf H\Mod$
the length of $\Xirad(\mathscr X)$ is finite (Theorem \ref{thm:FD2HC}).
Secondly, $\Ximin$ is a functor extending the correspondence \eqref{corr:Ximin} (Corollary \ref{cor:Ximinext}).
If $\mathscr X\in\mathbf H\Mod$ has finite dimension,
then $\Ximin(\mathscr X)$ can be embedded into
the $G$-module induced from $\mathscr X$ viewed naturally as an $MAN$-module (Theorem \ref{thm:minind}).
Using this realization we can prove that
if $(\cdot,\cdot)^{\mathbf H}$ is an invariant
sesquilinear form on
two finite-dimensional $\mathbf H$-modules $\mathscr X_1,\mathscr X_2$ (cf.~Definition \ref{defn:sesquiHinv})
then there exists a natural invariant
sesquilinear form $(\cdot,\cdot)^{G}$ on $\Ximin(\mathscr X_1)\times \Ximin(\mathscr X_2)$
(Theorem \ref{thm:sesquilift}).
Finally, $\Xi$ extends the correspondence \eqref{eq:Analcor} (Corollary \ref{cor:Xi} (iii)). 
If $\mathcal M=(\mathcal M_G, \mathcal M_{\mathbf H})\in\Crad$
then there exists a unique $(\Lieg_\CC,K)$-homomorphism
$\mathcal I_G:\Ximin_{\mathcal M}(\mathcal M_{\mathbf H})\to\Xi(\mathcal M_{\mathbf H})$
such that
\[
(\mathcal I_G,\id_{\mathcal M_{\mathbf H}}) :
\bigl(\Ximin_{\mathcal M}(\mathcal M_{\mathbf H}),\mathcal M_{\mathbf H}\bigr)
\to \bigl(\Xi(\mathcal M_{\mathbf H}),\mathcal M_{\mathbf H}\bigr)
\]
is a morphism in $\Crad$ (Theorem \ref{thm:Xiunivprop} (ii)).
If $(\cdot,\cdot)^{\mathbf H}$ is as above
then $(\cdot,\cdot)^{G}$ induces a sesquilinear form on $\Xi(\mathscr X_1)\times \Xi(\mathscr X_2)$
(Theorem \ref{thm:Xiindform}).
This form is non-degenerate when $(\cdot,\cdot)^{\mathbf H}$ is non-degenerate.
In \S\ref{sec:SL2R}
we shall restrict ourselves to the case of $G=SL(2,\RR)$ and
completely describe the behaviors of these functors
for the irreducible $\mathbf H$-modules
(Theorem \ref{thm:SL2Rfunctors}).
Our three functors are closely related to the functors
given by \cite{AS, EFM, CT}.
Roughly speaking, our functors are right inverses of their functors.
The author wants to discuss
the relations between them in a subsequent paper.

Now, let us return to the first part of the paper
and consider the infinitesimal (or ``tangential'')
version of radial part formulas.
Let $\Lies$ be the $-1$-eigenspace of $\theta$ in $\Lieg$.
For the \emph{Cartan motion group} $G_\CM:=K\ltimes\Lies$
and the \emph{rational Dunkl operator}s introduced by \cite{Dun},
we have similar results to the case of $G/K$.
Using the $K$-module isomorphism
\begin{equation}\label{eq:K-isom}
C^\infty(G/K)\simarrow
C^\infty(\Lies)\,;\quad
f\longmapsto f(\exp\cdot).
\end{equation}
and the $W$-module isomorphism
\begin{equation}\label{eq:W-isom}
C^\infty(A)\simarrow C^\infty(\Liea)\,;\quad
\varphi\longmapsto \varphi(\exp\cdot),
\end{equation}
we identify $\gamma_0$ with the natural restriction map
$C^\infty(\Lies)\rightarrow C^\infty(\Liea)$.
Via the Killing form $B(\cdot,\cdot)$ of $\Lieg$,
$S(\Lies_\CC)$ is identified with the algebra $\mathscr P(\Lies)$ of polynomial functions on $\Lies$,
and $S(\Liea_\CC)$ with $\mathscr P(\Liea)$.
Thus $\gamma_0$ induces a map $S(\Lies_\CC)\rightarrow S(\Liea_\CC)$,
which is also denoted by $\gamma_0$.
For $X \in \Lies$ let $\der(X)$ denote
the $X$-directional derivative operator on $\Lies$.
Extend the linear map $\der : \Lies\rightarrow \End_\CC C^\infty(\Lies)$
to an algebra homomorphism $\der : S(\Lies_\CC)\rightarrow \End_\CC C^\infty(\Lies)$.
In \cite{Je} de Jeu gives a simple proof for the fact that
the Dunkl operator
$\mathscr D : S(\Liea_\CC)\rightarrow
\End_\CC C^\infty(\Liea)$
with a special parameter (Definition~\ref{defn:Dunkl}) satisfies a radial part formula
\begin{equation}\label{eq:rad3}
\gamma_0\bigl(\der(\Delta)f\bigr)=
\mathscr D\bigl(\gamma_0(\Delta)\bigr)\gamma_0(f)
\end{equation}
for $\Delta\in S(\Lies_\CC)^K$ and $f\in C^\infty(\Lies)^K$.
In \S\ref{sec:Dunkl} 
we show that
\eqref{eq:rad3} holds for more general combinations of $\Delta$ and $f$ (Theorems~\ref{thm:radA}, \ref{thm:radB}).
But we do not try to for maximally possible combinations
because in this paper our stress is persistently on the case of $G/K$.
Nevertheless, after recalling
the generalized Chevalley restriction theorem in \S\ref{sec:Ch},
we shall discuss the case of $G_\CM/K$ first in \S\ref{sec:Dunkl}.
One reason for it is the easiness:
de Jeu's simple method still works in our generalization without any change.
Another reason is that the resulting radial part formulas will be good prototypes
for the Riemannian symmetric space case in subsequent sections.
\medskip

\noindent
{\bfseries Acknowledgments.}
This paper was written during the stay of the author at MIT.
He would like to express his gratitude to the Department of Mathematics of MIT
for their help and hospitality, especially to David A.~Vogan, Jr. 
The author is also grateful to Sigurdur Helgason, Fulton Gonzalez, and Dan Ciubotaru
for valuable discussions.

\section{The Chevalley restriction theorem, I}\label{sec:Ch}
In this section we review a generalization
of the Chevalley restriction theorem given in \cite{Oda:HC}.
Let $\Sigma$ be
the restricted root system for $(\Lieg,\Liea)$, 
and $\Sigma^+$ the positive system
of $\Sigma$ corresponding to $\Lien$.
The root space for each $\alpha\in\Sigma$ is denoted by
$\Lieg_\alpha$ and we fix $E_\alpha\in\Lieg_\alpha$ so that
$-B(E_\alpha,\theta E_\alpha)=\frac{2}{B(H_\alpha,H_\alpha)}$.
Here for $\mu\in \Liea_\CC^*$, $H_\mu$ denotes a unique element in $\Liea_\CC$
such that $\mu(H)=B(H_\mu,H)$ for $H\in\Liea$.
Putting $Z_\alpha=E_\alpha+\theta E_\alpha \in \Liek$,
we introduce some classes of $K$-types.
\begin{defn}[$K$-types]\label{defn:K-types}
(i)
In this paper by the terminology ``$K$-type''
we mean an irreducible unitary representation of $K$
or its equivalence class.
The set of $K$-types is denoted by $\widehat K$.

\noindent
(ii) $V\in \widehat K$ is called \emph{$M$-spherical}
if $V^M\ne\{0\}$.
The set of $M$-spherical $K$-types is denoted by $\Km$.

\noindent
(iii) For $V \in \Km$ put
\begin{equation}\label{eq:VMsingle}
V^M_\single
=\bigl\{
v\in V^M;\,
Z_\alpha(Z_\alpha^2+4)v=0\quad\text{for any }\alpha\in\Sigma
\bigr\}.
\end{equation}
$V\in \Km$ is called \emph{quasi-single-petaled}
if $V^M_\single\ne\{0\}$.
The set of quasi-single-petaled $K$-types is denoted by $\Kqsp$.

\noindent
(iv)
$V \in \Km$ is called \emph{single-petaled}
if $V^M_\single=V^M$.
The set of single-petaled $K$-types is denoted by $\Ksp$.
\end{defn}
\begin{rem}
For $V\in\Km$, $V^M$ is naturally a $W$-module.
Its subspace $V^M_\single$ is $W$-stable
since the definition \eqref{eq:VMsingle}
is independent of the choice of $E_\alpha$'s (cf.~\cite[Remark 1.2]{Oda:HC}).
\end{rem}
Let $\mathscr F$ be either one of the following function classes on $G/K$ or $A$:
$C^\infty$ (smooth functions), $C^\infty_\cpt$ (smooth functions with compact support), $\mathscr P$ (polynomial functions on $\Lies\simeq G/K$ or $\Liea\simeq A$).
\begin{thm}[the generalized Chevalley restriction theorem \cite{Oda:HC}]\label{thm:Ch}
Suppose $V\in\Km$.

\noindent{\normalfont (i)}
The map
\begin{equation}\label{eq:GammaV0}
\begin{aligned}
\Gamma_0^V :\,&
 \Hom_K(V,\mathscr F(G/K))
\longrightarrow
\Hom_W(V^M,\mathscr F(A))\,;\\
&\Phi\longmapsto
\bigl(
\varphi:
V^M\hookrightarrow
V
\xrightarrow{\Phi} \mathscr F(G/K)
\xrightarrow{\gamma_0} \mathscr F(A)
\bigr)
\end{aligned}\end{equation}
is well defined.

\noindent{\normalfont (ii)}
Let $p^V : V\rightarrow V^M$ be the orthogonal projection
with respect to a $K$-invariant inner product of\/ $V$.
Then for any $\Phi\in \Hom_K(V,\mathscr F(G/K))$,
$\gamma_0\circ\Phi[v]=\Gamma_0^V(\Phi)\circ p^V[v]$.

\noindent{\normalfont (iii)}
$\Gamma_0^V$ is injective.

\noindent{\normalfont (iv)}
Define
\[
V^M_\double=V^M\cap
\sum
\bigl\{
m\,
Z_\alpha(Z_\alpha^2+4) V;\,
\alpha\in\Sigma, m\in M
\bigr\}.
\]
Then $V^M=V^M_\single\oplus V^M_\double$ is a direct sum decomposition into two $W$-submodules.
A $W$-submodule $U\subset V^M$ satisfies the condition
\[
\Image \Gamma_0^V
\supset
\bigl\{
\varphi\in \Hom_W(V^M, \mathscr F(A));\,
\varphi[U]=\{0\}
\bigr\}
\]
if and only if $U\supset V^M_\double$.
In particular
$\Image \Gamma_0^V$ contains
any $\varphi\in \Hom_W(V^M, \mathscr F(A))$ such that
$\varphi\bigl[ V^M_\double\bigr]=\{0\}$.

\noindent{\normalfont (v)}
$\Gamma_0^V$ is surjective
if and only if $V\in\Ksp$.
\end{thm}
\begin{rem}
In view of \eqref{eq:K-isom} and \eqref{eq:W-isom},
$\Gamma_0^V$ is naturally identified with
the following map:
\begin{equation*}%\label{eq:varGamma}
\begin{aligned}
\Gamma_0^V :\,&
 \Hom_K(V,\mathscr F(\Lies))
\longrightarrow
\Hom_W(V^M,\mathscr F(\Liea))\,;\\
&\Phi\longmapsto
\bigl(
\varphi:
V^M\hookrightarrow
V
\xrightarrow{\Phi} \mathscr F(\Lies)
\xrightarrow{\gamma_0} \mathscr F(\Liea)
\bigr).
\end{aligned}
\end{equation*}
In fact,
\cite{Oda:HC}
originally proved Theorem~\ref{thm:Ch}
in this form. 
\end{rem}
\begin{defn}\label{defn:22}
{\normalfont (i)}
For $V\in\Km$ we put
\[
\Hom_K^{\ttt}(V,\mathscr F(G/K))
=\bigl\{
\Phi\in \Hom_K(V,\mathscr F(G/K));\,
\Gamma^V_0(\Phi)\bigl[ V^M_\double\bigr]=\{0\}
\bigr\}.
\]
$\Hom_K^{\ttt}(V,\mathscr F(\Lies))\simeq \Hom_K^{\ttt}(V,\mathscr F(G/K))$
is similarly defined.
(This symbol comes from
the analogy between $\Hom_K^{\ttt}(V,\mathscr F(G/K))$
and $\Hom_K^\ttt(V, U(\Lieg_\CC)\otimes_{U(\Liek_\CC)}\CC_\triv)$.
The latter will be defined in \S\ref{sec:HC}.)

\noindent{\normalfont (ii)}
We define the map
\begin{equation*}
\begin{aligned}
\tilde\Gamma_0^V :\,&
 \Hom_K(V,\mathscr F(G/K))
\longrightarrow
\Hom_W(V^M_\single,\mathscr F(A))\,;\\
&\Phi\longmapsto
\bigl(
\varphi:
V^M_\single\hookrightarrow
V
\xrightarrow{\Phi} \mathscr F(G/K)
\xrightarrow{\gamma_0} \mathscr F(A)
\bigr).
\end{aligned}\end{equation*}
This is also identified with $\tilde \Gamma_0^V :\,
 \Hom_K(V,\mathscr F(\Lies))
\longrightarrow
\Hom_W(V^M_\single,\mathscr F(\Liea))$.
\end{defn}

\begin{rem}\label{rem:22}
{\normalfont (i)}
If $V\in \Ksp$ then $\Hom_K^{\ttt}(V,\mathscr F(G/K))=\Hom_K(V,\mathscr F(G/K))$.
If $V\in \Km\setminus\Kqsp$ then $\Hom_K^{\ttt}(V,\mathscr F(G/K))=\{0\}$.

\noindent
{\normalfont (ii)}
If $V\in\Km$ then it follows from Theorem \ref{thm:Ch} that
the restriction of $\tilde\Gamma^V_0$ to $\Hom_K^\ttt$ gives a linear bijection
\begin{equation}\label{eq:TGammaV0}
\tilde\Gamma_0^V :\,
 \Hom_K^\ttt(V,\mathscr F(G/K))
\simarrow
\Hom_W(V^M_\single,\mathscr F(A)).
\end{equation}
\end{rem}

The following lemma will be used repeatedly to deduce generalized radial part formulas:
\begin{lem}\label{lem:qsp}
Suppose $V\in \Km$ and
let $p^V$ be as in {\normalfont Theorem \ref{thm:Ch} (ii)}.
Suppose $\alpha\in\Sigma$ and $X_\alpha\in\Lieg_\alpha$.

\noindent{\normalfont (i)}
$(X_\alpha+\theta X_\alpha)^2v
=|\alpha|^2B(X_\alpha,\theta X_\alpha)(1-s_\alpha)v$ for $v\in V^M_\single$,
where $|\alpha|=\sqrt{B(H_\alpha,H_\alpha)}$
and $s_\alpha\in W$ is the reflection corresponding to $\alpha$.

\noindent{\normalfont (ii)}
$p^V\bigl((X_\alpha+\theta X_\alpha)^2v\bigr)\in V^M_\double$
for $v\in V^M_\double$.
\end{lem}
\begin{proof}
(i) is only a restatement of \cite[(3.10)]{Oda:HC}.
To show (ii) let $(\cdot,\cdot)$ be a $K$-invariant inner product of $V$.
Then the proof of \cite[Lemma~3.3]{Oda:HC} shows
$V^M_\single \perp V^M_\double$ with respect to $(\cdot,\cdot)$.
If $v_1\in V^M_\single$ and $v_2\in V^M_\double$ then
\begin{align*}
\bigl(v_1, p^V\bigl((X_\alpha+\theta X_\alpha)^2 v_2\bigr)\bigr)
&=\bigl(v_1, (X_\alpha+\theta X_\alpha)^2 v_2\bigr)\\
&=\bigl((X_\alpha+\theta X_\alpha)^2 v_1,  v_2\bigr)\\
&=|\alpha|^2B(X_\alpha,\theta X_\alpha)
\bigl((1-s_\alpha)v_1, v_2\bigr)\\
&=0.
\end{align*}
Hence $p^V\bigl((X_\alpha+\theta X_\alpha)^2 v_2\bigr)
\in V^M\cap (V^M_\single)^\perp=V^M_\double$.
\end{proof}
\section{Dunkl operators}\label{sec:Dunkl}
Put $R:=2\Sigma$, $R^+:=2\Sigma^+$, $R_1:=R \setminus 2R$, and $R_1^+:=R_1\cap R^+$.
Then $R$ and $R_1$ are root systems sharing the same Weyl group $W$ with $\Sigma$.
In \cite{Dun} Dunkl introduces the following operators:
\begin{defn}[Dunkl operators]\label{defn:Dunkl}
Suppose $\mathbf k : W\backslash R_1\rightarrow \CC$
is a \emph{multiplicity function}, namely, a $\CC$-valued function on $R_1$ which is
constant on each $W$-orbit.
For $\xi\in\Liea$ define $\mathscr D_{\mathbf k}(\xi)\in\End_\CC C^\infty(\Liea)$ by
\[
\mathscr D_{\mathbf k}(\xi)=
\der(\xi)+\sum_{\alpha\in R_1^+}
\mathbf k(\alpha)
\frac{\alpha(\xi)}{\alpha}(1-s_\alpha)
\]
where $\der(\xi)$ is the $\xi$-directional derivative operator.
When $\mathbf k$ is a special multiplicity function $\mathbf m_0 : W\backslash R_1 \rightarrow\CC$ specified by
\begin{equation*}
\mathbf m_0(\alpha)=
\begin{cases}
\frac12 \dim \Lieg_{\alpha/2} & \text{if }2\alpha\notin R,\\
\frac12 \bigl(\dim \Lieg_{\alpha/2}+ \dim \Lieg_\alpha\bigr) & \text{if }2\alpha\in R,
\end{cases}
\end{equation*}
we use the brief symbol $\mathscr D$
for $\mathscr D_{\mathbf m_0}$.
\end{defn}
\begin{rem}
There is no significant meaning in using $R$ or $R_1$ instead of $\Sigma$.
We do so only for the compatibility with the case of Cherednik operators (Definition \ref{defn:Chered}).
\end{rem}
The following are well-known properties of
Dunkl operators:
\begin{prop}[\cite{Dun}]\label{prop:Dunkl}
{\normalfont (i)}
For $\xi,\eta\in\Liea$,
$\mathscr D_{\mathbf k}(\xi)\mathscr D_{\mathbf k}(\eta)
=\mathscr D_{\mathbf k}(\eta)\mathscr D_{\mathbf k}(\xi)$.

\noindent
{\normalfont (ii)}
For $\xi\in\Liea$ and $w\in W$,
$w \mathscr D_{\mathbf k}(\xi) w^{-1}= \mathscr D_{\mathbf k}(w\xi)$.

\noindent
{\normalfont (iii)}
Let $\{\xi_1,\ldots,\xi_\ell\}$ be an orthonormal basis of
the Euclidean space $(\Liea, B(\cdot,\cdot))$
and put $L_\Liea=\sum_{i=1}^\ell \xi_i^2\in S(\Liea_\CC)$.
Then
\[
\mathscr D_{\mathbf k}(L_\Liea)
=
\der(L_\Liea)+
\sum_{\alpha\in R^+_1} \mathbf k(\alpha)
\Bigl(
\frac2{\alpha}\der(H_\alpha)-\frac{|\alpha|^2}{\alpha^2}(1-s_\alpha)
\Bigr).
\]
\end{prop}
Now let us state a generalization of \eqref{eq:rad3}
for $f\notin C^\infty(\Lies)^K$.
\begin{thm}[the radial part formula]\label{thm:radA}
Suppose $V\in \Km$,
$\Delta \in S(\Lies_\CC)^K$
and $\Phi\in\Hom_K(V,C^\infty(\Lies))$.

\noindent{\normalfont (i)}
We have
\begin{align*}
\der(\Delta) \circ \Phi &\in \Hom_K(V,C^\infty(\Lies)),\\
\mathscr D\bigl(\gamma_0(\Delta)\bigr)\circ \Gamma_0^V(\Phi)
&\in \Hom_W(V^M,C^\infty(\Liea)),\\ 
\mathscr D\bigl(\gamma_0(\Delta)\bigr)\circ \tilde\Gamma_0^V(\Phi)
&\in \Hom_W(V^M_\single,C^\infty(\Liea)),
\end{align*}
and it holds that
\begin{equation}\label{eq:rad4}
\tilde \Gamma_0^V\bigl(
\der(\Delta) \circ \Phi
\bigr)
=\mathscr D\bigl(\gamma_0(\Delta)\bigr)\circ \tilde \Gamma_0^V(\Phi).
\end{equation}

\noindent{\normalfont (ii)}
If $\Phi\in\Hom_K^\ttt(V,C^\infty(\Lies))$
then
$\der(\Delta) \circ \Phi\in\Hom_K^\ttt(V,C^\infty(\Lies))$.
Hence for such $\Phi$
\begin{equation}\label{eq:rad5}
\Gamma_0^V\bigl(
\der(\Delta) \circ \Phi
\bigr)
=\mathscr D\bigl(\gamma_0(\Delta)\bigr)\circ \Gamma_0^V(\Phi).
\end{equation}

\noindent{\normalfont (iii)}
If $V\in\Ksp$ then \eqref{eq:rad5} always holds.
\end{thm}
On the other hand, as a generalization of \eqref{eq:rad3}
for $\Delta \notin S(\Lies_\CC)^K$ we have
\begin{thm}[the radial part formula]\label{thm:radB}
Suppose $V\in \Km$,
$\Phi\in\Hom_K(V,S(\Lies_\CC))$
and $f \in C^\infty(\Lies)^K$.
Then $\Gamma_0^V(\Phi)\in \Hom_W(V^M, S(\Liea_\CC))$
is naturally defined by the identifications $S(\Lies_\CC)\simeq\mathscr P(\Lies)$
and $S(\Liea_\CC)\simeq\mathscr P(\Liea)$.
Let $\der(\Phi)f$ denote the map
\[V \ni v \longmapsto \der(\Phi[v])f \in C^\infty(\Lies).\]
Let $\mathscr D\bigl(\Gamma_0^V(\Phi)\bigr) \gamma_0(f)$ denote
the map
\[V^M \ni v \longmapsto \mathscr D\bigl(\Gamma_0^V(\Phi)[v]\bigr) \gamma_0(f) \in C^\infty(\Liea).\]

\noindent{\normalfont (i)}
We have
\begin{align*}
\der(\Phi)f &\in \Hom_K(V, C^\infty(\Lies)),\\
\mathscr D\bigl(\Gamma_0^V(\Phi)\bigr) \gamma_0(f)
&\in \Hom_W(V^M, C^\infty(\Liea)),
\end{align*}
and it holds that
\begin{equation}\label{eq:rad6}
\Gamma_0^V\bigl(\der(\Phi)f\bigr)[v]
=
\mathscr D\bigl(\Gamma_0^V(\Phi)\bigr) \gamma_0(f) [v].
\quad\text{for }v\in V^M_\single.
\end{equation}

\noindent{\normalfont (ii)}
If $\Phi\in\Hom_K^\ttt(V,S(\Lies_\CC))
\,\bigl(\,:=
\Hom_K^\ttt(V,\mathscr P(\Lies))\,\bigr)$,
then
$\der(\Phi)f\in\Hom_K^\ttt(V,C^\infty(\Lies))$.
Hence for such $\Phi$
\begin{equation}\label{eq:rad7}
\Gamma_0^V\bigl(\der(\Phi)f\bigr)
=
\mathscr D\bigl(\Gamma_0^V(\Phi)\bigr) \gamma_0(f).
\end{equation}

\noindent{\normalfont (iii)}
If $V\in\Ksp$ then \eqref{eq:rad7} always holds.
\end{thm}
From now on we shall prove these theorems by following the method of de Jeu \cite{Je},
which uses only three simple lemmas.
\begin{lem}\label{lem:DunklL}
Let $\{X_1,\ldots,X_{\dim\Lies}\}$ be an orthonormal basis of
the Euclidean space $(\Lies, B(\cdot,\cdot))$
and put $L_\Lies=\sum_{i=1}^{\dim\Lies} X_i^2\in S(\Lies_\CC)$.
Then {\normalfont Theorem \ref{thm:radA}} for $\Delta=L_\Lies$ is true.
\end{lem}
\begin{proof}
Suppose $V\in\Km$ and $\Phi\in\Hom_K(V,C^\infty(\Lies))$.
The first assertion of Theorem \ref{thm:radA} (i) is trivial.
Although \eqref{eq:rad4} for $\Delta=L_\Lies$ is equivalent
to \cite[Lemma 3.10]{Oda:HC}, we recall the outline of its proof.
For each $\alpha\in\Sigma^+$
choose an orthonormal basis $\bigl\{X_\alpha^{(1)},\ldots,X_\alpha^{(\dim\Lieg_\alpha)}\bigr\}$
of the Euclidean space $(\Lieg_\alpha,-B(\cdot,\theta\cdot))$.
Let $v\in V^M$ and $H\in\Liea$.
Then a direct calculation (cf. the proof of \cite[Lemma 3.10]{Oda:HC}) leads to
\begin{equation}\label{eq:Lrad1}
\begin{aligned}
\der(L_\Lies)\Phi[v](H)
&=
\der(L_\Liea)\Phi[v](H)
+\sum_{\alpha\in R_1^+}
\frac{2\mathbf m_0(\alpha)}{\alpha(H)}\der(H_\alpha)\Phi[v](H)\\
&\qquad+\sum_{\alpha\in\Sigma^+}\sum_{i=1}^{\dim\Lieg_\alpha}
\frac{\Phi\Bigl[p^V\Bigl(\bigl(X_\alpha^{(i)}+\theta X_\alpha^{(i)}\bigr)^2v\Bigr)\Bigr](H)}{2\alpha(H)^2}.
\end{aligned}
\end{equation}
If $v\in V^M_\single$, then Lemma~\ref{lem:qsp} (i) reduces
the right-hand side of \eqref{eq:Lrad1} to
\[
\der(L_\Liea)\Phi[v](H)
+\sum_{\alpha\in R_1^+}
\mathbf m_0(\alpha)\Bigl(
\frac{2}{\alpha(H)}\der(H_\alpha)\Phi[v](H)
-|\alpha|^2\frac{\Phi[v](H)-\Phi[v](s_\alpha H)}{\alpha(H)^2}
\Bigr),
\]
which equals $\mathscr D(L_\Liea)\Phi[v](H)$ by Proposition~\ref{prop:Dunkl} (iii).
Since $\gamma_0(L_\Lies)=L_\Liea$, we get \eqref{eq:rad4} for $\Delta=L_\Lies$.

In order to show (ii) for $\Delta=L_\Lies$,
suppose $\Gamma_0^V(\Phi)\bigl[V^M_\double\bigr]=\{0\}$.
Then by virtue of Lemma~\ref{lem:qsp} (ii)
the right-hand side of \eqref{eq:Lrad1} is $0$ for $v\in V^M_\double$.
This means $\Gamma_0^V(\der(L_\Lies)\circ\Phi)\bigl[V^M_\double\bigr]=\{0\}$,
proving (ii) for $\Delta=L_\Lies$.

Finally (iii) is immediate from (i) or (ii) (cf.~Remark~\ref{rem:22} (i)).
\end{proof}
\begin{lem}[\cite{He:Dunkl, Je}]\label{lem:L1}
Suppose $p$ is a homogeneous element of $S(\Liea_\CC)$ with degree $d$.
Identifying $p$ with a polynomial function on $\Liea$,
let $\underline p\in \End_\CC C^\infty(\Liea)$ denote
the multiplication operator by $p$.
Then in the algebra $\End_\CC C^\infty(\Liea)$ the following identity holds:
\[
\frac1{d!} \biggl( \ad\frac{\mathscr D_{\mathbf k}(L_\Liea)}2
\biggr)^d \underline p
=
\mathscr D_{\mathbf k}(p).
\]
\end{lem}
\begin{lem}\label{lem:L2}
Suppose $P$ is a homogeneous element of $S(\Lies_\CC)$ with degree $d$.
Identifying $P$ with a polynomial function on $\Lies$,
let $\underline P\in \End_\CC C^\infty(\Lies)$ denote
the multiplication operator by $P$.
Then in the algebra $\End_\CC C^\infty(\Lies)$ the following identity holds:
\[
\frac1{d!} \biggl( \ad\frac{\der(L_\Lies)}2
\biggr)^d \underline P
=
\der(P).
\]
\end{lem}
\begin{proof}
By an elementary calculation.
\end{proof}
\begin{proof}[Proof of {\normalfont Theorem~\ref{thm:radA}}]
Suppose $V\in\Km$,
$\Delta\in S(\Lies_\CC)^K$,
and $\Phi\in\Hom_K(V,C^\infty(\Lies))$.
We may assume $\Delta$ is homogeneous with degree $d$.
By Lemma~\ref{lem:L2} we have for $v\in V$
\begin{equation}\label{eq:derDelta}
\der(\Delta) \Phi[v]
=\frac1{d!} \biggl(\biggl(\ad\frac{\der(L_\Lies)}2\biggr)^d\underline\Delta\biggr)\Phi[v].
\end{equation}
Now let $v\in V^M_\single$.
Then restricting both sides of \eqref{eq:derDelta} to $\Liea$ we have
\begin{align*}
\gamma_0\bigl(
\der(\Delta) \Phi[v]
\bigr)
&=\frac1{d!}
\biggl(\biggl(\ad\frac{\mathscr D(L_\Liea)}2\biggr)^d
\underline{\gamma_0(\Delta)}\biggr)
\gamma_0\bigl(
\Phi[v]
\bigr)
&(\because \text{Lemma~\ref{lem:DunklL}})\\
&=\mathscr D\bigl(\gamma_0(\Delta)\bigr)\gamma_0\bigl(
\Phi[v]
\bigr),
&(\because \text{Lemma~\ref{lem:L1}})
\end{align*}
which proves Theorem \ref{thm:radA} (i) and hence (iii).

In general, if $\Psi\in\Hom_K^\ttt(V,C^\infty(\Lies))$,
then $\der(L_\Lies)\circ \Psi\in\Hom_K^\ttt(V,C^\infty(\Lies))$ by Lemma \ref{lem:DunklL}
and
clearly $\underline \Delta\circ \Psi\in\Hom_K^\ttt(V,C^\infty(\Lies))$.
Hence it follows from
\eqref{eq:derDelta} 
$\Phi\in\Hom_K^\ttt(V,C^\infty(\Lies))$
implies $\der(\Delta)\circ\Phi\in\Hom_K^\ttt(V,C^\infty(\Lies))$.
We thus get Theorem \ref{thm:radA} (ii).
\end{proof}
\begin{proof}[Proof of {\normalfont Theorem \ref{thm:radB}}]
Suppose $V\in \Km$,
$\Phi\in\Hom_K(V,S(\Lies_\CC))$,
and $f \in C^\infty(\Lies)^K$.
We may assume
$\Phi[v]$ for any $v\in V$ is a homogeneous element with a constant degree $d$.
The first assertion of Theorem \ref{thm:radB} (i) is clear.
By Lemma \ref{lem:L2} we have for $v\in V$
\begin{equation}\label{eq:derPhi}
\begin{aligned}
\der\bigl(\Phi[v]\bigr)f
&=\frac1{d!} \biggl(\biggl(\ad\frac{\der(L_\Lies)}2\biggr)^d\underline{\Phi[v]}\biggr)f\\
&=\sum_{d_1=0}^d\frac{(-1)^{d_1}}{(d-d_1)!d_1!}
\biggl(\frac{\der(L_\Lies)}2\biggr)^{d-d_1}
\!\!\!\underline{\Phi[v]}
\biggl(\frac{\der(L_\Lies)}2\biggr)^{d_1}
\!\!f.
\end{aligned}
\end{equation}
For $d_1=0,\ldots,d$ define $\Psi_{d_1}\in\Hom_K(V,C^\infty(\Lies))$ by
\[
V\ni v\longmapsto \underline{\Phi[v]}
\biggl(\frac{\der(L_\Lies)}2\biggr)^{d_1}\!\!f.
\]
Then by Lemma~\ref{lem:DunklL} with $V=\CC_\triv$ (the trivial $K$-type) we have
\begin{equation}\label{eq:resPsi}
\Gamma_0^V(\Psi_{d_1})[v]
=
\underline{\Gamma_0^V(\Phi)[v]}
\biggl(\frac{\mathscr D(L_\Liea)}2\biggr)^{d_1}\!\!\gamma_0(f)
\quad\text{for }v\in V^M.
\end{equation}
Now if $v\in V^M_\single$ then the restriction of \eqref{eq:derPhi} is calculated to be
\begin{align*}
\gamma_0\bigl(
\der\bigl(&\Phi[v]\bigr)f
\bigr)\\
&=\sum_{d_1=0}^d\frac{(-1)^{d_1}}{(d-d_1)!d_1!}
\biggl(\frac{\mathscr D(L_\Liea)}2\biggr)^{d-d_1}
\!\!\Gamma_0^V(\Psi_{d_1})[v]
&&(\because\text{Lemma~\ref{lem:DunklL}})\\
&=\sum_{d_1=0}^d\frac{(-1)^{d_1}}{(d-d_1)!d_1!}
\biggl(\frac{\mathscr D(L_\Liea)}2\biggr)^{d-d_1}
\!\!\underline{\Gamma_0^V(\Phi)[v]}
\biggl(\frac{\mathscr D(L_\Liea)}2\biggr)^{d_1}\!\!\gamma_0(f)
&&(\because\eqref{eq:resPsi})\\
&=\frac1{d!} \biggl(\biggl(\ad\frac{\mathscr D(L_\Liea)}2\biggr)^d
\underline{\Gamma_0^V(\Phi)[v]}\biggr)\gamma_0(f)\\
&=
\mathscr D\bigl(
\Gamma_0^V(\Phi)[v]
\bigr)\gamma_0(f).
&&(\because\text{Lemma~\ref{lem:L1}})
\end{align*}
It proves \eqref{eq:rad6} and therefore (iii).

In order to show (ii)
suppose $\Phi\in\Hom_K^\ttt(V,S(\Lies_\CC))$.
Then by \eqref{eq:resPsi} we have $\Psi_{d_1}\in\Hom_K^\ttt(V,C^\infty(\Lies))$
for $d_1=0,\ldots,d$.
Since
\[
\der(\Phi)f=\sum_{d_1=0}^d \frac{(-1)^{d_1}}{(d-d_1)!d_1!}
\biggl(\frac{\der(L_\Lies)}2\biggr)^{d-d_1}
\!\!\!\circ\Psi_{d_1},
\]
it follows from Lemma \ref{lem:DunklL} that $\der(\Phi)f\in\Hom_K^\ttt(V,C^\infty(\Lies))$.
We thus get (ii).
\end{proof}

\section{Graded Hecke algebras and Cherednik operators}\label{sec:GHA}
Let $\Pi$ be the system of simple roots in $R_1=2\Sigma\setminus4\Sigma$
corresponding to the positive system $R_1^+=2\Sigma^+\setminus4\Sigma^+$.
\begin{defn}[graded Heck algebras \cite{Lu}]\label{defn:H}
Let $\mathbf k : W\backslash R_1\rightarrow\CC$ be a multiplicity function.
Then there exists uniquely (up to equivalence) an algebra $\mathbf H_{\mathbf k}$ over $\CC$
with the following properties:\smallskip

\noindent{\normalfont (i)}
$\mathbf H_{\mathbf k}\simeq S(\Liea_\CC)\otimes\CC W$ as a $\CC$-linear space;

\noindent{\normalfont (ii)}
The maps $S(\Liea_\CC)\rightarrow \mathbf H_{\mathbf k}, \varphi\mapsto \varphi\otimes1$ and
 $\CC W\rightarrow \mathbf H_{\mathbf k}, w\mapsto1\otimes w$ are algebra homomorphisms;

\noindent{\normalfont (iii)}
$(\varphi\otimes1)\cdot(1\otimes w)=\varphi\otimes w$ for any $\varphi\in S(\Liea_\CC)$ and $w\in W$;

\noindent{\normalfont (iv)}
$(1\otimes s_\alpha)\cdot(\xi\otimes1)
 =s_\alpha(\xi)\otimes s_\alpha-\mathbf k(\alpha)\,\alpha(\xi)$
for any $\alpha\in\Pi$ and $\xi\in\Liea_\CC$.
\smallskip

We call $\mathbf H_{\mathbf k}$ the graded Hecke algebra
associated to the data $(\Liea, \Pi, \mathbf k)$.
\end{defn}
By (ii) we identify $S(\Liea_\CC)$ and $\CC W$
with subalgebras of $\mathbf H_{\mathbf k}$.
Then (iv) is simply written as
\begin{equation}\label{eq:Hrel}
s_\alpha\cdot\xi
 =s_\alpha(\xi)\cdot s_\alpha
 - \mathbf k(\alpha)\,\alpha(\xi)\quad\forall\alpha\in\Pi\ \forall\xi\in\Liea_\CC.
\end{equation}
It is well known that the center of $\mathbf H_{\mathbf k}$
equals $S(\Liea_\CC)^W$ (cf.~\cite[Theorem 6.5]{Lu}).
Define the multiplicity function $\mathbf m_1 : W\backslash R_1\rightarrow \CC$ by
\begin{equation*}
\mathbf m_1(\alpha)=\begin{cases}
\frac12 \dim \Lieg_{\alpha/2} & \text{if }2\alpha\notin R,\\
\frac12 \dim \Lieg_{\alpha/2}+ \dim \Lieg_\alpha & \text{if }2\alpha\in R,
\end{cases}
\end{equation*}
and put $\mathbf H=\mathbf H_{\mathbf m_1}$.
We consider $\mathbf H$ is a special graded Hecke algebra attached to the Iwasawa decomposition $G=NAK$.

Now recall $R=2\Sigma$ and $R^+=2\Sigma^+$.
By \eqref{eq:W-isom} we identify $C^\infty(\Liea)$ with $C^\infty(A)$,
so that $\End_\CC C^\infty(A)\ni \der(\xi)$ $(\xi\in\Liea_\CC) $.
For $\mu\in \Liea_\CC^*$ let $e^{\mu}\in C^\infty(A)$ denote
the function $a \mapsto \exp{\mu(\log a)}$.
\begin{defn}[Cherednik operators \cite{Ch1}]\label{defn:Chered}
Suppose $\mathbf k : W\backslash R\rightarrow \CC$ is a multiplicity function
and put $\rho_{\mathbf k}=\frac12\sum_{\alpha\in R^+}\mathbf k(\alpha) \alpha$.
For $\xi\in\Liea$ define $\mathscr T_{\mathbf k}(\xi)\in\End_\CC C^\infty(A)$ by
\[
\mathscr T_{\mathbf k}(\xi)=
\der(\xi)+\sum_{\alpha\in R^+}
\mathbf k(\alpha)\frac{\alpha(\xi)}{1-e^{-\alpha}}(1-s_\alpha)
-\rho_{\mathbf k}(\xi).
\]
If $\mathbf k$ equals
the multiplicity function $\mathbf m : W\backslash R \rightarrow\CC$ defined by
\begin{equation*}
\mathbf m(\alpha)=\frac12\dim \Lieg_{\alpha/2},
\end{equation*}
then we use the brief symbol $\mathscr T$
for $\mathscr T_{\mathbf m}$.
\end{defn}
The following are fundamental properties of Cherednik operators:
\begin{prop}\label{prop:Cherednik}
{\normalfont (i)}
For $\xi,\eta\in\Liea$,
$\mathscr T_{\mathbf k}(\xi)\mathscr T_{\mathbf k}(\eta)
=\mathscr T_{\mathbf k}(\eta)\mathscr T_{\mathbf k}(\xi)$.

\noindent
{\normalfont (ii)}
Let\/ $\mathbf k_1 : W\backslash R_1\rightarrow\CC$ be the multiplicity function defined by
\[
\mathbf k_1(\alpha)=\begin{cases}
\mathbf k(\alpha) & \text{if }2\alpha\notin R,\\
\mathbf k(\alpha)+2\mathbf k(2\alpha) & \text{if }2\alpha\in R.
\end{cases}
\]
Then
$s_\alpha \mathscr T_{\mathbf k}(\xi)= \mathscr T_{\mathbf k}(s_\alpha(\xi))s_\alpha-\mathbf k_1(\alpha)\alpha(\xi)$
for $\xi\in\Liea$ and $\alpha\in \Pi$.
Hence
$\mathscr T_{\mathbf k} : \Liea\rightarrow \End_\CC C^\infty(A)$
uniquely extends to an algebra homomorphism
of\/ $\mathbf H_{\mathbf k_1}$ into $\End_\CC C^\infty(A)$.

\noindent
{\normalfont (iii)}
Let $L_\Liea\in S(\Liea_\CC)$ be as in {\normalfont Proposition~\ref{prop:Dunkl} (iii)}.
Then
\begin{equation}\label{eq:ChredLap}
\mathscr T_{\mathbf k}(L_\Liea)
=
\der(L_\Liea)+
\sum_{\alpha\in R^+} \mathbf k(\alpha)
\Bigl(
\coth \frac{\alpha}2 \der(H_\alpha)-\frac{|\alpha|^2}{4\sinh^2 \frac{\alpha}2}(1-s_\alpha)
\Bigr)
+B(H_{\rho_{\mathbf k}},H_{\rho_{\mathbf k}}).
\end{equation}
\end{prop}
\begin{proof}
Proposition~\ref{prop:Cherednik} (i), (ii) are given in \cite{Ch1}
(see also \cite{Op:Cherednik}).
(iii) is calculated in \cite{Shapira}.
\end{proof}
\begin{rem}
If $\mathbf k=\mathbf m$ then
$\rho_{\mathbf k}=\rho$, $\mathbf k_1=\mathbf m_1$ and $\mathbf H_{\mathbf k_1}=\mathbf H$.
\end{rem}

\section{Radial part formula, I}\label{sec:radI}
The next theorem is a generalization of \eqref{eq:rad1}
to some cases where $f$ is no longer $K$-invariant.
The tangential counterpart of this theorem is Theorem~\ref{thm:radA}.
\begin{thm}[the radial part formula]\label{thm:radC}
Suppose $V\in \Km$,
$\Delta \in U(\Lieg_\CC)^K$,
and $\Phi\in\Hom_K(V,C^\infty(G/K))$.

\noindent{\normalfont (i)}
We have
\begin{align*}
r(\Delta) \circ \Phi &\in \Hom_K(V,C^\infty(G/K)),\\
\mathscr T\bigl(\gamma(\Delta)\bigr)\circ \Gamma_0^V(\Phi)
&\in \Hom_W(V^M,C^\infty(A)),\\
\mathscr T\bigl(\gamma(\Delta)\bigr)\circ \tilde\Gamma_0^V(\Phi)
&\in \Hom_W(V^M_\single,C^\infty(A)),
\end{align*}
and it holds that
\begin{equation}\label{eq:rad8}
\tilde\Gamma_0^V\bigl(
r(\Delta) \circ \Phi
\bigr)
=\mathscr T\bigl(\gamma(\Delta)\bigr)\circ \tilde\Gamma_0^V(\Phi).
\end{equation}

\noindent{\normalfont (ii)}
If $\Phi\in\Hom_K^\ttt(V,C^\infty(G/K))$
then
$r(\Delta) \circ \Phi
\in\Hom_K^\ttt(V,C^\infty(G/K)).$
Hence for such $\Phi$
\begin{equation}\label{eq:rad9}
\Gamma_0^V\bigl(
r(\Delta) \circ \Phi
\bigr)
=\mathscr T\bigl(\gamma(\Delta)\bigr)\circ \Gamma_0^V(\Phi).
\end{equation}

\noindent{\normalfont (iii)}
If $V\in\Ksp$ then \eqref{eq:rad9} always holds.
\end{thm}
\begin{cor}\label{cor:spec-cor}
Suppose $\lambda\in\Liea_\CC^*$.
For each $V\in\Km$ put
\[
\Hom_K^{\ttt}(V,\mathscr A(G/K,\lambda))
=
\Hom_K(V,\mathscr A(G/K,\lambda))
\cap
\Hom_K^{\ttt}(V,\mathscr C^\infty(G/K)).
\]
Then $\tilde\Gamma^V_0$ 
induces a linear bijection
\begin{equation}\label{eq:spec-cor2}
\tilde\Gamma^V_0:
\Hom_K^\ttt(V, \mathscr A(G/K,\lambda))\simarrow
\Hom_W(V^M_\single, \mathscr A(A, \lambda)).
\end{equation}
In particular, if $V\in\Ksp$ then \eqref{eq:spec-cor} holds. 
\end{cor}
\begin{proof}
By Remark \ref{rem:22} (ii), \eqref{eq:HC-isom}, and the above theorem.
\end{proof}

The rest of this section is devoted to the proof of Theorem \ref{thm:radC}.
Let $\mathscr R$ be the algebra of those power series of $\{e^{\alpha};\,\alpha\in\Sigma^+\}$
which absolutely converge on
\[A_-:=\{e^H;\,H\in\Liea\text{ with }\alpha(H)<0\text{ for any }\alpha\in\Sigma^+\}.\]
Each element $c\in\mathscr R$ is uniquely expanded as
\[
c=\sum_{\lambda\in \ZZ_{\ge0}\Sigma^+} c_\lambda e^\lambda\quad\text{with }c_\lambda\in\CC.
\]
Using this expansion we put
\[\spec c=\{\lambda;\,c_\lambda\ne0\}.\]
This is a subset of $\ZZ_{\ge0}\Sigma^+$.
In the argument below, the maximal ideal
\[
\mathscr M:=\{c\in\mathscr R;\, \spec c\not\ni 0\}
\]
 has important roles.

Suppose $x\in A_-$ and put $\bar x=xK\in G/K$.
Let $C^\infty_{wx}$ ($w\in W$) be the space of germs of $C^\infty$ functions at $wx$, that is
\[C^\infty_{wx}=\varinjlim_{\mathscr U;\,wx\in\mathscr U\subset A,\text{ open}}C^\infty(\mathscr U).\]
If we put $C^\infty(Wx)=\bigoplus_{w\in W}C^\infty_{wx}$,
then $W$ naturally acts on $C^\infty(Wx)$.
Since $x$ is a regular point of $A$,
for any $W$-module $U$
\begin{equation}\label{eq:restId}
\Hom_W(U, C^\infty(Wx))\simarrow \Hom_\CC(U, C^\infty_{x})
\end{equation}
by restriction.
Let us think of $\mathscr R\otimes \partial(S(\Liea_\CC))\otimes \End_\CC U^*$
as the tensor product of the ring $\mathscr R\otimes \partial(S(\Liea_\CC))$
of differential operators on $A_-$ with coefficients in $\mathscr R$ and the endomorphism ring $\End_\CC U^*$ for the dual space $U^*$ of $U$.
(In this paper $\otimes$ is taken over $\CC$ unless otherwise specified.)
Then $\mathscr R\otimes \partial(S(\Liea_\CC))\otimes \End_\CC U^*$ acts on
$\varphi\in\Hom_\CC(U, C^\infty_{x})$ by
\[
((D\otimes\tau)\varphi)[u]=D(\varphi[\,\trans\tau(u)])
\quad\text{with }D\in \mathscr R\otimes \partial(S(\Liea_\CC)),
\tau\in \End_\CC U^*, u\in U,
\]
where $\trans\tau \in \End_\CC U$ is the transpose of $\tau$.

Now, if we put
\[
C^\infty({K\bar x}):=\varinjlim_{\mathscr V\supset K\bar x\text{; open}}C^\infty(\mathscr V),
\]
then $\ell(U(\Lieg_\CC))$, $\ell(K)$ and $r\bigl(U(\Lieg_\CC)^K\bigr)$ naturally act on it. 
For any $V\in \Km$
define a localized restriction map 
\begin{align*}
\Gamma^V_{0,Wx}:\,& \Hom_K(V, C^\infty({K\bar x}))
\longrightarrow
\Hom_W(V^M, C^\infty(Wx))\,\bigl(\,\simeq \Hom_\CC(V^M, C^\infty_x)\,\bigr)
\,;\\
&\Phi\longmapsto
\bigl(
\varphi:
V^M\hookrightarrow
V
\xrightarrow{\Phi} C^\infty({K\bar x})
\xrightarrow{\gamma_0} C^\infty(Wx)\ (\text{ or }C^\infty_x\,)\,
\bigr).
\end{align*}
When we consider the target space is $\Hom_\CC(V^M, C^\infty_x)$,
we use the symbol $\Gamma^V_{0,x}$ instead of $\Gamma^V_{0,Wx}$.
\begin{lem}\label{lem:localCh}
$\Gamma^V_{0,x}$ (or equivalently $\Gamma^V_{0,Wx}$) is a bijection.
\end{lem}
\begin{proof}
Let $p^V:V\to V^M$ be as in Theorem \ref{thm:Ch} (ii).
For any open neighbourhood $\mathscr V\subset G/K$ of $K\bar x$,
we can take a sufficiently small open neighbourhood
$\mathscr U\subset A_-$ of $x$ so that
$K\mathscr U\simeq K/M\times \mathscr U\subset \mathscr V$.
If $\Phi\in \Hom_K(V, C^\infty(K\mathscr U))$
then for $v\in V$ and $(kM, y)\in K/M\times \mathscr U$ we have
\[
\Phi[v](ky)=\Phi[k^{-1}v](y)=\Phi[p^V(k^{-1}v)](y)=\Gamma^V_{0,\mathscr U}(\Phi)[p^V(k^{-1}v)](y).
\]
This shows the injectivity of $\Gamma^V_{0,x}$.
Conversely, for any $\varphi\in \Hom_\CC(V^M, C^\infty(\mathscr U))$,
\[
\Phi[v](ky):=\varphi[p^V(k^{-1}v)](y)
\qquad\bigl(v\in V,(kM, y)\in K/M\times\mathscr U \bigr)
\]
defines its lift. This shows the surjectivity.
\end{proof}
\begin{lem}\label{lem:abstRad}
Suppose $V\in \Km$. For any $\Delta\in U(\Lieg_\CC)^K$
there exists a unique $E \in \mathscr M\otimes \partial(S(\Liea_\CC))\otimes \End_\CC (V^M)^*$
such that
\[
\Gamma^V_{0,x}\bigl(r(\Delta)\circ\Phi\bigr)
=
\bigl(\partial(\gamma(\Delta)(\cdot-\rho))+E\bigr)\,\Gamma^V_{0,x}(\Phi)
\quad\text{for any }\Phi\in\Hom_K(V, C^\infty({K\bar x})).
\]
\end{lem}
\begin{rem}\label{rem:matform}
{\normalfont (i)}
$\mathscr M\otimes \partial(S(\Liea_\CC))$ is an ideal of $\mathscr R\otimes \partial(S(\Liea_\CC))$.

\noindent
{\normalfont (ii)}
$\gamma(\Delta)(\cdot-\rho)$ is nothing but the second summand
of $\Delta$ in the direct sum decomposition
$U(\Lieg_\CC)=\bigl(\Lien_\CC U(\Lieg_\CC)+U(\Lieg_\CC)\Liek_\CC
\bigr)\oplus U(\Liea_\CC)$.

\noindent
{\normalfont (iii)}
Let $\{v_1,\ldots,v_{m'}\}$ and $\{v_{m'+1},\ldots,v_m\}$ be bases of $V^M_\single$
and $V^M_\double$ respectively.
Then $\{v_1,\ldots,v_{m}\}$ is a basis of $V^M$.
Let $\{v^*_1,\ldots,v^*_{m}\}\subset (V^M)^*$ be the dual basis of $\{v_1,\ldots,v_{m}\}$.
With repect to these bases, we can express any element of $\mathscr R\otimes \partial(S(\Liea_\CC))\otimes \End_\CC (V^M)^*$ in a matrix form.
More precisely, the correspondence
\[
\sum_{i,j}D_{ij}\otimes(v_i^*\otimes v_j)\longmapsto (D_{ij})
\]
gives an algebra isomorphism $\mathscr R\otimes \partial(S(\Liea_\CC))\otimes \End_\CC (V^M)^*\simarrow \Mat(m, m;\mathscr R\otimes \partial(S(\Liea_\CC)))$.
Moreover, if we identify $\varphi \in \Hom_\CC(V^M, C^\infty_x)$
with a column vector $\trans(\varphi[v_1],\ldots,\varphi[v_m])\in (C^\infty_x)^m$,
then the action of $\mathscr R\otimes \partial(S(\Liea_\CC))\otimes \End_\CC (V^M)^*$
reduces to the left multiplication.
In the proof of Theorem \ref{thm:radC}
we use the following matrix expression of $E$:
\begin{equation*}
E=\begin{pmatrix}
E_\single & P \\
Q & E_\double
\end{pmatrix}.
\end{equation*}
Here the matrix is divided into four blocks according to the 
division of the basis
\[\{v_1,\ldots,v_{m}\}=\{v_1,\ldots,v_{m'}\} \sqcup \{v_{m'+1},\ldots,v_m\}.\]
\end{rem}
\begin{proof}[Proof of {\normalfont Lemma~\ref{lem:abstRad}}]
Suppose $X_\alpha\in\Lieg_\alpha$ ($\alpha\in\Sigma^+$).
For any $y=e^H \in A_-$
we have
\[\begin{aligned}
\Ad(y^{-1})(X_\alpha+\theta X_\alpha)
&=e^{-\alpha(H)}X_\alpha + e^{\alpha(H)}\theta X_\alpha\\
&=(e^{-\alpha(H)}-e^{\alpha(H)})X_\alpha + e^{\alpha(H)}(X_\alpha+\theta X_\alpha)\end{aligned}\]
and therefore
\[
X_\alpha=\frac{e^{\alpha(H)}}{1-e^{2\alpha(H)}}\Ad(y^{-1})(X_\alpha+\theta X_\alpha)
-\frac{e^{2\alpha(H)}}{1-e^{2\alpha(H)}}(X_\alpha+\theta X_\alpha).
\]
Hence if $D\in U(\Lieg_\CC)$ is given,
we can take $c_i\in\mathscr M$, $D'_i\in U(\Liek_\CC)$,
and $D''_i\in S(\Liea_\CC)$ ($i=1,\ldots,q$) so that it holds that
\begin{equation}\label{eq:radialmod}
D\equiv\gamma(D)(\cdot-\rho)
+\sum_{i=1}^{q}c_i(y)\Ad(y^{-1})(D'_i)D''_i
\pmod{U(\Lieg_\CC)\Liek_\CC}
\end{equation}
for any $y\in A_-$.
(This is shown by induction on the order of $D$.
A more detailed argument can be found in the proof of \cite[Ch.\,II, Proposition~5.23]{Hel4}.)
Applying \eqref{eq:radialmod} to $D=\Delta$,
we have for any $\Phi\in\Hom_K(V, C^\infty({K\bar x}))$ and $v\in V^M$
\[
\gamma_{0,x}\bigl(r(\Delta)\Phi[v]\bigr)
=
\partial(\gamma(\Delta)(\cdot-\rho))\gamma_{0,x}\bigl(\Phi[v]\bigr)
+\sum_{i=1}^{q}c_i\,\partial(D''_i)\,\gamma_{0,x}\bigl(\ell(\trans\! D'_i)\Phi[v]\bigr),
\]
where $\gamma_{0,x}$ is the restriction map $C^\infty({K\bar x})\to C^\infty_x$
and $\trans\cdot$ is the anti-automorphism of $U(\Lieg_\CC)$
defined by $\trans X=-X$ for $X\in\Lieg_\CC$.
Let $\pi_V$ be the $U(\Liek_\CC)$-action on $V$.
Since $\gamma_{0,x}\circ\Phi=\gamma_{0,x}\circ\Phi\circ p^V$, we have for $i=1,\ldots,q$
\[
\gamma_{0,x}\bigl(\ell(\trans\! D'_i)\Phi[v]\bigr)
=
\gamma_{0,x}\bigl(\Phi\bigl[\pi_V(\trans\! D'_i)v\bigr]\bigr)
=
\gamma_{0,x}\bigl(\Phi\bigl[p^V\!\circ\pi_V(\trans\! D'_i)v\bigr]\bigr).
\]
Hence we can take
\[
E=\sum_{i=1}^{q}c_i\otimes\partial(D''_i)\otimes {}^{\mathrm t}\bigl(p^V\!\circ\pi_V(\trans\! D'_i)\bigr|_{V^M}\bigr)
\]
in the lemma.
The uniqueness is clear from the surjectivity of $\Gamma^V_{0,x}$.
\end{proof}
Let $\tilde{\mathscr R}$ be the subalgebra of $\mathscr R$ generated by
$\frac1{1-e^{-2\alpha}}\in\mathscr M$ ($\alpha\in\Sigma^+$).
If we identify $\tilde{\mathscr R}$ with a subalgebra of $\End_\CC C^\infty(Wx)$,
then it clearly has the following properties:
\[\left\{\begin{aligned}
&w\,\tilde{\mathscr R}\,w^{-1}=\tilde{\mathscr R}\quad\text{for }w\in W,\\
&\bigl[\partial(\xi), \tilde{\mathscr R}\bigr]\subset \tilde{\mathscr R}
\quad\text{for }\xi\in \Liea.
\end{aligned}\right.
\]
Hence it follows from Definition~\ref{defn:Chered} that
\[
\mathscr T(D)-\partial(D(\cdot-\rho))
\in (\tilde{\mathscr R}\cap\mathscr M)\,\partial(S(\Liea_\CC))\,W
\quad\text{for }D\in S(\Liea_\CC).
\]
This implies the next lemma, which can be considered as the Cherednik operator version of Lemma~\ref{lem:abstRad}.
\begin{lem}\label{lem:abstChered}
Suppose $U$ is a $W$-module.
For any $\Delta\in S(\Liea_\CC)^W$, 
the action of $\mathscr T(\Delta)$ on $C^\infty(Wx)$ induces 
its action on
$\Hom_W(U,C^\infty(Wx))\simeq\Hom_\CC(U,C^\infty_x)$.
On this action there exists a unique $F\in \mathscr M\otimes \partial(S(\Liea_\CC))\otimes \End_\CC U^*$
such that
\[
\mathscr T(\Delta)\varphi
=
\bigl(\partial(\Delta(\cdot-\rho))+F\bigr)\varphi
\quad\text{for any }\varphi\in\Hom_\CC(U, C^\infty_x).
\]
\end{lem}
Suppose $V\in \Km$ and $\Delta\in U(\Lieg_\CC)^K$.
Let $E_\Delta$ be the $E$ of Lemma \ref{lem:abstRad} and
\begin{equation*}
E_\Delta=\begin{pmatrix}
E_\single & P \\
Q & E_\double
\end{pmatrix}
\end{equation*}
its matrix expression by Remark \ref{rem:matform} (iii).
Moreover let $F_{\gamma(\Delta)}$ be the $F$ of Lemma \ref{lem:abstChered}
for the $W$-module $U=V^M_\single$ and $\gamma(\Delta)\in S(\Liea_\CC)^W$.
This can also be expressed in the matrix form
with respect to the basis $\{v_1,\ldots,v_{m'}\}$.
By comparing Lemma~\ref{lem:abstRad} with Lemma~\ref{lem:abstChered}
we can see:
Theorem~\ref{thm:radC} (i) asserts that $E_\single=F_{\gamma(\Delta)}$ and $P=0$;
Theorem~\ref{thm:radC} (ii) is equivalent to $Q=0$.
(Note that (iii) in the theorem is a corollary of (i).)

It is very interesting that if we confirm these things by some concrete calculations
for one special case where $\Delta=L_\Lieg$ (the Casimir element of $\Lieg$),
then all the other cases follow from it.
Let us see this mechanism first.

Assume it is proved that we can take
\begin{equation*}
E_{L_\Lieg}=\begin{pmatrix}
F_{\gamma(L_\Lieg)} & 0\\
0 & L_\double
\end{pmatrix}
\end{equation*}
as the $E$ of Lemma~\ref{lem:abstRad} for $\Delta=L_\Lieg$.
On the one hand, the commutativity
\[
\bigl[r(L_\Lieg),\,r(\Delta)\bigr]=0
\]
implies
\[
\Biggl[
\partial(\gamma(L_\Lieg)(\cdot-\rho))+\begin{pmatrix}
F_{\gamma(L_\Lieg)} & 0\\
0 & L_\double
\end{pmatrix},\,
\partial(\gamma(\Delta)(\cdot-\rho))+\begin{pmatrix}
E_\single & P \\
Q & E_\double
\end{pmatrix}
\Biggr]=0.
\]
This reduces to
\begin{gather}
\bigl[
\partial(\gamma(L_\Lieg)(\cdot-\rho))+F_{\gamma(L_\Lieg)},\,
\partial(\gamma(\Delta)(\cdot-\rho))+E_\single
\bigr]=0,\label{eq:commu1}\\
\bigl(\partial(\gamma(L_\Lieg)(\cdot-\rho))+F_{\gamma(L_\Lieg)}\bigr)\,P
-P\,\bigl(\partial(\gamma(L_\Lieg)(\cdot-\rho))+L_\double \bigr)
=0,\label{eq:commu21}\\
\bigl(\partial(\gamma(L_\Lieg)(\cdot-\rho))+L_\double \bigr)\,Q
-
Q\,\bigl(\partial(\gamma(L_\Lieg)(\cdot-\rho))+F_{\gamma(L_\Lieg)}\bigr)
=0,\label{eq:commu2}\\
\bigl[
\partial(\gamma(L_\Lieg)(\cdot-\rho))+L_\double,\,
\partial(\gamma(\Delta)(\cdot-\rho))+E_\double
\bigr]=0.\notag%\label{eq:commu22}
\end{gather}
On the other hand, the commutativity
\[
\bigl[
\mathscr T(\gamma(L_\Lieg)),\,
\mathscr T(\gamma(\Delta))
\bigr]=0\]
implies
\begin{equation}\label{eq:commu3}
\bigl[
\partial(\gamma(L_\Lieg)(\cdot-\rho))+F_{\gamma(L_\Lieg)},\,
\partial(\gamma(\Delta)(\cdot-\rho))+F_{\gamma(\Delta)}
\bigr]=0.
\end{equation}
From \eqref{eq:commu1} and \eqref{eq:commu3}
we have
\begin{equation*}
\bigl[
\partial(\gamma(L_\Lieg)(\cdot-\rho)),\,
E_\single-F_{\gamma(\Delta)}
\bigr]=
\bigl[
E_\single-F_{\gamma(\Delta)},\,
F_{\gamma(L_\Lieg)}
\bigr],
\end{equation*}
from \eqref{eq:commu21}
\begin{equation*}
\bigl[
\partial(\gamma(L_\Lieg)(\cdot-\rho)),\,
P
\bigr]
=
PL_\double-
F_{\gamma(L_\Lieg)}P,
\end{equation*}
and from \eqref{eq:commu2}
\begin{equation*}
\bigl[
\partial(\gamma(L_\Lieg)(\cdot-\rho)),\,
Q
\bigr]=
QF_{\gamma(L_\Lieg)}-L_\double Q.
\end{equation*}
Now applying the next lemma to these relations,
we can get $E_\single-F_{\gamma(\Delta)}=0$, $P=0$
and $Q=0$.

\begin{lem}\label{lem:Ltrick}
Suppose matrices $S\in\Mat(k,\ell;\mathscr M\otimes \partial(S(\Liea_\CC)))$,
$T\in\Mat(\ell,\ell;\mathscr M\otimes \partial(S(\Liea_\CC)))$ and
$U\in\Mat(k,k;\mathscr M\otimes \partial(S(\Liea_\CC)))$ satisfy
\begin{equation}\label{eq:matrel}
\bigl[
\partial(\gamma(L_\Lieg)(\cdot-\rho)),\,
S
\bigr]=
ST-US.
\end{equation}
Then $S=0$.
\end{lem}
\begin{proof}
In general, any $S\in \Mat(k,\ell;\mathscr R\otimes \partial(S(\Liea_\CC)))$
is uniquely expanded as
\[
S=\sum_{\lambda\in \ZZ_{\ge0}\Sigma^+} e^\lambda S_\lambda \quad\text{with }S_\lambda\in\Mat(k,\ell;\partial(S(\Liea_\CC)))
\]
in the obvious way. Using this expansion we define
\[
\spec S=\{\lambda \in \ZZ_{\ge0}\Sigma^+;\,S_\lambda\ne0\}.
\]
The condition $S\in \Mat(k,\ell;\mathscr M\otimes \partial(S(\Liea_\CC)))$ is equivalent to $0\notin\spec S$.
This is also equivalent to
\[
\spec\bigl[\partial(\gamma(L_\Lieg)(\cdot-\rho)),\,S\bigr] = \spec S,
\]
because a direct calculation shows 
\[
\bigl[
\partial(\gamma(L_\Lieg)(\cdot-\rho)),\,
e^\lambda
\bigr]
=
\bigl[
\partial(L_\Liea-2H_\rho),\,
e^\lambda
\bigr]
=
e^\lambda\bigl(2\partial(H_\lambda)+B(H_\lambda,H_\lambda-2H_\rho)\bigr)
\]
and this is non-zero unless $\lambda=0$.

Now suppose $S,T$ and $U$ are as in the lemma and
assume $S\ne0$. 
Then there exists
a minimal weight $\lambda_0\ne0$ in $\spec S$
with respect to the partial order $\preceq$ in the root lattice defined by
\[
\lambda \preceq \mu \Longleftrightarrow \mu-\lambda \in \ZZ_{\ge0}\Sigma^+.
\]
The above argument shows
the `$\spec$' of the left-hand side of \eqref{eq:matrel} must contain $\lambda_0$.
But
it is easy to see that the `$\spec$' of the right-hand side of \eqref{eq:matrel}
does not contain $\lambda_0$, a contradiction.
\end{proof}

To make the above argument effective,
we must prove Theorem~\ref{thm:radC} for $\Delta=L_\Lieg$.
It is enough to show the following:
\begin{prop}
Let $x\in A_-$ and put $\bar x=xK\in G/K$ as before.
Suppose $V\in \Km$ and
$\Phi\in\Hom_K(V,C^\infty(K\bar x))$.
Then $r(L_\Lieg)\Phi[v]=\ell(L_\Lieg)\Phi[v]$.
If $\gamma_{0,Wx}$ stands for the restriction map $C^\infty({K\bar x})\to C^\infty(Wx)$,
then for any $v\in V^M_\single$ it holds that
\begin{equation*}
\gamma_{0,Wx}\bigl(\ell(L_\Lieg)\Phi[v]\bigr)=
\mathscr T\bigl(\gamma(L_\Lieg)\bigr)\gamma_{0,Wx}(\Phi[v]).
\end{equation*}
Moreover, if
\[
\Phi\in \Hom_K^\ttt(V,C^\infty(K\bar x))
:=\bigl\{
\Phi\in \Hom_K(V,C^\infty(K\bar x));\,
\Gamma^V_{0,Wx}(\Phi)\bigl[ V^M_\double\bigr]=\{0\}
\bigr\},
\]
then $\ell(L_\Lieg)\circ\Phi\in \Hom_K^\ttt(V,C^\infty(K\bar x))$. 
\end{prop}
\begin{proof}
The first assertion is clear
since $L_\Lieg$ is a central element of $U(\Lieg_\CC)$
and $\trans L_\Lieg=L_\Lieg$.
Suppose $y=e^H$ ($H\in\Liea$)  is in a neighborhood of $Wx$. 
We may assume $y$ is a regular point in $A$.
Let $X_\alpha\in\Lieg_\alpha$ ($\alpha\in\Sigma$)
and normalize it so that $-B(X_\alpha,\theta X_\alpha)=1$.
From
$\Ad(y)(X_\alpha+\theta X_\alpha)=\cosh \alpha(H) (X_\alpha+\theta X_\alpha)
+\sinh \alpha(H) (X_\alpha-\theta X_\alpha)$
we have
\[
X_\alpha-\theta X_\alpha
\equiv -\coth\alpha(H) (X_\alpha+\theta X_\alpha)
\pmod{\Ad(y)(X_\alpha+\theta X_\alpha) U(\Lieg_\CC)}
\]
and hence
\[
\begin{aligned}
(X_\alpha-\theta X_\alpha)^2
&\equiv -\coth\alpha(H) (X_\alpha+\theta X_\alpha)(X_\alpha-\theta X_\alpha)\\
&=-2\coth\alpha(H) H_\alpha -\coth\alpha(H) (X_\alpha-\theta X_\alpha)(X_\alpha+\theta X_\alpha)\\
&\equiv -2\coth\alpha(H) H_\alpha + \coth^2\alpha(H)(X_\alpha+\theta X_\alpha)^2\\
&\qquad\qquad\qquad\qquad\qquad \pmod{\Ad(y)(X_\alpha+\theta X_\alpha) U(\Lieg_\CC)}.
\end{aligned}
\]
Now suppose $v\in V^M$. Then
\[
\begin{aligned}
\bigl\{
\ell(X_\alpha +&\theta X_\alpha)^2
-
\ell(X_\alpha -\theta X_\alpha)^2
\bigr\}\Phi[v](y)\\
&=
\bigl\{
2\coth\alpha(H) \ell(H_\alpha)
+(1-\coth^2\alpha(H))\ell(X_\alpha +\theta X_\alpha)^2
\bigr\}\Phi[v](y)\\
&=
2\coth\alpha(H) \ell(H_\alpha) \Phi[v](y)
-\frac{1}{\sinh^2\alpha(H)}
\Phi\bigl[(X_\alpha +\theta X_\alpha)^2 v\bigr](y)\\
&=
-\coth\alpha(H) \partial(H_{2\alpha}) \gamma_{0,Wx}(\Phi[v])(y)
-\frac{1}{\sinh^2\alpha(H)}
\Phi\bigl[(X_\alpha +\theta X_\alpha)^2 v\bigr](y).
\end{aligned}
\]
Let $L_\Liea\in S(\Liea_\CC)$ be as in Proposition~\ref{prop:Dunkl}.
Let $\Liem=\Lie M$ and choose an orthonormal basis 
$\bigl\{Y_1,\dots,Y_{\dim\Liem}\bigr\}$ of the Euclidean space $(\Liem,-B(\cdot,\cdot))$.
Put $L_\Liem=-\sum_{i=1}^{\dim\Liem} Y_i^2$. Then
$\ell(L_\Liem)\Phi[v]=\Phi[L_\Liem v]=0$.
For each $\alpha\in\Sigma^+$ take a basis $\bigl\{X_\alpha^{(1)},\ldots,X_\alpha^{(\dim\Lieg_\alpha)}\bigr\}$
of $\Lieg_\alpha$ 
as in the proof of Lemma~\ref{lem:DunklL}.
Since
\[
\begin{aligned}
L_\Lieg&=
L_\Liem+L_\Liea
-\sum_{\alpha\in\Sigma^+}\sum_{i=1}^{\dim\Lieg_\alpha}
(X_\alpha^{(i)} \theta X_\alpha^{(i)}+\theta X_\alpha^{(i)} X_\alpha^{(i)})\\
&=
L_\Liem+L_\Liea
-\frac12\sum_{\alpha\in\Sigma^+}\sum_{i=1}^{\dim\Lieg_\alpha}
\Bigl\{
\bigl(X_\alpha^{(i)} +\theta X_\alpha^{(i)}\bigr)^2
-
\bigl(X_\alpha^{(i)} -\theta X_\alpha^{(i)}\bigr)^2
\Bigr\},
\end{aligned}
\]
we get
\begin{multline}\label{eq:Lg}
\ell(L_\Lieg) \Phi[v](y)=
\left\{
\der(L_\Liea)
+\sum_{\alpha\in\Sigma^+}\frac{\dim\Lieg_\alpha}2
\coth\alpha(H) \der(H_{2\alpha})
\right\}
\gamma_{0,Wx}(\Phi[v])(y)\\
+
\frac12\sum_{\alpha\in\Sigma^+}\sum_{i=1}^{\dim\Lieg_\alpha}
\frac{1}{\sinh^2\alpha(H)}
\Phi\bigl[(X_\alpha^{(i)} +\theta X_\alpha^{(i)})^2 v\bigr](y).
\end{multline}
If $v\in V^M_\single$, then from Lemma \ref{lem:qsp} (i) we have
\begin{align*}
\Phi\bigl[(X_\alpha^{(i)} +\theta X_\alpha^{(i)})^2 v\bigr](y)
&=
-|\alpha^2|
\Phi\bigl[(1-s_\alpha) v\bigr](y)\\
&=
-|\alpha^2|(1-s_\alpha)\,
\gamma_{0,Wx}(\Phi[v])(y)
\end{align*}
and therefore
\begin{align*}
\ell&(L_\Lieg) \Phi[v](y)\\
&=\left\{
\der(L_\Liea)
+\sum_{\alpha\in\Sigma^+}\frac{\dim\Lieg_\alpha}2
\Bigl(
\coth\alpha(H) \der(H_{2\alpha})
-\frac{|\alpha|^2}{\sinh^2\alpha(H)}(1-s_\alpha)
\Bigr)
\right\}
\gamma_{0,Wx}(\Phi[v])(y)\\
&=
\left\{
\der(L_\Liea)
+\sum_{\beta\in R^+}\mathbf m(\beta)
\Bigl(
\coth\frac{\beta(H)}2 \der(H_{\beta})
-\frac{|\beta|^2}{4\sinh^2\frac{\beta(H)}2}(1-s_{\beta})
\Bigr)
\right\}
\gamma_{0,Wx}(\Phi[v])(y)\\
&=\mathscr T\bigl(L_\Liea-|\rho|^2\bigr) \gamma_{0,Wx}(\Phi[v])(y)
\qquad\qquad\qquad\qquad\quad(\because \eqref{eq:ChredLap})\\
&=\mathscr T\bigl(\gamma(L_\Lieg)\bigr) \gamma_{0,Wx}(\Phi[v])(y).
\qquad\qquad\qquad\qquad\qquad(\because \gamma(L_\Lieg)=L_\Liea-|\rho|^2)
\end{align*}
This proves the second assertion.
Finally, if $\Phi\in \Hom_K^\ttt(V,C^\infty(K\bar x))$ and $v\in V^M_\double$,
then it follows from Theorem \ref{thm:Ch} (ii) and Lemma \ref{lem:qsp} (ii) that
\[
\Phi\bigl[(X_\alpha^{(i)} +\theta X_\alpha^{(i)})^2 v\bigr](y)
=\Phi\bigl[p^V\bigl((X_\alpha^{(i)} +\theta X_\alpha^{(i)})^2 v\bigr)\bigr](y)
=0.
\]
Thus, in this case, the right-hand side of \eqref{eq:Lg} vanishes.
\end{proof}

\section{Harish-Chandra homomorphisms}\label{sec:HC}
The radial part formula given in the last section
concerns the action of $r(U(\Lieg_\CC)^K)$ on $C^\infty(G/K)$.
Its formulation is relatively simple since operators are always $K$-invariant.
In the next section we shall develop another kind of radial part formula
concerning the action of $\ell(U(\Lieg_\CC))$ on $C^\infty(G/K)$,
in which we treat the case where both operators and functions are non-$K$-invariant.
To do so, we must first prepare a non-$K$-invariant generalization of
the Harish-Chandra homomorphism.

In general, for a $(\Lieg_\CC, K)$-module $\mathscr Y$ we put $\Gamma(\mathscr Y)=(\mathscr Y/\Lien_\CC\mathscr Y)^M$.
Since the $0$th $\Lien_\CC$-homology $\mathscr Y/\Lien_\CC\mathscr Y$ of $\mathscr Y$ is $(\Liem_\CC+\Liea_\CC, M)$-module,
its $M$-fixed part $\Gamma(\mathscr Y)$ is naturally an $\Liea_\CC$-module.
But in various contexts it is useful to endow $\Gamma(\mathscr Y)$ with a shifted (dotted) $\Liea_\CC$-module structure.
That is, we let $\xi\in\Liea_\CC$ act on $y\in \Gamma(\mathscr Y)$ by $\xi\cdot y=(\xi-\rho(\xi))y$.
Note there is a natural linear surjective map
\begin{equation}\label{eq:nhomogamma}
\gamma^{\mathscr Y}:\,\mathscr Y \twoheadrightarrow \mathscr Y/\Lien_\CC\mathscr Y
\twoheadrightarrow \Gamma(\mathscr Y)=(\mathscr Y/\Lien_\CC\mathscr Y)^M
\end{equation}
where the second map in this composition is the projection to the isotypic component of the trivial representation of $M$.
If there is no fear of confusion,
we use a brief symbol $\gamma$ for $\gamma^{\mathscr Y}$
because when $\mathscr Y=U(\Lieg_\CC)\otimes_{U(\Liek_\CC)}\CC_\triv$
this map essentially coincides with $\gamma$ of \eqref{eq:HC-homo} (Example \ref{exmp:triv}).

For any $V\in\Km$ we define 
a $(\Lieg_\CC,K)$-module
$P_G(V)=U(\Lieg_\CC)\otimes_{U(\Liek_\CC)}V$.
If $(V^M)^\perp$ denotes the orthogonal compliment of $V^M$ with respect to
a $K$-invariant inner product of $V$, then we have
the direct sum decomposition
\begin{equation}\label{eq:PGVdec}
P_G(V)
=\bigl(\Lien_\CC U(\Lieg_\CC)\otimes V
\,\oplus\, S(\Liea_\CC)\otimes(V^M)^\perp\bigr)
\,\oplus\, S(\Liea_\CC)\otimes V^M.
\end{equation}
Hence $\Gamma(P_G(V))\simeq S(\Liea_\CC)\otimes V^M$.
On the other hand, for any $W$-module $U$
we put $P_{\mathbf H}(U):=\mathbf H\otimes_{\CC W}U$.
This is an $\mathbf H$-module and naturally
$P_{\mathbf H}(U)\simeq S(\Liea_\CC)\otimes U$
as an $\Liea_\CC$-module.
Now let us identify $\Gamma(P_G(V))$ with $P_{\mathbf H}(V^M)$ by
\begin{equation}\label{eq:PHVid}
\Gamma(P_G(V))\simeq S(\Liea_\CC)\otimes V^M \ni
\varphi(\lambda)\otimes v\longmapsto
\varphi(\lambda+\rho)\otimes v\in
S(\Liea_\CC)\otimes V^M\simeq P_{\mathbf H}(V^M).
\end{equation}
Note this is an isomorphism of $\Liea_\CC$-modules.

\begin{exmp}\label{exmp:triv}
If $V=\CC_\triv$
then for any fixed $v_\triv\in \CC_\triv\setminus \{0\}$ we have
the surjection $U(\Lieg_\CC)\ni D\mapsto D\otimes v_\triv \in P_G(\CC_\triv)$
and the bijection $S(\Liea_\CC)\ni \varphi\mapsto \varphi\otimes v_\triv\in P_{\mathbf H}(\CC_\triv)$, for which the diagram
\[\xymatrix{
U(\Lieg_\CC)
\ar[d]_{\text{$\gamma$ in \eqref{eq:HC-homo}}}
\ar@{->>}[r]
&
P_G(\CC_\triv)
\ar[d]^{\gamma=\gamma^{P_G(\CC_\triv)}}
\\
S(\Liea_\CC)
\ar[r]^-{\sim}
&
P_{\mathbf H}(\CC_\triv)
}
\]
commutes.
\end{exmp}

Now suppose $\mathscr Y_1$ and $\mathscr Y_1$ are two $(\Lieg_\CC,K)$-modules
and $\Psi\in\Hom_{\Lieg_\CC,K}(\mathscr Y_1,\mathscr Y_2)$.
Then there exists a unique $\Gamma(\Psi)\in\Hom_{\Liea_\CC}(\Gamma(\mathscr Y_1),\Gamma(\mathscr Y_2))$ such that the diagram
\begin{equation}\label{cd:HCfunct}
\xymatrix{
\mathscr Y_1 \ar[r]^{\Psi}
\ar@{->>}[d]_{\gamma}
& \mathscr Y_2 \ar@{->>}[d]^{\gamma}\\
\Gamma(\mathscr Y_1) \ar[r]_{\Gamma(\Psi)} & \Gamma(\mathscr Y_2)
}
\end{equation}
commutes.
Thus $\Gamma$ defines a (right exact) functor from the category of $(\Lieg_\CC,K)$-modules
to the category of $\Liea_\CC$-modules.
\begin{defn}
Suppose $E, V \in \Km$.
Then naturally
\[
P_{\mathbf H}(E^M)=P_{\mathbf H}(E^M_\single)\oplus P_{\mathbf H}(E^M_\double),\quad
P_{\mathbf H}(V^M)=P_{\mathbf H}(V^M_\single)\oplus P_{\mathbf H}(V^M_\double).
\]
For $\Psi\in\Hom_{\Lieg_\CC,K}(P_G(E),P_G(V))$
we define
$\tilde\Gamma(\Psi)\in\Hom_{\Liea_\CC}(P_{\mathbf H}(E^M_\single),P_{\mathbf H}(V^M_\single))$ by
\[
P_{\mathbf H}(E^M_\single)\hookrightarrow
P_{\mathbf H}(E^M)\xrightarrow{\Gamma(\Psi)}
P_{\mathbf H}(V^M)\twoheadrightarrow
P_{\mathbf H}(V^M_\single).
\]
In addition, we put
\begin{align*}
\Hom_{\Lieg_\CC,K}^\oto(P_G(E)&,P_G(V))\\
=&\bigl\{
\Psi\in \Hom_{\Lieg_\CC,K}(P_G(E),P_G(V));\,
\Gamma(\Psi)\bigl[ P_{\mathbf H}(E^M_\single)\bigr]\subset P_{\mathbf H}(V^M_\single)
\bigr\},\\
\Hom_{\Lieg_\CC,K}^\ttt(P_G(E)&,P_G(V))\\
=&\bigl\{
\Psi\in \Hom_{\Lieg_\CC,K}(P_G(E),P_G(V));\,
\Gamma(\Psi)\bigl[ P_{\mathbf H}(E^M_\double)\bigr]\subset P_{\mathbf H}(V^M_\double)
\bigr\}.
\end{align*}
\end{defn}
In general,
the correspondence $\tilde\Gamma$ does not commute with composition of morphisms.
But it does in the following cases:
\begin{prop}\label{prop:chain}
Suppose $E,F, V\in\Km$, $\Psi\in\Hom_{\Lieg_\CC,K}(P_G(E),P_G(F))$
and $\Phi\in\Hom_{\Lieg_\CC,K}(P_G(F),P_G(V))$.
If $\Psi\in\Hom_{\Lieg_\CC,K}^\oto$ or $\Phi\in\Hom_{\Lieg_\CC,K}^\ttt$ then
$\tilde\Gamma(\Phi\circ\Psi)=\tilde\Gamma(\Phi)\circ\tilde\Gamma(\Psi)$.
\end{prop}

\begin{rem}
Suppose $E, V \in \Km$.

\noindent{\normalfont (i)}
If $V\in\Ksp$ or $E\in\Km\setminus\Kqsp$, then
\[
\Hom_{\Lieg_\CC,K}^\oto(P_G(E),P_G(V))
=\Hom_{\Lieg_\CC,K}(P_G(E),P_G(V)).
\]
If $E\in\Ksp$ or $V\in\Km\setminus\Kqsp$, then
\[
\Hom_{\Lieg_\CC,K}^\ttt(P_G(E),P_G(V))
=\Hom_{\Lieg_\CC,K}(P_G(E),P_G(V)).
\]

\noindent{\normalfont (ii)}
If $V\in\Ksp$, then
\[
\Hom_{\Lieg_\CC,K}^\ttt(P_G(E),P_G(V))
=\bigl\{
\Psi\in \Hom_{\Lieg_\CC,K}(P_G(E),P_G(V));\,
\Gamma(\Psi)\bigl[ P_{\mathbf H}(E^M_\double)\bigr]=\{0\}
\bigr\}.
\]
\end{rem}

We are now in the position to state the main result of this section.
\begin{thm}[the generalized Harish-Chandra homomorphism]\label{thm:HC}
Suppose $E, V \in \Km$.

\noindent
{\normalfont (i)}
$\tilde\Gamma\bigl(
\Hom_{\Lieg_\CC,K}(P_G(E),P_G(V))
\bigr)\subset
\Hom_{\mathbf H}(P_{\mathbf H}(E^M_\single),P_{\mathbf H}(V^M_\single)).
$

\noindent
{\normalfont (ii)}
If $V=\CC_\triv$ then $\tilde\Gamma$ induces a bijection
\[
\tilde\Gamma:\,
\Hom_{\Lieg_\CC,K}^\ttt(P_G(E),P_G(\CC_\triv))
\simarrow
\Hom_{\mathbf H}(P_{\mathbf H}(E^M_\single),P_{\mathbf H}(\CC_\triv)).
\]

\noindent
{\normalfont (iii)}
If $E=\CC_\triv$ then $\tilde\Gamma$ induces a bijection
\[
\tilde\Gamma:\,
\Hom_{\Lieg_\CC,K}^\oto(P_G(\CC_\triv),P_G(V))
\simarrow
\Hom_{\mathbf H}(P_{\mathbf H}(\CC_\triv),P_{\mathbf H}(V^M_\single)).
\]

\noindent
{\normalfont (iv)}
If $E=V=\CC_\triv$ then there exist natural identifications
$\End_{\Lieg_\CC,K}(P_G(\CC_\triv))\simeq U(\Lieg_\CC)^K/(U(\Lieg_\CC)\Liek_\CC)^K$ and
$\End_{\mathbf H}(P_{\mathbf H}(\CC_\triv))\simeq S(\Liea_\CC)^W$,
under which the algebra isomorphism
\[
\tilde\Gamma:\,
\End_{\Lieg_\CC,K}(P_G(\CC_\triv))
\simarrow
\End_{\mathbf H}(P_{\mathbf H}(\CC_\triv))
\]
coincides with \eqref{eq:HC-isom}.
\end{thm}

   Before proving the theorem
we introduce some notation, which will also be used
to formulate the radial part formula in the next section.

\begin{defn}\label{defn:homhom}
Suppose $E, V \in \Km$.
By the Frobenius reciprocity,
we can identify
\begin{align*}
\Hom_{\Lieg_\CC,K}(P_G(E),P_G(V))&\simeq\Hom_{K}(E,P_G(V)),\\
\Hom_{\Liea_\CC}(P_{\mathbf H}(E^M),P_{\mathbf H}(V^M))
&\simeq\Hom_{\CC}(E^M,P_{\mathbf H}(V^M)),\\
\Hom_{\mathbf H}(P_{\mathbf H}(E^M_\single),P_{\mathbf H}(V^M_\single))
&\simeq\Hom_{W}(E^M_\single,P_{\mathbf H}(V^M_\single)).
\end{align*}
Under these identifications, the map
\[
\Gamma :\, \Hom_{\Lieg_\CC,K}(P_G(E),P_G(V))\to\Hom_{\Liea_\CC}(P_{\mathbf H}(E^M),P_{\mathbf H}(V^M))
\]
reduces to
\begin{align*}
\Gamma^E_V :\,&
 \Hom_K(E,P_G(V))
\longrightarrow
\Hom_\CC(E^M,P_{\mathbf H}(V^M))\,;\\
&\Psi\longmapsto
\bigl(
\psi:
E^M\hookrightarrow
E
\xrightarrow{\Psi} P_G(V)
\xrightarrow{\gamma} P_{\mathbf H}(V^M)
\bigr).
\end{align*}
This interpretation
is distinguished by attaching super- and sub-scripts to $\Gamma$.
We define
\[
\tilde\Gamma^E_V :\,
 \Hom_K(E,P_G(V))
\longrightarrow
\Hom_\CC(E^M_\single,P_{\mathbf H}(V^M_\single))
\]
similarly (Theorem \ref{thm:HC} (i) asserts
that the target space of this map can be replaced with $\Hom_W$).
Finally we put
\begin{align*}
\Hom_{K}^\oto(E,P_G(V))
&=\bigl\{
\Psi\in \Hom_{K}(E,P_G(V));\,
\Gamma^E_V(\Psi)\bigl[ E^M_\single ]\subset P_{\mathbf H}(V^M_\single)
\bigr\},\\
\Hom_{K}^\ttt(E,P_G(V))
&=\bigl\{
\Psi\in \Hom_{K}(E,P_G(V));\,
\Gamma^E_V(\Psi)\bigl[ E^M_\double \bigr]\subset P_{\mathbf H}(V^M_\double)
\bigr\}.
\end{align*}
\end{defn}
\begin{proof}[Proof of {\normalfont Theorem \ref{thm:HC}}]
From
\cite[Thorem 4.7]{Oda:HC}
it holds that
\[
\tilde\Gamma^E_{\CC_\triv}\bigl(
\Hom_{K}(E,P_G(\CC_\triv))\bigr)
\subset
\Hom_{W}(E^M_\single,P_{\mathbf H}(\CC_\triv)),
\]
which is equivalent to (i) for $V=\CC_\triv$.
Also, by \cite[Thorem 4.11]{Oda:HC}
we have an isomorphism
\[
\tilde\Gamma^E_{\CC_\triv}:\,
\Hom_{K}^\ttt(E,P_G(\CC_\triv))
\simarrow
\Hom_{W}(E^M_\single,P_{\mathbf H}(\CC_\triv)),
\]
which is equivalent to (ii).

Next, fix a non-zero $v_\triv\in\CC_\triv$.
Then (iv) is clear from
\begin{align*}
\End_{\Lieg_\CC,K}(P_G(\CC_\triv))
&=\Hom_{K}(\CC_\triv,P_G(\CC_\triv))\\
&=\Hom_{\CC}\bigl(\CC\,v_\triv, U(\Lieg_\CC)^K/(U(\Lieg_\CC)\Liek_\CC)^K \otimes v_\triv\bigr),\\
\End_{\mathbf H}(P_{\mathbf H}(\CC_\triv))
&=\Hom_W(\CC_\triv,P_{\mathbf H}(\CC_\triv))\\
&=\Hom_\CC(\CC\,v_\triv, S(\Liea_\CC)^W\otimes v_\triv)
\end{align*}
and Example \ref{exmp:triv}.

To prove (i) suppose $E,V\in\Km$ are arbitrary.
Choose a basis $\{v_1,\ldots,v_{m'}\}$ of $V^M_\single$.
Let $H_W(\Liea_\CC)\subset S(\Liea_\CC)$
be the space of $W$-harmonic polynomials on $\Liea^*$.
Then there exist $m'\, (=\dim V^M_\single$) linearly independent $W$-homomorphisms
$\varphi_j:\,V^M_\single\to H_W(\Liea_\CC)$ ($j=1,\ldots,m'$)
such that $\varphi_j[v_1],\ldots,\varphi_j[v_{m'}]$ are all homogeneous with the same degree
for each fixed $j$.
For $j=1,\ldots,m'$ choose
$\hat\varphi_j\in \Hom_W(V^M_\single,P_{\mathbf H}(\CC_\triv))$
so that 
the top degree part of $\hat\varphi_j$ coincides with $\varphi_j$.
Here we are identifying
$P_{\mathbf H}(\CC_\triv)=S(\Liea_\CC)\otimes v_\triv$ with $S(\Liea_\CC)$ naturally.
Note that this identification respects $\Liea_\CC$-module structures
and that $\det(\hat\varphi_j[v_i])_{1\le i,j\le m'}\ne 0$
since $\det(\varphi_j[v_i])_{1\le i,j\le m'}\ne 0$ (cf.~\cite[\S2]{HC}).
Thus the $\mathbf H$-homomorphism
\begin{equation*}
\begin{aligned}
P_{\mathbf H}(V^M_\single)&
\xrightarrow{
\prod_{j}\hat\varphi_j
}
P_{\mathbf H}(\CC_\triv)^{m'};\\
\sum_i f_i \otimes v_i
&\longmapsto
\Bigl(
\sum_i f_i\hat\varphi_1[v_i]\otimes v_\triv,
\ldots,
\sum_i f_i\hat\varphi_{m'}[v_i]\otimes v_\triv,
\Bigr)
\end{aligned}
\end{equation*}
is injective.
Now using (ii) we can lift $\hat\varphi_j$ to
$\Phi_j\in \Hom_{\Lieg_\CC,K}^\ttt(P_G(V),P_G(\CC_\triv))$.
By virtue of Proposition \ref{prop:chain},
for any $\Psi\in\Hom_{\Lieg_\CC,K}(P_G(E),P_G(V))$ we have
\[
\prod_j\tilde\Gamma(\Phi_j \circ\Psi)
=
\prod_j(\tilde\Gamma(\Phi_j)\circ\tilde\Gamma(\Psi))
=
\prod_j(\hat\varphi_j\circ\tilde\Gamma(\Psi))
=
\Bigl(\prod_j\hat\varphi_j\Bigl)\circ\tilde\Gamma(\Psi).
\]
The leftmost side shows this is an $\mathbf H$-homomorphism of
$P_{\mathbf H}(E^M_\single)$ into $P_{\mathbf H}(\CC_\triv)^{m'}$
since we already know (i) is valid for
$\Hom_{\Lieg_\CC,K}(P_G(E),P_G(\CC_\triv))$.
But since $\prod_j\hat\varphi_j$ is an injective $\mathbf H$-homomorphism,
the rightmost side shows $\tilde\Gamma(\Psi)$ is also an $\mathbf H$-homomorphism.

We postpone the proof of (iii) until we introduce the notion of
\emph{star operation}s for morphisms in \S\ref{sec:star}.
With that notion, (iii) is equivalent to (ii) (Corollary \ref{cor:2-3}).
\end{proof}

\section{Radial part formula, II}\label{sec:radII}

In this section we try to generalize \eqref{eq:rad2}
to some cases where $\Delta$ and $f$ are not necessarily $K$-invariant.
In view of \eqref{eq:rad2} the radial part of a left Lie algebra action
on the $K$-invariant functions
is twisted by the Cartan involution $\theta$.
Related to this, we introduce the ``Cartan involution" for $\mathbf H$.
\begin{defn}\label{defn:HCartan}
Let $w_0$ be the longest element of $W$.
We define the algebra automorphism $\theta_{\mathbf H}$ of $\mathbf H$
so that it satisfies the following relations:
\[\left\{
\begin{aligned}
&\theta_{\mathbf H}\,w=w&&\text{for }w\in W,\\
&\theta_{\mathbf H}\,\xi=-w_0\,w_0(\xi)\,w_0
&&\text{for }\xi\in \Liea_\CC.
\end{aligned}\right.
\]
The automorphism is well defined by \eqref{eq:Hrel}.
\end{defn}
For $V\in\Km$, $\theta$ naturally induces a $K$-linear automorphism of
$P_G(V)=U(\Lieg_\CC)\otimes_{\Liek_\CC} V $:
\[
P_G(V)\ni D\otimes v\longmapsto (\theta D)\otimes v
\in P_G(V).
\]
For a $W$-module $U$,
$\theta_{\mathbf H}$ naturally induces a $W$-linear automorphism of
$P_{\mathbf H}(U)$ likewise.

\begin{prop}\label{prop:2Cartans}
Let $\bar w_0\in N_K(\Liea)=\{k\in K;\,\Ad(k)(\Liea)\subset\Liea\}$ be an element 
normalizing $\Liea$ with the same action as $w_0$.
Suppose $V\in \Km$. Then it holds that
\begin{equation}\label{eq:Cartan0}
\gamma(\theta (\bar w_0D))=\theta_{\mathbf H}(w_0\gamma(D))
\quad\text{for any }D\in P_G(V).
\end{equation}
Moreover, if $E,V\in\Km$ then we have
\begin{align}
&\tilde\Gamma^E_V(\theta\circ\Psi)=\theta_{\mathbf H}\circ \tilde\Gamma^E_V(\Psi)
&&\text{for any }\Psi\in \Hom_K(E,P_G(V)),\label{eq:Cartan1}\\
&\theta\circ\Psi\in \Hom_K^\oto(E,P_G(V))
&&\text{for any }\Psi\in \Hom_K^\oto(E,P_G(V)),\label{eq:Cartan2}\\
&\theta\circ\Psi\in \Hom_K^\ttt(E,P_G(V))
&&\text{for any }\Psi\in \Hom_K^\ttt(E,P_G(V)).\label{eq:Cartan3}
\end{align}
\end{prop}
\begin{proof}
Suppose $V\in\Km$ and $D\in P_G(V)$.
Let $D=D_1+D_2$ be the decomposition corresponding to the direct sum decomposition \eqref{eq:PGVdec}.
Applying $\theta\circ\bar w_0$ to this, we get
$\theta(\bar w_0D)=\theta(\bar w_0D_1)+\theta(\bar w_0D_2)$,
which is nothing but the decomposition of $\theta(\bar w_0D)$
corresponding to \eqref{eq:PGVdec}.
Thus if we write the identification \eqref{eq:PHVid} in the form
\begin{equation*}
\iota:\,S(\Liea_\CC)\otimes V^M \ni
\varphi(\lambda)\otimes v\longmapsto
\varphi(\lambda+\rho)\otimes v\in
S(\Liea_\CC)\otimes V^M\simeq P_{\mathbf H}(V^M)
\end{equation*}
and prove the equality $\iota(\theta(\bar w_0D_2))=\theta_{\mathbf H}(w_0\,\iota(D_2))$, then
\eqref{eq:Cartan0} follows.
But the equality holds since
\begin{align*}
\iota\bigl(\theta(\bar w_0\,\varphi(\lambda)\otimes v)\bigr)
&=\iota\bigl(\varphi(-w_0^{-1}\lambda)\otimes w_0v\bigr)
=\iota\bigl(\varphi(-w_0\lambda)\otimes w_0v\bigr)\\
&=\varphi(-w_0(\lambda+\rho))\otimes w_0v
=\varphi(-w_0\lambda+\rho)\otimes w_0v\\
&=w_0\,w_0 \varphi(-w_0\lambda+\rho) w_0\otimes v
=w_0\,\theta_{\mathbf H}\bigl(\varphi(\lambda+\rho)\otimes v\bigr)\\
&=w_0\,\theta_{\mathbf H}\bigl(\iota\bigl(\varphi(\lambda)\otimes v\bigr)\bigr)
=
\theta_{\mathbf H}\bigl(w_0\,\iota\bigl(\varphi(\lambda)\otimes v\bigr)\bigr).
\end{align*}

Now, suppose $E,V\in\Km$ and $\Psi\in \Hom_K(E,P_G(V))$.
For any $e\in E^M$
\begin{align*}
\Gamma^E_V(\theta\circ\Psi)[e]&=\gamma(\theta(\Psi[e]))=\gamma(\theta(\Psi[w_0\,w_0e]))\\
&=\gamma(\theta(\bar w_0\Psi[w_0e])) &&(\because \Psi\in\Hom_K)\\
&=\theta_{\mathbf H}(w_0\,\gamma(\Psi[w_0e])) &&(\because \eqref{eq:Cartan0})\\
&=\theta_{\mathbf H}(w_0\,\Gamma^E_V(\Psi)[w_0e]).
\end{align*}
This expression proves \eqref{eq:Cartan2} since
$w_0 E^M_\single=E^M_\single$ and both $\theta_{\mathbf H}$ and the left multiplication by $w_0$
leave $P_{\mathbf H}(V^M_\single)$ stable.
Similar is \eqref{eq:Cartan3}.
Lastly, since the projection $P_{\mathbf H}(V^M)\to P_{\mathbf H}(V^M_\single)$
commutes with $\theta_{\mathbf H}$ and the left multiplication by $w_0$,
for $e\in E^M_\single$ we have
\begin{align*}
\tilde\Gamma^E_V(\theta\circ\Psi)[e]
&=\theta_{\mathbf H}(w_0\,\tilde\Gamma^E_V(\Psi)[w_0e])\\
&=\theta_{\mathbf H}(w_0\,w_0\,\tilde\Gamma^E_V(\Psi)[e])
=\theta_{\mathbf H}(\tilde\Gamma^E_V(\Psi)[e]) &&(\because\text{Theorem \ref{thm:HC} (i)})\\
&=(\theta_{\mathbf H}\circ \tilde\Gamma^E_V(\Psi))[e].
\end{align*}
This shows \eqref{eq:Cartan1}. 
\end{proof}
Suppose $V\in\Km$.
As in Definition \ref{defn:homhom} we identify
\begin{equation*}
\Hom_K(V,C^\infty(G/K))\simeq \Hom_{\Lieg_\CC,K}(P_G(V), C^\infty(G/K)_\Kf)
\end{equation*}
using the action of $\ell(U(\Lieg_\CC))$ on $C^\infty(G/K)_\Kf$.
In view of \eqref{eq:rad2} and Proposition \ref{prop:2Cartans}
we let the analogous identification
\begin{equation*}
\Hom_W(U,C^\infty(A))\simeq \Hom_{\mathbf H}(P_{\mathbf H}(U), C^\infty(A))
\end{equation*}
for any $W$-module $U$
be based on the $\mathbf H$-module structure of $C^\infty(A)$
defined by $\mathscr T(\theta_{\mathbf H}\cdot)$.
Under these identifications the map $\Gamma^V_0$ defined by \eqref{eq:GammaV0}
can be rewritten as
\[
\begin{aligned}
\Gamma_0:\,&
\Hom_{\Lieg_\CC,K}(P_G(V), C^\infty(G/K)_\Kf)
\to
\Hom_{\mathbf H}(P_{\mathbf H}(V^M), C^\infty(A));\\
&\Phi\longmapsto
\Bigl(
\varphi:
P_{\mathbf H}(V^M)\ni \sum_{i=1}^mh_i\otimes v_i
\longmapsto
\sum_{i=1}^m\mathscr T(\theta_{\mathbf H}h_i)\gamma_0(\Phi[v_i])
\Bigr).
\end{aligned}
\]
We distinguish
this interpretation by the symbol $\Gamma_0$
with no superscript.
We remark in contrast to \eqref{cd:HCfunct}
the diagram
\begin{equation*}
\xymatrix{
P_G(V) \ar[r]^-{\Phi}
\ar@{->>}[d]_{\gamma}
& {C^\infty(G/K)_\Kf \ar[d]^{\gamma_0}\hspace{-5ex}}\\
P_{\mathbf H}(V^M) \ar[r]_{\Gamma_0(\Phi)} & C^\infty(A)
}
\end{equation*}
cannot be assumed commutative at all.
Similarly,
for $\Phi\in\Hom_K(V,C^\infty(G/K))=\Hom_{\Lieg_\CC,K}(P_G(V), C^\infty(G/K)_\Kf)$
we let
$\tilde\Gamma_0(\Phi)\in\Hom_{\mathbf H}(P_{\mathbf H}(V^M_\single), C^\infty(A))$
be a map identified with $\tilde\Gamma^V_0(\Phi)\in\Hom_W(V^M_\single, C^\infty(A))$
(cf.~Definition \ref{defn:22} (ii)).
We also use the following identification:
\begin{align*}
\Hom_{K}^\ttt&(V,C^\infty(G/K))\\
\simeq&\Hom_{\Lieg_\CC,K}^\ttt(P_G(V),C^\infty(G/K)_\Kf)\\
&:=\bigl\{
\Phi\in \Hom_{\Lieg_\CC,K}(P_G(V),C^\infty(G/K)_\Kf);\,
\Gamma_0(\Phi)\bigl[ P_{\mathbf H}(V^M_\double)\bigr]=\{0\}
\bigr\}.
\end{align*}

\begin{thm}[the radial part formula]\label{thm:radD}
Suppose $E,V\in \Km$,
$\Psi\in\Hom_K(E,P_G(V))$
and $\Phi \in \Hom_K(V,C^\infty(G/K))=\Hom_{\Lieg_\CC,K}(P_G(V), C^\infty(G/K)_\Kf)$.

\noindent{\normalfont (i)}
Suppose $\Phi \in \Hom_K^\ttt(V,C^\infty(G/K))=\Hom_{\Lieg_\CC,K}^\ttt(P_G(V), C^\infty(G/K)_\Kf)$.
Then it holds that
\begin{equation*}
\tilde \Gamma^E_0(\Phi\circ\Psi)=\tilde\Gamma_0(\Phi)\circ\tilde\Gamma^E_V(\Psi).
\end{equation*}
This means the diagram
\[
\xymatrix{
E
\ar[r]^-{\Psi}
&
P_G(V)
\ar[r]^-{\Phi}
& {C^\infty(G/K)_\Kf
\ar[d]^{\gamma_0}\hspace{-5ex}}\\
E^M_\single
\ar@{_{(}->}[u]
\ar[r]_-{\tilde\Gamma^E_V(\Psi)}
&
P_{\mathbf H}(V^M_\single)
\ar[r]_{\tilde\Gamma_0(\Phi)}
& C^\infty(A)
}
\]
is commutative.
In other words, if we take a basis $\{v_1,\ldots,v_n\}$ of $V$
so that $\{v_1,\ldots,v_{m'}\}$, $\{v_{m'+1},\ldots,v_{m}\}$ and $\{v_{m+1},\ldots,v_{n}\}$
are bases of $V^M_\single$, $V^M_\double$ and $(V^M)^\perp$ respectively $(m'\le m\le n)$,
and if for any $e\in E^M_\single$ we write
\[
\Psi[e]=\sum_{i=1}^n D_i\otimes v_i \text{ with }D_i\in U(\Lieg_\CC),\qquad
\Gamma^E_V(\Psi)[e]=\sum_{i=1}^{m} h_i\otimes v_i \text{ with }h_i\in \mathbf H,
\]
then
\[
\gamma_0\biggl(\sum_{i=1}^n \ell(D_i)\Phi[v_i]\biggr)
=
\sum_{i=1}^{m'}\mathscr T(\theta_{\mathbf H}h_i)\gamma_0(\Phi[v_i]).
\]

\noindent
{\normalfont (ii)}
If $\Phi \in \Hom_{\Lieg_\CC,K}(P_G(V), C^\infty(G/K)_\Kf)$
and
$\Psi\in\Hom_K^\oto(E,P_G(V))$ (see {\normalfont Definition \ref{defn:homhom}}),
then the same assertion as {\normalfont (i)} holds.

\noindent
{\normalfont (iii)}
If $\Phi \in \Hom_{\Lieg_\CC,K}^\ttt(P_G(V), C^\infty(G/K)_\Kf)$
and
$\Psi\in\Hom_K^\ttt(E,P_G(V))$ (see {\normalfont Definition \ref{defn:homhom}}),
then $\Phi\circ\Psi \in\Hom_K^\ttt(E,C^\infty(G/K))$ and hence it holds that
\begin{equation*}
\Gamma^E_0(\Phi\circ\Psi)=\Gamma_0(\Phi)\circ\Gamma^E_V(\Psi).
\end{equation*}
\end{thm}
The proof of the theorem is similar to that of Theorem~\ref{thm:radC}.
Suppose $x\in A_-$ and 
let $\gamma_{0,Wx}: C^\infty(K\bar x)\to C^\infty(Wx)$,
$\gamma_{0,x}: C^\infty(K\bar x)\to C^\infty_x$, $\mathscr R$, $\mathscr M$
and $\mathscr R\otimes \partial(S(\Liea_\CC))$
be as in \S\ref{sec:radI}.
In addition to the basis $\{v_1,\ldots,v_n\}$ of $V$ in (i), take
a basis $\{e_1,\ldots,e_\mu\}$ of $E^M$
so that $\{e_1,\ldots,e_{\mu'}\}$, $\{e_{\mu'+1},\ldots,e_{\mu}\}$
are bases of $E^M_\single$, $E^M_\double$ respectively $(\mu'\le \mu)$.
We use the same identifications with \eqref{eq:restId}:
\[
\Hom_W(V^M,C^\infty(Wx))\simeq \Hom_\CC(V^M,C^\infty_x),\
\Hom_W(E^M,C^\infty(Wx))\simeq\Hom_\CC(E^M,C^\infty_x).
\]
They are still identified with
the spaces of column vectors with entries in $C^\infty_x$
by using $\{v_1,\ldots,v_m\}$ and $\{e_1,\ldots,e_\mu\}$.
Thus each element of $\Mat(\mu,m;\mathscr R\otimes \partial(S(\Liea_\CC)))$ naturally
defines a map $\Hom_W(V^M,C^\infty(Wx)) \to \Hom_W(E^M,C^\infty(Wx))$.
\begin{lem}\label{lem:abstRad2}
Suppose $\Psi\in\Hom_K(E,P_G(V))$ and take $D_{ij}\in U(\Lien_\CC+\Liea_\CC)$
so that
\begin{equation}\label{eq:thetaPsi}
(\theta\circ\Psi)[e_i]=\sum_{j=1}^n D_{ij}\otimes v_j
\qquad\text{for }i=1,\ldots,\mu.
\end{equation}
Then there exists a unique
$S=(S_{ij}) \in \Mat(\mu,m; \mathscr M\otimes \partial(S(\Liea_\CC)))$
such that
\begin{equation}\label{eq:dMdecomp}
\Gamma^E_{0,x}(\Phi\circ\Psi)[e_i]
=
\sum_{j=1}^{m}
\Bigl(
\partial\bigl(\gamma(D_{ij})(\cdot-\rho)\bigr)+S_{ij}
\Bigr)\,\Gamma^V_{0,x}(\Phi)[v_j]
\end{equation}
for any $i=1,\ldots,\mu$ and
$\Phi\in\Hom_K(V,C^\infty(K\bar x))\simeq\Hom_{\Lieg_\CC,K}(P_G(V),C^\infty(K\bar x)_\Kf)$.
 \end{lem}
\begin{proof}
By the same method used in obtaining \eqref{eq:radialmod}
we can take $c_{ijk}\in\mathscr M$, $D'_{ijk}\in U(\Liek_\CC)$,
and $D''_{ijk}\in S(\Liea_\CC)$ ($1\le i\le \mu, 1\le j\le n, 1\le k\le q_{ij}$) so that it holds that
\begin{equation*}
(\theta\circ\Psi)[e_i]
=\sum_{j=1}^n\gamma(D_{ij})(\cdot-\rho)\otimes v_j
+\sum_{j=1}^{n}\sum_{k=1}^{q_{ij}}c_{ijk}(y)\Ad(y^{-1})(D'_{ijk})D''_{ijk}\otimes v_j
\end{equation*}
for any $i=1,\ldots,\mu$ and $y\in A_-$.
Applying $\theta$ to both the side
we get
\begin{equation*}
\Psi[e_i]
=\sum_{j=1}^n{\mathstrut}^{\mathrm t}\!\bigl(\gamma(D_{ij})(\cdot-\rho)\bigr)\otimes v_j
+\sum_{j=1}^{n}\sum_{k=1}^{q_{ij}}c_{ijk}(y)
\Ad(y)(D'_{ijk})\,\trans\!D''_{ijk}\otimes v_j.\end{equation*}
Hence letting $\epsilon:U(\Liek_\CC)\to\CC$ be the projection 
of the decomposition $U(\Liek_\CC)=\CC\oplus U(\Liek_\CC)\Liek_\CC$
to the first summand,
we calculate for $i=1,\ldots,\mu$ and $y$ in a neighborhood of $x$
\begin{align*}
(\Phi&\circ\Psi)[e_i](y)\\
&=\sum_{j=1}^n\ell\Bigl({\mathstrut}^{\mathrm t}\!\bigl(\gamma(D_{ij})(\cdot-\rho)\bigr)\Bigr)\Phi[v_j] (y)+\sum_{j=1}^{n}\sum_{k=1}^{q_{ij}}c_{ijk}(y)
\ell\bigl(
\Ad(y)(D'_{ijk})\,\trans\!D''_{ijk}
\bigr)\Phi[v_j](y)\\
&=
\sum_{j=1}^n\partial\bigl(\gamma(D_{ij})(\cdot-\rho)\bigr)\gamma_{0,x}\bigl(\Phi[v_j]\bigr) (y)
+\sum_{j=1}^{n}\sum_{k=1}^{q_{ij}}c_{ijk}(y)\,
\epsilon(D'_{ijk})\,\partial(D''_{ijk})
\gamma_{0,x}\bigl(\Phi[v_j]\bigr)(y)\\
&=
\sum_{j=1}^m\partial\bigl(\gamma(D_{ij})(\cdot-\rho)\bigr)\gamma_{0,x}\bigl(\Phi[v_j]\bigr) (y)
+\sum_{j=1}^{m}\sum_{k=1}^{q_{ij}}c_{ijk}(y)\,
\epsilon(D'_{ijk})\,\partial(D''_{ijk})
\gamma_{0,x}\bigl(\Phi[v_j]\bigr)(y).
\end{align*}
Here the last equality follows from the fact that
$\gamma_{0,x}\bigl(\Phi[v_j]\bigr)=0$ for $j=m+1,\ldots,n$.
Now putting $S_{ij}=\sum_{k=1}^{q_{ij}}c_{ijk}(y)\,
\epsilon(D'_{ijk})\,\partial(D''_{ijk})$ we get \eqref{eq:dMdecomp}.
Its uniqueness is due to the surjectivity of $\Gamma^V_{0,x}$
(Lemma \ref{lem:localCh}).
\end{proof}
\begin{lem}\label{lem:abstChered2}
Suppose $\psi\in\Hom_W(E^M_\single,P_{\mathbf H}(V^M_\single))$
and take $f_{ij}\in S(\Liea_\CC)$ so that
\[
(\theta_{\mathbf H}\circ\psi)[e_i]=\sum_{j=1}^{m'}f_{ij}\otimes v_j\qquad\text{for }i=1,\ldots,\mu'.
\]
Then there exists a unique $T=(T_{ij})\in\Mat(\mu',m';\mathscr M\otimes\partial(S(\Liea_\CC)))$
such that
\begin{equation}\label{eq:Chred2}
(\varphi\circ\psi)[e_i]
=\sum_{j=1}^{m'}
\mathscr T(f_{ij})\varphi[v_j]
=
\sum_{j=1}^{m'}\Bigl(\partial\bigl(f_{ij}(\cdot-\rho)\bigr)+T_{ij}\Bigr)\varphi[v_j]
\end{equation}
for any $i=1,\ldots,\mu'$ and
\[
\varphi\in
\Hom_{\mathbf H}(P_{\mathbf H}(V^M_\single),C^\infty(Wx))\simeq
\Hom_W(V^M_\single,C^\infty(Wx))\simeq
\Hom_\CC(V^M_\single,C^\infty_x).
\]
Note $h\in\mathbf H$ acts on $C^\infty(Wx)$
by $\mathscr T(\theta_{\mathbf H}h)$.
We consider
$\varphi[v_j]\in C^\infty(Wx)$ in the second expression of \eqref{eq:Chred2}
while $\varphi[v_j]\in C^\infty_x$ in the third expression .
\end{lem}
\begin{proof}
Immediate from
the argument just before Lemma \ref{lem:abstChered}.
\end{proof}
Now suppose $\Psi$, $D_{ij}$ and $S$ are as in Lemma \ref{lem:abstRad2}.
According to the devisions of bases
\begin{align*}
\{e_1,\ldots,e_{\mu}\}&=\{e_1,\ldots,e_{\mu'}\} \sqcup \{e_{\mu'+1},\ldots,e_\mu\},\\
\{v_1,\ldots,v_{m}\}&=\{v_1,\ldots,v_{m'}\} \sqcup \{v_{m'+1},\ldots,v_m\},
\end{align*}
we divide $S$ into four blocks:
\[
S=\begin{pmatrix}
S_\single & P \\
Q & S_\double
\end{pmatrix}.
\]
Similarly we divide the matrix $\bigl(
\partial\bigl(\gamma(D_{ij})(\cdot-\rho)\bigr)
\bigr)_{\substack{1\le i\le \mu\\1\le j\le m}}
\in \Mat(\mu,m;\partial(S(\Liea_\CC)))$ into four blocks:
\[
\bigl(
\partial\bigl(\gamma(D_{ij})(\cdot-\rho)\bigr)
\bigr)_{\substack{1\le i\le \mu\\1\le j\le m}}
=\begin{pmatrix}
\partial_\single & \partial_P \\
\partial_Q & \partial_\double
\end{pmatrix}.
\]
For $\psi:=\tilde\Gamma^E_V(\Psi)$
let $f_{ij}$ and $T$ be as in Lemma \ref{lem:abstChered2}.
Because of \eqref{eq:Cartan1} we clearly have
\[
f_{ij}=\gamma(D_{ij})\qquad\text{for }i=1,\ldots,\mu'\text{ and }
j=1,\ldots,m',
\]
and hence
\begin{equation}\label{eq:partilapart}
\bigl(
\partial\bigl(f_{ij}(\cdot-\rho)\bigr)
\bigr)_{\substack{1\le i\le \mu'\\1\le j\le m'}}
=\partial_\single.
\end{equation}
Let $F_{\gamma(L_\Lieg)}$ (resp., $F'_{\gamma(L_\Lieg)}$) be the $F$ of Lemma \ref{lem:abstChered}
for $U=V^M_\single$ (resp., $U=E^M_\single$) and $\Delta=\gamma(L_\Lieg)\in S(\Liea_\CC)^W$.
By a result of \S\ref{sec:radI},
there exist a matrix $L_\double\in \Mat(m-m',m-m';\mathscr M\otimes\partial(S(\Liea_\CC)))$
such that for any $\Phi\in\Hom_K(V^M,C^\infty(K\bar x))$
\[
\Gamma^V_{0,x}\bigl(r(L_\Lieg)\circ\Phi\bigr)=
\biggl(
\partial(\gamma(L_\Lieg)(\cdot-\rho))+\begin{pmatrix}
F_{\gamma(L_\Lieg)} & 0\\
0 & L_\double
\end{pmatrix}
\biggr)
{\mathstrut}^{\mathrm t}\!\bigl(\Gamma^V_{0,x}(\Phi)[v_1],\ldots,\Gamma^V_{0,x}(\Phi)[v_m]\bigr),
\]
and a matrix $L'_\double\in \Mat(\mu-\mu',\mu-\mu';\mathscr M\otimes\partial(S(\Liea_\CC)))$
such that for any $\Phi'\in\Hom_K(E^M,C^\infty(K\bar x))$
\[
\Gamma^E_{0,x}\bigl(r(L_\Lieg)\circ\Phi'\bigr)=
\biggl(
\partial(\gamma(L_\Lieg)(\cdot-\rho))+\begin{pmatrix}
F'_{\gamma(L_\Lieg)} & 0\\
0 & L'_\double
\end{pmatrix}
\biggr)
{\mathstrut}^{\mathrm t}\!\bigl(\Gamma^E_{0,x}(\Phi')[e_1],\ldots,\Gamma^E_{0,x}(\Phi')[e_\mu]\bigr).
\]
Since $r(\cdot)$ and $\ell(\cdot)$ commute, we have $r(L_\Lieg)\circ (\Phi\circ\Psi)=(r(L_\Lieg)\circ \Phi)\circ\Psi$
for any $\Phi\in\Hom_K(V^M,C^\infty(K\bar x))$.
Therefore Lemma \ref{lem:abstRad2} and the surjectivity of $\Gamma^V_{0,x}$ imply
the matrix identity:
\begin{multline}\label{eq:matId1}
\biggl(
\partial(\gamma(L_\Lieg)(\cdot-\rho))+\begin{pmatrix}
F'_{\gamma(L_\Lieg)} & 0\\
0 & L'_\double
\end{pmatrix}
\biggr)
\begin{pmatrix}
\partial_\single+S_\single & \partial_P+P \\
\partial_Q+Q & \partial_\double+S_\double
\end{pmatrix}
\\
=
\begin{pmatrix}
\partial_\single+S_\single & \partial_P+P \\
\partial_Q+Q & \partial_\double+S_\double
\end{pmatrix}
\biggl(
\partial(\gamma(L_\Lieg)(\cdot-\rho))+\begin{pmatrix}
F_{\gamma(L_\Lieg)} & 0\\
0 & L_\double
\end{pmatrix}
\biggr).
\end{multline}
On the other hand, it follows from Proposition \ref{prop:Cherednik} (i)
that $\mathscr T(\gamma(L_\Lieg))\circ(\varphi \circ \psi)
=(\mathscr T(\gamma(L_\Lieg))\circ\varphi) \circ \psi$
for any $\varphi\in\Hom_{\mathbf H}(P_{\mathbf H}(V^M_\single),C^\infty(Wx))
\simeq\Hom_\CC(V^M_\single,C^\infty_x)$.
Hence Lemma \ref{lem:abstChered},
Lemma \ref{lem:abstChered2} and \eqref{eq:partilapart}
imply the matrix identity:
\begin{equation}\label{eq:matId2}
\bigl(\partial(\gamma(L_\Lieg)(\cdot-\rho))+F'_{\gamma(L_\Lieg)}\bigr)
\bigl(\partial_\single+T\bigr)
=
\bigl(\partial_\single+T\bigr)
\bigl(\partial(\gamma(L_\Lieg)(\cdot-\rho))+F_{\gamma(L_\Lieg)}\bigr).
\end{equation}
Now by comparing the upper-left blocks in \eqref{eq:matId1}
we get
\begin{equation}\label{eq:matId3}
\bigl(\partial(\gamma(L_\Lieg)(\cdot-\rho))+F'_{\gamma(L_\Lieg)}\bigr)
\bigl(\partial_\single+S_\single\bigr)
=
\bigl(\partial_\single+S_\single\bigr)
\bigl(\partial(\gamma(L_\Lieg)(\cdot-\rho))+F_{\gamma(L_\Lieg)}\bigr).
\end{equation}
Subtracting \eqref{eq:matId2} from \eqref{eq:matId3},
\begin{equation}\label{eq:matId4}
\bigl[
\partial(\gamma(L_\Lieg)(\cdot-\rho)),\,
S_\single-T
\bigr]=
\bigl(S_\single-T\bigr)F_{\gamma(L_\Lieg)}-F'_{\gamma(L_\Lieg)}\bigl(S_\single-T\bigr).
\end{equation}
Applying Lemma \ref{lem:Ltrick} to \eqref{eq:matId4} we conclude 
\begin{equation}\label{eq:ST}
S_\single=T.
\end{equation}

If $\Phi \in \Hom_K^\ttt(V,C^\infty(K\bar x))$,
that is, if $\Gamma^V_{0,x}(\Phi)[v_j]=0$ for $j=m'+1,\ldots,m$, then
for $i=1,\ldots,\mu'$
\begin{align*}
\tilde\Gamma^E_{0,x}(\Phi\circ\Psi)[e_i]
&=\Gamma^E_{0,x}(\Phi\circ\Psi)[e_i]
&&(\text{by definition})\\
&=\sum_{j=1}^{m}
\Bigl(
\partial\bigl(\gamma(D_{ij})(\cdot-\rho)\bigr)+S_{ij}
\Bigr)\,\Gamma^V_{0,x}(\Phi)[v_j]
&&(\because\text{Lemma \ref{lem:abstRad2}})\\
&=\sum_{j=1}^{m'}
\Bigl(
\partial\bigl(\gamma(D_{ij})(\cdot-\rho)\bigr)+S_{ij}
\Bigr)\,\tilde\Gamma^V_{0,x}(\Phi)[v_j]\\
&=\sum_{j=1}^{m'}
\Bigl(
\partial\bigl(f_{ij}(\cdot-\rho)\bigr)+T_{ij}\Bigr)\,\tilde\Gamma^V_{0,x}(\Phi)[v_j]
&&(\because\text{\eqref{eq:partilapart} and \eqref{eq:ST}})\\
&=\bigl(\tilde\Gamma_{0,x}(\Phi)\circ \tilde\Gamma^E_V(\Psi)\bigr)[e_i].
&&(\because\text{Lemma \ref{lem:abstChered2}})
\end{align*}
This proves Theorem \ref{thm:radD} (i).

Next, suppose $\Phi \in \Hom_K(V,C^\infty(K\bar x))$ is arbitrary
and $\Psi\in\Hom_K^\oto(E,P_G(V))$.
Then $\theta\circ\Psi\in\Hom_K^\oto(E,P_G(V))$ by \eqref{eq:Cartan2}.
Thus $\gamma(D_{ij})=0$ for $i=1,\ldots,\mu'$ and $j=m'+1,\ldots,m$
in view of \eqref{eq:thetaPsi},
which means $\partial_P=0$.
Hence by comparing the upper-right blocks in \eqref{eq:matId1} we get
\[
\bigl[
\partial(\gamma(L_\Lieg)(\cdot-\rho)),\,
P
\bigr]=
P L_\double - F'_{\gamma(L_\Lieg)} P.
\]
Applying Lemma \ref{lem:Ltrick} to this, we obtain $P=0$.
Thus for $i=1,\ldots,\mu'$
\begin{align*}
\tilde\Gamma^E_{0,x}(\Phi\circ\Psi)[e_i]
&=\sum_{j=1}^{m}
\Bigl(
\partial\bigl(\gamma(D_{ij})(\cdot-\rho)\bigr)+S_{ij}
\Bigr)\,\Gamma^V_{0,x}(\Phi)[v_j]\\
&=\sum_{j=1}^{m'}
\Bigl(
\partial\bigl(\gamma(D_{ij})(\cdot-\rho)\bigr)+S_{ij}
\Bigr)\,\Gamma^V_{0,x}(\Phi)[v_j]\\
&\qquad\qquad
+\sum_{j=1}^{m-m'}
\Bigl(
(\partial_P)_{ij}+P_{ij}
\Bigr)\,\Gamma^V_{0,x}(\Phi)[v_{m'+j}]\\
&=\sum_{j=1}^{m'}
\Bigl(
\partial\bigl(\gamma(D_{ij})(\cdot-\rho)\bigr)+S_{ij}
\Bigr)\,\tilde\Gamma^V_{0,x}(\Phi)[v_j].
\end{align*}
Hence the same calculation as before proves Theorem \ref{thm:radD} (ii).

Finally, suppose $\Phi \in \Hom_K^\ttt(V,C^\infty(K\bar x))$
and $\Psi\in\Hom_K^\ttt(E,P_G(V))$.
Since $\theta\circ\Psi\in\Hom_K^\ttt(E,P_G(V))$ by \eqref{eq:Cartan3},
$\gamma(D_{ij})=0$ for $i=\mu'+1,\ldots,\mu$ and $j=1,\ldots,m'$.
This means $\partial_Q=0$.
By comparing the lower-left blocks in \eqref{eq:matId1} we get
\[
\bigl[
\partial(\gamma(L_\Lieg)(\cdot-\rho)),\,
Q
\bigr]=
Q F_{\gamma(L_\Lieg)} - L'_\double Q
\]
and hence $Q=0$.
Since $\Gamma^V_{0,x}(\Phi)[v_j]=0$ for $j=m'+1,\ldots,m$,
we calculate for $i=\mu'+1,\ldots,\mu$
\begin{align*}
\Gamma^E_{0,x}(\Phi\circ\Psi)[e_i]
&=\sum_{j=1}^{m}
\Bigl(
\partial\bigl(\gamma(D_{ij})(\cdot-\rho)\bigr)+S_{ij}
\Bigr)\,\Gamma^V_{0,x}(\Phi)[v_j]\\
&=\sum_{j=1}^{m'}
\Bigl(
(\partial_Q)_{i-\mu'\,j}+Q_{i-\mu'\,j}
\Bigr)\,\Gamma^V_{0,x}(\Phi)[v_{j}]\\
&=0.
\end{align*}
This shows $\Phi\circ\Psi\in \Hom_K^\ttt(E,C^\infty(K\bar x))$,
proving Theorem \ref{thm:radD} (iii).

\section{Correspondences of submodules}\label{sec:cor}
In this section we study
various correspondences
which send $\mathbf H$-modules to $(\Lieg_\CC,K)$-modules.
One example is $\Ximin_0$ which maps
an $\mathbf H$-submodule of $C^\infty(A)$ to a $(\Lieg_\CC,K)$-submodule
of $C^\infty(G/K)_\Kf$.
Another example
is $\Ximin$ which maps
an $\mathbf H$-submodule of $P_{\mathbf H}(\CC_\triv)$ to
a $(\Lieg_\CC,K)$-submodule of $P_{G}(\CC_\triv)$.
Both correspondences have a nice property on the multiplicities of
(quasi-) \!\!single-petaled $K$-types.
This property for $\Ximin_0$ comes from
the generalized Chevalley restriction theorem (Theorem \ref{thm:Ch})
and the radial part formula in the last section (Theorem \ref{thm:radD}),
and for $\Ximin$ from the generalized Harish-Chandra isomorphism (Theorem \ref{thm:HC} (ii))
and the functoriality of
the generalized Harish-Chandra homomorphism (Proposition \ref{prop:chain}).
Motivated by this obvious formal parallelism
we shall develop a unified argument by
introducing the following three categories $\CCh$,
$\Cwrad$ and $\Crad$:
\begin{defn}[the category $\CCh$]\label{defn:CCh}
An object of $\CCh$ is
a pair $\mathcal M=(\mathcal M_G, \mathcal M_{\mathbf H})$
of a $(\Lieg_\CC,K)$-module $\mathcal M_G$
and an $\mathbf H$-module $\mathcal M_{\mathbf H}$
satisfying:
all the $K$-types of $\mathcal M_G$ belong to $\Km$;
to each $V\in\Km$ there attach a linear map
$\tilde\Gamma^V_{\mathcal M}: \Hom_K(V,\mathcal M_G)\to \Hom_W(V^M_\single,\mathcal M_{\mathbf H})$
and a linear subspace $\Hom_K^\ttt(V,\mathcal M_G)$ of $\Hom_K(V,\mathcal M_G)$ such that
the restriction of $\tilde\Gamma^V_{\mathcal M}$ to $\Hom_K^\ttt(V,\mathcal M_G)$
gives a bijection
\begin{gather}\label{cond:Ch}
\tilde\Gamma_{\mathcal M}^V :\,
 \Hom_K^\ttt(V,\mathcal M_{G})
\simarrow
\Hom_W(V^M_\single,\mathcal M_{\mathbf H}).\tag{Ch-0}
\end{gather}
Suppose $\mathcal M=(\mathcal M_G, \mathcal M_{\mathbf H})$,
$\mathcal N=(\mathcal N_G, \mathcal N_{\mathbf H})\in\CCh$.
Then a morphism of $\CCh$ between them
is a pair $\mathcal I=(\mathcal I_G, \mathcal I_{\mathbf H})$
of a $(\Lieg_\CC,K)$-homomorphism $\mathcal I_G : \mathcal M_G\to \mathcal N_G$
and an $\mathbf H$-homomorphism $\mathcal I_{\mathbf H} : \mathcal M_{\mathbf H}\to \mathcal N_{\mathbf H}$
which satisfies the following two conditions:
\begin{enumerate}[label=(Ch-\arabic{*}), ref=Ch-\arabic{*}]
\item\label{cond:Ch1}
For any $V\in\Km$ the diagram
\[
\xymatrix{
\Hom_K(V,\mathcal M_G) \ar[r]^{\mathcal I_G\circ\cdot}
\ar[d]_{\tilde\Gamma^V_{\mathcal M}}
& \Hom_K(V,\mathcal N_G) \ar[d]^{\tilde\Gamma^V_{\mathcal N}}\\
\Hom_W(V^M_\single,\mathcal M_{\mathbf H}) \ar[r]_{\mathcal I_{\mathbf H}\circ\cdot} & \Hom_W(V^M_\single,\mathcal N_{\mathbf H})
}
\]
commutes.
\item\label{cond:Ch2}
For any $V\in\Km$ and $\Phi\in \Hom_K^\ttt(V,\mathcal M_G)$,
$\mathcal I_G\circ\Phi\in \Hom_K^\ttt(V,\mathcal N_G)$,
namely, the following map is well defined:
\[
\xymatrix{
\Hom_K^\ttt(V,\mathcal M_G) \ar[r]^{\mathcal I_G\circ\cdot}
& \Hom_K^\ttt(V,\mathcal N_G).
}
\]
\end{enumerate}
\end{defn}
\begin{rem}
Suppose $\mathcal M=(\mathcal M_G, \mathcal M_{\mathbf H})\in\CCh$.
For each $V\in\Km$ we have the direct sum decomposition
\[
\Hom_K(V,\mathcal M_{G})=\Ker \tilde\Gamma^V_{\mathcal M} \,\oplus\, \Hom_K^\ttt(V,\mathcal M_{G}).
\]
For $V\in\Km\setminus\Kqsp$ it necessarily holds that
 $\tilde\Gamma^V_{\mathcal M}=0$ and $\Hom_K^\ttt(V,\mathcal M_G)=\{0\}$.
Hence any pair of 
a $(\Lieg_\CC,K)$-homomorphism $\mathcal I_G : \mathcal M_G\to \mathcal N_G$
and an $\mathbf H$-homomorphism
$\mathcal I_{\mathbf H} : \mathcal M_{\mathbf H}\to \mathcal N_{\mathbf H}$
automatically satisfies 
Conditions \eqref{cond:Ch1} and \eqref{cond:Ch2}
for $V\in\Km\setminus\Kqsp$. 
\end{rem}
\begin{defn}[the category $\Cwrad$]\label{defn:Cwrad}
We call an object of $\mathcal M=(\mathcal M_G, \mathcal M_{\mathbf H})\in\CCh$ a \emph{weak radial pair} if it satisfies
\begin{enumerate}[label=(w-rad),ref=w-rad]
\item\label{cond:w-rad}
for any $E,V\in\Km$, $\Phi \in \Hom_{\Lieg_\CC,K}^{\ttt}(P_G(V), \mathcal M_{G})$,
and $\Psi\in \Hom_W(E,P_G(V))$ it holds that
\[
\tilde\Gamma^E_{\mathcal M}(\Phi\circ\Psi)
=\tilde\Gamma_{\mathcal M}(\Phi)\circ\tilde\Gamma^E_V(\Psi).
\]
\end{enumerate}
Here $\Hom_{\Lieg_\CC,K}^{\ttt}(P_G(V), \mathcal M_{G})$ denotes
the subspace of $\Hom_{\Lieg_\CC,K}(P_G(V), \mathcal M_{G})$
corresponding to $\Hom_K^\ttt(V,\mathcal M_G)$
under the identification $\Hom_{\Lieg_\CC,K}(P_G(V), \mathcal M_{G})
\simeq \Hom_K(V,\mathcal M_G)$ and
\[
\tilde \Gamma_{\mathcal M}:\,
\Hom_{\Lieg_\CC,K}(P_G(V), \mathcal M_{G})
\to
\Hom_{\mathbf H}(P_{\mathbf H}(V^M_\single), \mathcal M_{\mathbf H})
\]
is the map identified with
\[
\tilde\Gamma_{\mathcal M}^V :\,
 \Hom_K(V,\mathcal M_G)
\longrightarrow
\Hom_W(V^M_\single,\mathcal M_{\mathbf H}).
\]
The category $\Cwrad$ is the full subcategory of $\CCh$
consisting of the weak radial pairs.
\end{defn}
\begin{rem}
If $E\in \Km\setminus \Kqsp$ or $V\in \Km\setminus \Kqsp$
then the formula in \eqref{cond:w-rad} is automatic.
\end{rem}
\begin{defn}[the category $\Crad$]\label{defn:Crad}
We call an object of $\mathcal M=(\mathcal M_G, \mathcal M_{\mathbf H})\in\Cwrad$
a \emph{radial pair} if it satisfies the following two additional conditions:
\begin{enumerate}[label=(rad-\arabic{*}), ref=rad-\arabic{*}]
\item\label{cond:rad1}
For any $E,V\in\Km$, $\Phi \in \Hom_{\Lieg_\CC,K}(P_G(V), \mathcal M_{G})$,
and $\Psi\in \Hom_W^{\oto}(E,P_G(V))$ it holds that
\[
\tilde\Gamma^E_{\mathcal M}(\Phi\circ\Psi)
=\tilde\Gamma_{\mathcal M}(\Phi)\circ\tilde\Gamma^E_V(\Psi).
\]
\item\label{cond:rad2}
For any $E,V\in\Km$, $\Phi \in \Hom_{\Lieg_\CC,K}^{\ttt}(P_G(V), \mathcal M_{G})$,
and $\Psi\in \Hom_W^{\ttt}(E,P_G(V))$ it holds that
\[
\Phi\circ\Psi \in \Hom_{K}^{\ttt}(E, \mathcal M_{G}).
\]
\end{enumerate}
The category $\Crad$ is the full subcategory of $\Cwrad$
consisting of the radial pairs.
\end{defn}
\begin{rem}
If $E\in \Km\setminus \Kqsp$ then 
the formula in \eqref{cond:rad1} is automatic.
If $V\in \Km\setminus \Kqsp$ then 
the condition in \eqref{cond:rad2} is automatic.
\end{rem}
\begin{defn}[radial restrictions]\label{defn:RadRest}
Suppose $\mathcal M=(\mathcal M_G, \mathcal M_{\mathbf H})\in\CCh$.
We say a linear map $\gamma_{\mathcal M}: \mathcal M_G\to \mathcal M_{\mathbf H}$
is a \emph{radial restriction} of $\mathcal M$
if it satisfies the following two conditions:
\begin{enumerate}[label=(rest-\arabic{*}), ref=rest-\arabic{*}]
\item\label{cond:rest1}
For any $V\in\Km$
\begin{equation*}
\Hom_K^{\ttt}(V, \mathcal M_{G})
=\bigl\{
\Phi\in \Hom_K(V, \mathcal M_{G});\,
(\gamma_{\mathcal M}\circ\Phi)\bigl[ V^M_\double\bigr]=\{0\}
\bigr\}.
\end{equation*}
\item\label{cond:rest2}
For any $V\in\Km$,
$\tilde\Gamma_{\mathcal M}^V$ coincides with the linear map
\begin{equation*}
\quad\qquad\Hom_K(V,\mathcal M_G)\ni 
\Phi\longmapsto
\bigl(
V^M_\single\hookrightarrow
V
\xrightarrow{\Phi} \mathcal M_{G}
\xrightarrow{\gamma_{\mathcal M}} \mathcal M_{\mathbf H}
\bigr)
\in
\Hom_\CC(V^M_\single,\mathcal M_{\mathbf H}).
\end{equation*}
(Hence actually the rightmost part can be replaced by $\Hom_W$.)
\end{enumerate}
Note $\gamma_{\mathcal M}$ completely determines the structure of $\mathcal M$.
\end{defn}
\begin{rem}\label{rem:RadR}
Suppose $\mathcal M=(\mathcal M_G, \mathcal M_{\mathbf H})\in\CCh$
has a radial restriction $\gamma_{\mathcal M}$.
Then $\Hom_K^\ttt(V,\mathcal M_{G})=\Hom_K(V,\mathcal M_{G})$
for $V\in\Ksp$.
For each $V\in\Km$
we can define in addition to $\tilde\Gamma_{\mathcal M}^V$
a linear map
\begin{equation}\label{eq:genGammaV}
\begin{aligned}
\Gamma_{\mathcal M}^V :\,&
 \Hom_K(V,\mathcal M_G)
\longrightarrow
\Hom_\CC(V^M,\mathcal M_{\mathbf H})\,;\\
&\Phi\longmapsto
\bigl(
\varphi:
V^M\hookrightarrow
V
\xrightarrow{\Phi} \mathcal M_{G}
\xrightarrow{\gamma_{\mathcal M}} \mathcal M_{\mathbf H}
\bigr).
\end{aligned}
\end{equation}
This is injective since for $\Phi\in \Hom_K(V,\mathcal M_{G})$ we have
\[
\Gamma_{\mathcal M}^V(\Phi)=0
\Longleftrightarrow
\tilde\Gamma_{\mathcal M}^V(\Phi)=0 \text{ and }
\Phi\in\Hom_K^\ttt(V,\mathcal M_{G})
\Longleftrightarrow
\Phi=0.
\]
\end{rem}
Let us look at some examples:
\begin{prop}\label{prop:R3GK}
$\bigl((\ell(\cdot),C^\infty(G/K)_\Kf),\,
(\mathscr T(\theta_{\mathbf H}\cdot),C^\infty(A))\bigr)$ is a radial pair with
radial restriction $\gamma_0$.
\end{prop}
\begin{proof}
It is well known (or easy to show) that all the $K$-types of $C^\infty(G/K)_\Kf$
belong to $\Km$.
The other conditions are satisfied
by Theorem \ref{thm:Ch}, \eqref{eq:TGammaV0} and Theorem \ref{thm:radD}.
\end{proof}

Recall the Harish-Chandra homomorphism $\gamma : P_G(\CC_\triv)\to P_{\mathbf H}(\CC_\triv)$ in {\normalfont Example \ref{exmp:triv}}.
\begin{prop}\label{prop:R3HC}
$\bigl(P_G(\CC_\triv), P_{\mathbf H}(\CC_\triv) \bigr)$ is a radial pair with
radial restriction $\gamma$.
\end{prop}
\begin{proof}
Recall $\Lies$ is the vector part of the Cartan decomposition of $\Lieg$.
If $\mathbf S$ denotes the image of the symmetrization map of $S(\Lies_\CC)$
then one has $P_G(\CC_\triv)\simeq \mathbf S\otimes v_\triv$.
Thus a $K$-type of $P_G(\CC_\triv)$ is
a $K$-type of $S(\Lies_\CC)$ and hence is belonging to $\Km$
\cite{KR}.
The other conditions follow from Theorem \ref{thm:HC} (i), (ii) and Proposition \ref{prop:chain}.
\end{proof}
\begin{prop}\label{prop:class1}
Suppose $\mathcal M=(\mathcal M_G, \mathcal M_{\mathbf H})\in \Cwrad$
satisfies \eqref{cond:rad2}
and $\mathcal M_{\mathbf H}$ has a $W$-invariant element $\phi_W$. 
Fix a non-zero $v_\triv\in \CC_\triv$.
Then there exists a unique $K$-invariant element
$\phi_K$ in $\mathcal M_G$ such that
under the bijection
\[
\tilde\Gamma_{\mathcal M}^{\CC_\triv}:
\Hom_K^\ttt(\CC_\triv, \mathcal M_G)
\simarrow
\Hom_W(\CC_\triv, \mathcal M_{\mathbf H})
\]
$(\CC_\triv\ni c\,v_\triv\mapsto c\,\phi_K \in \mathcal M_G)$
corresponds to $(\CC_\triv\ni c\,v_\triv\mapsto c\,\phi_W \in \mathcal M_{\mathbf H})$.
(If $\mathcal M$ has a radial restriction $\gamma_{\mathcal M}$
then $\phi_K$ is a unique $K$-invariant element such that $\gamma_{\mathcal M}(\phi_K)=\phi_W$.)
If we define
\begin{align*}
\mathcal I_G &: P_G(\CC_\triv)=U(\Lieg_\CC)\otimes_{U(\Liek_\CC)}\CC_\triv \ni
D\otimes v_\triv \longmapsto D\,\phi_K \in \mathcal M_G,\\
\mathcal I_{\mathbf H} &: P_{\mathbf H}(\CC_\triv)={\mathbf H}\otimes_{\CC W}\CC_\triv \ni
h\otimes v_\triv \longmapsto h\,\phi_W \in \mathcal M_{\mathbf H}
\end{align*}
then $\mathcal I=(\mathcal I_G,\mathcal I_{\mathbf H}) :
\bigl(P_G(\CC_\triv), P_{\mathbf H}(\CC_\triv)\bigr) \to \mathcal M$ is a morphism of $\Cwrad$.
\end{prop}
\begin{proof}
Let $\Phi:=(\CC_\triv\ni c\,v_\triv\mapsto c\,\phi_K \in \mathcal M_G)\in \Hom_K^\ttt(\CC_\triv, \mathcal M_G)$.
Thus 
$\tilde\Gamma^{\CC_\triv}_{\mathcal M}(\Phi)=
(\CC_\triv\ni c\,v_\triv\mapsto c\,\phi_W \in \mathcal M_{\mathbf H})\in
\Hom_W(\CC_\triv, \mathcal M_{\mathbf H})$.
Now for $V\in\Km$ and $\Psi\in\Hom_K(V, P_G(\CC_\triv))$ we have
\begin{align*}
&\mathcal I_G\circ\Psi = \Phi\circ\Psi \in \Hom_K^\ttt(V, \mathcal M_G)
\quad\text{if }\Psi\in\Hom_K^\ttt,
&&(\because \eqref{cond:rad2})\\
&\tilde\Gamma_{\mathcal M}^{V}(\mathcal I_G\circ\Psi)
=\tilde\Gamma_{\mathcal M}^{V}(\Phi\circ\Psi)
=\tilde\Gamma_{\mathcal M}(\Phi)\circ \tilde\Gamma_{\CC_\triv}^V(\Psi),
&&(\because \eqref{cond:w-rad})\\
&\mathcal I_{\mathbf H}\circ\tilde\Gamma_{\CC_\triv}^{V}(\Psi)
=\tilde\Gamma_{\mathcal M}(\Phi)\circ \tilde\Gamma_{\CC_\triv}^V(\Psi).
\end{align*}
The first property shows $\mathcal I$ satisfies \eqref{cond:Ch2}.
The second and third ones show $\mathcal I$ satisfies \eqref{cond:Ch1}.
\end{proof}

One can check the following basic properties of our categories
in a straightforward way.
\begin{prop}\label{prop:catpro}
The categories $\CCh$, $\Cwrad$ and $\Crad$ are Abelian categories.
Suppose 
$\mathcal I=(\mathcal I_G, \mathcal I_{\mathbf H}):
\mathcal M=(\mathcal M_G, \mathcal M_{\mathbf H})\to
\mathcal N=(\mathcal N_G, \mathcal N_{\mathbf H})$
is a morphism in $\CCh$.

\noindent
{\normalfont (i)}
$\mathcal I$ is mono if and only if both $\mathcal I_G$ and $\mathcal I_{\mathbf H}$
are injective.
$\mathcal I$ is epi if and only if both $\mathcal I_G$ and $\mathcal I_{\mathbf H}$
are surjective.

\noindent
{\normalfont (ii)}
Put $\mathcal K_G=\Ker \mathcal I_G$ and $\mathcal K_{\mathbf H}=\Ker \mathcal I_{\mathbf H}$.
Then we can identify\/
$\Ker \mathcal I\simeq\mathcal K:=(\mathcal K_G, \mathcal K_{\mathbf H})$
where for each $V\in\Km$ 
\[\begin{aligned}
\Hom_K^\ttt(V,\mathcal K_G)
&= \Hom_K(V,\mathcal K_G) \cap \Hom_K^\ttt(V,\mathcal M_G)\\
&=\bigl\{
\Phi\in \Hom_K^\ttt(V, \mathcal M_{G});\,
\tilde \Gamma^V_{\mathcal M}(\Phi) \in \Hom_W(V^M_\single, \mathcal K_{\mathbf H})
\bigr\}.
\end{aligned}\]
If $\mathcal M$ satisfies \eqref{cond:w-rad} (that is $\mathcal M\in\Cwrad$)
then so does $\mathcal K$.
The same thing holds for \eqref{cond:rad1} or \eqref{cond:rad2}.

\noindent
{\normalfont (iii)}
Put $\mathcal Q_G=\Coker \mathcal I_G$ and $\mathcal Q_{\mathbf H}=\Coker \mathcal I_{\mathbf H}$.
Then we can identify
$\Coker \mathcal I\simeq \mathcal Q:=(\mathcal Q_G, \mathcal Q_{\mathbf H})$
where for each $V\in\Km$ 
\[\begin{aligned}
\Hom_K^\ttt(V,\mathcal Q_G)
=\, \text{the image of }&
\Hom_K^\ttt(V, \mathcal N_{G})\\
\text{ under }&
\Hom_K(V,\mathcal N_G)\twoheadrightarrow\Hom_K(V,\mathcal Q_G).
\end{aligned}\]
If $\mathcal N$ satisfies \eqref{cond:w-rad} (that is $\mathcal N\in\Cwrad$)
then so does $\mathcal Q$.
The same thing holds for \eqref{cond:rad1} or \eqref{cond:rad2}.
\end{prop}
The next lemma will be repeatedly used:
\begin{lem}\label{lem:rest3}
Suppose
$\mathcal M=(\mathcal M_G,\mathcal M_{\mathbf H})\in\CCh$
and a linear map $\gamma_{\mathcal M}: \mathcal M_G\to \mathcal M_{\mathbf H}$
satisfies {\normalfont Condition \eqref{cond:rest2}} 
and the following two conditions:
\begin{enumerate}[label={\normalfont (rest-1')}, ref=rest-1']
\item\label{cond:rest1'}
For any $V\in\Km$
\begin{equation*}
\Hom_K^{\ttt}(V, \mathcal M_{G})
\subset\bigl\{
\Phi\in \Hom_K(V, \mathcal M_{G});\,
(\gamma_{\mathcal M}\circ\Phi)\bigl[ V^M_\double\bigr]=\{0\}
\bigr\}.
\end{equation*}
\end{enumerate}
\begin{enumerate}[label={\normalfont (rest-3)}, ref=rest-3]
\item\label{cond:rest3}
If $y\in\mathcal M_G$ then
\begin{align}
\gamma_{\mathcal M}(Hy)&=(H+\rho(H))\gamma_{\mathcal M}(y)
&\text{for }H\in\Liea_\CC,\label{eq:gammaH}\\
\gamma_{\mathcal M}(Xy)&=0
&\text{for }X\in\Lien_\CC,\label{eq:gammaX}\\
\gamma_{\mathcal M}(my)&=0
&\text{for }m\in M.\label{eq:gammam}
\end{align}
\end{enumerate}
Then $\mathcal M$ is a weak radial pair satisfying \eqref{cond:rad1}.
Moreover, if $\gamma_{\mathcal M}$ satisfies \eqref{cond:rest1}
then $\mathcal M$ is a radial pair.
\end{lem}
\begin{rem}\label{eq:goodRR}
In the setting of the lemma
let $\gamma^{\mathcal M_G}: \mathcal M_G\to\Gamma(\mathcal M_G)$ be the linear map defined by \eqref{eq:nhomogamma}.
Then Condition \eqref{cond:rest3} holds if and only if there exists an $\Liea_\CC$-homomorphism
$\delta^{\mathcal M_G}: \Gamma(\mathcal M_G)\to \mathcal M_{\mathbf H}$
such that $\gamma_{\mathcal M}=\delta^{\mathcal M_G}\circ\gamma^{\mathcal M_G}$.
\end{rem}
\begin{exmp}\label{exmp:rest3}
The radial restriction $\gamma$ for the radial pair
$\bigl(P_G(\CC_\triv), P_{\mathbf H}(\CC_\triv) \bigr)$
satisfies \eqref{cond:rest3} since $\gamma^{P_G(\CC_\triv)}=\gamma$.
\end{exmp}
\begin{proof}[Proof of {\normalfont Lemma \ref{lem:rest3}}]
Suppose $E,V\in\Km$
and let $\{v_1,\ldots,v_{m'}\}$, $\{v_{m'+1},\ldots,v_m\}$ and $\{v_{m+1},\ldots,v_n\}$
be bases of\/ $V^M_\single$, $V^M_\double$ and $(V^M)^\perp$.
Let $\Phi\in\Hom_{\Lieg_\CC,K}(P_G(V),\mathcal M_G)$,
$\Psi\in\Hom_K(E,P_G(V))$.
For any $e\in E^M$ let
\[
\Psi[e]=\sum_{i=1}^n D_i\otimes v_i\quad
\text{with }D\in U(\Lien_\CC+\Liea_\CC).
\]
Then $\Gamma^E_V(\Psi)[e]=\sum_{i=1}^{m}\gamma(D_i)\otimes v_i$ and
\begin{equation}\label{eq:lemrest}
\begin{aligned}
(\gamma_{\mathcal M}\circ\Phi\circ\Psi)[e]
&=\sum_{i=1}^n \gamma_{\mathcal M}(D_i\Phi[v_i])\\
&=\sum_{i=1}^n \gamma(D_i) \gamma^{\mathcal M}(\Phi[v_i])
&&(\because \eqref{eq:gammaH}, \eqref{eq:gammaX})\\
&=\sum_{i=1}^m \gamma(D_i) \gamma^{\mathcal M}(\Phi[v_i])
&&(\because \eqref{eq:gammam})\\
&=\sum_{i=1}^{m'} \gamma(D_i) \tilde\Gamma_{\mathcal M}(\Phi)[v_i]
+\sum_{i=m'+1}^m \gamma(D_i) (\gamma^{\mathcal M}\circ \Phi)[v_i].
&&(\because \eqref{cond:rest2})
\end{aligned}
\end{equation}
Here the first summand is zero when $\Psi\in\Hom_{K}^\ttt$ and $e\in E^M_\double$
since in this case $\gamma(D_i)=0$ for $i=1,\ldots,m'$.
Also, the second summand is zero when $\Phi\in\Hom_{\Lieg_\CC,K}^\ttt$
since in this case $(\gamma^{\mathcal M}\circ \Phi)[v_i]=0$
for $i=m'+1,\ldots,m$ by \eqref{cond:rest1'}.

Hence if $\Phi\in\Hom_{\Lieg_\CC,K}^\ttt$ and
$\Psi\in\Hom_{K}^\ttt$ then
\[
\gamma_{\mathcal M}\circ(\Phi\circ\Psi)\bigl[ V^M_\double \bigr] =\{0\}.
\]
This means
\eqref{cond:rest1} implies \eqref{cond:rad2}
under the assumption of the lemma.

Now suppose $e\in E^M_\single$.
Then \eqref{eq:lemrest} reduces to
\begin{equation*}
\tilde\Gamma^E_{\mathcal M}(\Phi\circ\Psi)[e]
= 
(\tilde\Gamma_{\mathcal M}(\Phi)\circ \tilde\Gamma^E_V(\Psi))[e]
+\sum_{i=m'+1}^m \gamma(D_i) (\gamma^{\mathcal M}\circ \Phi)[v_i].
\end{equation*}
Since the second summand is zero when $\Phi\in\Hom_{\Lieg_\CC,K}^\ttt$,
Condition \eqref{cond:w-rad} is satisfied.
The second summand also vanishes when $\Psi\in\Hom_{K}^\oto$
since in this case $\gamma(D_i)=0$ for $i=m'+1,\ldots,m$.
Thus Condition \eqref{cond:rad1} is satisfied.
\end{proof}
In this paper we shall introduce various functors
connecting the category $\mathbf H\Mod$ of $\mathbf H$-modules
to the category $(\Lieg_\CC,K)\Mod$ of $(\Lieg_\CC,K)$-modules.
Each one is a descendant of the following:
\begin{defn}[the functor $\Xiwrad$]\label{defn:Xiwrad}
For each $V\in\Km$ 
we define a $(\Lieg_\CC,K)$-subspace $Q_G(V)$
of $P_G(V)$ by
\[
Q_G(V)=
\sum_{E\in\widehat K\setminus\Km}
U(\Lieg_\CC)\,\bigl(\text{the $E$-isotypic component of $P_G(V)$}\bigr)
\]
and put $\bar P_G(V):=P_G(V)/Q_G(V)$.
For $\mathscr X\in \mathbf H\Mod$ we define a $(\Lieg_\CC,K)$-module
\begin{align*}
\Xiwrad(\mathscr X)
&=\bigoplus_{V\in\Km} \bar P_G(V)\otimes \Hom_W(V^M_\single,\mathscr X)\\
&\simeq \bigoplus_{V\in\Km} \bar P_G(V)\otimes \Hom_{\mathbf H}(P_{\mathbf H}(V^M_\single),\mathscr X)\end{align*}
where $\Lieg_\CC$ and $K$ act only on the $\bar P_G(V)$-parts.
The correspondence
$\Xiwrad: \mathbf H\Mod \to (\Lieg_\CC,K)\Mod$
clearly defines an exact functor .
\end{defn}
The next lemma accumulates easy properties of $\bar P_G(V)$. 
\begin{lem}\label{lem:PGbar}
Suppose $F,V\in\Km$.

\noindent
{\normalfont (i)}
All the $K$-types of $\bar P_G(V)$ belong to $\Km$.
The right exact functor $\Gamma$ defined in {\normalfont \S\ref{sec:HC}}
maps $Q_G(V)$ to $0$.
Hence $\Gamma(\bar P_G(V))=\Gamma(P_G(V))=P_{\mathbf H}(V^M)$
and the map $\gamma^{P_G(V)} : P_G(V)\to P_{\mathbf H}(V^M)$ defined by \eqref{eq:nhomogamma} factors through $\gamma^{\bar P_G(V)} : \bar P_G(V)\to P_{\mathbf H}(V^M)$.

\noindent
{\normalfont (ii)}
If all the $K$-types of a $(\Lieg_\CC,K)$-module $\mathscr Y$ belong to $\Km$
then
\[
\Hom_K(V, \mathscr Y)\simeq
\Hom_{\Lieg_\CC,K}(P_G(V), \mathscr Y)\simeq
\Hom_{\Lieg_\CC,K}(\bar P_G(V), \mathscr Y).
\]

\noindent
{\normalfont (iii)}
The surjective map
\begin{equation}\label{eq:gkPbar}
\Hom_{\Lieg_\CC,K}(P_G(V),P_G(F))\twoheadrightarrow
\Hom_{\Lieg_\CC,K}(\bar P_G(V),\bar P_G(F))
\end{equation}
is naturally induced from (and is identified with) the surjective map
\begin{equation}\label{eq:KPbar}
\Hom_{K}(V,P_G(F))\twoheadrightarrow
\Hom_{K}(V,\bar P_G(F)).
\end{equation}

\noindent
{\normalfont (iv)}
Let $\tilde\gamma^{\bar P_G(F)} : \bar P_G(F)\to P_{\mathbf H}(F^M_\single)$
be the composition of 
$\gamma^{\bar P_G(F)}:\bar P_G(F)\to P_{\mathbf H}(F^M)$
and
the projection $P_{\mathbf H}(F^M) = P_{\mathbf H}(F^M_\single)\oplus P_{\mathbf H}(F^M_\double) \to P_{\mathbf H}(F^M_\single)$.
Then the linear map $\tilde\Gamma^V_F : \Hom_K(V, P_G(F)) \to
\Hom_{W}(V^M_\single,P_{\mathbf H}(F^M_\single))$
equals the composition of \eqref{eq:KPbar}
and the linear map $\tilde{\bar\Gamma}^V_F
: \Hom_{K}(V,\bar P_G(F)) \to \Hom_{\CC}(V^M_\single,P_{\mathbf H}(F^M_\single))$ defined by
\[
\Psi\longmapsto
\bigl(
V^M_\single\hookrightarrow
V
\xrightarrow{\Psi} \bar P_G(F)
\xrightarrow{\tilde\gamma^{\bar P_G(F)}} P_{\mathbf H}(F^M_\single)
\bigr).
\]
Hence $\tilde{\bar\Gamma}^V_F(\Psi)\in\Hom_W$ for any $\Psi$.
\end{lem}
We give the pair
$(\Xiwrad(\mathscr X),\mathscr X)$
a structure as an object of $\Cwrad$.
For each $V\in\Km$  we identify
\[
\Hom_K(V,\Xiwrad(\mathscr X))
\simeq\bigoplus_{F\in\Km} \Hom_K(V,\bar P_G(F))\otimes
\Hom_{\mathbf H}(P_{\mathbf H}(F^M_\single),\mathscr X)
\]
and define the linear map
$\tilde\Gamma^V_\wrad: \Hom_K(V, \Xiwrad(\mathscr X))
\to \Hom_W(V^M_\single,\mathscr X)$ by
\begin{equation}\label{eq:Gwraddef}
\tilde\Gamma^V_\wrad:\,
\sum_{F\in\Km} \sum_{i=1}^{q_F}\Psi_F^i\otimes \varphi_F^i
\longmapsto
\sum_{F\in\Km} \sum_{i=1}^{q_F}\varphi_F^i\circ \tilde{\bar\Gamma}^V_F(\Psi_F^i).
\end{equation}
Moreover for each $V\in\Km$ we put
\[
\Hom_K^\ttt(V,\Xiwrad(\mathscr X))=
I_V\otimes \Hom_{\mathbf H}(P_{\mathbf H}(V^M_\single),\mathscr X)
\]
where $I_V\in\Hom_K(V,\bar P_G(V))$ denotes the map $V\ni v\mapsto 1\otimes v \bmod Q_G(V)\in \bar P_G(V)$.
Then clearly $(\Xiwrad(\mathscr X),\mathscr X)\in\CCh$ by these data. 

We define a linear map
\[
\gamma_\wrad:\Xiwrad(\mathscr X)=\bigoplus_{F\in\Km} \bar P_G(F)\otimes \Hom_{\mathbf H}(P_{\mathbf H}(F^M_\single),\mathscr X)
\longrightarrow\mathscr X
\]
by
\[
\bar P_G(F)\otimes \Hom_{\mathbf H}(P_{\mathbf H}(F^M_\single),\mathscr X) \ni 
D\otimes\varphi \longmapsto \varphi(\tilde\gamma^{\bar P_G(F)}(D)) \in \mathscr X.
\]
One can easily observe
\begin{lem}\label{lem:gammawrad}
For $(\Xiwrad(\mathscr X),\mathscr X)\in\CCh$
the linear map $\gamma_\wrad$
satisfies \eqref{cond:rest1'}, \eqref{cond:rest2} and \eqref{cond:rest3}.
\end{lem}
Hence from Lemma \ref{lem:rest3} we have
\begin{prop}\label{prop:Xiwrad0}
For $\mathscr X\in \mathbf H\Mod$,
$(\Xiwrad(\mathscr X),\mathscr X)$
is a weak radial pair satisfying \eqref{cond:rad1}.
\end{prop}
We thus get the exact functor
$\mathbf H\Mod\ni\mathscr X\mapsto (\Xiwrad(\mathscr X),\mathscr X)\in\Cwrad$,
which has the following universal property:
\begin{prop}\label{prop:Xiwrad}
The functor $\mathscr X\mapsto (\Xiwrad(\mathscr X),\mathscr X)$ 
is left adjoint to the functor $\Cwrad\ni (\mathcal M_G, \mathcal M_{\mathbf H})
\mapsto \mathcal M_{\mathbf H}\in \mathbf H\Mod$.
More precisely, if $\mathcal M=(\mathcal M_G, \mathcal M_{\mathbf H})\in\Cwrad$
and an $\mathbf H$-homomorphism $\mathcal I_{\mathbf H}:\mathscr X\to \mathcal M_{\mathbf H}$ are given,
then
there exists a unique $(\Lieg_\CC,K)$-homomorphism
$\mathcal I_G:\Xiwrad(\mathscr X)\to\mathcal M_G$ such that
$(\mathcal I_G,\mathcal I_{\mathbf H}): (\Xiwrad(\mathscr X),\mathscr X)
\to (\mathcal M_G, \mathcal M_{\mathbf H})$ is a morphism of $\Cwrad$.
\end{prop}
\begin{proof}
Using \eqref{cond:Ch} for $\mathcal M$
and Lemma \ref{lem:PGbar} (ii), we define $\mathcal I_G : \Xiwrad(\mathscr X)\to\mathcal M_G$ by
\begin{gather*}
\bigoplus_{F\in\Km} \bar P_G(F)\otimes \Hom_W(F^M_\single,\mathscr X)
\to
\bigoplus_{F\in\Km} \bar P_G(F)\otimes \Hom_{\Lieg_\CC,K}^\ttt(\bar P_G(F),\mathcal M_G)
\to \mathcal M_G;\\
 D_F\otimes \varphi_F\,\longmapsto\,
 D_F\otimes \Phi_F\text{ with }
\tilde\Gamma_{\mathcal M}(\Phi_F)=\mathcal I_{\mathbf H}\circ\varphi_F\,
\longmapsto\,
 \Phi_F(D_F).
\end{gather*}
This is clearly a $(\Lieg_\CC, K)$-homomorphism.
Now suppose $V\in\Km$ and express
any $\Phi\in\Hom_K(V,\Xiwrad(\mathscr X))$ as
\begin{equation}\label{eq:PhiPshiphi}
\Phi=\sum_{F\in\Km}\sum_{i=1}^{q_F} \bar \Psi_F^i\otimes \varphi_F^i
\quad
\text{with }
\left\{\begin{aligned}
&\Psi_F^i\in
\Hom_K(V, P_G(F)),\\
&\varphi_F^i\in\Hom_W(F^M_\single,\mathscr X),
\end{aligned}\right.
\end{equation}
where $\bar \Psi_F^i$ is the image of $\Psi_F^i$ under \eqref{eq:KPbar}.
For each $\varphi_F^i$ take $\Phi_F^i\in \Hom_K^\ttt(F,\mathcal M_G)$
so that 
$\tilde\Gamma^F_{\mathcal M}(\Phi_F^i)=\mathcal I_{\mathbf H}\circ\varphi_F^i$.
Then
\begin{align*}
\tilde\Gamma^V_{\mathcal M}(\mathcal I_G\circ\Phi)
&=\tilde\Gamma^V_{\mathcal M}
\Bigl(\sum_{F\in\Km}\sum_{i=1}^{q_F}
\Phi_F^i\circ\Psi_F^i\Bigr)\\
&=\sum_{F\in\Km}\sum_{i=1}^{q_F}
\tilde\Gamma_{\mathcal M}(\Phi_F^i)\circ\tilde\Gamma^V_F(\Psi_F^i)
&&(\because \eqref{cond:w-rad}\text{ for }\mathcal M)\\
&=\mathcal I_{\mathbf H}\circ\Bigl(
\sum_{F\in\Km}\sum_{i=1}^{q_F}
\varphi_F^i\circ\tilde\Gamma^V_F(\Psi_F^i)\Bigr)\\
&=\mathcal I_{\mathbf H}\circ\Bigl(
\sum_{F\in\Km}\sum_{i=1}^{q_F}
\varphi_F^i\circ\tilde{\bar\Gamma}^V_F(\bar\Psi_F^i)\Bigr)
&&(\because \text{Lemma \ref{lem:PGbar} (iv)})\\
&=\mathcal I_{\mathbf H}\circ \tilde\Gamma^V_\wrad(\Phi),
&&(\because \eqref{eq:Gwraddef})
\end{align*}
proving \eqref{cond:Ch1} for $(\mathcal I_G,\mathcal I_{\mathbf H})$.
If $\Phi\in\Hom_K^\ttt(V,\Xiwrad(\mathscr X))$
then $\Phi=I_V\otimes \varphi_V$
with some $\varphi_V\in \Hom_W(V^M_\single,\mathscr X)$.
In this case, by taking
a unique element $\Phi_V \in \Hom_K^\ttt(V,\mathcal M_G)$
such that
$\tilde\Gamma^V_{\mathcal M}(\Phi_V)=\mathcal I_{\mathbf H}\circ\varphi_V$,
we have $\mathcal I_G\circ\Phi=\Phi_V$.
This shows $(\mathcal I_G,\mathcal I_{\mathbf H})$ satisfies \eqref{cond:Ch2}.
Finally, to prove the uniqueness of $\mathcal I_G$,
assume $(\mathcal I'_G, \mathcal I_{\mathbf H}): (\Xiwrad(\mathscr X),\mathscr X)\to(\mathcal M_G,\mathcal M_{\mathbf H})$
is a morphism of $\Cwrad$.
For each $V\in\Km$,
$V\otimes \Hom_W(V^M_\single,\mathscr X)$ is contained in the $V$-isotypic component of 
$\Xiwrad(\mathscr X)$ and
\[
\Hom_K\bigl(V, V\otimes \Hom_W(V^M_\single,\mathscr X)\bigr)
=\Hom_K^\ttt(V,\, \Xiwrad(\mathscr X))
\]
by definition.
It follows that $\mathcal I'_G$ on $V\otimes \Hom_W(V^M_\single,\mathscr X)$
is determined by $\mathcal I_{\mathbf H}$,
and hence is equal to $\mathcal I_G$.
But since $\Xiwrad(\mathscr X)$ is spanned by
$\bigcup_{V\in\Km} V\otimes \Hom_W(V^M_\single,\mathscr X)$,
we conclude $\mathcal I'_G=\mathcal I_G$.
\end{proof}
Now suppose $\mathcal M=(\mathcal M_G,\mathcal M_{\mathbf H})\in\Cwrad$
and let us define two correspondences
$\Ximin_{\mathcal M}$ and $\Ximax_{\mathcal M}$ sending $\mathbf H$-submodules of $\mathcal M_{\mathbf H}$
to $(\Lieg_\CC,K)$-submodules of $\mathcal M_G$.
\begin{defn}[the correspondence $\Ximin_{\mathcal M}$]\label{defn:Ximin}
Suppose $\mathscr X$ is an $\mathbf H$-submodule of $\mathcal M_{\mathbf H}$.
For any $V\in\Km$ the $V$-isotypic component of $\mathcal M_G$
is naturally identified with $V\otimes\Hom_K(V,\mathcal M_G)$.
Using \eqref{cond:Ch}
we can think of
$\Hom_W(V^M_\single, \mathscr X)\subset \Hom_W(V^M_\single, \mathcal M_{\mathbf H})$ as a subspace of $\Hom_K^\ttt(V,\mathcal M_G)$.
Thus
\[
V\otimes \Hom_W(V^M_\single, \mathscr X)
\subset V\otimes \Hom_K^\ttt(V,\mathcal M_G)
\subset V\otimes \Hom_K(V,\mathcal M_G) \subset \mathcal M_G.
\]
Now we set
\begin{align*}
\Ximin_{\mathcal M}(\mathscr X)
&=\text{ the }U(\Lieg_\CC)\text{-span of }\sum_{V\in\Km}
V\otimes \Hom_W(V^M_\single, \mathscr X)\\
&=\sum_{V\in\Km}
U(\Lieg_\CC)\bigl(V\otimes \Hom_W(V^M_\single, \mathscr X)\bigr).
\end{align*}
This is a $(\Lieg_\CC,K)$-submodule
of $\mathcal M_G$.

If $\mathcal M=\bigl(C^\infty(G/K),C^\infty(A)\bigr)$,
we write $\Ximin_0$ for $\Ximin_{\mathcal M}$.
If $\mathcal M=\bigl(P_G(\CC_\triv), P_{\mathbf H}(\CC_\triv) \bigr)$,
we write $\Ximin$ for $\Ximin_{\mathcal M}$.
\end{defn}
\begin{defn}[the correspondence $\Ximax_{\mathcal M}$]\label{defn:Ximax}
Suppose $\mathscr X$ is an $\mathbf H$-submodule of $\mathcal M_{\mathbf H}$.
We define a $(\Lieg_\CC,K)$-submoudle $\Ximax_{\mathcal M}(\mathscr X) \subset \mathcal M_G$ by
\[
\Ximax_{\mathcal M}(\mathscr X)
=\sum \left\{
\mathscr Y\subset \mathcal M_G\,;\,
\begin{aligned}
\text{a }(\Lieg_\CC,K&)\text{-submodule such that}&\\
\tilde\Gamma^V_{\mathcal M}&\bigl(\Hom_K(V,\mathscr Y)\bigr)\subset
\Hom_W(V^M_\single,\mathscr X)\\
&\text{for any $V\in\Km$}
\end{aligned}
\right\}.
\]

If $\mathcal M=\bigl(C^\infty(G/K),C^\infty(A)\bigr)$,
we write $\Ximax_0$ for $\Ximax_{\mathcal M}$.
If $\mathcal M=\bigl(P_G(\CC_\triv), P_{\mathbf H}(\CC_\triv) \bigr)$,
we write $\Ximax$ for $\Ximax_{\mathcal M}$.
\end{defn}
\begin{thm}\label{thm:multcor}
Retain the setting of the above definitions.

\noindent{\normalfont (i)}
$\Ximin_{\mathcal M}(\mathscr X)\subset\Ximax_{\mathcal M}(\mathscr X)$.

\noindent{\normalfont (ii)}
Suppose a $(\Lieg_\CC,K)$-submodule
$\mathscr Y\subset \mathcal M_G$ is
such that\/ $\Ximin_{\mathcal M}(\mathscr X)\subset 
\mathscr Y\subset\Ximax_{\mathcal M}(\mathscr X)$.
Then $(\mathscr Y,\mathscr X)$ is a weak radial pair such that
the pair of inclusion maps $\mathscr Y\hookrightarrow \mathcal M_G$
and $\mathscr X\hookrightarrow \mathcal M_{\mathbf H}$
is a morphism.
In particular, for any $V\in\Km$ we have
\[
\Hom_K^\ttt(V,\mathscr Y)=
\Hom_K(V,\mathscr Y)\cap \Hom_K^\ttt(V,\mathcal M_G),
\]
and the bijection
\begin{equation}\label{eq:mod-cor}
\tilde\Gamma^V_{\mathcal M}:
\Hom_K^\ttt(V,\, \mathscr Y)\simarrow
\Hom_W(V^M_\single, \mathscr X).
\end{equation}
If $\mathcal M$ is a radial pair then
so is $(\mathscr Y,\mathscr X)$ by {\normalfont Proposition \ref{prop:catpro} (ii)}.

\noindent{\normalfont (iii)}
The weak radial pair $(\Ximin_{\mathcal M}(\mathscr X), \mathscr X)$ always
satisfies \eqref{cond:rad1}.
If $\mathcal M$ satisfies \eqref{cond:rad2}
then $(\Ximin_{\mathcal M}(\mathscr X), \mathscr X)$ is a radial pair.
In this case, for any $V\in\Ksp$ we have
\begin{equation}\label{eq:genspcor}
\Hom_K(V,\, \Ximin_{\mathcal M}(\mathscr X))=\Hom_K^\ttt(V,\, \Ximin_{\mathcal M}(\mathscr X))
\end{equation}
and hence the bijection
\[
\tilde\Gamma^V_{\mathcal M}:
\Hom_K(V,\, \Ximin_{\mathcal M}(\mathscr X))\simarrow
\Hom_W(V^M, \mathscr X).
\]
\end{thm}
\begin{exmp}\label{exmp:Wmax}
For any $\mathcal M=(\mathcal M_G,\mathcal M_{\mathbf H})\in\Cwrad$,
$\Ximin_{\mathcal M}(\{0\})=\{0\}$ and
$\Ximax_{\mathcal M}(\mathcal M_{\mathbf H})=\mathcal M_G$.
\end{exmp}
\begin{exmp}\label{exmp:AA}
Suppose $\lambda\in\Liea_\CC^*$.
Then \eqref{eq:rad8} and \eqref{eq:spec-cor2} imply
\[
\Ximin_0(\mathscr A(A,\lambda))
\subset \mathscr A(G/K,\lambda)_\Kf \subset \Ximax_0(\mathscr A(A,\lambda)).
\]
In particular $\bigl(\mathscr A(G/K,\lambda)_\Kf, \mathscr A(A,\lambda)\bigr)$
is a radial pair.
\end{exmp}
We regard that \eqref{eq:mod-cor} describes the correspondence of
the multiplicity of the $K$-representation $V$ in $\mathscr Y$
with the multiplicity of the $W$-representation $V^M_\single$ in $\mathscr X$
(cf.~Remark \ref{rem:RadR}).
\begin{proof}[Proof of {\normalfont Theorem \ref{thm:multcor}}]
Let $\mathcal I_{\mathbf H}: \mathscr X \hookrightarrow \mathcal M_{\mathbf H}$
is the inclusion map. It follows from Proposition \ref{prop:Xiwrad} that
there exists a unique $\mathcal I_G: \Xiwrad(\mathscr X)\to \mathcal M_G$
such that $(\mathcal I_G,\mathcal I_{\mathbf H})$ is a morphism of $\Cwrad$.
From the proof of Proposition \ref{prop:Xiwrad}
we can see
$\mathcal I_G$ maps
$v\otimes \varphi_V \in V\otimes \Hom_W(V^M_\single,\mathscr X)$
($V\in\Km$) to $\Phi_V[v]$
where $\Phi_V\in \Hom_K^\ttt(V,\mathcal M_G)$ is
a unique element
such that $\tilde\Gamma^V_{\mathcal M}(\Phi_V)=\varphi_V$.
Note $\Phi_V[v]$ is written as $v\otimes \Phi_V$
if we identify the $V$-isotypic component of $\mathcal M_G$
with $V\otimes\Hom_K(V,\mathcal M_G)$.
This shows $\Image \mathcal I_G=\Ximin_{\mathcal M}(\mathscr X)$.
Hence 
by Proposition \ref{prop:Xiwrad0} and Proposition \ref{prop:catpro} (iii),
$(\Ximin_{\mathcal M}(\mathscr X), \mathscr X)$
is a weak radial pair satisfying \eqref{cond:rad1}.
Hence in particular for any $V\in\Km$ we have
\begin{gather}
\tilde\Gamma^V_{\mathcal M}\bigl(\Hom_K(V,\,\Ximin_{\mathcal M}(\mathscr X))\bigr)
=
\Hom_W(V^M_\single,\mathscr X),\label{eq:gammamin}\\
\tilde\Gamma^V_{\mathcal M}\bigl(
\Hom_K(V,\,\Ximin_{\mathcal M}(\mathscr X))
\cap
\Hom_K^\ttt(V,\mathcal M_G)
\bigr)
=
\Hom_W(V^M_\single,\mathscr X).
\end{gather}
From \eqref{eq:gammamin} we get $\Ximin_{\mathcal M}(\mathscr X)\subset \Ximax_{\mathcal M}(\mathscr X)$.

In general, if $\{\mathscr Y_\nu\}$ is a family of $(\Lieg_\CC,K)$-submodules of $\mathcal M_G$ then
\[
\Hom_K\Bigl(V,\, \sum_\nu\mathscr Y_\nu\Bigr)=\sum_\nu\Hom_K(V, \mathscr Y_\nu)
\subset \Hom_K(V, \mathcal M_G).
\]
Hence for any $V\in\Km$
\begin{equation}\label{eq:gammamax}
\tilde\Gamma^V_{\mathcal M}\bigl(\Hom_K(V,\Ximax_{\mathcal M}(\mathscr X))\bigr)\subset
\Hom_W(V^M_\single,\mathscr X).
\end{equation}
Now suppose a $(\Lieg_\CC,K)$-submodule $\mathscr Y\subset \mathcal M_G$
is such that $\Ximin_{\mathcal M}(\mathscr X) \subset
\mathscr Y \subset \Ximax_{\mathcal M}(\mathscr X)$.
Then by virtue of \eqref{eq:gammamin}--\eqref{eq:gammamax} we have
\begin{gather*}
\tilde\Gamma^V_{\mathcal M}\bigl(\Hom_K(V,\mathscr Y)\bigr)
=
\Hom_W(V^M_\single,\mathscr X),\\
\tilde\Gamma^V_{\mathcal M}\bigl(
\Hom_K(V,\mathscr Y)
\cap
\Hom_K^\ttt(V,\mathcal M_G)
\bigr)
=
\Hom_W(V^M_\single,\mathscr X).
\end{gather*}
Hence $(\mathscr Y,\mathscr X)$ is a subobject of $(\mathcal M_G,\mathcal M_{\mathbf H})\in\CCh$.
Accordingly 
$(\mathscr Y,\mathscr X)\in\Cwrad$
by Proposition \ref{prop:catpro} (ii).
Similarly, if $(\mathcal M_G,\mathcal M_{\mathbf H})$ additionally
satisfies \eqref{cond:rad1} or \eqref{cond:rad2} then
so does $(\mathscr Y,\mathscr X)$.

What remains to be shown is \eqref{eq:genspcor}
when $\mathcal M$ satisfies \eqref{cond:rad2} and $V\in\Ksp$.
In this case it holds that
\[
\Hom_K(V,P_G(F))=\Hom_K^\ttt(V,P_G(F))
\quad\text{for any }F\in\Km
\]
by Proposition \ref{prop:chain} (i).
Let us consider the surjective map
\[
\Hom_K(V, \Xiwrad(\mathscr X))\xrightarrow{\mathcal I_G\circ\cdot}
\Hom_K(V,\Ximin_{\mathcal M}(\mathscr X)).
\]
Express any element $\Phi$ in
the left-hand side as in \eqref{eq:PhiPshiphi}
and take $\Phi_F^i\in \Hom_K^\ttt(F,\mathcal M_G)$
so that 
$\tilde\Gamma^F_{\mathcal M}(\Phi_F^i)=\mathcal I_{\mathbf H}\circ\varphi_F^i$.
Then
\[
\mathcal I_G\circ\Phi=\sum_{F\in\Km}\sum_{i=1}^{q_F}
\Phi_F^i\circ\Psi_F^i
\]
where $\Psi_F^i\in \Hom_K(V, P_G(F))=\Hom_K^\ttt(V,P_G(F))$.
Thus from \eqref{cond:rad2} for $\mathcal M$
we have
\[
\mathcal I_G\circ\Phi\in 
\Hom_K(V,\Ximin_{\mathcal M}(\mathscr X)) \cap \Hom_K^\ttt(V,\mathcal M_G)
=\Hom_K^\ttt(V,\Ximin_{\mathcal M}(\mathscr X)).
\]
This proves \eqref{eq:genspcor}.
\end{proof}
In general, when $\mathcal M=(\mathcal M_G,\mathcal M_{\mathbf H})\in\Cwrad$
has a radial restriction $\gamma_{\mathcal M}$
and $\mathscr X$ is an $\mathbf H$-submodule of $\mathcal M_{\mathbf H}$,
$\gamma_{\mathcal M}$ is not necessarily a radial restriction of
$(\Ximin_{\mathcal M}(\mathscr X), \mathscr X)$
because it may happen that $\gamma_{\mathcal M}\bigl(\Ximin_{\mathcal M}(\mathscr X)\bigr)\not\subset \mathscr X$.
But such a thing never happens
if $\mathcal M=\bigl(P_G(\CC_\triv), P_{\mathbf H}(\CC_\triv)\bigr)$
and $\gamma_{\mathcal M}=\gamma$.
\begin{prop}[the correspondence $\Xi^\natural_{\mathcal M}$]\label{prop:XinaturalGen}
Suppose
$\mathcal M=(\mathcal M_G,\mathcal M_{\mathbf H})\in\Cwrad$
and a linear map $\gamma_{\mathcal M}$
satisfies the assumption of {\normalfont Lemma \ref{lem:rest3}},
that is, {\normalfont Conditions} \eqref{cond:rest1'}, \eqref{cond:rest2}
and \eqref{cond:rest3}.
For an $\mathbf H$-submodule $\mathscr X$ of $\mathcal M_{\mathbf H}$
set
\[
\Xi^\natural_{\mathcal M}(\mathscr X)=\sum\bigl\{
\mathscr V\subset \mathcal M_G;\, \text{a $K$-stable $\CC$-subspace with }\gamma_{\mathcal M}(\mathscr V)\subset \mathscr X
\bigr\}.
\]
Then we have the following:

\noindent{\normalfont (i)}
$\Xi^\natural_{\mathcal M}(\mathscr X)$
is a $(\Lieg_\CC,K)$-submodule of ${\mathcal M}_G$
such that\/ $\Ximin_{\mathcal M}(\mathscr X)\subset 
\Xi^\natural_{\mathcal M}(\mathscr X)\subset\Ximax_{\mathcal M}(\mathscr X)$.

\noindent{\normalfont (ii)}
Suppose $\mathscr Y\subset {\mathcal M}_G$ is a $(\Lieg_\CC,K)$-submodule
such that\/ $\Ximin_{\mathcal M}(\mathscr X)\subset 
\mathscr Y\subset\Xi^\natural_{\mathcal M}(\mathscr X)$.
Then
the linear map $\gamma_{\mathcal M}|_{\mathscr Y}:\mathscr Y\to\mathscr X$
satisfies \eqref{cond:rest1'}, \eqref{cond:rest2} and \eqref{cond:rest3}
for the sub weak radial pair $(\mathscr Y, \mathscr X)$.
Also, a linear map 
$
\gamma_{\mathcal Q}: \mathcal M_G /\mathscr Y\to
\mathcal M_{\mathbf H}/\mathscr X
$
is naturally induced from $\gamma_{\mathcal M}$,
which satisfies \eqref{cond:rest1'}, \eqref{cond:rest2} and \eqref{cond:rest3}
for the quotient weak radial pair
$\mathcal Q:=(\mathcal M_G /\mathscr Y,\mathcal M_{\mathbf H}/\mathscr X)$.
If $\gamma_{\mathcal M}$ satisfies \eqref{cond:rest1}
then so does $\gamma_{\mathcal M}|_{\mathscr Y}$.

\noindent{\normalfont (iii)}
If $\mathscr Y=\Xi^\natural_{\mathcal M}(\mathscr X)$ in {\normalfont (ii)},
then 
$\mathcal Q=\bigl(\mathcal M_G /\, \Xi^\natural_{\mathcal M}(\mathscr X),\,
 \mathcal M_{\mathbf H}/\mathscr X\bigr)$
is a radial pair 
and 
$\gamma_{\mathcal Q}: \mathcal M_G /\, \Xi^\natural_{\mathcal M}(\mathscr X)\to
\mathcal M_{\mathbf H}/\mathscr X$
is a radial restriction of $\mathcal Q$
satisfying \eqref{cond:rest3}.

\noindent{\normalfont (iv)}
The radial restriction $\gamma$ for 
$\bigl(P_G(\CC_\triv), P_{\mathbf H}(\CC_\triv)\bigr)$
satisfies the assumption of the proposition (cf.~{\normalfont Example \ref{exmp:rest3}}).
In this case we use the symbol $\Xi^\natural$ instead of
$\Xi^\natural_\mathcal M$.
\end{prop}
\begin{proof}
Suppose $\mathscr V\subset \mathcal M_G$ is $K$-stable and
$\gamma_{\mathcal M}(\mathscr V)\subset \mathscr X$.
Then $\Lieg_\CC \mathscr V$ is also $K$-stable
and $\gamma_{\mathcal M}(\Lieg_\CC \mathscr V)=\gamma_{\mathcal M}((\Lien_\CC+\Liea_\CC+\Liek_\CC) \mathscr V)=\Liea_\CC\gamma_{\mathcal M}(\mathscr V)+\gamma_{\mathcal M}(\mathscr V)\subset \mathscr X$  by \eqref{eq:gammaH} and \eqref{eq:gammaX}.
Thus $\Xi^\natural_{\mathcal M}(\mathscr X)$ is stable under the action of $(\Lieg_\CC, K)$.
It is clear that $\Xi^\natural_{\mathcal M}(\mathscr X) \subset \Ximax_{\mathcal M}(\mathscr X)$.
Recall $\Ximin_{\mathcal M}(\mathscr X)$ is generated by
\[
v \otimes \Phi \in V\otimes \Hom^\ttt_K(V,\mathcal M_G)
\quad
\text{with }V\in\Km\text{ and }\tilde\Gamma^V_{\mathcal M}(\Phi)\in\Hom_W(V^M_\single,\mathscr X).
\]
Since $v\otimes\Phi=\Phi[v]$ we have
\begin{align*}
\gamma_{\mathcal M}(V\otimes\Phi) &= \gamma_{\mathcal M}(\Phi[V])\\
&=\gamma_{\mathcal M}(\Phi[V^M]) &&(\because \eqref{eq:gammam})\\
&=\gamma_{\mathcal M}(\Phi[V^M_\single]) &&(\because \eqref{cond:rest1'})\\
&=\tilde\Gamma^V_{\mathcal M}(\Phi)[V^M_\single] &&(\because \eqref{cond:rest2})\\
&\subset\mathscr X.
\end{align*}
This proves $\Ximin_{\mathcal M}(\mathscr X) \subset \Xi^\natural_{\mathcal M}(\mathscr X)$ and hence (i).

Suppose $\mathscr Y$ is as in (ii).
It is immediate from Proposition \ref{prop:catpro} (ii)
that $\gamma_{\mathcal M}|_{\mathscr Y}$ satisfies \eqref{cond:rest1'}, \eqref{cond:rest2} and \eqref{cond:rest3} for $(\mathscr Y,\mathscr X)$
(and \eqref{cond:rest1} if $\gamma_{\mathcal M}$ satisfies \eqref{cond:rest1}).
The pair of inclusion maps
$(\mathscr Y,\mathscr X)\to (\mathcal M_G,\mathcal M_{\mathbf H})$
induces a morphism $(\mathcal J_G,\mathcal J_{\mathbf H}): \mathcal M\to (\mathcal M_G/\mathscr Y,\mathcal M_{\mathbf H}/\mathscr X)$ of $\Cwrad$.
Since 
$\gamma_{\mathcal Q}\circ\mathcal J_G=
\mathcal J_{\mathbf H}\circ\gamma_{\mathcal M}$
we can easily observe from Proposition \ref{prop:catpro} (iii)
that $\gamma_{\mathcal Q}$ satisfies \eqref{cond:rest1'}, \eqref{cond:rest2} and \eqref{cond:rest3}.
Thus (ii) is proved.

Finally suppose $\mathscr Y=\Xi^\natural_{\mathcal M}(\mathscr X)$
and let us prove $\gamma_{\mathcal Q}$ satisfies \eqref{cond:rest1}.
For this purpose let $V\in\Km$ and take any $\Phi\in \Hom_K(V,\mathcal M_G/\mathscr Y)$
such that $(\gamma_{\mathcal Q}\circ\Phi)\bigl[V^M_\double\bigr]=\{0\}$.
Then there exists $\Phi'\in \Hom_K(V,\mathcal M_G)$ such that $\Phi=\mathcal J_G\circ\Phi'$.
Since $(\mathcal J_{\mathbf H}\circ\gamma_{\mathcal M}\circ\Phi')\bigl[V^M_\double\bigr]
=(\gamma_{\mathcal Q}\circ \mathcal J_G\circ\Phi')\bigl[V^M_\double\bigr]
=(\gamma_{\mathcal Q}\circ \Phi)\bigl[V^M_\double\bigr]=\{0\}$,
we have $(\gamma_{\mathcal M}\circ\Phi')\bigl[V^M_\double\bigr]\subset \mathscr X$.
Now it follows from \eqref{cond:Ch} that
there exists $\Phi''\in \Hom_K^\ttt(V,\mathcal M_G)$ such that $\tilde\Gamma^V_{\mathcal M}(\Phi'')=\tilde\Gamma^V_{\mathcal M}(\Phi')$.
But then we have
\begin{align*}
&(\gamma_{\mathcal M}\circ(\Phi'-\Phi''))\bigr|_{V^M_\single}
=\tilde\Gamma^V_{\mathcal M}(\Phi'-\Phi'')=0,
&&(\because \eqref{cond:rest2})\\
&(\gamma_{\mathcal M}\circ(\Phi'-\Phi''))\bigr|_{V^M_\double}
=(\gamma_{\mathcal M}\circ\Phi')\bigr|_{V^M_\double}\subset\mathscr X,
&&(\because \eqref{cond:rest1'})\\
\therefore\quad
&\gamma_{\mathcal M}\bigl((\Phi'-\Phi'')[V^M]\bigr)\subset \mathscr X,\\
\therefore\quad
&\gamma_{\mathcal M}\bigl((\Phi'-\Phi'')[V]\bigr)\subset \mathscr X,
&&(\because \eqref{eq:gammam})\\
\therefore\quad
&(\Phi'-\Phi'')[V]\subset \Xi^\natural_{\mathcal M}(\mathscr X)=\mathscr Y.
\end{align*}
Hence $\Phi=\mathcal J_G\circ\Phi'=\mathcal J_G\circ\Phi''\in \Hom_K^\ttt(V,\mathcal M_G/\mathscr Y)$.
This shows $\gamma_{\mathcal Q}$ satisfies \eqref{cond:rest1}, proving (iii).
\end{proof}
In \S\ref{sec:Ximin} the correspondence
\[
\Ximin : \{\mathbf H\text{-submodules of }P_{\mathbf H}(\CC_\triv)\}\to\{(\Lieg_\CC,K)\text{-submodules of }P_G(\CC_\triv)\}
\]
will be extended to a functor
sending an $\mathbf H$-module $\mathscr X$ to a $(\Lieg_\CC,K)$-module $\Ximin(\mathscr X)$ such that $(\Ximin(\mathscr X), \mathscr X)$ is a radial pair
with some canonical radial restriction satisfying \eqref{cond:rest3}.
Proposition \ref{prop:XinaturalGen} will play a key role in that argument.
We conclude this section with the following:

\begin{thm}\label{thm:liftincl}
Suppose $\mathcal M=(\mathcal M_G, \mathcal M_{\mathbf H})$,
$\mathcal N=(\mathcal N_G, \mathcal N_{\mathbf H})\in\Cwrad$ and\/
$\mathcal I=(\mathcal I_G, \mathcal I_{\mathbf H}) : \mathcal M\to \mathcal N$
is a morphism.
Then for any $\mathbf H$-submodule $\mathscr X\subset \mathcal M_{\mathbf H}$
\begin{equation}
\Ximin_{\mathcal N}(\mathcal I_{\mathbf H}(\mathscr X))
=\mathcal I_G(\Ximin_{\mathcal M}(\mathscr X))
\subset \mathcal I_G(\Ximax_{\mathcal M}(\mathscr X)) \subset \Ximax_{\mathcal N}(\mathcal I_{\mathbf H}(\mathscr X)).
\label{eq:R3Mimage}
\end{equation}
For any $\mathbf H$-submodule $\mathscr X'\subset \mathcal N_{\mathbf H}$
\begin{equation}
\Ximin_{\mathcal M}(\mathcal I_{\mathbf H}^{-1}(\mathscr X'))\subset
\mathcal I_G^{-1}(\Ximin_{\mathcal N}(\mathscr X'))
\subset \mathcal I_G^{-1}(\Ximax_{\mathcal N}(\mathscr X'))
=\Ximax_{\mathcal M}(\mathcal I_{\mathbf H}^{-1}(\mathscr X')).
\label{eq:R3Mpreimage}
\end{equation}
In particular
\begin{gather}
\Ximin_{\mathcal N}(\Image \mathcal I_{\mathbf H})
\subset \Image \mathcal I_G \subset \Ximax_{\mathcal N}(\Image \mathcal I_{\mathbf H}),
\label{eq:R3MWimage}\\
\Ximin_{\mathcal M}(\Ker \mathcal I_{\mathbf H})
\subset \Ker \mathcal I_G \subset \Ximax_{\mathcal M}(\Ker \mathcal I_{\mathbf H}).
\label{eq:R3MWpreimage}
\end{gather}
\end{thm}
\begin{proof}
Consider the commutative diagram
\[
\xymatrix @M=0.4pc{
\mathscr X \ar @{->>} ^-{\mathcal I_{\mathbf H}} [r]
\ar @{^{(}->} _-{\mathcal J_{\mathbf H}} [d]
&
\mathcal I_{\mathbf H}(\mathscr X)
\ar @{^{(}->} ^-{\mathcal J'_{\mathbf H}} [d]\\
\mathcal M_{\mathbf H}
\ar _-{\mathcal I_{\mathbf H}} [r]
&
\mathcal N_{\mathbf H}
}
\]
where $\mathcal J_{\mathbf H}$ and $\mathcal J'_{\mathbf H}$
are inclusion maps.
Then by Proposition \ref{prop:Xiwrad}
there uniquely exist a set of $(\Lieg_\CC,K)$-homomorphisms
$\mathcal J_G$, $\mathcal J'_G$ and $\overline{\mathcal I_G}$, 
such that the diagram
\[
\xymatrix @M=0.4pc{
(\Xiwrad(\mathscr X),\mathscr X) \ar @{->>} ^-{(\overline{\mathcal I_G},\mathcal I_{\mathbf H})} [rr]
\ar _-{(\mathcal J_{G},\mathcal J_{\mathbf H})} [d]
&&
(\Xiwrad(\mathcal I_{\mathbf H}(\mathscr X)),\mathcal I_{\mathbf H}(\mathscr X))
\ar ^-{(\mathcal J'_{G},\mathcal J'_{\mathbf H})} [d]\\
(\mathcal M_{G},\mathcal M_{\mathbf H})
\ar _-{(\mathcal I_{G},\mathcal I_{\mathbf H})} [rr]
&&
(\mathcal N_{G},\mathcal N_{\mathbf H})
}
\]
commutes in $\Cwrad$.
Here 
as is shown in the proof of Theorem \ref{thm:multcor},
$\Ximin_{\mathcal M}(\mathscr X)=\Image \mathcal J_G$ and
$\Ximin_{\mathcal N}(\mathcal I_{\mathbf H}(\mathscr X))=\Image \mathcal J'_G$.
We also note $\overline{\mathcal I_G}$ is surjective since $\Xiwrad$ is exact.
Hence we have
\[
\Ximin_{\mathcal N}(\mathcal I_{\mathbf H}(\mathscr X))=\Image \mathcal J'_G
=\Image \mathcal J'_G\circ \overline{\mathcal I_G}
=\Image \mathcal I_G\circ \mathcal J_G
=\mathcal I_G(\Image \mathcal J_G)
=\mathcal I_G(\Ximin_{\mathcal M}(\mathscr X)),
\]
namely, the first equality in \eqref{eq:R3Mimage}.
Applying this to the case where $\mathscr X=\mathcal I^{-1}_{\mathbf H}(\mathscr X')$ we have
\[
\mathcal I_G(\Ximin_{\mathcal M}(\mathcal I^{-1}_{\mathbf H}(\mathscr X')))
=
\Ximin_{\mathcal N}(\mathcal I_{\mathbf H}(\mathcal I^{-1}_{\mathbf H}(\mathscr X')))\subset
\Ximin_{\mathcal N}(\mathscr X').
\]
This implies the first inclusion relation of \eqref{eq:R3Mpreimage}.
The first inclusion in \eqref{eq:R3Mimage}
and the second inclusion in \eqref{eq:R3Mpreimage}
are obvious from Theorem \ref{thm:multcor} (i).
Now, for a $(\Lieg_\CC,K)$-submodule $\mathscr Y$ of $\mathcal M_G$
\begin{align*}
\mathscr Y\subset \mathcal I_G^{-1}(\Ximax_{\mathcal N}(\mathscr X'))
&\Longleftrightarrow
\mathcal I_G(\mathscr Y)\subset \Ximax_{\mathcal N}(\mathscr X')\\
&\Longleftrightarrow
\forall V\in\Km\ 
\tilde\Gamma_{\mathcal N}^V\bigl(
\Hom_K(V, \mathcal I_G(\mathscr Y))
\bigr) \subset \Hom_W(V^M_\single, \mathscr X' ).
\end{align*}
But since
\begin{align*}
\tilde\Gamma_{\mathcal N}^V\bigl(
\Hom_K(V, \mathcal I_G(\mathscr Y))
\bigr)
&=(\tilde\Gamma_{\mathcal N}^V \circ (\mathcal I_G\circ\cdot))
\bigl(
\Hom_K(V, \mathscr Y )
\bigr)\\
&=((\mathcal I_{\mathbf H}\circ\cdot)\circ \tilde\Gamma_{\mathcal M}^V )
\bigl(
\Hom_K(V, \mathscr Y )
\bigr),
\end{align*}
the above condition is still equivalent to
\begin{align*}
&\phantom{\Leftrightarrow}
\forall V\in\Km\ 
\tilde\Gamma_{\mathcal M}^V\bigl(
\Hom_K(V, \mathscr Y)
\bigr) \subset \Hom_W(V^M_\single, \mathcal I_{\mathbf H}^{-1}(\mathscr X') )
\Longleftrightarrow
\mathscr Y\subset \Ximax_{\mathcal M}(\mathcal I^{-1}_{\mathbf H}(\mathscr X')).
\end{align*}
Thus we get the last equality in \eqref{eq:R3Mpreimage}.
Applying this to the case where $\mathscr X'=\mathcal I_{\mathbf H}(\mathscr X)$ we get
\[
\mathcal I_G(\Ximax_{\mathcal M}(\mathscr X))
\subset
\mathcal I_G(\Ximax_{\mathcal M}(\mathcal I^{-1}_{\mathbf H}(\mathcal I_{\mathbf H}(\mathscr X))))
=
\mathcal I_G(\mathcal I_G^{-1}(\Ximax_{\mathcal N}(\mathcal I_{\mathbf H}(\mathscr X)))
)
\subset
\Ximax_{\mathcal N}(\mathcal I_{\mathbf H}(\mathscr X)),
\]
namely the last inclusion relation in \eqref{eq:R3Mimage}.

Finally, \eqref{eq:R3MWimage} and \eqref{eq:R3MWpreimage}
follow from \eqref{eq:R3Mimage}, \eqref{eq:R3Mpreimage} and Example \ref{exmp:Wmax}.
\end{proof}

\section{Spherical principal series}\label{sec:sps}
In this section we review the spherical principal series representation
$B_{G}(\lambda)$ for $G$ and the corresponding
standard representation $B_{\mathbf H}(\lambda)$ for $\mathbf H$.
For the latter we employ an unusual realization
so that a certain ``restriction map'' $\gamma_{B(\lambda)}:
B_{G}(\lambda) \to B_{\mathbf H}(\lambda)$ can easily be defined.
It will turn out that $B(\lambda):=\bigl(B_{G}(\lambda)_\Kf, B_{\mathbf H}(\lambda)\bigr)$ is a radial pair with radial restriction $\gamma_{B(\lambda)}$.
We also review standard invariant sesquilinear forms for principal series representations, which play important roles in later sections.

We discuss these things in a slightly more general setting
for the sake of application in \S\ref{sec:Ximin}.
Suppose a finite-dimensional $\Liea_\CC$-module $(\sigma,\mathscr U)$ is given.
The action $\sigma$ of $\Liea$ on $\mathscr U$ can be integrated to
the action of the simply connected Lie group $A$:
\[
A \ni a\longmapsto a^\sigma:=\exp \sigma(\log a)\in\End_\CC\mathscr U.
\]
We denote by $\mathscr U^\star$ the linear space of
antilinear functionals on $\mathscr U$.
Let $(\cdot,\cdot)_{\mathscr U}$ denote the canonical sesquilinear form on $\mathscr U^\star\times\mathscr U$.
Then $\mathscr U^\star$ is naturally an $\Liea_\CC$-module by
\[
(\sigma^\star(\xi) u^\star,u)_{\mathscr U}=
-(u^\star,\sigma(\xi) u)_{\mathscr U}
\quad\text{for }\xi\in\Liea,
u^\star\in\mathscr U^\star\text{ and }u\in\mathscr U.
\]
Suppose $\lambda\in\Liea_\CC^*$
and let $\CC_\lambda$ be the $\CC$ endowed with the $\Liea_\CC$-module structure
by $\Liea_\CC\ni\xi\mapsto\lambda(\xi)\in\End_\CC \CC_\lambda$.
We naturally identify $\CC_{-\lambda}^\star$ with $\CC_{\bar\lambda}$.

\begin{defn}\label{defn:BG}
We put
\begin{align*}
\Ind_{MAN}^G(\mathscr U)&=\left\{F:G\xrightarrow{C^\infty} \mathscr U;\,
\begin{aligned}
F(gman)
&=a^{-\sigma-\rho}F(g)\\
\text{for }&
(g,m,a,n)\in G\times M\times A\times N
\end{aligned}\right\},\\
B_G(\lambda)&=\Ind_{MAN}^G(\CC_{-\lambda})\\
&=\left\{
F \in C^\infty(G);\, 
\begin{aligned}F(gman)
&=a^{\lambda-\rho} F(g)\\
\text{for }&
(g,m,a,n)\in G\times M\times A\times N
\end{aligned}\right\}.
\end{align*}
We consider $\Ind_{MAN}^G(\mathscr U)$ as a $G$-module by the similar action to $\ell(\cdot)$.
Furthermore we define a sesquilinear form
$\threeset{\mathscr U}{\cdot,\cdot}{G}{\mathscr U^\star}$
on $\Ind_{MAN}^G(\mathscr U)\times \Ind_{MAN}^G(\mathscr U^\star)$ by
\[
\threeset{\mathscr U}{F_1,F_2}{G}{\mathscr U^\star}
=\int_K
(F_1(k),F_2(k))_{\mathscr U^\star}\,dk.
\]
Here $dk$ is the Haar measure on $K$ with $\int dk=1$.
In particular the sesquilinear form
$\threeset{\lambda}{\cdot,\cdot}{G}{-\bar\lambda}$
on $B_{G}(\lambda)\times B_{G}(-\bar\lambda)$
is defined by
\[
\threeset{\lambda}{F_1,F_2}{G}{-\bar\lambda}
=\int_K
F_1(k)\overline{F_2(k)}\,dk.
\]
\end{defn}
\begin{prop}\label{prop:sesquiGinv}
The sesquiliner form 
$\threeset{\mathscr U}{\cdot,\cdot}{G}{\mathscr U^\star}$
is invariant and non-degenerate.
\end{prop}
\begin{proof}
Use the $G$-invariance of the linear functional
\[
B_{G}(-\rho)\ni F(g)\longmapsto \int_K F(k)\,dk\in\CC
\]
(cf.~\cite[Ch.\,I, Lemma 5.19]{Hel4}).
\end{proof}
\begin{defn}\label{defn:BH}
Let $w_0$ be as in {\normalfont Definition \ref{defn:HCartan}}
and let $\trans\cdot : \mathbf H\to\mathbf H$ be a unique algebra anti-automorphism such that
\[
\left\{\begin{aligned}
&\trans w=w^{-1}&&\text{for }w\in W,\\
&\trans \xi=-w_0\,w_0(\xi)\,w_0&&\text{for }\xi\in\Liea_\CC.
\end{aligned}\right.
\]
We put
\begin{align*}
&\Ind_{S(\Liea_\CC)}^{\mathbf H}(\mathscr U)
=\{
F \in \Hom_\CC(\mathbf H,\mathscr U);\, F(h\,\trans\xi)=\sigma(\xi) F(h)
\quad
\text{for }h\in\mathbf H\text{ and }\xi\in\Liea_\CC
\},\\
&B_{\mathbf H}(\lambda)=\Ind_{S(\Liea_\CC)}^{\mathbf H}(\CC_{-\lambda})\\
&\qquad\quad=\{
F \in \mathbf H^*;\, F(h\,\trans\xi)=-\lambda(\xi) F(h)
\quad
\text{for }h\in\mathbf H\text{ and }\xi\in\Liea_\CC
\}.
\end{align*}
We consider $\Ind_{S(\Liea_\CC)}^{\mathbf H}(\mathscr U)$ as a left $\mathbf H$-module
by $hF(\cdot)=F(\trans h\,\cdot)$.
Furthermore we define a sesquilinear form
$\threeset{\mathscr U}{\cdot,\cdot}{\mathbf H}{\mathscr U^\star}$
on $\Ind_{S(\Liea_\CC)}^{\mathbf H}(\mathscr U)\times
\Ind_{S(\Liea_\CC)}^{\mathbf H}(\mathscr U^\star)$ by
\[
\threeset{\mathscr U}{F_1,F_2}{\mathbf H}{\mathscr U^\star}
=\frac1{|W|}\sum_{w\in W} (F_1(w), {F_2(w)})_{\mathscr U^\star}.
\]
In particular the sesquilinear form 
$\threeset{\lambda}{\cdot,\cdot}{\mathbf H}{-\bar\lambda}$
on $B_{\mathbf H}(\lambda)\times B_{\mathbf H}(-\bar\lambda)$
is defined by
\[
\threeset{\lambda}{F_1,F_2}{\mathbf H}{-\bar\lambda}
=\frac1{|W|}\sum_{w\in W} F_1(w)\overline{F_2(w)}.
\]
\end{defn}
\begin{defn}\label{defn:sesquiHinv}
let $\cdot^\star : \mathbf H\to\mathbf H$ be a unique antilinear anti-automorphism such that
\[
\left\{\begin{aligned}
&w^\star=w^{-1}&&\text{for }w\in W,\\
&\xi^\star=-w_0\,w_0(\xi)\,w_0&&\text{for }\xi\in\Liea.
\end{aligned}\right.
\]
Suppose $\mathscr X_1$ and $\mathscr X_2$ are $\mathbf H$-modules.
Then a sesquilinear form
$(\cdot,\cdot):\,\mathscr X_1\times \mathscr X_2\to\CC$ is
called invariant if it satisfies
\[
(hx_1,x_2)=(x_1,h^\star x_2)
\qquad\text{for any }
x_1\in\mathscr X_1,
x_2\in\mathscr X_2\text{ and }
h\in\mathbf H.
\]
\end{defn}
\begin{prop}\label{prop:sesquiHinv}
The sesquiliner form 
$\threeset{\mathscr U}{\cdot,\cdot}{\mathbf H}{\mathscr U^\star}$
is invariant and non-degenerate.
\end{prop}

\begin{proof}
Since $\mathbf H=\trans \mathbf H=\trans(S(\Liea_\CC)\otimes \CC W)= \CC W\otimes \trans S(\Liea_\CC)$,
restriction to $\CC W$ gives the $W$-isomorphism
$\Ind_{S(\Liea_\CC)}^{\mathbf H}(\mathscr U)\simarrow \Hom_\CC(\CC W, \mathscr U)$.
Hence the non-degeneracy and
the $W$-invariance are clear.
In general, for $\xi\in\Liea_\CC$ and $w\in W$ it holds that
\begin{equation}\label{eq:wxiw}
w\,\xi\,w^{-1}=w(\xi)+\!\!\sum_{\alpha\in R_1^+\cap ww_0R_1^+} \!\!\mathbf m_1(\alpha)
(w^{-1}\alpha)(\xi)s_\alpha
\end{equation}
in $\mathbf H$ (cf.~\cite[Proposition 1.1 (1)]{Op:Cherednik}).
We also note 
$\xi\,w_0=-w_0\trans(w_0(\xi))$ for $\xi\in\Liea_\CC$.
Hence for any $\xi\in\Liea_\CC$
\begin{align*}
\threeset{\mathscr U}{\xi F_1,F_2}{\mathbf H}{\mathscr U^\star}\hspace{-3.5em}&\\
&=\frac1{|W|}\sum_{w\in W} 
\bigl( F_1(-w_0\,w_0(\xi)\,w_0w) ,\, F_2(w) \bigr)_{\mathscr U^\star}\\
&=\frac1{|W|}\sum_{w\in W}
\bigl( F_1(-w^{-1}ww_0\,w_0(\xi)\,w_0w^{-1}w_0) ,\, F_2(w^{-1}w_0) \bigr)_{\mathscr U^\star}\\
&=\frac1{|W|}\sum_{w\in W}
\bigl( F_1\bigl(
w^{-1}
\bigl(-w(\xi)-\!\!\sum_{\alpha\in R_1^+\cap wR_1^+} \!\!\mathbf m_1(\alpha)(w^{-1}\alpha)(\xi)s_\alpha\bigr)w_0
\bigr) ,\, F_2(w^{-1}w_0) \bigr)_{\mathscr U^\star}\\
&=\frac1{|W|}\sum_{w\in W}
\bigl( \sigma(w_0w(\xi))F_1(w^{-1}w_0) ,\, F_2(w^{-1}w_0) \bigr)_{\mathscr U^\star}\\
&
\quad-\frac1{|W|}\sum_{\alpha\in R_1^+} \!\!\sum_{\substack{w\in W;\\\,w^{-1}\!\!\alpha \in R^+}}
\mathbf m_1(\alpha)(w^{-1}\alpha)(\xi)
\bigl( F_1(w^{-1} s_\alpha w_0) ,\, F_2(w^{-1}w_0) \bigr)_{\mathscr U^\star}\\
&=\frac1{|W|}\sum_{w\in W} 
\bigl( F_1(w^{-1}w_0) ,\, \sigma^\star(-w_0w(\bar\xi)) F_2(w^{-1}w_0) \bigr)_{\mathscr U^\star}\\
&
\quad-\frac1{|W|}\sum_{\alpha\in R_1^+} \!\!\sum_{\substack{w\in W;\\\,w^{-1}\alpha \in w_0R^+}}\!\!
\mathbf m_1(\alpha)({(s_\alpha w)}^{-1}\alpha)(\xi)
\bigl( F_1({(s_\alpha w)}^{-1} s_\alpha w_0) ,\, F_2({(s_\alpha w)}^{-1}w_0) \bigr)_{\mathscr U^\star}\\
&=\frac1{|W|}\sum_{w\in W} 
\bigl( F_1(w^{-1}w_0) ,\, F_2(w^{-1}\,w(\overline \xi)\,w_0) \bigr)_{\mathscr U^\star}\\
&
\quad+\frac1{|W|}\sum_{\alpha\in R_1^+} \!\!\sum_{\substack{w\in W;\\\,w^{-1}\alpha \in w_0R^+}}\!\!
\mathbf m_1(\alpha)({w}^{-1}\alpha)(\xi)
\bigl( F_1(w^{-1} w_0) ,\, F_2(w^{-1}s_\alpha w_0) \bigr)_{\mathscr U^\star}\\
&=\frac1{|W|}\sum_{w\in W}
\bigl( F_1(
w^{-1}w_0) ,\, F_2\bigl(w^{-1}
\bigl(w(\overline \xi)+\!\!\sum_{\alpha\in R_1^+\cap ww_0R_1^+} \!\!\mathbf m_1(\alpha)(w^{-1}\alpha)(\overline \xi)s_\alpha\bigr)w_0
\bigr) \bigr)_{\mathscr U^\star}\\
&=\frac1{|W|}\sum_{w\in W}
\bigl( F_1(w^{-1}w_0) ,\, F_2(\overline\xi\,w^{-1}w_0) \bigr)_{\mathscr U^\star}
\\
&=\threeset{\mathscr U}{F_1,\xi^\star F_2}{\mathbf H}{\mathscr U^\star}.\qedhere
\end{align*}
\end{proof}
\begin{rem}
This proof is essentially the same as Opdam's proof of \cite[Theorem 4.2 (1)]{Op:Cherednik}.
Indeed $B_{\mathbf H}(-w_0\lambda)$ is isomorphic to $I_{\lambda}$ in \cite{Op:Cherednik}
and $\threeset{-w_0\lambda}{\cdot,\cdot}{\mathbf H}{w_0\bar\lambda}$
equals the sesquilinear form $(\cdot,\cdot)$ of \cite[Definition 7.2]{Op:Cherednik} up to a scalar multiple.
\end{rem}
\begin{defn}\label{defn:restInd}
Define the linear map $\gamma_{\Ind(\mathscr U)} : \Ind_{MAN}^G(\mathscr U)\to \Ind_{S(\Liea_\CC)}^{\mathbf H}(\mathscr U)$
by
\[
\Ind_{MAN}^G(\mathscr U)
\ni F(g)
\longmapsto
\bigl(
W\ni w\longmapsto F(\bar w)\in\mathscr U
\bigr)\in\Hom_\CC(\CC W,\mathscr U)
\simeq \Ind_{S(\Liea_\CC)}^{\mathbf H}(\mathscr U).
\]
Here for $w\in W$,
$\bar w$ is any lift of $w$ in $N_K(\Liea)$.
If $\mathscr U=\CC_{-\lambda}$ then
the corresponding map $\gamma_{B(\lambda)}:
B_{G}(\lambda) \to B_{\mathbf H}(\lambda)$ is defined by
\[
B_G(\lambda)\ni F(g)\longmapsto \bigl(\,W\ni w\longmapsto F(\bar w)\in\CC \,\bigr)
\in (\CC W)^*\simeq B_{\mathbf H}(\lambda).
\]
\end{defn}
\begin{thm}\label{thm:R3B}
$\Ind(\mathscr U):=\bigl(\Ind_{MAN}^G(\mathscr U)_\Kf,\, \Ind_{S(\Liea_\CC)}^{\mathbf H}(\mathscr U) \bigr)$
is a radial pair with radial restriction $\gamma_{\Ind(\mathscr U)}$.
Moreover, for any $V\in \Km$,
$\Gamma^V_{\Ind(\mathscr U)}$ defined by \eqref{eq:genGammaV}
gives a linear bijection
\begin{equation}\label{eq:R3B}
\Gamma^V_{\Ind(\mathscr U)} :\,
 \Hom_K(V,\Ind_{MAN}^G(\mathscr U))
\simarrow
\Hom_W(V^M,\Ind_{S(\Liea_\CC)}^{\mathbf H}(\mathscr U)).
\end{equation}
In particular $B(\lambda):=\bigl(B_{G}(\lambda)_\Kf, B_{\mathbf H}(\lambda)\bigr)$ is a radial pair with radial restriction $\gamma_{B(\lambda)}$.
\end{thm}
\begin{proof}
Suppose $V\in\widehat K$.
Let $(V^M)^\perp$ be
the orthogonal complement of $V^M$ in $V$ with respect to a $K$-invariant inner product and $p^V:V\to V^M$ the orthogonal projection.
If $\Phi\in\Hom_K(V, \Ind_{MAN}^G(\mathscr U))$ then
$\Phi[v](1)=\Phi[p^V(v)](1)$ for $v\in V$ since the linear functional
$V\ni v\mapsto \Phi[v](1)\in \CC$ is $M$-invariant.
Hence
\begin{equation}\label{eq:Phirecover}
\Phi[v](kan)=a^{-\sigma-\rho}\Phi[p^V(k^{-1}v)](1)
\quad\text{for }v\in V\text{ and }(k,a,n)\in K\times A\times N.
\end{equation}
From this one sees all the $K$-types of $\Gamma^V_{\Ind(\mathscr U)}$
belong to $\Km$.
So we assume $V\in\Km$.
Now clearly $\Gamma^V_{\Ind(\mathscr U)}$ maps $\Hom_K$
into $\Hom_W$:
\[
\Gamma^V_{\Ind(\mathscr U)}:\,
\Hom_K(V, \Ind_{MAN}^G(\mathscr U))\longrightarrow
\Hom_W(V^M, \Ind_{S(\Liea_\CC)}^{\mathbf H}(\mathscr U)).
\]
We assert this is a bijection.
Indeed,
the injectivity follows from \eqref{eq:Phirecover}.
In addition, if $\varphi\in \Hom_W(V^M, \Ind_{S(\Liea_\CC)}^{\mathbf H}(\mathscr U))$
then
\[
\Phi:\,V\ni v\longmapsto \bigl(\,G\ni g=kan \longmapsto 
a^{-\sigma-\rho}\varphi[p^V(k^{-1}v)](1) \in\mathscr X\,
\bigr) \in \Ind_{MAN}^G(\mathscr U)
\] 
defines an element of $\Hom_K(V, \Ind_{MAN}^G(\mathscr U))$ such that $\varphi=\Gamma^V_{\Ind(\mathscr U)}(\Phi)$.
We thus get \eqref{eq:R3B}.
Hence if we define $\Hom^\ttt_K(V,\Ind_{MAN}^G(\mathscr U))$ and $\tilde\Gamma_{\Ind(\mathscr U)}^V$
so that \eqref{cond:rest1} and \eqref{cond:rest2} are valid,
then \eqref{cond:Ch} holds and $\Ind(\mathscr U) \in \CCh$.
Note that \eqref{eq:R3B} induces a bijection
\[
\Gamma_{\Ind(\mathscr U)} :\,
 \Hom_{\Lieg_\CC,K}(P_G(V),\Ind_{MAN}^G(\mathscr U)_\Kf)
\simarrow
\Hom_{\mathbf H}(P_{\mathbf H}(V^M),\Ind_{S(\Liea_\CC)}^{\mathbf H}(\mathscr U)).
\]

In order to prove \eqref{cond:w-rad}, \eqref{cond:rad1} and \eqref{cond:rad2},
suppose $\Phi\in\Hom_K(V, \Ind_{MAN}^G(\mathscr U))
\simeq \Hom_{\Lieg_\CC,K}(P_G(V),\Ind_{MAN}^G(\mathscr U)_\Kf)$, $E\in\Km$,
and $\Psi\in\Hom_K(E, P_{G}(V))$.
Take a basis $\{v_1,\ldots,v_m,\ldots,v_n\}$ of $V$
so that $\{v_1,\ldots,v_m\}$ and $\{v_{m+1},\ldots,v_n\}$
are respectively bases of $V^M$ and $(V^M)^\perp$.
If for any $e\in E^M$ we write
\[
\Psi[e]=\sum_{i=1}^n D_i\otimes v_i \text{ with }D_i\in U(\Lien_\CC+\Liea_\CC),
\]
then we have
\[
\Gamma^E_V(\Psi)[e]=\sum_{i=1}^{m} \gamma(D_i)\otimes v_i \text{ with }\gamma(D_i) \in S(\Liea_\CC), 
\]
and
\begin{align*}
\bigl(\Phi\circ\Psi\bigr)[e](1)&=\sum_{i=1}^n\bigl(D_i\Phi[v_i]\bigr)(1)
=\sum_{i=1}^n\sigma(\gamma(D_i))\bigl(\Phi[v_i](1)\bigr)\\
&=\sum_{i=1}^m\sigma(\gamma(D_i))\bigl(\Phi[v_i](1)\bigr)
=\sum_{i=1}^m\sigma(\gamma(D_i))\bigl(\Gamma^V_{\Ind(\mathscr U)}(\Phi)[v_i](1)\bigr)\\
&=\sum_{i=1}^m\bigl(\gamma(D_i)\Gamma^V_{\Ind(\mathscr U)}(\Phi)[v_i]\bigr)(1)
=\bigl(\Gamma_{\Ind(\mathscr U)}(\Phi)\circ\Gamma^E_V(\Psi)\bigr)[e](1).
\end{align*}
Now since $\Gamma^E_{\Ind(\mathscr U)}(\Phi\circ\Psi)\in \Hom_W(E^M,\Ind_{S(\Liea_\CC)}^{\mathbf H}(\mathscr U))$,
for any $w\in W$
\begin{equation}\label{eq:R3PS}
\begin{aligned}
\Gamma^E_{\Ind(\mathscr U)}(\Phi\circ\Psi)[e](w)
&=\Gamma^E_{\Ind(\mathscr U)}(\Phi\circ\Psi)[w^{-1}e](1)
=\bigl(\Phi\circ\Psi\bigr)[\bar w^{-1}e](1)\\
&=\bigl(\Gamma_{\Ind(\mathscr U)}(\Phi)\circ\Gamma^E_V(\Psi)\bigr)[w^{-1}e](1).
\end{aligned}\end{equation}
If $\Phi\in\Hom^\ttt$, $\Psi\in\Hom^\ttt$, $e\in E^M_\double$ and $w\in W$ then
\begin{align*}
\bigl(\Gamma_{\Ind(\mathscr U)}(\Phi)\circ\Gamma^E_V(\Psi)\bigr)[w^{-1}e]
&\subset
\Gamma_{\Ind(\mathscr U)}(\Phi)\bigl[\Gamma^E_V(\Psi)\bigl[ E^M_\double \bigr]\bigr]\\
&\subset
\Gamma_{\Ind(\mathscr U)}(\Phi)\bigl[ P_{\mathbf H}(V^M_\double)\bigr]
=\{0\}
\end{align*}
and hence
$\Gamma^E_{\Ind(\mathscr U)}(\Phi\circ\Psi)[e](w)=0$.
This shows \eqref{cond:rad2}.
Finally, if $\Phi\in\Hom^\ttt$ or $\Psi\in\Hom^\oto$, then
$\bigl(\Gamma_{\Ind(\mathscr U)}(\Phi)\circ\Gamma^E_V(\Psi)\bigr)\bigr|_{E^M_\single}
=\tilde \Gamma_{\Ind(\mathscr U)}(\Phi)\circ\tilde \Gamma^E_V(\Psi)$ is a $W$-homomorphism.
Hence in this case for each $e\in E^M_\single$, \eqref{eq:R3PS} reduces to
\[
\tilde\Gamma^E_{\Ind(\mathscr U)}(\Phi\circ\Psi)[e](w)
=\bigl(\tilde \Gamma_{\Ind(\mathscr U)}(\Phi)\circ\tilde \Gamma^E_V(\Psi)\bigr)[w^{-1}e](1)
=\bigl(\tilde \Gamma_{\Ind(\mathscr U)}(\Phi)\circ\tilde \Gamma^E_V(\Psi)\bigr)[e](w),
\]
proving \eqref{cond:w-rad} and \eqref{cond:rad1}.
\end{proof}
In general, for $V\in \Km$ 
we denote by $V^\star$ the linear space of antilinear functionals on $V$.
This has a natural $K$-module structure isomorphic to $V$ 
and the map $\,\bar\cdot: V^*\ni v^*\mapsto
\overline{v^*}=\overline{\ang{\cdot}{v^*}}\in V^\star$
is a $K$-antilinear isomorphism.
The inverse of $\,\bar\cdot\,$ is also denoted by $\,\bar\cdot$.

\begin{defn}\label{defn:compRM}
Suppose $\mathcal M^1=(\mathcal M^1_G, \mathcal M^1_{\mathbf H})$,
$\mathcal M^2=(\mathcal M^2_G, \mathcal M^2_{\mathbf H}) \in \CCh$.
A pair of
two invariant sesquilinear froms,
$(\cdot,\cdot)^{G}$ on $\mathcal M^1_G\times\mathcal M^2_G$ and
$(\cdot,\cdot)^{\mathbf H}$
on $\mathcal M^1_{\mathbf H}\times\mathcal M^2_{\mathbf H}$,
is said to be \emph{compatible with restriction}
if it satisfies the following condition:
\smallskip

\noindent
For any $V\in\Kqsp$ take a basis $\{v_1,\ldots,v_{m'},\ldots,v_n\}$ of $V$
so that $\{v_1,\ldots,v_{m'}\}$ is a basis of $V^M_\single$ and
$v_{m'+1},\ldots,v_{n}$ are orthogonal to $V^M_\single$ with respect to
a $K$-invariant inner product of $V$;
Let $\{v^*_1,\ldots,v^*_n\}\subset V^*$ be the dual basis of
$\{v_1,\ldots,v_n\}$ and put $v_i^\star=\overline{v_i^*}$ for $i=1,\ldots, n$;
Then for any $(\Phi_1, \Phi_2)\in
\Hom_K^\ttt(V,\mathcal M^1_G) \times \Hom_K(V^\star,\mathcal M^2_G)$
$\cup$
$\Hom_K(V,\mathcal M^1_G) \times \Hom_K^\ttt(V^\star,\mathcal M^2_G)$
\[
\sum_{i=1}^n
\bigl(\Phi_1[v_i],\Phi_2[v_i^\star]\bigr)^G
=\sum_{i=1}^{m'}\bigl(\tilde\Gamma^{V}_{\mathcal M^1}(\Phi_1)[v_i],
\tilde\Gamma^{V^\star}_{\mathcal M^2}(\Phi_2)[v_i^\star]\bigr)^{\mathbf H}.
\]
\end{defn}
\begin{prop}\label{prop:Blcomp}
Suppose $V\in\Km$ and
take a basis $\{v_1,\ldots,v_{m},\ldots,v_n\}$\/ of $V$
so that $\{v_1,\ldots,v_{m}\}$ is a basis of\/ $V^M$ and
$v_{m+1},\ldots,v_{n}$ are orthogonal to $V^M$.
Define $\{v_i^\star\}\subset V^\star$
as in {\normalfont Definition \ref{defn:compRM}}.
Then for any $\Phi_1\in\Hom_K(V,\Ind_{MAN}^G(\mathscr U))$
and $\Phi_2\in\Hom_K(V^\star,\Ind_{MAN}^G(\mathscr U^\star))$
\[
\sum_{i=1}^n
\bthreeset{\mathscr U}{\Phi_1[v_i],\Phi_2[v_i^\star]}{G}{\mathscr U^\star}
=\sum_{i=1}^{m}
\bthreeset{\mathscr U}{
\Gamma^{V}_{\Ind(\mathscr U)}(\Phi_1)[v_i],
\Gamma^{V^\star}_{\Ind(\mathscr U^\star)}(\Phi_2)[v_i^\star]
}{\mathbf H}{\mathscr U^\star}.
\]
In particular, the pair of $\threeset{\mathscr U}{\cdot,\cdot}{G}{\mathscr U^\star}$
and $\threeset{\mathscr U}{\cdot,\cdot}{\mathbf H}{\mathscr U^\star}$
is compatible with restriction.
\end{prop}
\begin{proof}
Since $\sum_{i=1}^n v_i\otimes v_i^*=
\sum_{i=1}^n k^{-1}v_i\otimes k^{-1}v_i^*$
for any $k\in K$, we calculate
\begin{align*}
\sum_{i=1}^n
\bthreeset{\mathscr U}{\Phi_1[v_i],\Phi_2[v_i^\star]}{G}{\mathscr U^\star}
&=
\sum_{i=1}^n
\int_K 
\bigl( \Phi_1[v_i](k) ,\, \Phi_2[v_i^\star](k) \bigr)_{\mathscr U^\star}\,dk\\
&=
\sum_{i=1}^n
\int_K 
\bigl( \Phi_1[k^{-1}v_i](1) ,\, \Phi_2\bigl[\,\overline{k^{-1}v_i^*}\,\bigr](1) \bigr)_{\mathscr U^\star}\,dk\\
&=
\sum_{i=1}^n
\int_K 
\bigl( \Phi_1[v_i](1) ,\, \Phi_2\bigl[\,\overline{v_i^*}\,\bigr](1) \bigr)_{\mathscr U^\star}\,dk\\
&=
\sum_{i=1}^n
\bigl( \Phi_1[v_i](1) ,\, \Phi_2\bigl[\,\overline{v_i^*}\,\bigr](1) \bigr)_{\mathscr U^\star}.
\end{align*}
But since $\Phi_1[v_i](1)=0$ for $i>m$ (cf.~the proof of Theorem \ref{thm:R3B})
and since $\sum_{i=1}^m v_i\otimes v_i^*=
\sum_{i=1}^m w^{-1}v_i\otimes w^{-1}v_i^*$ for any $w\in W$, 
the last expression equals
\begin{align*}
\sum_{i=1}^m
\bigl( \Phi_1[v_i](1) ,\, &\Phi_2\bigl[\,\overline{v_i^*}\,\bigr](1) \bigr)_{\mathscr U^\star}\\
&=
\sum_{i=1}^m
\bigl( \Gamma^{V}_{\Ind(\mathscr U)}(\Phi_1)[v_i](1) ,\, 
\Gamma^{V^\star}_{\Ind(\mathscr U^\star)}(\Phi_2)\bigl[\,\overline{v_i^*}\,\bigr](1)
 \bigr)_{\mathscr U^\star}\\
&=
\frac1{|W|}\sum_{w\in W}\sum_{i=1}^m
\bigl( \Gamma^{V}_{\Ind(\mathscr U)}(\Phi_1)[w^{-1}v_i](1) ,\, 
\Gamma^{V^\star}_{\Ind(\mathscr U^\star)}(\Phi_2)\bigl[\,\overline{w^{-1}v_i^*}\,\bigr](1)
 \bigr)_{\mathscr U^\star}\\
&=
\frac1{|W|}\sum_{w\in W}\sum_{i=1}^m
\bigl( \Gamma^{V}_{\Ind(\mathscr U)}(\Phi_1)[v_i](w) ,\, 
\Gamma^{V^\star}_{\Ind(\mathscr U^\star)}(\Phi_2)[v_i^\star](w)
 \bigr)_{\mathscr U^\star}\\
&=\sum_{i=1}^{m}
\bthreeset{\mathscr U}{
\Gamma^{V}_{\Ind(\mathscr U)}(\Phi_1)[v_i],
\Gamma^{V^\star}_{\Ind(\mathscr U^\star)}(\Phi_2)[v_i^\star]
}{\mathbf H}{\mathscr U^\star}.
\end{align*}
The last assertion of the proposition is easy from this.
(Note $V^M_\single\perp V^M_\double$ \cite[Lemma~3.3]{Oda:HC}.)
\end{proof}

\section{Star operations on morphisms}\label{sec:star}
In this section
we introduce two antilinear operations
\begin{align*}
\cdot^\star: \Hom_K(E,P_G(V))&\longrightarrow\Hom_K(V^\star,P_G(E^\star)),\\
\cdot^\star: \Hom_{\CC}(E^M,P_{\mathbf H}(V^M))&\longrightarrow\Hom_\CC((V^\star)^M,P_{\mathbf H}((E^\star)^M))
\end{align*}
for $E,V\in\Km$ and study their properties.
These tools will be used to prove Theorem \ref{thm:HC} (iii)
and other results in later sections.
\begin{lem}\label{lem:hanten1}
Suppose $E,V\in\Km$. 
Then naturally
\begin{equation*}
\begin{aligned}
\Hom_K(E,P_G(V))
&\simeq(E^*\otimes U(\Lieg_\CC)\otimes_{U(\Liek_\CC)} V)^K\\
&\simeq(E^*\otimes_{U(\Liek_\CC)} U(\Lieg_\CC)\otimes_{U(\Liek_\CC)} V)^K\\
&\simeq(E^*\otimes_{U(\Liek_\CC)} U(\Lieg_\CC)\otimes V)^K\\
&\simeq \Hom_K(V^*,E^*\otimes_{U(\Liek_\CC)}U(\Lieg_\CC)).
\end{aligned}
\end{equation*}
Here we consider $D\in U(\Liek_\CC)$ acts on $e^*\in E^*$ from the right by
$e^*D=\trans\! D\,e^*$.
(Recall $\trans\cdot:U(\Lieg_\CC)\to U(\Lieg_\CC)$ is a unique anti-automorphism such that $\trans X=-X$ for $X\in\Lieg_\CC$.)
\end{lem}
\begin{proof}
Note that $K$ acts on $E^*\otimes U(\Lieg_\CC)\otimes_{U(\Liek_\CC)} V$
and $E^*\otimes_{U(\Liek_\CC)} U(\Lieg_\CC)\otimes_{U(\Liek_\CC)} V$
diagonally.
Since these actions are locally finite,
we have the following projections to the trivial isotypic components:
\begin{align*}
&p_1:\,E^*\otimes U(\Lieg_\CC)\otimes_{U(\Liek_\CC)} V
\longrightarrow
(E^*\otimes U(\Lieg_\CC)\otimes_{U(\Liek_\CC)} V)^K,\\
&p_2:\,E^*\otimes_{U(\Liek_\CC)} U(\Lieg_\CC)\otimes_{U(\Liek_\CC)} V
\longrightarrow
(E^*\otimes_{U(\Liek_\CC)} U(\Lieg_\CC)\otimes_{U(\Liek_\CC)} V)^K.
\end{align*}
Thus if
\[
\iota_1: \,
(E^*\otimes U(\Lieg_\CC)\otimes_{U(\Liek_\CC)} V)^K
\longrightarrow
E^*\otimes U(\Lieg_\CC)\otimes_{U(\Liek_\CC)} V
\]
is the inclusion map and if
\[
\pi:\, E^*\otimes U(\Lieg_\CC)\otimes_{U(\Liek_\CC)} V
\longrightarrow
E^*\otimes_{U(\Liek_\CC)} U(\Lieg_\CC)\otimes_{U(\Liek_\CC)} V
\]
is the canonical surjection, then $p_2\circ\pi\circ\iota_1$ is a surjection
and $p_1\circ\iota_1=\id$.
One easily checks $p_1$ induces a $K$-homomorphism
\[
\tilde p_1:\,E^*\otimes_{U(\Liek_\CC)} U(\Lieg_\CC)\otimes_{U(\Liek_\CC)} V
\longrightarrow
(E^*\otimes U(\Lieg_\CC)\otimes_{U(\Liek_\CC)} V)^K
\]
such that $\tilde p_1\circ\pi=p_1$.
Since $\tilde p_1$ factors through $p_2$ there exists
\[
t:\,(E^*\otimes_{U(\Liek_\CC)} U(\Lieg_\CC)\otimes_{U(\Liek_\CC)} V)^K
\longrightarrow
(E^*\otimes U(\Lieg_\CC)\otimes_{U(\Liek_\CC)} V)^K
\]
such that $t\circ p_2=\tilde p_1$.
Now since $t\circ(p_2\circ\pi\circ\iota_1)=(t\circ p_2)\circ\pi\circ\iota_1
=(\tilde p_1\circ \pi)\circ\iota_1=p_1\circ\iota_1=\id$,
\[
p_2\circ\pi\circ\iota_1:\,
(E^*\otimes U(\Lieg_\CC)\otimes_{U(\Liek_\CC)} V)^K
\longrightarrow
(E^*\otimes_{U(\Liek_\CC)} U(\Lieg_\CC)\otimes_{U(\Liek_\CC)} V)^K
\]
is a bijection.
Likewise $(E^*\otimes_{U(\Liek_\CC)} U(\Lieg_\CC)\otimes_{U(\Liek_\CC)} V)^K
\simeq (E^*\otimes_{U(\Liek_\CC)} U(\Lieg_\CC)\otimes V)^K$
and the other isomorphisms are obvious.
\end{proof}
\begin{defn}[the star operation]\label{defn:starOpG}
Let $\cdot^\star: U(\Lieg_\CC)\to U(\Lieg_\CC)$ be a unique antilinear anti-automorphism
such that $X^\star=-X$ for $X\in\Lieg$ and
let $\sigma:E^*\otimes_{U(\Liek_\CC)}U(\Lieg_\CC)\to U(\Lieg_\CC) \otimes_{U(\Liek_\CC)} E^\star$ be the $K$-antilinear isomorphism defined by
$e^*\otimes D\mapsto D^\star\otimes\overline{e^*}$.
For $\Psi\in \Hom_K(E,P_G(V))$
let
$\trans\Psi \in \Hom_K(V^*,E^*\otimes_{U(\Liek_\CC)}U(\Lieg_\CC))$
be the corresponding element by Lemma \ref{lem:hanten1}.
Then we define $\Psi^\star \in \Hom_K(V^\star,P_G(E^\star))$
by the following the composition:
\[
\Psi^\star:\,V^\star\xrightarrow{\bar\cdot} V^*
\xrightarrow{\trans\Psi} 
E^*\otimes_{U(\Liek_\CC)}U(\Lieg_\CC)
\xrightarrow{\sigma}
U(\Lieg_\CC) \otimes_{U(\Liek_\CC)} E^\star=P_G(E^\star).
\]
\end{defn}
The next proposition is easy from the definition:
\begin{prop}\label{prop:starMor}
Suppose an $\Ad(K)$-stable subspace $\mathbf S$ of $U(\Lieg_\CC)$ satisfies
$\mathbf S \otimes U(\Liek_\CC)\simarrow U(\Lieg_\CC)$
by multiplication (or equivalently $U(\Liek_\CC) \otimes \mathbf S\simarrow U(\Lieg_\CC)$).
Let $\{e_1,\ldots,e_\nu\}$ and $\{v_1,\ldots,v_n\}$
be bases of $E$ and $V$.
Let $\{e_1^*,\ldots,e_\nu^*\}\subset E^*$
be the dual basis of $\{e_1,\ldots,e_\nu\}$
and put $e_i^\star=\overline{e_i^*}$ $(i=1,\ldots,\nu)$.
Define $\{v_1^\star,\ldots,v_n^\star\}\subset V^\star$ similarly.
If for a given $\Psi\in\Hom_K(E,P_G(V))$ we take
$S_{ij}\in\mathbf S$ $(1\le i\le\nu, 1\le j\le n)$ so that
\[
\Psi[e_i]=\sum_{j=1}^n S_{ij}\otimes v_j
\qquad\text{for }i=1,\ldots,\nu,
\]
then
\[
\Psi^\star[v_j^\star]=\sum_{i=1}^\nu S_{ij}^\star\otimes  e_i^\star
\qquad\text{for }j=1,\ldots,n.
\]
Hence $\Psi^{\star\star}=\Psi$.%, $(\theta\circ\Psi)^\star=\theta\circ\Psi^\star$.
\end{prop}
\begin{rem}{\normalfont (i)}
As an $\mathbf S$ we can take the image of the symmetrization map for $S(\Lies_\CC)$ (Recall $\Lies$ is the vector part of the Cartan decomposition of $\Lieg$.)

\noindent{\normalfont (ii)}
We cannot use $U(\Lien_\CC+\Liea_\CC)$
as an $\mathbf S$ since it is not $\Ad(K)$-stable.
\end{rem}
\begin{cor}\label{cor:sesqui}
Suppose $\mathscr Y_1$ and $\mathscr Y_2$ are $(\Lieg_\CC,K)$-modules
and $(\cdot,\cdot)$ is
an invariant sesquilinear form on $\mathscr Y_1\times \mathscr Y_2$.
Suppose bases $\{e_i\}, \{e_i^\star\},
\{v_j\}$ and $\{v_j^\star\}$ are as in {\normalfont Proposition \ref{prop:starMor}}.
For $\Psi\in \Hom_K(E,P_G(V))$,
$\Phi_1\in\Hom_K(V,\mathscr Y_1)\simeq \Hom_{\Lieg_\CC,K}(P_G(V),\mathscr Y_1)$
and $\Phi_2\in \Hom_K(E^\star,\mathscr Y_2)\simeq \Hom_{\Lieg_\CC,K}(P_G(E^\star),\mathscr Y_2)$ it holds that
\[
\sum_{i=1}^\nu\bigl(\Phi_1
\circ\Psi[e_i],\, \Phi_2[e_i^\star]
\bigr)
=
\sum_{j=1}^n\bigl(
\Phi_1[v_j],\, \Phi_2\circ\Psi^\star[v_j^\star]
\bigr).
\]
\end{cor}
\begin{proof}
Using the expressions of $\Psi$ and $\Psi^\star$ in Proposition \ref{prop:starMor}.
\[
\text{L.H.S.}=\sum_{i=1}^\nu\Bigl(\sum_{j=1}^n S_{ij}\Phi_1[v_j]
,\, \Phi_2[e_i^\star]
\Bigr)=\sum_{j=1}^n\Bigl( \Phi_1[v_j]
,\,
\sum_{i=1}^\nu S_{ij}^\star\Phi_2[e_i^\star]
\Bigr)
=\text{R.H.S.}\qedhere
\]
\end{proof}

\begin{cor}\label{cor:compRM}
Let $\mathcal M^1=(\mathcal M^1_G, \mathcal M^1_{\mathbf H})$,
$\mathcal M^2=(\mathcal M^2_G, \mathcal M^2_{\mathbf H}) \in \Cwrad$
and suppose $\Ximin_{\mathcal M^1}(\mathcal M^1_{\mathbf H})=\mathcal M^1_G$.
Let $(\cdot,\cdot)^{G}$ and ${\vphantom|}^\backprime(\cdot,\cdot)^{G}$ be
invariant sesquilinear forms on $\mathcal M^1_G\times\mathcal M^2_G$
and
$(\cdot,\cdot)^{\mathbf H}$ an invariant sesquilinear form
on $\mathcal M^1_{\mathbf H}\times\mathcal M^2_{\mathbf H}$.
Suppose both pairs $\bigl((\cdot,\cdot)^{G},\,(\cdot,\cdot)^{\mathbf H})$
and $\bigl({\vphantom|}^\backprime(\cdot,\cdot)^{G},\,(\cdot,\cdot)^{\mathbf H})$
are compatible with restriction
in the sense of\/ {\normalfont Definition \ref{defn:compRM}}.
Then $(\cdot,\cdot)^{G}={\vphantom|}^\backprime(\cdot,\cdot)^{G}$.
\end{cor}

\begin{proof}
It suffices to show for $E\in\Km$, $\Phi'\in\Hom_K(E,\mathcal M^1_G)$
and $\Phi''\in\Hom_K(E^\star,\mathcal M^2_G)$
\[
\sum_{i=1}^\nu
\bigl(\Phi'[e_i],\Phi''[e_i^\star]\bigr)^G
=\sum_{i=1}^\nu
{\vphantom{\bigm|}}^{\backprime}\bigl(\Phi'[e_i],\Phi''[e_i^\star]\bigr)^G.
\]
Here $\{e_1,\ldots,e_\nu\}$ is a basis of $E$
and $e_i^\star=\overline{e_i^*}$ for $i=1,\ldots, \nu$.
Since $\Ximin_{\mathcal M^1}(\mathcal M^1_{\mathbf H})=\mathcal M^1_G$,
$\Hom_K(E,\mathcal M^1_G)$ is spanned 
over $\CC$ by elements of the form
\[
\Phi_1\circ\Psi_1
\quad
\text{with }V\in\Kqsp, \Phi_1\in\Hom^\ttt_{\Lieg_\CC,K}(P_G(V),\mathcal M^1_G)
\text{ and }\Psi_1\in\Hom_K(E,P_G(V)).
\]
Hence we have only to show
\begin{equation}\label{eq:GsesGses}
\sum_{i=1}^\nu
\bigl(\Phi_1\circ\Psi_1[e_i],\,\Phi''[e_i^\star]\bigr)^G
=\sum_{i=1}^\nu
{\vphantom{\bigm|}}^{\backprime}\bigl(\Phi_1\circ\Psi_1[e_i],\,\Phi''[e_i^\star]\bigr)^G.
\end{equation}
Take a basis $\{v_1,\ldots,v_{m'},\ldots,v_n\}$ of $V$
so that $\{v_1,\ldots,v_{m'}\}$ is a basis of $V^M_\single$ and
$v_{m'+1},\ldots,v_n$ are orthogonal to $V^M_\single$.
Put $v_j^\star=\overline{v_j^*}$ for $j=1,\ldots, n$.
Then we have
\begin{align*}
\text{L.H.S.}&=
\sum_{j=1}^n
\bigl(\Phi_1[v_j],\,\Phi''\circ\Psi_1^\star[v_j^\star]\bigr)^G
&(\because \text{Corollary \ref{cor:sesqui}})\\
&=\sum_{j=1}^{m'}
\bigl(\tilde\Gamma^V_{\mathcal M_1}(\Phi_1)[v_j],\,
\tilde\Gamma^{V^\star}_{\mathcal M_2}(\Phi''\circ\Psi_1^\star)[v_j^\star]\bigr)^{\mathbf H}.
&(\because \Phi_1\in\Hom^\ttt)
\end{align*}
But since the right-hand side of \eqref{eq:GsesGses} can be calculated in the same way, \eqref{eq:GsesGses} holds.
\end{proof}
Identifying
$\Hom_{K}(E,P_G(V))$ with $\Hom_{\Lieg_\CC,K}(P_G(E),P_G(V))$
as in Definition \ref{defn:homhom},
we can apply the star operation to the latter.
\begin{prop}\label{prop:chainStar}
In addition to $E, V$ suppose $F\in\Km$.
Then it holds that $(\Psi_2\circ\Psi_1)^\star=\Psi_1^\star \circ \Psi_2^\star$
for any $\Psi_1\in\Hom_{\Lieg_\CC,K}(P_G(E),P_G(F))$
and $\Psi_2\in\Hom_{\Lieg_\CC,K}(P_G(F),P_G(V))$.
\end{prop}
\begin{proof}
Put
\[
\mathscr A(G\times_K E):=\bigl\{
f:G\xrightarrow{\text{analytic}} E;\,
f(gk)=k^{-1}f(g)\quad\text{for }k\in K
\bigr\}.
\]
Then there is a sesquilinear form $(\cdot,\cdot)$
on $P_G(E^\star) \times \mathscr A(G\times_K E)_\Kf$ defined by
\[
\bigl(
D\otimes e^\star, f
\bigr)
=
r(D)(e^\star,f(g))\bigr|_{g=1}.
\]
One easily checks this is invariant and non-degenerate.
Let $I\in\End_{\Lieg_\CC,K}(P_G(E^\star))$ be the identity map.
Then it follows from Corollary \ref{cor:sesqui} that
for any $\Phi\in\Hom_K(V,\mathscr A(G\times_K E))
\simeq \Hom_{\Lieg_\CC,K}(V,\mathscr A(G\times_K E)_\Kf)$,
\begin{align*}
\sum_{j=1}^n
\Bigl(
\bigl((\Psi_2\circ\Psi_1)^\star - \Psi_1^\star \circ \Psi_2^\star\bigr)[v^\star_j],\,
\Phi[v_j]
\Bigr)&=
\sum_{i=1}^\nu
\Bigl(
I[e_i^\star],\,
\bigl(\Phi\circ(\Psi_2\circ\Psi_1) - \Phi\circ\Psi_2 \circ \Psi_1\bigr)[e_i]
\Bigr)\\
&=0.
\end{align*}
The non-degeneracy of the form implies $(\Psi_2\circ\Psi_1)^\star = \Psi_1^\star \circ \Psi_2^\star$.
\end{proof}
Suppose $Y$ and $U$ are finite-dimensional $W$-modules.
Let us define the star operation for $\Hom_\CC(Y,P_{\mathbf H}(U))$.

\begin{defn}[the star operation]\label{defn:starOpH}
Let $\{y_1,\ldots,y_\mu\}$ and $\{u_1,\ldots,u_m\}$
be bases of $Y$ and $U$.
Let $\{y_1^*,\ldots,y_\mu^*\}\subset Y^*$
be the dual basis of $\{y_1,\ldots,y_\mu\}$
and put $y_i^\star=\overline{y_i^*}$ $(i=1,\ldots,\mu)$.
Define $\{u_1^\star,\ldots,u_m^\star\}\subset U^\star$ similarly.
For a given $\psi\in\Hom_\CC(Y,P_{\mathbf H}(U))$ we take
$f_{ij}(\lambda)\in S(\Liea_\CC)$ $(1\le i\le\mu, 1\le j\le m)$ so that
\begin{equation}\label{eq:stardef}
\psi[y_i]=\sum_{j=1}^m f_{ij}(\lambda)\otimes u_j
\qquad\text{for }i=1,\ldots,\mu.
\end{equation}
We define $\psi^\star\in \Hom_\CC(U^\star,P_{\mathbf H}(Y^\star))$ by 
\begin{equation*}
\psi^\star[u_j^\star]=\sum_{i=1}^\mu
\overline{f_{ij}(-\bar\lambda)}\otimes  y_i^\star
\qquad\text{for }j=1,\ldots,m.
\end{equation*}
\end{defn}
\begin{rem}
{\normalfont (i)}
It is easy to see the definition is independent of the choice of bases.

\noindent
{\normalfont (ii)}
The star operation $\psi\mapsto\psi^\star$ is involutive.

\noindent
{\normalfont (iii)}
If we consider $f_{ij}$ in the definition as elements of $U(\Lieg_\CC)$,
then $f_{ij}^\star(\lambda)=\overline{f_{ij}(-\bar\lambda)}$.
\end{rem}
Identifying
$\Hom_{\CC}(Y,P_{\mathbf H}(U))$ with
$\Hom_{\Liea_\CC}(P_{\mathbf H}(Y),P_{\mathbf H}(U))$
as in Definition \ref{defn:homhom},
we can apply the star operation to the latter.
The next is immediate from Definition \ref{defn:starOpH}.
\begin{prop}
In addition to $Y, U$ suppose $Z$ is a finite-dimensional $W$-module.
Then it holds that $(\psi_2\circ\psi_1)^\star=\psi_1^\star \circ \psi_2^\star$
for any $\psi_1\in\Hom_{\Liea_\CC}(P_{\mathbf H}(Y),P_{\mathbf H}(Z))$
and $\psi_2\in\Hom_{\Liea_\CC}(P_{\mathbf H}(Z),P_{\mathbf H}(U))$.
\end{prop}
The operation $\psi\mapsto\psi^\star$
has a similar property to Corollary \ref{cor:sesqui}.
\begin{prop}\label{prop:sesquiH}
Suppose $\psi\in\Hom_W(Y,P_{\mathbf H}(U))$. Then
$\psi^\star\in \Hom_W(U^\star,P_{\mathbf H}(Y^\star))$.
Furthermore,
suppose $\mathscr X_1$ and $\mathscr X_2$ are $\mathbf H$-modules
and $(\cdot,\cdot)$ is
an invariant sesquilinear form on $\mathscr X_1\times \mathscr X_2$.
Suppose bases $\{y_i\}, \{y_i^\star\},
\{u_j\}$ and $\{u_j^\star\}$ are as in {\normalfont Definition \ref{defn:starOpH}}.
Then for
$\varphi_1\in\Hom_W(U,\mathscr X_1)
\simeq \Hom_{\mathbf H}(P_{\mathbf H}(U),\mathscr X_1)$
and $\varphi_2\in \Hom_W(Y^\star,\mathscr X_2)
\simeq \Hom_{\mathbf H}(P_{\mathbf H}(Y^\star),\mathscr X_2)$
it holds that
\[
\sum_{i=1}^\mu\bigl(\varphi_1
\circ\psi[y_i],\, \varphi_2[y_i^\star]
\bigr)
=
\sum_{j=1}^m\bigl(
\varphi_1[u_j],\, \varphi_2\circ\psi^\star[u_j^\star]
\bigr).
\]
\end{prop}
\begin{proof}
Express $\psi\in\Hom_W(Y,P_{\mathbf H}(U))$ as in \eqref{eq:stardef}
and define $\psi^\vee\in \Hom_\CC(U^\star,P_{\mathbf H}(Y^\star))$ by
\[
\psi^\vee[u^\star]
=
\frac1{|W|}
\sum_{i=1}^\mu
\sum_{k=1}^m \sum_{w\in W} (w u^\star, u_k)
\,w^{-1}w_0\,\overline{f_{ik}(-w_0 \bar \lambda)}\,w_0\otimes y_i^\star
\qquad\text{for }u^\star\in U.
\]
Since $w^{-1}u_k=\sum_{j=1}^m\overline{(wu_j^\star,u_k)}u_j$ for $w\in W$,
we have
\begin{align*}
\sum_{i=1}^\mu\bigl(\varphi_1
\circ&\psi[y_i],\, \varphi_2[y_i^\star]
\bigr)\\
&=
\sum_{i=1}^\mu\Bigl(\sum_{j=1}^m f_{ij}(\lambda)\varphi_1[u_j],\, \varphi_2[y_i^\star]
\Bigr)\\
&=
\sum_{i=1}^\mu\sum_{j=1}^m
\bigl(
\varphi_1[u_j],\, w_0\,\overline{f_{ij}(-w_0\bar\lambda)}\,w_0 \varphi_2[y_i^\star]
\bigr)\\
&=
\frac1{|W|}
\sum_{i=1}^\mu\sum_{k=1}^m
\sum_{w\in W}
\bigl(
w\varphi_1[w^{-1}u_k],\, w_0\,\overline{f_{ik}(-w_0\bar\lambda)}\,w_0 \varphi_2[y_i^\star]
\bigr)\\
&=\sum_{j=1}^m
\frac1{|W|}
\sum_{i=1}^\mu
\sum_{k=1}^m \sum_{w\in W}
\bigl(
\varphi_1[u_j],\,
 (w u_j^\star, u_k)
\,w^{-1}w_0\,\overline{f_{ik}(-w_0 \bar \lambda)}\,w_0\varphi_2[y_i^\star]
\bigr)\\
&=\sum_{j=1}^m \bigl(\varphi_1
[u_j],\, \varphi_2\circ\psi^\vee[u_j^\star]
\bigr).
\end{align*}
Also, one easily checks $\psi^\vee\in \Hom_W$.
Hence the proposition follows if we can show $\psi^\vee=\psi^\star$.
To do so,
for any fixed $\lambda_0\in\Liea_\CC^*$
suppose $\varphi_1\in\Hom_W(U,B_{\mathbf H}(\lambda_0))$
and $\varphi_2\in \Hom_W(Y^\star,B_{\mathbf H}(-\bar\lambda_0))$.

On the one hand,
\begin{align*}
\sum_{j=1}^m 
\threeset{\lambda_0}{\varphi_1
[u_j],\, \varphi_2\circ\psi^\vee[u_j^\star]
}{\mathbf H}{-\bar\lambda_0}
&=
\sum_{i=1}^\mu
\threeset{\lambda_0}{\varphi_1
\circ\psi[y_i],\, \varphi_2[y_i^\star]
}{\mathbf H}{-\bar\lambda_0}\\
&=
\frac1{|W|}
\sum_{w\in W}
\sum_{i=1}^\mu
\varphi_1\circ\psi[y_i](w)
\overline{\varphi_2[y_i^\star](w)}\\
&=
\frac1{|W|}
\sum_{w\in W}
\sum_{i=1}^\mu
\varphi_1\circ\psi[w^{-1}y_i](1)
\overline{\varphi_2[w^{-1}y_i^\star](1)}\\
&=
\frac1{|W|}
\sum_{w\in W}
\sum_{i=1}^\mu
\varphi_1\circ\psi[w^{-1}y_i](1)
\overline{\varphi_2\bigl[\,\overline{w^{-1} y_i^*}\,\bigr](1)}.
\end{align*}
But since $\sum_{i=1}^\mu w^{-1}y_i\otimes w^{-1}y_i^*=\sum_{i=1}^\mu y_i\otimes y_i^*$,
the last expression equals
\begin{align*}
\sum_{i=1}^\mu
\varphi_1\circ\psi[y_i](1)
\overline{\varphi_2[y_i^\star](1)}
&=
\sum_{i=1}^\mu\sum_{j=1}^m
\bigl(f_{ij}(\lambda)\varphi_1[u_j]\bigr)(1)
\overline{\varphi_2[y_i^\star](1)}\\
&=
\sum_{i=1}^\mu\sum_{j=1}^m
\varphi_1[u_j](w_0\,f_{ij}(-w_0\lambda)\,w_0)
\overline{\varphi_2[y_i^\star](1)}\\
&=
\sum_{i=1}^\mu\sum_{j=1}^m
f_{ij}(-\lambda_0)
\varphi_1[u_j](1)
\overline{\varphi_2[y_i^\star](1)}.
\end{align*}
One the other hand,
take $g_{ij}(\lambda)\in S(\Liea_\CC)$ $(1\le i\le\mu, 1\le j\le m)$ so that
\begin{equation*}
\psi^\vee[u_j^\star]=\sum_{i=1}^\mu g_{ij}(\lambda)\otimes y_i^\star
\qquad\text{for }j=1,\ldots,m,
\end{equation*}
and apply 
the same calculation as above to the right entry.
Then we have
\begin{align*}
\sum_{j=1}^m 
\threeset{\lambda_0}{\varphi_1
[u_j],\, \varphi_2\circ\psi^\vee[u_j^\star]
}{\mathbf H}{-\bar\lambda_0}
&=
\frac1{|W|}\sum_{w\in W}\sum_{j=1}^m
\varphi_1[u_j](w)
\overline{\varphi_2\circ\psi^\vee[u_j^\star](w)}\\
&=
\sum_{j=1}^m
\varphi_1[u_j](1)
\overline{\varphi_2\circ\psi^\vee[u_j^\star](1)}\\
&=
\sum_{j=1}^m\sum_{i=1}^\nu
\varphi_1[u_j](1)
\overline{\bigl(g_{ij}(\lambda)\varphi_2[y_i^\star]\bigr)(1)}\\
&=
\sum_{j=1}^m\sum_{i=1}^\nu
\varphi_1[u_j](1)
\overline{\varphi_2[y_i^\star](w_0\,g_{ij}(-w_0\lambda)\,w_0)}\\
&=
\sum_{j=1}^m\sum_{i=1}^\nu
\overline{g_{ij}(\bar\lambda_0)}
\varphi_1[u_j](1)
\overline{\varphi_2[y_i^\star](1)}.
\end{align*}
Therefore
\begin{equation}\label{eq:f-g}
\sum_{i=1}^\mu\sum_{j=1}^m
\Bigl(f_{ij}(-\lambda_0)
-
\overline{g_{ij}(\bar\lambda_0)}\Bigr)
\varphi_1[u_j](1)
\overline{\varphi_2[y_i^\star](1)}
=0.
\end{equation}

Now for $k=1,\ldots,m$ and $\ell=1,\ldots,\mu$
let us take
\[
\varphi_1^{(k)} : U\ni u\longmapsto
\Bigl( w \longmapsto \overline{(u_k^\star, w^{-1}u)} \Bigr)
\in (\CC W)^*\simeq B_{\mathbf H}(\lambda_0)
\]
as $\varphi_1\in\Hom_W(U,B_{\mathbf H}(\lambda_0))$ 
and
\[
\varphi_2^{(\ell)} : Y^\star\ni y^\star\longmapsto
\Bigl( w \longmapsto (w^{-1} y^\star, y_\ell) \Bigr)
\in (\CC W)^*\simeq B_{\mathbf H}(-\bar \lambda_0)
\]
as $\varphi_2\in\Hom_W(Y^\star,B_{\mathbf H}(-\bar \lambda_0))$.
Then \eqref{eq:f-g} reduces to
\[
f_{\ell k}(-\lambda_0)-\overline{g_{\ell k}(\bar\lambda_0)}=0
\]
since
$\varphi_1^{(k)}[u_j](1)=\delta_{kj}$ and $\varphi_2^{(\ell)}[y_i^\star](1)=\delta_{\ell i}$.
Since $\lambda_0$ is arbitrary,
$g_{\ell k}(\lambda)=\overline{f_{\ell k}(-\bar\lambda)}$ for each $k$ and $\ell$.
This shows $\psi^\vee=\psi^\star$.
\end{proof}

The following theorem shows how the two star operations
(Definitions \ref{defn:starOpG} and \ref{defn:starOpH})
relate under the functor $\Gamma:\{P_G(V)\}\to \{P_{\mathbf H}(V^M)\}$ (\S\ref{sec:HC}).
\begin{thm}\label{thm:starGam}
Suppose $E,V\in \Km$. For any $\Psi\in\Hom_K(E,P_G(V))$
\begin{align}
\Gamma^{V^\star}_{E^\star}(\Psi^\star)&=\Gamma^E_V(\Psi)^\star,\label{eq:starGam1}\\
\tilde \Gamma^{V^\star}_{E^\star}(\Psi^\star)&=\tilde\Gamma^E_V(\Psi)^\star.\label{eq:starGam2}
\end{align}
Moreover we have
\begin{align}
\Hom_K^\oto(P_G(E),P_G(V))^\star&=\Hom_K^\ttt(P_G(V^\star),P_G(E^\star)),\label{eq:starGam3}\\
\Hom_K^\ttt(P_G(E),P_G(V))^\star&=\Hom_K^\oto(P_G(V^\star),P_G(E^\star)).\label{eq:starGam4}
\end{align}
\end{thm}
The proof of the theorem is very similar to that of Proposition \ref{prop:sesquiH}.
In both proofs the key tools are
invariant sesquilinear forms for
principal series representations.
This method is inspired by Kostant's argument in \cite[Chapter I, \S8]{Ko:Mat2}
where he shows a theorem essentially equivalent to \eqref{eq:starGam1} for $V=\CC_\triv$.
\begin{proof}[Proof of {\normalfont Theorem \ref{thm:starGam}}]
Let $\{e_1,\ldots,e_\mu\}$ and $\{e_{\mu+1},\ldots,e_\nu\}$ be bases of $E^M$ and $(E^M)^\perp$.
Let $\{e_i^\star\}\subset E^\star$ be as in {\normalfont Proposition \ref{prop:starMor}}.
Take bases $\{v_1,\ldots,v_m\}\sqcup\{v_{m+1}\ldots,v_n\}\subset V$
and $\{v_j^\star\}\subset V^\star$ similarly.
Choose
$f_{ij}\in U(\Lien_\CC+\Liea_\CC)$ $(1\le i\le\nu, 1\le j\le n)$
so that
\[
\Psi[e_i]=\sum_{j=1}^n f_{ij}\otimes v_j
\qquad\text{for }i=1,\ldots,\nu
\]
and
$g_{ij}\in U(\Lien_\CC+\Liea_\CC)$ $(1\le i\le\nu, 1\le j\le n)$
so that
\[
\Psi^\star[v_j^\star]=\sum_{i=1}^\nu g_{ij}\otimes e_i^\star
\qquad\text{for }j=1,\ldots,n.
\]
Then
\begin{align*}
\Gamma^E_V(\Psi)[e_i]&=\sum_{j=1}^m \gamma(f_{ij})\otimes v_j
\qquad\text{for }i=1,\ldots,\mu,\\
\Gamma^{V^\star}_{E^\star}(\Psi^\star)[v_j^\star]
&=\sum_{i=1}^\mu \gamma(g_{ij})\otimes e_i^\star
\qquad\text{for }j=1,\ldots,m.
\end{align*}
Now fix any $\lambda_0\in\Liea_\CC^*$ and suppose $\Phi_1\in\Hom_K(V,B_G(\lambda_0))$,
$\Phi_2\in\Hom_K(E^\star,B_G(-\bar\lambda_0))$.
By a similar calculation to the proofs of 
Propositions \ref{prop:Blcomp} and \ref{prop:sesquiH}
\begin{align*}
\sum_{i=1}^\nu
\bthreeset{\lambda_0}{\Phi_1
\circ\Psi[e_i],\, \Phi_2[e_i^\star]
}{G}{-\bar\lambda_0}
&=\sum_{i=1}^\mu
\Phi_1\circ\Psi[e_i](1)
\overline{\Phi_2[e_i^\star](1)}\\
&=
\sum_{i=1}^\mu\sum_{j=1}^n
\bigl(\ell(f_{ij})\Phi_1[v_j]\bigr)(1)
\overline{\Phi_2[e_i^\star](1)}\\
&=
\sum_{i=1}^\mu\sum_{j=1}^n
\gamma(f_{ij})(-\lambda_0)\,\Phi_1[v_j](1)
\overline{\Phi_2[e_i^\star](1)}\\
&=
\sum_{i=1}^\mu\sum_{j=1}^m
\gamma(f_{ij})(-\lambda_0)\,\Phi_1[v_j](1)
\overline{\Phi_2[e_i^\star](1)}.
\end{align*}
Here in the first and last equalities we used respectively
\[
\Phi_2[e_i^\star](1)=0\text{ for }i>\mu,\qquad
\Phi_1[v_j](1)=0\text{ for }j>m
\]
(see the proof of Theorem \ref{thm:R3B}).
But by virtue of Corollary \ref{cor:sesqui} this also equals
\begin{align*}
\sum_{j=1}^n
\bthreeset{\lambda_0}{\Phi_1[v_j]
,\, \Phi_2\circ\Psi^\star[v_j^\star]
}{G}{-\bar\lambda_0}
&=\sum_{j=1}^m
\Phi_1[v_j](1)
\overline{\Phi_2\circ\Psi^\star[v_j^\star](1)}\\
&=
\sum_{j=1}^m\sum_{i=1}^\nu
\Phi_1[v_j](1)
\overline{\bigl(\ell(g_{ij})\Phi_2[e_i^\star]\bigr)(1)}\\
&=
\sum_{j=1}^m\sum_{i=1}^\mu
\overline{\gamma(g_{ij})(\bar\lambda_0)}\,\Phi_1[v_j](1)
\overline{\Phi_2[e_i^\star](1)}.
\end{align*}
Therefore
\begin{equation}\label{eq:f-g2}
\sum_{i=1}^\mu\sum_{j=1}^m
\Bigl(
\gamma(f_{ij})(-\lambda_0)
-
\overline{\gamma(g_{ij})(\bar\lambda_0)}
\Bigr)\,\Phi_1[v_j](1)
\overline{\Phi_2[e_i^\star](1)}=0.
\end{equation}

Now for $k=1,\ldots,m$ and $\ell=1,\ldots,\mu$
let us take
\[
\Phi_1^{(k)} : V\ni v\longmapsto
\Bigl( k \longmapsto \overline{(v_k^\star, k^{-1}v)} \Bigr)
\in C^\infty(K/M)\simeq B_G(\lambda_0)
\]
as $\Phi_1\in\Hom_K(V,B_G(\lambda_0))$ 
and
\[
\Phi_2^{(\ell)} : E^\star\ni e^\star\longmapsto
\Bigl( k \longmapsto (k^{-1} e^\star, e_\ell) \Bigr)
\in C^\infty(K/M) \simeq B_G(-\bar \lambda_0)
\]
as $\Phi_2\in\Hom_K(E^\star,B_G(-\bar \lambda_0))$.
Then \eqref{eq:f-g2} reduces to
\[
\gamma(f_{\ell k})(-\lambda_0)-\overline{\gamma(g_{\ell k})(\bar\lambda_0)}=0
\]
since
$\Phi_1^{(k)}[v_j](1)=\delta_{kj}$ and $\Phi_2^{(\ell)}[e_i^\star](1)=\delta_{\ell i}$.
Since $\lambda_0$ is arbitrary,
$\gamma(g_{\ell k})(\lambda)=\overline{\gamma(f_{\ell k})(-\bar\lambda)}$ for each
$k=1,\ldots,m$ and $\ell=1,\ldots,\mu$.
This shows \eqref{eq:starGam1},
from which
one can easily deduce \eqref{eq:starGam2}--\eqref{eq:starGam4}.
\end{proof}
\begin{cor}\label{cor:2-3}
In {\normalfont Theorem \ref{thm:HC}}
the assertions {\normalfont (ii)} and {\normalfont (iii)} are equivalent
via star operations.
Hence {\normalfont (iii)} is valid.
\end{cor}

\section{Module structure of $\mathscr A(A,\lambda)$}\label{sec:modstr}
Suppose $\lambda\in\Liea_\CC^*$ and
let us study the structure of
\[
\mathscr A(A,\lambda)=\bigl\{
\varphi\in C^\infty(A);\,
\mathscr T(\Delta)\varphi=\Delta(\lambda)\varphi\quad\text{for }
\Delta \in S(\Liea_\CC)^W
\bigr\}
\]
as an $\mathbf H$-module.
Since the center of $\mathbf H$ is $ S(\Liea_\CC)^W$, this is actually
an $\mathbf H$-submodule of $C^\infty(A)$.
(Recall $h\in\mathbf H$ acts on $\varphi\in C^\infty(A)$
by $\mathscr T(\theta_{\mathbf H}h)\varphi$.)
As it turned out to be in \S\ref{sec:radI} and Example \ref{exmp:AA},
$\mathscr A(A,\lambda)$ is the radial counterpart of
\[
\mathscr A(G/K,\lambda)
=\bigl\{
f \in C^\infty(G/K);\,r(\Delta)f=\gamma(\Delta)(\lambda)f
\quad\text{for }\Delta \in U(\Lieg_\CC)^K
\bigr\}.
\]
So we review some fundamental facts of $\mathscr A(G/K,\lambda)$ first.
Throughout the section we fix a non-zero vector $v_\triv$ of the trivial representation $\CC_\triv$
of $K$ or $W$.
Let $(\cdot,\cdot)^G_{r}$ be the sesquilinear from on $C^\infty(G/K)\times P_G(\CC_\triv)$ defined by
\begin{equation}\label{eq:sesquiGr}
(f,D\otimes v_\triv)^G_{r}=r(\bar D)f(1)=\ell(D^\star)f(1)
\quad\text{for }f\in C^\infty(G/K)\text{ and }D\in U(\Lieg_\CC)
\end{equation}
(the conjugation $\bar\cdot$ is with respect to the real form $\Lieg$).
This is clearly $(\Lieg_\CC,K)$-invariant, and non-degenerate when $C^\infty(G/K)$
is replaced with the space $\mathscr A(G/K)$ of analytic functions on $G/K$.
The following results are well known:
\begin{prop}\label{prop:modstrAGl}
{\normalfont (i)}
All functions in $\mathscr A(G/K,\lambda)$ are analytic on $G/K$.

\noindent
{\normalfont (ii)}
Put
\[
P_G(\CC_\triv,\bar\lambda)=P_G(\CC_\triv) \bigm/
\sum_{\Delta \in U(\Lieg_\CC)^K} U(\Lieg_\CC)(\Delta-\gamma(\Delta)(\bar\lambda))
\otimes v_\triv.
\]
Then this is a $(\Lieg_\CC,K)$-module with a $K$-invariant cyclic vector.
Any $K$-invariant vector in $P_G(\CC_\triv,\bar\lambda)$ is a scalar multiple of $1\otimes v_\triv$.
An invariant non-degenerate sesquilinear form on $\mathscr A(G/K,\lambda)\times P_G(\CC_\triv,\bar\lambda)$ is induced from $(\cdot,\cdot)^G_{r}$.

\noindent
{\normalfont (iii)}
By {\normalfont (ii)}, any non-zero $(\Lieg_\CC,K)$-submodule of $\mathscr A(G/K,\lambda)_\Kf$ contains
the spherical function $\phi_\lambda$.
Hence there is a unique irreducible submodule $X_G(\lambda)\subset\mathscr A(G/K,\lambda)_\Kf$
generated by $\phi_\lambda$.
\end{prop}

Using $\mathscr A(G/K,\lambda)$ as a model case, we introduce
an invariant sesquilinear form on $C^\infty(A)\times P_{\mathbf H}(\CC_\triv)$
defined by
\begin{equation}\label{eq:sesquiHT}
(\varphi,h\otimes v_\triv)^{\mathbf H}_{\mathscr T}
=\mathscr T(\theta_{\mathbf H}(h^\star))\varphi(1)
\quad\text{for }\varphi\in C^\infty(A)\text{ and }h\in \mathbf H,
\end{equation}
and an $\mathbf H$-module
\[
P_{\mathbf H}(\CC_\triv,\bar\lambda)
=P_{\mathbf H}(\CC_\triv) \bigm/
\sum_{\Delta\in S(\Liea_\CC)^W} \mathbf H(\Delta-\Delta(\bar\lambda)) \otimes v_\triv.
\]
Note \eqref{eq:sesquiHT} is simply rewritten as
\begin{equation}\label{eq:sesquiHT2}
(\varphi,f\otimes v_\triv)^{\mathbf H}_{\mathscr T}
=\mathscr T\bigl(\,\overline{f(\bar\cdot)}\,\bigr)\varphi(1)
\quad\text{for }\varphi\in C^\infty(A)\text{ and }f\in S(\Liea_\CC).
\end{equation}

Opdam's \emph{non-symmetric hypergeometric function}s are key tools of our investigation:
\begin{thm}[\cite{Op:Cherednik}]\label{thm:G}
For any $\lambda\in\Liea_\CC^*$ there exists a unique analytic function $\mathbf G(\lambda,x)$ on $A$
satisfying
\begin{equation}\label{eq:Gfunc}
\left\{\begin{aligned}
&\mathscr T(\xi)\mathbf G(\lambda,x)=\lambda(\xi)\mathbf G(\lambda,x)\quad\text{for }\xi\in\Liea_\CC,\\
&\mathbf G(\lambda,1)=1.\end{aligned}\right.
\end{equation}
Under the identification $A\simeq \Liea$ by \eqref{eq:W-isom},
there exists an open neighborhood $U$ at $0\in \Liea$
such that $\mathbf G(\lambda, x)$ extends to a holomorphic function on $\Liea_\CC^* \times (\Liea + iU)$.
\end{thm}
The uniqueness assures that an analytic function satisfying 
the first condition of \eqref{eq:Gfunc} is a scalar multiple of $\mathbf G(\lambda,x)$.
In Appendix \ref{apndx:NSHG} we prove a $C^\infty$ function satisfying 
the first condition of \eqref{eq:Gfunc} is necessarily analytic.
Hence we get
\begin{lem}\label{lem:Glem}
If $\varphi(x)\in C^\infty(A)$ satisfies
\[
\mathscr T(\xi)\varphi=\lambda(\xi)\varphi\quad\text{for }\xi\in\Liea_\CC,
\]
then it is a scalar multiple of $\mathbf G(\lambda,x)$.
In particular, if moreover $\varphi\ne0$ then $\varphi(1)\ne0$.
\end{lem}
Observe that $\mathbf G(w\lambda,x)\in \mathscr A(A,\lambda)$ for any $w\in W$.
\begin{thm}\label{thm:modstr}
{\normalfont (i)}
All functions in $\mathscr A(A,\lambda)$ are analytic on $A$.

\noindent
{\normalfont (ii)}
As a $W$-module $P_{\mathbf H}(\CC_\triv,\bar\lambda)$ is isomorphic to the regular representation of $W$.
A non-zero $W$-invariant vector (unique up to a scalar multiple) spans $P_{\mathbf H}(\CC_\triv,\bar\lambda)$
as an $\mathbf H$-module.
An invariant non-degenerate sesquilinear form on $\mathscr A(A,\lambda)\times P_{\mathbf H}(\CC_\triv,\bar\lambda)$ is induced from $(\cdot,\cdot)^{\mathbf H}_{\mathscr T}$.

\noindent
{\normalfont (iii)}
By {\normalfont (ii)}, any non-zero $\mathbf H$-submodule of $\mathscr A(A,\lambda)$ contains
the restriction $\gamma_0(\phi_\lambda)$ of the spherical function (the Heckman-Opdam hypergeometric function
with a special parameter).
Hence there is a unique irreducible submodule $X_\mathbf H(\lambda)\subset\mathscr A(A,\lambda)$
generated by $\gamma_0(\phi_\lambda)$.

\noindent
{\normalfont (iv)}
$\mathscr A(A,\lambda)$ is spanned by
$\mathbf G(\lambda,x)$ as an $\mathbf H$-module if and only if
\begin{equation}\label{HPbij}
\lambda(\alpha^\vee)\ne -\mathbf m_1(\alpha)
\quad\text{for any }\alpha\in R_1^+,
\end{equation}
where $\alpha^\vee:=\frac{2H_\alpha}{|\alpha|^2}$ is the coroot for $\alpha$.

\noindent
{\normalfont (v)}
Let $\CC_\sign$ be the sign representation of $W$ and $v_\sign$ its fixed generator.
Put
\[
P_{\mathbf H}(\CC_\sign,-\lambda)
=P_{\mathbf H}(\CC_\sign) \bigm/
\sum_{\Delta\in S(\Liea_\CC)^W} \mathbf H(\Delta-\Delta(-\lambda)) \otimes v_\sign.
\]
Then $\mathscr A(A,\lambda)\simeq P_{\mathbf H}(\CC_\sign,-\lambda)$
as an $\mathbf H$-module.
\end{thm}
\begin{rem}
Essentially the same result as (iv)  
and the duality between $P_{\mathbf H}(\CC_\triv,\bar\lambda)$ and $P_{\mathbf H}(\CC_\sign,-\lambda)$
are stated in \cite {Ch2}.
\end{rem}
\begin{proof}[Proof of {\normalfont Theorem \ref{thm:modstr}}]
First, the identification $S(\Liea_\CC)\simarrow P_{\mathbf H}(\CC_\triv);\,
f\mapsto f\otimes v_\triv$ induces the identification
\[
S(\Liea_\CC) / \sum_{\Delta\in S(\Liea_\CC)^W}
S(\Liea_\CC)(\Delta-\Delta(\bar\lambda))
\simarrow
P_{\mathbf H}(\CC_\triv,\bar\lambda)
\]
as $S(\Liea_\CC)$-modules.
The left-hand side is isomorphic to 
the space $H_W(\Liea_\CC)$ of $W$-harmonic polynomials on $\Liea^*$
as a $\CC$-linear space
since
\[
S(\Liea_\CC)
=H_W(\Liea_\CC) \otimes S(\Liea_\CC)^W
=H_W(\Liea_\CC)\oplus \sum_{\Delta\in S(\Liea_\CC)^W}
S(\Liea_\CC)(\Delta-\Delta(\bar\lambda)).
\]
But since we also have the decomposition
\[
S(\Liea_\CC)
=H_W(\Liea_\CC) \oplus S(\Liea_\CC)\bigl(S(\Liea_\CC)\Liea_\CC\bigr)^W,
\]
which is compatible with the decomposition to the homogeneous parts,
it holds that for $d=0,1,2,\ldots$
\[
\begin{aligned}
\mathcal F^d P_{\mathbf H}(\CC_\triv,\bar\lambda)&:=\bigl\{f\otimes v_\triv\in P_{\mathbf H}(\CC_\triv,\bar\lambda) ;\,f\in S(\Liea_\CC)\text{ with }\deg f\le d\bigr \}\\
&=
\bigl\{f\otimes v_\triv\in P_{\mathbf H}(\CC_\triv,\bar\lambda) ;\,f\in H_W(\Liea_\CC)\text{ with }\deg f\le d\bigr \}.
\end{aligned}
\]
Using these subspaces as a filtered $W$-module structure of $P_{\mathbf H}(\CC_\triv,\bar\lambda)$,
we get the $W$-module isomorphisms
\[
P_{\mathbf H}(\CC_\triv,\bar\lambda)\simeq
\gr_{\mathcal F}P_{\mathbf H}(\CC_\triv,\bar\lambda)
\simeq H_W(\Liea_\CC) \simeq \CC W.
\]
Likewise $P_{\mathbf H}(\CC_\sign,-\lambda)\simeq
\CC W\otimes \CC_\sign\simeq \CC W$ as $W$-modules.

Secondly, we assert that $\CC\prod_{\alpha\in R_1^+}(\alpha^\vee+\mathbf m_1(\alpha))\otimes v_\triv\subset P_{\mathbf H}(\CC_\triv,\bar\lambda)$ is a unique $W$-submodule
isomorphic to $\CC_\sign$ and that 
\[
P_{\mathbf H}(\CC_\triv,\bar\lambda)
=
\mathcal F^{\,|R_1^+|-1} P_{\mathbf H}(\CC_\triv,\bar\lambda)
\ \oplus\
\CC\prod_{\alpha\in R_1^+}(\alpha^\vee+\mathbf m_1(\alpha))\otimes v_\triv
\]
is the direct sum decomposition as a $W$-module.
Indeed, since $\prod_{\alpha\in R_1^+} \alpha^\vee \in H_W(\Liea_\CC)$,
\[
\prod_{\alpha\in R_1^+}(\alpha^\vee+\mathbf m_1(\alpha))\otimes v_\triv
\equiv
\prod_{\alpha\in R_1^+} \alpha^\vee \otimes v_\triv
\not\equiv0
\pmod{\mathcal F^{\,|R_1^+|-1} P_{\mathbf H}(\CC_\triv,\bar\lambda)}.
\]
Moreover for any $\beta\in\Pi$
we have $\prod_{\alpha\in R_1^+\setminus\{\beta\}}(\alpha^\vee+\mathbf m_1(\alpha))
\in S(\Liea_\CC)^{s_\beta}=S(\Liea_\CC^{s_\beta})\cdot\CC[(\beta^\vee)^2]$
where $\Liea_\CC^{s_\beta}=\{\xi\in\Liea_\CC;\,\beta(\xi)=0\}$.
By using \eqref{eq:Hrel}
one can check
$S(\Liea_\CC)^{s_\beta}$ commutes with
$s_\beta$ also in $\mathbf H$.
Hence using \eqref{eq:Hrel} again
we calculate
\begin{align*}
s_\beta \prod_{\alpha\in R_1^+}(\alpha^\vee+\mathbf m_1(\alpha)) \otimes v_\triv
&=
\Bigl(\prod_{\alpha\in R_1^+\setminus\{\beta\}}(\alpha^\vee+\mathbf m_1(\alpha))\Bigr)
s_\beta\, (\beta^\vee+\mathbf m_1(\beta))
\otimes v_\triv\\
&=
\Bigl(\prod_{\alpha\in R_1^+\setminus\{\beta\}}(\alpha^\vee+\mathbf m_1(\alpha))\Bigr)
(-\beta^\vee s_\beta-\mathbf m_1(\beta)(2-s_\beta) )
\otimes v_\triv\\
&=-\prod_{\alpha\in R_1^+}(\alpha^\vee+\mathbf m_1(\alpha)) \otimes v_\triv.
\end{align*}
This shows $\CC\prod_{\alpha\in R_1^+}(\alpha^\vee+\mathbf m_1(\alpha))\otimes v_\triv \simeq \CC_\sign$.
The other assertions are obvious.

Thirdly, let us prove any non-zero
$\mathbf H$-submodule of $P_{\mathbf H}(\CC_\triv,\bar\lambda)$
contains $\prod_{\alpha\in R_1^+}(\alpha^\vee+\mathbf m_1(\alpha))\otimes v_\triv$.
For this purpose take any non-zero $f\in H_W(\Liea_\CC)$.
Let $f'$ denote the highest homogeneous part of $f$.
By the theory of $W$-harmonic polynomials there exists a homogeneous $g \in S(\Liea_\CC)$ such that
$\der(g)\der(f')\prod_{\alpha\in R^+_1}\alpha=1$.
For such $g$ we easily see
\[
\sum_{w\in W}(\sgn w)g(w\,\cdot)f(w\,\cdot)=
\sum_{w\in W}(\sgn w)g(w\,\cdot)f'(w\,\cdot)=
c \prod_{\alpha_\in R^+_1}\alpha^\vee
\quad\text{with }c\ne0.
\]
Hence in $P_{\mathbf H}(\CC_\triv,\bar\lambda)$ it holds that
\[
\sum_{w\in W}(\sgn w)g(w\,\cdot)\,w^{-1}(f\otimes v_\triv)
\equiv
c\prod_{\alpha_\in R^+_1}\alpha^\vee\otimes v_\triv
\not\equiv0
\pmod{\mathcal F^{\,|R_1^+|-1} P_{\mathbf H}(\CC_\triv,\bar\lambda)}.
\]
This shows $\mathbf H (f\otimes v_\triv) \not\subset \mathcal F^{\,|R_1^+|-1} P_{\mathbf H}(\CC_\triv,\bar\lambda)$,
which, combined with the previous argument, implies
$\prod_{\alpha\in R_1^+}(\alpha^\vee+\mathbf m_1(\alpha))\otimes v_\triv
\in
\mathbf H (f\otimes v_\triv)$.

Fourthly, it is clear from \eqref{eq:sesquiHT2} that
$(\cdot,\cdot)^{\mathbf H}_{\mathscr T}$ induces
an invariant sesquilinear form on
$\mathscr A(A,\lambda)\times P_{\mathbf H}(\CC_\triv,\bar\lambda)$
(we use the same symbol $(\cdot,\cdot)^{\mathbf H}_{\mathscr T}$ for this form).
We assert this is non-degenerate.
Indeed, if $\varphi\in \mathscr A(A,\lambda)$ is non-zero
then $\mathscr T(S(\Liea_\CC))\varphi=\mathscr T(H_W(\Liea_\CC))\varphi$
has finite dimension
and is annihilated by $\mathscr T(\Delta)-\Delta(\lambda)$
for any $\Delta\in S(\Liea_\CC)^W$.
Hence there exist some $f\in S(\Liea_\CC)$ and $w\in W$ such that
$\mathscr T(f)\varphi\ne 0$ and
\[
\mathscr T(\xi)\mathscr T(f)\varphi=(w\lambda)(\xi)\mathscr T(f)\varphi
\quad\text{for any }\xi\in\Liea_\CC.
\]
By Lemma \ref{lem:Glem}
$\mathscr T(f)\varphi$ is 
a non-zero multiple of $\mathbf G(w\lambda,x)$.
Thus it holds that
\[
\bigl(\varphi,\overline{f(\bar\cdot)} \otimes v_\triv\bigr)^{\mathbf H}_{\mathscr T}
=
\mathscr T(f)\varphi(1)\ne 0.
\]
Conversely, if $D \in P_{\mathbf H}(\CC_\triv,\bar\lambda)$ is non-zero
then there exists some $h\in\mathbf H$ such that
\[
h D=\prod_{\alpha\in R_1^+}(\alpha^\vee+\mathbf m_1(\alpha))\otimes v_\triv.
\]
Since
for any $w\in W$
\begin{equation}\label{eq:Gsgn}
\biggl(
\mathbf G(w\lambda,x),
\prod_{\alpha\in R_1^+}(\alpha^\vee+\mathbf m_1(\alpha))\otimes v_\triv
\biggr)^{\!\mathbf H}_{\!\mathscr T}
=\prod_{\alpha\in R_1^+}((w\lambda)(\alpha^\vee)+\mathbf m_1(\alpha))
\end{equation}
and this is non-zero for a suitable choice of $w\in W$,
we get for such $w$
\[
\bigl(
h^\star\,\mathbf G(w\lambda,x), D
\bigr)^{\!\mathbf H}_{\!\mathscr T}
=
\biggl(
\mathbf G(w\lambda,x),
\prod_{\alpha\in R_1^+}(\alpha^\vee+\mathbf m_1(\alpha))\otimes v_\triv
\biggr)^{\!\mathbf H}_{\!\mathscr T}
\ne 0.
\]
This completes the proof of (ii) and hence (iii).

Fifthly, since $\mathbf H=\theta_{\mathbf H}(\CC W S(\Liea_\CC))=\CC W\theta_{\mathbf H}(S(\Liea_\CC))$,
we have
\[
\mathscr T(\theta_{\mathbf H}\mathbf H)\,\mathbf G(\lambda,x)
=\CC W\mathscr T(S(\Liea_\CC))\,\mathbf G(\lambda,x)
=\CC W\, \mathbf G(\lambda,x).
\]
Hence, $\mathbf G(\lambda,x)$ spans $\mathscr A(A,\lambda)$
if and only if 
the orthogonal complement of $\CC W\, \mathbf G(\lambda,x)$ with respect to the sesquilinear form
does not contain $\prod_{\alpha\in R_1^+}(\alpha^\vee+\mathbf m_1(\alpha))\otimes v_\triv$,
if and only if
$t\,\mathbf G(\lambda,x)$ and $\prod_{\alpha\in R_1^+}(\alpha^\vee+\mathbf m_1(\alpha))\otimes v_\triv$
are not orthogonal for some $t\in W$,
if and only if 
$\mathbf G(\lambda,x)$ and $\prod_{\alpha\in R_1^+}(\alpha^\vee+\mathbf m_1(\alpha))\otimes v_\triv$
are not orthogonal,
if and only if \eqref{eq:Gsgn} with $w=1$ is non-zero,
if and only if \eqref{HPbij} is satisfied.
Thus (iv) is proved.
If we choose $w\in W$ so that \eqref{eq:Gsgn} is non-zero then
$\mathscr A(A,\lambda)=\mathscr A(A,w\lambda)$ is spanned by the analytic function $\mathbf G(w\lambda,x)$.
This implies (i).

Finally it follows from the non-degeneracy of the sesquilinear form that
there exists a non-zero $\phi_\sign \in \mathscr A(A,\lambda)$
such that $\CC \phi_\sign\simeq \CC_\sign$ as a $W$-module.
Such $\phi_\sign$ is unique up to a non-zero scalar and
is not orthogonal to $\prod_{\alpha\in R_1^+}(\alpha^\vee+\mathbf m_1(\alpha))\otimes v_\triv$.
Now any $\Delta\in S(\Liea_\CC)^W$ acts on $\mathscr A(A,\lambda)$
by
\[
\mathscr T(\theta_{\mathbf H}\Delta)=\mathscr T(w_0\, \Delta(-w_0\cdot)\,w_0)
=\mathscr T(w_0\, \Delta(-\cdot)\,w_0)
=\mathscr T(\Delta(-\cdot))=\Delta(-\lambda).
\]
Hence by the obvious universal property of $P_{\mathbf H}(\CC_\sign,-\lambda)$
there exists a unique $\mathbf H$-homomorphism $P_{\mathbf H}(\CC_\sign,-\lambda)\to\mathscr A(A,\lambda)$
such that $1\otimes v_\sign\mapsto\phi_\sign$.
This is surjective since the orthogonal compliment of $\mathscr T(\theta_{\mathbf H}\mathbf H)\,\phi_\sign$
does not contain $\prod_{\alpha\in R_1^+}(\alpha^\vee+\mathbf m_1(\alpha))\otimes v_\triv$.
The injectivity is clear from the dimension argument.
Thus we get (v).
\end{proof}
As we saw in Example \ref{exmp:AA},
$\bigl(\mathscr A(G/K,\lambda)_\Kf, \mathscr A(A,\lambda)\bigr)$
is a radial pair.
So each $\mathbf H$-submodule of $\mathscr A(A,\lambda)$ is lifted to
a $(\Lieg_\CC,K)$-submodule of $\mathscr A(G/K,\lambda)_\Kf$ by
the correspondence $\Ximin_0$ introduced in Definition \ref{defn:Ximin}.
\begin{thm}\label{thm:Xpair}
$\Ximin_0(X_{\mathbf H}(\lambda))=X_G(\lambda)$.
In particular $\bigl(X_G(\lambda), X_{\mathbf H}(\lambda)\bigr)$ is a radial pair.
\end{thm}
\begin{proof}
We apply Proposition \ref{prop:class1} to the case where
$\mathcal M=\bigl(\mathscr A(G/K,\lambda)_\Kf, \mathscr A(A,\lambda)\bigr)$
and $\phi_W=\gamma_0(\phi_\lambda)$.
Thus $\phi_K=\phi_\lambda$ and there exists a morphism
\[\mathcal I=(\mathcal I_G,\mathcal I_{\mathbf H}) :
\bigl(P_G(\CC_\triv), P_{\mathbf H}(\CC_\triv)\bigr) \to \bigl(\mathscr A(G/K,\lambda)_\Kf, \mathscr A(A,\lambda)\bigr)\]
in $\Crad$
such that $\mathcal I_G(1\otimes v_\triv)=\phi_\lambda$
and $\mathcal I_{\mathbf H}(1\otimes v_\triv)=\gamma_0(\phi_\lambda)$.
Here clearly $\Image\mathcal I=\bigl(X_G(\lambda), X_{\mathbf H}(\lambda)\bigr)$ and $\Ximin(P_{\mathbf H}(\CC_\triv))=P_G(\CC_\triv)$.
Hence by \eqref{eq:R3Mimage} we have 
\[
\Ximin_0(X_{\mathbf H}(\lambda))
=\Ximin_0(\mathcal I_{\mathbf H}(P_{\mathbf H}(\CC_\triv)))
=\mathcal I_G(\Ximin(P_{\mathbf H}(\CC_\triv)))
=\mathcal I_G(P_G(\CC_\triv))=X_G(\lambda).\qedhere
\]
\end{proof}

\section{Poisson transforms}\label{sec:Poisson}
Suppose $\lambda\in\Liea_\CC^*$.
The Poisson transform $\mathcal P_G^\lambda$ is
the $G$-homomorphism of $B_G(\lambda)$ into $\mathscr A(G/K,\lambda)$ defined by
\begin{equation*}
\mathcal P_G^\lambda:\,
B_G(\lambda) \ni F(g)\longmapsto \int_KF(xk)\,dk \in \mathscr A(G/K,\lambda).
\end{equation*}
To construct an analogous $\mathbf H$-homomorphism of
$B_{\mathbf H}(\lambda)$ into $\mathscr A(A,\lambda)$
we recall
another description of $\mathcal P_G^\lambda$ according to \cite[Ch.\,II, \S3, No.\,4]{Hel5}.
For $g\in G$ let $A(g)$ be the unique element of $\Liea$ such that
$g=n\,\exp{A(g)}\,k$
for some $n\in N$ and $k\in K$.
Then the function
\[
G\times G \ni (g,x) \longmapsto e^{(\bar\lambda+\rho)(A(g^{-1}x))}\in\CC
\]
belongs 
to $B_G(-\bar\lambda)$ as a function in $g$
and
to $\mathscr A(G/K,\bar\lambda)$ as a function in $x$.
If $\threeset{\lambda}{\cdot,\cdot}{G}{-\bar\lambda}$
is the invariant sesquilinear form of Definition \ref{defn:BG}
then for $F(g)\in B_G(\lambda)$
\begin{align*}
\mathcal P_G^\lambda F(x)
&=\int_K F(xk)\,dk=\int_K F(xk)\,\overline{e^{(\bar\lambda+\rho)(A(k^{-1}))}}\,dk\\
&=\bthreeset{\lambda}{F(xg), e^{(\bar\lambda+\rho)(A(g^{-1}))} }{G}{-\bar\lambda}
=\bthreeset{\lambda}{F(g), e^{(\bar\lambda+\rho)(A((x^{-1}g)^{-1}))} }{G}{-\bar\lambda}\\
&=\bthreeset{\lambda}{F(g), e^{(\bar\lambda+\rho)(A(g^{-1}x))} }{G}{-\bar\lambda}
=\int_K F(k)\,\overline{e^{(\bar\lambda+\rho)(A(k^{-1}x))}}\,dk\\
&=\int_K F(k)\,e^{(\lambda+\rho)(A(k^{-1}x))}\,dk.
\end{align*}
Now the function
\[
\mathbf H \times A \ni (h,x)\longmapsto
\mathscr T(\theta_{\mathbf H}h)\,\mathbf G(\bar\lambda,x) \in \CC
\]
belongs to $B_{\mathbf H}(-\bar\lambda)$ as a function in $h$
and to $\mathscr A(A, \bar\lambda)$ as a function in $x$.
If $\threeset{\lambda}{\cdot,\cdot}{\mathbf H}{-\bar\lambda}$
is the invariant sesquilinear form of Definition \ref{defn:BH}
then for $F(h) \in B_{\mathbf H}(\lambda)$
\begin{equation}\label{eq:HPoisson}
\begin{aligned}
\bthreeset{\lambda}{F(h), \mathscr T(\theta_{\mathbf H}h)\,\mathbf G(\bar\lambda,x) }{{\mathbf H}}{-\bar\lambda}
&=\frac1{|W|}\sum_{w\in W} F(w)\,\overline{\mathscr T(\theta_{\mathbf H}w)\,\mathbf G(\bar\lambda,x) }\\
&=\frac1{|W|}\sum_{w\in W} F(w)\,\mathbf G(\lambda,w^{-1}x).
\end{aligned}\end{equation}
Here we used the relation
$\overline{\mathbf G(\bar\lambda,x)}=\mathbf G(\lambda,x)$,
which is obvious from \eqref{eq:Gfunc}.
Since the final expression of \eqref{eq:HPoisson}
belongs to $\mathscr A(A, \lambda)$,
we define the Poisson transform for $B_{\mathbf H}(\lambda)$ by
\begin{equation}\label{eq:HPoisson2}
\mathcal P_{\mathbf H}^\lambda:\,
B_{\mathbf H}(\lambda)\ni F(h)
\longmapsto \frac1{|W|}\sum_{w\in W} F(w)\,\mathbf G(\lambda,w^{-1}x)
\in \mathscr A(A, \lambda).
\end{equation}
This is an $\mathbf H$-homomorphism since for any $a\in\mathbf H$
\begin{align*}
\bthreeset{\lambda}{F(\trans a\, h), \mathscr T(\theta_{\mathbf H}h)\,\mathbf G(\bar\lambda,x) }{{\mathbf H}}{-\bar\lambda}
&=
\bthreeset{\lambda}{F(h), \mathscr T(\theta_{\mathbf H}(\trans(a^\star) h))\,\mathbf G(\bar\lambda,x) }{{\mathbf H}}{-\bar\lambda}\\
&=\bthreeset{\lambda}{F(h), \mathscr T\bigl(\overline{\theta_{\mathbf H} a}\bigr)\mathscr T(\theta_{\mathbf H}h)\,\mathbf G(\bar\lambda,x) }{{\mathbf H}}{-\bar\lambda}\\
&=\mathscr T(\theta_{\mathbf H} a)\,
\bthreeset{\lambda}{F(h), \mathscr T(\theta_{\mathbf H}h)\,\mathbf G(\bar\lambda,x) }{{\mathbf H}}{-\bar\lambda}.
\end{align*}
\begin{prop}\label{prop:PoiBij}
The Poisson transform $\mathcal P_{\mathbf H}^\lambda$ is bijective
if and only if \eqref{HPbij} is satisfied.
This condition is rewritten as
\[
\begin{cases}
\lambda(\alpha^\vee) + \dim \Lieg_{\alpha} \ne 0& \text{for }\alpha\in \Sigma^+\setminus (2\Sigma^+\cup\frac12\Sigma^+) ,\\
\lambda(\alpha^\vee)  +\dim \Lieg_{\alpha}+ 2\dim \Lieg_{2\alpha}\ne0 & \text{for }\alpha\in \Sigma^+\cap \frac12\Sigma^+.
\end{cases}
\]
\end{prop}
\begin{rem}\label{rem:PoissonBij}
(i)
The surjectivity of $\mathcal P_{\mathbf H}^\lambda$
is equivalent to its injectivity
since
$\dim B_{\mathbf H}(\lambda)=\dim \mathscr A(A, \lambda) =|W|$.

\noindent{\normalfont(ii)}
This condition is much weaker than the following well-known condition for
the bijectivity of $\mathcal P_G^\lambda : B_{G}(\lambda)_\Kf\to \mathscr A(G/K, \lambda)_\Kf$ (cf.~\cite{Hel2}):
\[
\begin{cases}
\lambda(\alpha^\vee) + \dim \Lieg_{\alpha} \notin \{0,-2,-4,\ldots\}& \text{for }\alpha\in \Sigma^+\setminus (2\Sigma^+\cup\frac12\Sigma^+) ,\\
\lambda(\alpha^\vee)  +\dim \Lieg_{\alpha}
\notin \{-2,-6,-10,\ldots\} & \text{for }\alpha\in \Sigma^+\cap \frac12\Sigma^+,\\
\lambda(\alpha^\vee)  +\dim \Lieg_{\alpha}+ 2\dim \Lieg_{2\alpha}
\notin \{0,-4,-8,\ldots\} & \text{for }\alpha\in \Sigma^+\cap \frac12\Sigma^+.
\end{cases}
\]
\end{rem}
\begin{proof}[Proof of {\normalfont Proposition \ref{prop:PoiBij}}]
By \eqref{eq:HPoisson2} we have
\begin{align*}
\Image \mathcal P_{\mathbf H}^\lambda
&=
\sum_{w\in W} \CC\, \mathbf G(\lambda,w^{-1}x)
= \sum_{w\in W} \bigl(\mathscr T(S(\Liea_\CC))\, \mathbf G\bigr)(\lambda,w^{-1}x)\\
&= \mathscr T(W S(\Liea_\CC) )\, \mathbf G(\lambda,x)
= \mathscr T(\mathbf H)\, \mathbf G(\lambda,x)= \mathscr T(\theta_\mathbf H \mathbf H )\, \mathbf G(\lambda,x).
\end{align*}
Hence the proposition follows from Theorem \ref{thm:modstr} (iv).
\end{proof}

The following is the main result of this section:
\begin{thm}\label{thm:Poisson}
The pair of\/ $\mathcal P_G^{\lambda}: B_G(\lambda)_\Kf\to \mathscr A(G/K,\lambda)_\Kf$
and\/ $\mathcal P_{\mathbf H}^\lambda: B_{\mathbf H}(\lambda)\to \mathscr A(A,\lambda)$
is a morphism of radial pairs:
\[
\mathcal P^{\lambda}=
(\mathcal P_G^{\lambda},\mathcal P_{\mathbf H}^\lambda):\,\bigl(B_G(\lambda)_\Kf, B_{\mathbf H}(\lambda)\bigr)
\longrightarrow\bigl(\mathscr A(G/K,\lambda)_\Kf,\mathscr A(A,\lambda)\bigr).
\]
That is, if\/ $V\in\Km$ then

\noindent{\normalfont (i)}
for any $\Phi\in\Hom_K(V,B_G(\lambda))$
the two maps
\begin{gather*}
V^M_\single \hookrightarrow V
\xrightarrow{\Phi} B_G(\lambda)
\xrightarrow{\mathcal P_G^\lambda}
\mathscr A(G/K,\lambda)
\xrightarrow{\gamma_0} C^\infty(A),\\
V^M_\single \hookrightarrow V
\xrightarrow{\Phi} B_G(\lambda)
\xrightarrow{\gamma_{B(\lambda)}}
B_{\mathbf H}(\lambda)
\xrightarrow{\mathcal P_{\mathbf H}^\lambda}
\mathscr A(A,\lambda)
\hookrightarrow C^\infty(A)
\end{gather*}
coincide and

\noindent{\normalfont (ii)}
for any $\Phi\in\Hom_K(V,B_G(\lambda))$ such that
\[
\Phi[v](1)=0\quad\forall v\in V^M_\double
\]
it holds that
\[
\mathcal P_G^\lambda(\Phi[v])(x)=0
\quad\forall v\in V^M_\double\text{ and\/ }\forall x\in A.
\]
These properties can be more explicitly stated as follows:

\noindent{\normalfont (i')}
For any $\Phi\in\Hom_K(V,C^\infty(K/M))$ it holds that
\[
\int_K\Phi[v](k)\,e^{(\lambda+\rho)(A(k^{-1}x))}\,dk
=\frac1{|W|}
\sum_{w\in W} \Phi[v](\bar w)\,\mathbf G(\lambda,w^{-1}x)
\quad
\forall v\in V^M_\single\text{ and\/ }\forall x\in A
\]
($\bar w\in N_K(\Liea)$ is a lift of $w$) and

\noindent{\normalfont (ii')}
for any $\Phi\in\Hom_K(V,C^\infty(K/M))$ such that
\[
\Phi[v](1)=0\quad\forall v\in V^M_\double
\]
it holds that
\[
\int_K\Phi[v](k)\,e^{(\lambda+\rho)(A(k^{-1}x))}\,dk
=0
\quad
\forall v\in V^M_\double\text{ and\/ }\forall x\in A.
\]
\end{thm}

The proof requires some preparation.
\begin{prop}
For $\bigl(C^\infty(G/K)_\Kf,C^\infty(A)\bigr)$ and
$\bigl(P_G(\CC_\triv), P_{\mathbf H}(\CC_\triv)\bigr)$,
the pair of invariant sesquilinear forms
$(\cdot,\cdot)^G_{r}$ and $(\cdot,\cdot)^{\mathbf H}_{\mathscr T}$
is compatible with restriction in the sense of {\normalfont Definition \ref{defn:compRM}}.
\end{prop}
\begin{proof}
Suppose $V\in\Kqsp$ and take bases $\{v_1,\ldots,v_{m'},\ldots,v_n\}\subset V$
and $\{v_1^\star,\ldots,v_n^\star\}\subset V^\star$
as in Definition \ref{defn:compRM}.
Let $I_G\in\End_{\Lieg_\CC,K}(P_G(\CC_\triv))$ be the identity
and take $v_\triv^\star\in\CC_\triv^\star$ so that $(v^\star_\triv,v_\triv)=1$.
For $(\Phi,\Psi)\in\Hom_K(V, C^\infty(G/K))\times \Hom_K(V^\star, P_G(\CC_\triv))$
\begin{align*}
\sum_{i=1}^n
\bigl(\Phi[v_i],\Psi[v_i^\star]\bigr)^G_r
&=\sum_{i=1}^n
\bigl(\Phi[v_i],\,I_G\circ\Psi[v_i^\star]\bigr)^G_r\\
&=\bigl(\Phi\circ\Psi^\star[v_\triv^\star],\,I_G[v_\triv]\bigr)^G_r
&&(\because \text{Corollary \ref{cor:sesqui}})\\
&=\bigl(\Phi\circ\Psi^\star[v_\triv^\star],\,1\otimes v_\triv\bigr)^G_r\\
&=\Phi\circ\Psi^\star[v_\triv^\star](1)
&&(\because \eqref{eq:sesquiGr})\\
&=\tilde\Gamma_0^{\CC^\star_\triv}(\Phi\circ\Psi^\star)[v_\triv^\star](1).
\end{align*}
Here if $(\Phi,\Psi)\in\Hom_K^\ttt\times\Hom_K$ then
\begin{equation}\label{eq:GPhiPsistar}
\tilde\Gamma_0^{\CC^\star_\triv}(\Phi\circ\Psi^\star)
=\tilde\Gamma_0(\Phi)\circ\tilde\Gamma^{\CC^\star_\triv}_V(\Psi^\star)
\end{equation}
by Theorem \ref{thm:radD} (i).
Formula \eqref{eq:GPhiPsistar} also holds
for $(\Phi,\Psi)\in\Hom_K\times\Hom_K^\ttt$ since in this case
$\Psi^\star\in\Hom_K^\oto$ by \eqref{eq:starGam4} and
Theorem \ref{thm:radD} (ii) can be applied.
Hence in either case, letting
$I_{\mathbf H}\in\End_{\Lieg_\CC,K}(P_{\mathbf H}(\CC_\triv))$ be the identity
we have
\begin{align*}
\sum_{i=1}^n
\bigl(\Phi[v_i],\Psi[v_i^\star]\bigr)^G_r
&=\tilde\Gamma_0(\Phi)\circ\tilde\Gamma^{\CC^\star_\triv}_V(\Psi^\star)
[v_\triv^\star](1)\\
&=\bigl(
\tilde\Gamma_0(\Phi)\circ\tilde\Gamma^{\CC^\star_\triv}_V(\Psi^\star)
[v_\triv^\star],\,
1\otimes v_\triv
\bigr)^{\mathbf H}_{\mathscr T}
&&(\because \eqref{eq:sesquiHT})\\
&=\bigl(
\tilde\Gamma_0(\Phi)\circ\tilde\Gamma^{\CC^\star_\triv}_V(\Psi^\star)
[v_\triv^\star],\,
I_{\mathbf H}[v_\triv]
\bigr)^{\mathbf H}_{\mathscr T}\\
&=\sum_{i=1}^{m'}
\bigl(
\tilde\Gamma_0^V(\Phi)
[v_i],\,
I_{\mathbf H}\circ\bigl(\tilde\Gamma^{\CC^\star_\triv}_V(\Psi^\star)\bigr)^\star\,
[v_i^\star]
\bigr)^{\mathbf H}_{\mathscr T}
&&(\because \text{\normalfont Proposition \ref{prop:sesquiH}})\\
&=\sum_{i=1}^{m'}
\bigl(
\tilde\Gamma_0^V(\Phi)
[v_i],\,
\tilde\Gamma^{V^\star}_{\CC_\triv}(\Psi)
[v_i^\star]
\bigr)^{\mathbf H}_{\mathscr T}
&&(\because \eqref{eq:starGam2})
\end{align*}
and the proposition.
\end{proof}
Let $\mathbf1^{-\bar\lambda}_G\in B_G(-\bar\lambda)$ and $\mathbf1^{-\bar\lambda}_{\mathbf H}\in B_{\mathbf H}(-\bar\lambda)$ denote functions
which constantly take $1$ on $K$ and $W$ respectively.
Then $\mathbf1^{-\bar\lambda}_G$ and $\mathbf1^{-\bar\lambda}_{\mathbf H}$
are  $K$- and $W$-invariant respectively.
Since $\gamma_{B(-\bar\lambda)}\bigl(\mathbf1^{-\bar\lambda}_G\bigr)=\mathbf1^{-\bar\lambda}_{\mathbf H}$ it follows from Proposition \ref{prop:class1} that
there exists a morphism $\mathcal I^{-\bar\lambda}=\bigl(\mathcal I^{-\bar\lambda}_G,\mathcal I^{-\bar\lambda}_{\mathbf H}\bigr):
\bigl(P_G(\CC_\triv), P_{\mathbf H}(\CC_\triv)\bigr) \to \bigl(B_{G}(-\bar\lambda),B_{\mathbf H}(-\bar\lambda)\bigr)$ in $\Crad$ such that
$\mathcal I^{-\bar\lambda}_G(1\otimes v_\triv)=\mathbf1^{-\bar\lambda}_G$
and $\mathcal I^{-\bar\lambda}_{\mathbf H}(1\otimes v_\triv)=\mathbf1^{-\bar\lambda}_{\mathbf H}$.

\begin{lem}\label{lem:PoissonAdj}
\begin{align*}
\bigl(\mathcal P_G^\lambda F,\, D\otimes v_\triv\bigr)^G_r
&=\bthreeset{\lambda}{F,\, \mathcal I^{-\bar\lambda}_G(D\otimes v_\triv) }{G}{-\bar\lambda}
&&\text{for }F\in B_G(\lambda)\text{ and }D\in U(\Lieg_\CC),\\
\bigl(\mathcal P_{\mathbf H}^\lambda F,\, h\otimes v_\triv\bigr)^{\mathbf H}_{\mathscr T}
&=\bthreeset{\lambda}{F,\, \mathcal I^{-\bar\lambda}_{\mathbf H}(h\otimes v_\triv) }{{\mathbf H}}{-\bar\lambda}
&&\text{for }F\in B_{\mathbf H}(\lambda)\text{ and }h\in \mathbf H.
\end{align*}
\end{lem}
\begin{proof}
The first formula holds since
\begin{align*}
\bigl(\mathcal P_G^\lambda F,\, D\otimes v_\triv\bigr)^G_r
&=\bigl(D^\star\,\mathcal P_G^\lambda F,\, 1\otimes v_\triv\bigr)^G_r
=\bigl(\mathcal P_G^\lambda (D^\star F),\, 1\otimes v_\triv\bigr)^G_r\\
&=\mathcal P_G^\lambda (D^\star F)(1)
=\bthreeset{\lambda}{(D^\star F)(g),\, e^{(\bar\lambda+\rho)(A(g^{-1}))} }{G}{-\bar\lambda}\\
&=\bthreeset{\lambda}{D^\star F,\, \mathbf1^{-\bar\lambda}_G }{G}{-\bar\lambda}
=\bthreeset{\lambda}{D^\star F,\, \mathcal I^{-\bar\lambda}_G(1\otimes v_\triv) }{G}{-\bar\lambda}\\
&=\bthreeset{\lambda}{F,\, D\,\mathcal I^{-\bar\lambda}_G(1\otimes v_\triv) }{G}{-\bar\lambda}
=\bthreeset{\lambda}{F,\, \mathcal I^{-\bar\lambda}_G(D\otimes v_\triv) }{G}{-\bar\lambda}.
\end{align*}
One can deduce the second formula by an analogous calculation.
\end{proof}
\begin{lem}\label{lem:adjMor}
Suppose $\mathcal M^1=(\mathcal M^1_G, \mathcal M^1_{\mathbf H})$,
$\mathcal M^2=(\mathcal M^2_G, \mathcal M^2_{\mathbf H})\in \CCh$
and a pair of
invariant sesquilinear froms
$(\cdot,\cdot)^{G}_{\mathcal M}$ on $\mathcal M^1_G\times\mathcal M^2_G$ and
$(\cdot,\cdot)^{\mathbf H}_{\mathcal M}$
on $\mathcal M^1_{\mathbf H}\times\mathcal M^2_{\mathbf H}$
is compatible with restriction.
Suppose $\mathcal N^1=(\mathcal N^1_G, \mathcal N^1_{\mathbf H})$,
$\mathcal N^2=(\mathcal N^2_G, \mathcal N^2_{\mathbf H})\in \CCh$
and a pair of
invariant sesquilinear froms
$(\cdot,\cdot)^{G}_{\mathcal N}$ on $\mathcal N^1_G\times\mathcal N^2_G$ and
$(\cdot,\cdot)^{\mathbf H}_{\mathcal N}$
on $\mathcal N^1_{\mathbf H}\times\mathcal N^2_{\mathbf H}$
is compatible with restriction.
Suppose $\mathcal I^2=(\mathcal I^2_G,\mathcal I^2_{\mathbf H}):\mathcal N^2\to\mathcal M^2$ is a morphism in $\CCh$.
Suppose $\mathcal I^2_G$ and $\mathcal I^2_{\mathbf H}$ have thier adjoint
$\mathcal I^1_G$ and $\mathcal I^1_{\mathbf H}$ satisfying
\begin{align*}
\bigl(\mathcal I^1_G (y_1),\, y_2\bigr)^G_{\mathcal N}
&=\bigl(y_1,\, \mathcal I^2_G (y_2)\bigr)^G_{\mathcal M}
&&\text{for }y_1\in \mathcal M^1_G\text{ and }y_2\in \mathcal N^2_G,\\
\bigl(\mathcal I^1_{\mathbf H} (x_1),\, x_2\bigr)^{\mathbf H}_{\mathcal N}
&=\bigl(x_1,\, \mathcal I^2_{\mathbf H} (x_2)\bigr)^{\mathbf H}_{\mathcal M}
&&\text{for }x_1\in \mathcal M^1_{\mathbf H}\text{ and }x_2\in \mathcal N^2_{\mathbf H}.
\end{align*}
Finally suppose $(\mathcal N^2_G)^\perp=\{0\}$ and $(\mathcal N^2_{\mathbf H})^\perp=\{0\}$, that is
\begin{align}
\bigl\{
y_1\in\mathcal N^1_G;\,
(y_1,y_2)^{G}_{\mathcal N}=0\text{ for any } y_2\in\mathcal N^2_G
\bigr\}&=\{0\},\label{eq:NGperp0}\\
\bigl\{
x_1\in\mathcal N^1_{\mathbf H};\,
(x_1,x_2)^{{\mathbf H}}_{\mathcal N}=0\text{ for any } x_2\in\mathcal N^2_{\mathbf H}
\bigr\}&=\{0\}.\label{eq:NHperp0}
\end{align}
Then $\mathcal I^1_G$ and $\mathcal I^1_{\mathbf H}$ constitute a morphism
$\mathcal I^1:=(\mathcal I^1_G,\mathcal I^1_{\mathbf H}):\mathcal M^1\to\mathcal N^1$ in $\CCh$.
\end{lem}
\begin{proof}
We have to check 
$\mathcal I^1$ satisfies
Conditions \eqref{cond:Ch1} and \eqref{cond:Ch2}
in Definition \ref{defn:CCh}.
To do so, let $V\in\Km$ be arbitrary and 
take bases $\{v_1,\ldots,v_{m'},\ldots,v_n\}\subset V$
and $\{v_i^\star\}\subset V^\star$
as in Definition \ref{defn:compRM}.
We first confirm \eqref{cond:Ch1},
namely 
$\tilde\Gamma^V_{\mathcal N^1}(\mathcal I_G^1\circ\Phi_1)=
\mathcal I_{\mathbf H}^1\circ\tilde\Gamma^V_{\mathcal M^1}(\Phi_1)$
for any $\Phi_1\in\Hom_K(V,\mathcal M^1_G)$.
Because of \eqref{eq:NHperp0} it is enough to show
\[
\sum_{i=1}^{m'} \bigl(
\tilde\Gamma^V_{\mathcal N^1}(\mathcal I_G^1\circ\Phi_1)[v_i],\,\varphi_2[v_i^\star]
\bigr)_{\mathcal N}^{\mathbf H}
=
\sum_{i=1}^{m'} \bigl(
\mathcal I_{\mathbf H}^1\circ\tilde\Gamma^V_{\mathcal M^1}(\Phi_1)[v_i],\,\varphi_2[v_i^\star]
\bigr)_{\mathcal N}^{\mathbf H}
\]
for any $\varphi_2\in\Hom_W((V^\star)^M_\single,\mathcal N_{\mathbf H}^2)$.
But if we take $\Phi_2\in\Hom_K^\ttt(V^\star,\mathcal N_{G}^2)$
so that 
$\tilde\Gamma^{V^\star}_{\mathcal N^2}(\Phi_2)=\varphi_2$ then
\begin{align*}
\sum_{i=1}^{m'} \bigl(
\tilde\Gamma^V_{\mathcal N^1}(\mathcal I_G^1&\circ\Phi_1)[v_i],\,\varphi_2[v_i^\star]
\bigr)_{\mathcal N}^{\mathbf H}\\
&=\sum_{i=1}^{n} \bigl(
\mathcal I_G^1\circ\Phi_1[v_i],\,\Phi_2[v_i^\star]
\bigr)_{\mathcal N}^G\\
&=\sum_{i=1}^{n} \bigl(
\Phi_1[v_i],\,\mathcal I_G^2\circ\Phi_2[v_i^\star]
\bigr)_{\mathcal M}^G\\
&=\sum_{i=1}^{m'} \bigl(
\tilde\Gamma^V_{\mathcal M^1}(\Phi_1)[v_i],\,
\tilde\Gamma^{V^\star}_{\mathcal M^2}(\mathcal I_G^2\circ\Phi_2)[v_i]
\bigr)_{\mathcal M}^{\mathbf H}
&&(\because \mathcal I_G^2\circ\Phi_2\in\Hom_K^\ttt)\\
&=\sum_{i=1}^{m'} \bigl(
\tilde\Gamma^V_{\mathcal M^1}(\Phi_1)[v_i],\,
\mathcal I_{\mathbf H}^2\circ\tilde\Gamma^{V^\star}_{\mathcal N^2}(\Phi_2)[v_i]
\bigr)_{\mathcal M}^{\mathbf H}
&&(\because \eqref{cond:Ch1}\text{ for }\mathcal I^2)\\
&=\sum_{i=1}^{m'} \bigl(
\mathcal I_{\mathbf H}^1\circ\tilde\Gamma^V_{\mathcal M^1}(\Phi_1)[v_i],\,\varphi_2[v_i^\star]
\bigr)_{\mathcal N}^{\mathbf H}.
\end{align*}
Thus $\mathcal I^1$ satisfies Condition \eqref{cond:Ch1}.

Secondly we assert
\begin{equation}\label{eq:22chara}
\Hom_K^\ttt(V,\mathcal N^1_G)
=\left\{
\Phi_1\in \Hom_K(V,\mathcal N^1_G);\,
\begin{aligned}
&\sum_{i=1}^n
\bigl(
\Phi_1[v_i],\,\Phi_2[v_i^\star]
\bigr)_{\mathcal N}^G=0\\
&\text{ for all }\Phi_2\in\Hom_K(V^\star,\mathcal N^2_G)\\
&
\text{ such that }\tilde\Gamma_{\mathcal N^2}^{V^\star}(\Phi_2)=0
\end{aligned}
\right\}.
\end{equation}
Indeed `$\subset$' is immediate from the compatibility of the sesquilinear forms with restriction.
In order to show the inverse inclusion,
take any $\Phi_1$ in the right-hand side of \eqref{eq:22chara}.
Then there exists a unique $\Phi_1'\in\Hom_K^\ttt(V,\mathcal N^1_G)$ such that
$\tilde\Gamma^V_{\mathcal N^1}(\Phi_1)=\tilde\Gamma^V_{\mathcal N^1}(\Phi_1')$.
Now suppose $\Phi_2\in\Hom_K(V^\star,\mathcal N^2_G)$ is arbitrary
and take $\Phi_2'\in\Hom_K^\ttt(V^\star,\mathcal N^2_G)$
so that
$\tilde\Gamma^{V^\star}_{\mathcal N^2}(\Phi_2)=\tilde\Gamma^{V^\star}_{\mathcal N^2}(\Phi_2')$.
Since $\tilde\Gamma_{\mathcal N^2}^{V^\star}(\Phi_2-\Phi_2')=0$
and $\Phi_1-\Phi_1'$ belongs to the right-hand side of \eqref{eq:22chara},
we have
\begin{align*}
\sum_{i=1}^n
\bigl(
(\Phi_1-\Phi_1')&[v_i],\,\Phi_2[v_i^\star]
\bigr)_{\mathcal N}^G\\
&=\sum_{i=1}^n
\bigl(
(\Phi_1-\Phi_1')[v_i],\,(\Phi_2-\Phi_2')[v_i^\star]
\bigr)_{\mathcal N}^G
+\sum_{i=1}^n
\bigl(
(\Phi_1-\Phi_1')[v_i],\,\Phi_2'[v_i^\star]
\bigr)_{\mathcal N}^G\\
&=0+\sum_{i=1}^{m'}
\bigl(
\tilde\Gamma^V_{\mathcal N^1}(\Phi_1-\Phi_1')[v_i],\,
\tilde\Gamma^{V^\star}_{\mathcal N^2}(\Phi_2')[v_i^\star]
\bigr)_{\mathcal N}^{\mathbf H}\\
&=0.
\end{align*}
Hence $\Phi_1-\Phi_1'=0$ by \eqref{eq:NGperp0}, proving \eqref{eq:22chara}.

Let us prove \eqref{cond:Ch2} for $\mathcal I^1$.
Suppose $\Phi_1\in\Hom_K^\ttt(V,\mathcal M_G^1)$.
Then
for any $\Phi_2\in\Hom_K(V^\star,\mathcal N^2_G)$ with 
$\tilde\Gamma_{\mathcal N^2}^{V^\star}(\Phi_2)=0$  it holds that
\begin{align*}
\sum_{i=1}^n
\bigl(
\mathcal I^1_G\circ\Phi_1[v_i],\,\Phi_2[v_i^\star]
\bigr)_{\mathcal N}^G
&=\sum_{i=1}^n
\bigl(
\Phi_1[v_i],\,\mathcal I^2_G\circ\Phi_2[v_i^\star]
\bigr)_{\mathcal M}^G\\
&=\sum_{i=1}^{m'}
\bigl(
\tilde\Gamma^V_{\mathcal M^1}(\Phi_1)[v_i],\,
\tilde\Gamma^{V^\star}_{\mathcal M^2}(\mathcal I^2_G\circ\Phi_2)[v_i^\star]
\bigr)_{\mathcal M}^{\mathbf H}\\
&=\sum_{i=1}^{m'}
\bigl(
\tilde\Gamma^V_{\mathcal M^1}(\Phi_1)[v_i],\,
\mathcal I^2_{\mathbf H}\circ\tilde\Gamma^{V^\star}_{\mathcal N^2}(\Phi_2)[v_i^\star]
\bigr)_{\mathcal M}^{\mathbf H}
&&(\because \eqref{cond:Ch1}\text{ for }\mathcal I^2)\\
&=0.
\end{align*}
Hence $\mathcal I^1_G\circ\Phi_1\in\Hom_K^\ttt(V,\mathcal N^1_G)$ by \eqref{eq:22chara}.
\end{proof}
\begin{proof}[Proof of {\normalfont Theorem \ref{thm:Poisson}}]
We have only to apply Lemma \ref{lem:adjMor}
to the case where
\begin{align*}
\mathcal M^1&=\bigl(B_G(\lambda)_\Kf, B_{\mathbf H}(\lambda)\bigr),
&\mathcal M^2&=\bigl(B_G(-\bar\lambda)_\Kf, B_{\mathbf H}(-\bar\lambda)\bigr),\\
\mathcal N^1&=\bigl(\mathscr A(G/K,\lambda)_\Kf,\mathscr A(A,\lambda)\bigr),
&\mathcal N^2&=\bigl(P_G(\CC_\triv), P_{\mathbf H}(\CC_\triv)\bigr),\\
\mathcal I^1&=(\mathcal P_G^{\lambda},\mathcal P_{\mathbf H}^\lambda),
&\mathcal I^2&=\mathcal I^{-\bar\lambda}.\qedhere
\end{align*}
\end{proof}
The following property of the Poisson transforms will be used in the next section:
\begin{prop}\label{prop:PoissonUniq}
\begin{gather*}
\Hom_{\Lieg_\CC,K}\bigl(B_G(\lambda)_\Kf,\,\mathscr A(G/K,\lambda)_\Kf\bigr)=\CC\,\mathcal P_G^\lambda,\\
\Hom_{\mathbf H}\bigl(B_{\mathbf H}(\lambda),\,\mathscr A(A,\lambda)\bigr)=\CC\,\mathcal P_{\mathbf H}^\lambda.
\end{gather*}
\end{prop}
\begin{proof}
We only prove the first equality since the second one can be proved in the same way.
Suppose $\mathcal I_G\in \Hom_{\Lieg_\CC,K}\bigl(B_G(\lambda)_\Kf,\,\mathscr A(G/K,\lambda)_\Kf\bigr)$ is given.
Since a $K$-invariant function in $\mathscr A(G/K,\lambda)_\Kf$ is a scalar multiple of
$\phi_\lambda=\mathcal P_G^\lambda\mathbf 1_G^\lambda$,
there exists a constant $c$ such that
$\mathcal I_G\mathbf 1_G^\lambda=c\,\mathcal P_G^\lambda\mathbf 1_G^\lambda$.
Hence $\Ker(\mathcal I_G-c\,\mathcal P_G^\lambda)\supset \CC\mathbf 1_G^\lambda$.
Since the multiplicity of the trivial $K$-type in $B_G(\lambda)_\Kf$ is $1$,
$\Image(\mathcal I_G-c\,\mathcal P_G^\lambda)$ does not contain $\phi_\lambda$.
It then follows from Propostion \ref{prop:modstrAGl} (iii) that 
$\Image(\mathcal I_G-c\,\mathcal P_G^\lambda)=\{0\}$.
This proves $\mathcal I_G=c\,\mathcal P_G^\lambda$.
\end{proof}

We conclude this section by presenting a new series of radial pairs.
Put
\begin{align*}
D_G(\CC_\triv,\bar\lambda)&:=
\sum_{\Delta \in U(\Lieg_\CC)^K} U(\Lieg_\CC)(\Delta-\gamma(\Delta)(\bar\lambda))
\otimes v_\triv \subset P_G(\CC_\triv),\\
D_{\mathbf H}(\CC_\triv,\bar\lambda)
&:= 
\sum_{\Delta\in S(\Liea_\CC)^W} \mathbf H(\Delta-\Delta(\bar\lambda)) \otimes v_\triv
\subset P_{\mathbf H}(\CC_\triv)
\end{align*}
and recall
\[
P_G(\CC_\triv,\bar\lambda)=P_G(\CC_\triv) / D_G(\CC_\triv,\bar\lambda),\quad
P_{\mathbf H}(\CC_\triv,\bar\lambda)
=P_{\mathbf H}(\CC_\triv) / D_{\mathbf H}(\CC_\triv,\bar\lambda).
\]
\begin{prop}\label{prop:Aduals}
Let
\[
\Xi^\natural 
: \{\mathbf H\text{-submodules of }P_{\mathbf H}(\CC_\triv)\}\to\{(\Lieg_\CC,K)\text{-submodules of }P_G(\CC_\triv)\}
\]
be the correspondence defined in {\normalfont Proposition \ref{prop:XinaturalGen} (iv)}.
Then
\[
\Xi^\natural\bigl(D_{\mathbf H}(\CC_\triv,\bar\lambda)\bigr)
=\Ximin\bigl(D_{\mathbf H}(\CC_\triv,\bar\lambda)\bigr)
=D_G(\CC_\triv,\bar\lambda).\]
Hence by {\normalfont Proposition \ref{prop:XinaturalGen} (iii)},
$\bigl(P_G(\CC_\triv,\bar\lambda),\,P_{\mathbf H}(\CC_\triv,\bar\lambda)\bigr)$
is a radial pair with a radial restriction satisfying \eqref{cond:rest3}.
\end{prop}
\begin{proof}
For any $\Delta\in U(\Lieg_\CC)^K$,
$(\CC_\triv\ni cv_\triv\mapsto
c(\gamma(\Delta)-\gamma(\Delta)(\bar\lambda))\otimes v_\triv
\in P_{\mathbf H}(\CC_\triv))
\in\Hom_W(\CC_\triv, P_{\mathbf H}(\CC_\triv))$
is lifted to
$(\CC_\triv\ni cv_\triv\mapsto
c(\Delta-\gamma(\Delta)(\bar\lambda))\otimes v_\triv
\in P_{G}(\CC_\triv))
\in\Hom_K^\ttt(\CC_\triv, P_{G}(\CC_\triv))$.
Hence
\[
D_G(\CC_\triv,\bar\lambda)\subset
\Ximin\bigl(D_{\mathbf H}(\CC_\triv,\bar\lambda)\bigr)\subset
\Xi^\natural\bigl(D_{\mathbf H}(\CC_\triv,\bar\lambda)\bigr).
\]
Suppose $w\in W$ and put
\[
U(w\lambda):=
\bigl\{
D\in P_G(\CC_\triv);\, 
\bigl(\mathcal P_G^{w\lambda} F,\, D\bigr)^G_r=0
\text{ for any }F\in B_G(w\lambda)
\bigr\}.
\]
Let us prove
\begin{equation}\label{eq:Uwlambda}
U(w\lambda)
=
\sum\bigl\{
\mathscr V\subset P_G(\CC_\triv);\, \text{a $K$-stable subspace with }
\gamma(D)(w\bar\lambda)=0\text{ for }\forall D\in \mathscr V
\bigr\}
\end{equation}
where we identify $P_{\mathbf H}(\CC_\triv)$ with $S(\Liea_\CC)$
as in Example \ref{exmp:triv}.
Since both sides of \eqref{eq:Uwlambda}
are $(\Lieg_\CC,K)$-submodules of $P_G(\CC_\triv)$,
it suffices to show for any $V\in\Km$
\[
\Hom_K(V,U(w\lambda))=\bigl\{\Phi\in
\Hom_K(V,P_G(\CC_\triv));\,\gamma(\Phi[v])(w\bar\lambda)=0\text{ for }v\in V
\bigr\}.
\]
Let $\Phi\in
\Hom_K(V,P_G(\CC_\triv))$ be arbitrary.
If we choose $D\in U(\Lien_\CC+\Liea_\CC)$ for $v\in V$
so that $\Phi[v]=D\otimes v_\triv$
then
\[
(\mathcal I^{-w\bar\lambda}_{G}
\circ \Phi[v])(1)
=\mathcal I^{-w\bar\lambda}_{G}
[D\otimes v_\triv](1)
=(D\, \mathbf 1_G^{-w\bar\lambda})(1)
=\gamma(D)(w\bar\lambda)
=\gamma(\Phi[v])(w\bar\lambda).
\]
Hence we have
\begin{align*}
\gamma(\Phi&[v])(w\bar\lambda)=0\quad\text{for any } v\in V\\
&\Longleftrightarrow
(\mathcal I^{-w\bar\lambda}_{G}
\circ \Phi[k^{-1}v])(1)=0
\quad\text{for any } v\in V\text{ and } k\in K\\
&\Longleftrightarrow
(\mathcal I^{-w\bar\lambda}_{G}
\circ \Phi[v])(k)=0
\quad\text{for any } v\in V\text{ and } k\in K\\
&\Longleftrightarrow
\bthreeset{w\lambda}{F,\, \mathcal I^{-w\bar\lambda}_{G}
\circ \Phi[v]
}{G}{-w\bar\lambda}=0\quad\text{for any } v\in V\text{ and }F\in B_G(w\lambda)\\
&\Longleftrightarrow
\bigl(\mathcal P_G^{w\lambda}F,\, \Phi[v]\bigr)^G_r=0\quad\text{for any }
v\in V\text{ and }F\in B_G(w\lambda)
\qquad(\because \text{Lemma \ref{lem:PoissonAdj}})\\
&\Longleftrightarrow
\Phi\in \Hom_K(V,U(w\lambda)).
\end{align*}
Thus we get \eqref{eq:Uwlambda}.
Hence in particular $\Xi^\natural(D_{\mathbf H}(\CC_\triv,\bar\lambda))\subset U(w\lambda)$.

Now choose $w\in W$ so that
\[
\mathcal P_G^{w\lambda}\bigl(B_G(w\lambda)_\Kf\bigr)
=\mathscr A(G/K, w\lambda)_\Kf=\mathscr A(G/K, \lambda)_\Kf
\]
(cf.~Remark \ref{rem:PoissonBij} (ii)).
Then $D_G(\CC_\triv,\bar\lambda)=U(w\lambda)$ by Proposition \ref{prop:modstrAGl} (ii).
Hence
\[
\Xi^\natural(D_{\mathbf H}(\CC_\triv,\bar\lambda))\subset U(w\lambda)=D_G(\CC_\triv,\bar\lambda).
\qedhere
\]
\end{proof}

\section{Intertwining operators}\label{sec:intertwining}
For an arbitrary $w\in W$ let $\bar w\in N_K(\Liea)$ be its lift.
Let $d\bar n$ be a Haar measure of $\bar w^{-1}\!N\bar w\cap \theta N$.
For $\lambda\in\Liea^*_\CC$, \cite{KuS} shows
the intertwining operator
$\mathcal A_G(w,\lambda):
B_G(\lambda)\to B_G(w \lambda)$
formally given by
\[
\mathcal A_G(w,\lambda)F\,(g)
=\int_{\bar w^{-1}\!N\bar w\,\cap\, \theta N} F(g\bar w\bar n)\,d\bar n
\]
converges and makes sense when $-\Real\lambda$ is sufficiently dominant.
(By \cite{Sch} the integral is convergent
when $\lambda$ satisfies $\Real \lambda(\alpha^\vee)<0$
for all $\alpha\in \Sigma^+\cap -w^{-1}\Sigma^+$.)
This operator clearly does not depend on the choice of $\bar w$.
In \cite{KnS1} Knapp and Stein prove that
$\mathcal A_G(w,\lambda)$,
as an operator acting on $C^\infty(K/M)\simeq B_G(\lambda)$
with the holomorphic parameter $\lambda$,
extends meromorphically in $\lambda$ to the whole $\Liea_\CC^*$.
Now let us assume for each $\alpha\in \Pi$
the Haar measure $d\bar n$ of $\bar s_\alpha^{-1} N \bar s_\alpha\cap\theta N$
is normalized so that
\[
\int_{\bar s_\alpha^{-1} N \bar s_\alpha\,\cap\, \theta N}e^{2\rho(A(\bar n))}\,d\bar n=1.
\]
It then follows from a result of \cite{Sch} and its correction by \cite[\S2]{KnS2}
that we can normalize other Haar measures so that
\begin{equation}\label{eq:AchainG}
\mathcal A_G(w,\lambda)=\mathcal A_G(w_1,w_2\lambda)\,\mathcal A_G(w_2,\lambda)
\end{equation}
whenever
$w=w_1w_2$ ($w_1,w_2\in W$) is a minimal decomposition
(namely, the length of $w$ is the sum of those of $w_1$ and $w_2$)
and all intertwining operators in the formula make sense.

\begin{defn}
For $\alpha\in R_1^+$ put
\begin{align*}
\mathbf e_\alpha(\lambda)&=
\Bigl\{
{\Gamma\bigl(\tfrac12\bigl(\tfrac12\dim \Lieg_{\alpha/2}+1+\lambda(\alpha^\vee)\bigr)\bigr)\,%
\Gamma\bigl(\tfrac12\bigl(\mathbf m_1(\alpha)+\lambda(\alpha^\vee)\bigr)\bigr)}
\Bigr\}^{-1},\\
\mathbf c_\alpha(\lambda)&=
2^{-\lambda(\alpha^\vee)}\,\Gamma(\lambda(\alpha^\vee))\,\mathbf e_\alpha(\lambda).
\end{align*}
For $\alpha\in \Pi$ define
\[
\Tilde{\mathcal A}_G(s_\alpha,\lambda)
=
\frac{\mathbf c_\alpha(\rho)}{\mathbf c_\alpha(-\lambda)}\,
\mathcal A_G(s_\alpha,\lambda).
\]
\end{defn}
\begin{prop}\label{prop:AG}
{\normalfont (i)}
For any $\alpha\in\Pi$
the intertwining operator $\Tilde{\mathcal A}_G(s_\alpha,\lambda):
B_G(\lambda)\to B_G(s_\alpha\lambda)$ makes sense
if and only if $\mathbf e_\alpha(-\lambda)\ne0$.
For such $\lambda$ it holds that
\begin{equation}\label{eq:image1G}
\Tilde{\mathcal A}_G(s_\alpha,\lambda)\mathbf 1_G^{\lambda}
=\mathbf 1_G^{s_\alpha\lambda}
\end{equation}
where $\mathbf 1_G^{\lambda}\in B_G(\lambda)$ and
$\mathbf 1_G^{s_\alpha\lambda}\in B_G(s_\alpha\lambda)$ are
functions taking the constant value $1$ on $K$ as in the last section.

\noindent
{\normalfont (ii)}
Suppose $w\in W$
and
let $w=s_{\alpha_1}\cdots s_{\alpha_k}$ be 
a reduced expression of $w$.
Then the intertwining operator
\[
\Tilde{\mathcal A}_G(w,\lambda)
=
\Tilde{\mathcal A}_G(s_{\alpha_1},s_{\alpha_2}\cdots s_{\alpha_k}\lambda)
\Tilde{\mathcal A}_G(s_{\alpha_2},s_{\alpha_3}\cdots s_{\alpha_k}\lambda)
\cdots
\Tilde{\mathcal A}_G(s_{\alpha_k},\lambda)
\]
is defined independently of the expression.

\noindent
{\normalfont (iii)}
For any $w_1,w_2\in W$ it holds that
\[
\Tilde{\mathcal A}_G(w_1w_2,\lambda)
=
\Tilde{\mathcal A}_G(w_1,w_2\lambda)
\Tilde{\mathcal A}_G(w_2,\lambda).
\]
\end{prop}
\begin{proof}
In the proof of (i) we may assume $G$ has real rank $1$
without loss of generality (cf.~\cite{KnS1}).
If $\Real\lambda(\alpha^\vee)<0$ then we have for $k\in K$
\begin{equation}\label{eq:calc1G}
\begin{aligned}
\mathcal A_G(s_\alpha,\lambda)\mathbf 1_G^\lambda\,(k)
&=\int_{\bar s_\alpha^{-1}N\bar s_\alpha\,\cap\, \theta N} \mathbf 1_G^\lambda(k\bar s_\alpha\bar n)\,d\bar n
=\int_{\bar s_\alpha^{-1}N\bar s_\alpha\,\cap\, \theta N} \mathbf 1_G^\lambda(\bar n)\,d\bar n\\
&=\int_{\bar s_\alpha^{-1}N\bar s_\alpha\,\cap\, \theta N} e^{(\lambda-\rho)(-A(\bar n^{-1}))}\,d\bar n
=\frac{\mathbf c_\alpha(-\lambda)}{\mathbf c_\alpha(\rho)}.
\end{aligned}
\end{equation}
Here the last equality follows from \cite[Ch.\,IV, Theorem 6.4]{Hel4}.
By analytic continuation,
\eqref{eq:calc1G} is valid
whenever $\mathcal A_G(s_\alpha,\lambda)$ is defined and
\eqref{eq:image1G} is valid
whenever $\Tilde{\mathcal A}_G(s_\alpha,\lambda)$ is defined.
Recall in general the Poisson transform $\mathcal P_G^\lambda : B_{G}(\lambda)_\Kf\to \mathscr A(G/K, \lambda)_\Kf$ is bijective if and only if $\mathbf e_\alpha(\lambda)\ne0$
(cf.~Remark \ref{rem:PoissonBij}).

Suppose $\mathbf e_\alpha(-\lambda_0)\ne0$.
Since $\mathbf e_\alpha(-\bar\lambda_0)=\overline{\mathbf e_\alpha(-\lambda_0)}
\ne0$
we have $B_{G}(-\bar\lambda_0)_\Kf\simeq \mathscr A(G/K, -\bar\lambda_0)_\Kf$
and it follows from Proposition \ref{prop:sesquiGinv}
and Proposition \ref{prop:modstrAGl} (ii)
that the $(\Lieg_\CC,K)$-module
$B_G(\lambda_0)_\Kf$ is equivalent to $P_G(\CC_\triv,-\lambda_0)$
and is generated by $\mathbf 1_G^{\lambda_0}$.
If $\lambda_0$ is a singular point for which
$\mathcal A_G(s_\alpha,\lambda_0)$ does not make sense
then it follows from \cite[Theorem 3]{KnS1} that 
the regularized intertwining operator
$(\lambda(\alpha^\vee)-\lambda_0(\alpha^\vee))\,\mathcal A_G(s_\alpha,\lambda)$ is well defined around $\lambda=\lambda_0$ and non-zero.
In this case since
\begin{align*}
(\lambda(\alpha^\vee)-\lambda_0(\alpha^\vee))\,
&\mathcal A_G(s_\alpha,\lambda)\Bigr|_{\lambda=\lambda_0}
\bigl(B_G(\lambda_0)_\Kf
\bigr)\\
&=
(\lambda(\alpha^\vee)-\lambda_0(\alpha^\vee))\,\mathcal A_G(s_\alpha,\lambda)\Bigr|_{\lambda=\lambda_0}
\bigl(U(\Lieg_\CC)\,\mathbf 1_G^{\lambda_0}
\bigr)\\
&=U(\Lieg_\CC)\biggl(
(\lambda(\alpha^\vee)-\lambda_0(\alpha^\vee))\,
\frac{\mathbf c_\alpha(-\lambda)}{\mathbf c_\alpha(\rho)}\biggr|_{\lambda=\lambda_0}
\biggr)\,\mathbf 1_G^{s_\alpha\lambda_0}
\ne\{0\}
\end{align*}
by \eqref{eq:calc1G}, we can conclude $\lambda=\lambda_0$
is a pole of $\mathbf c_\alpha(-\lambda)$
and $\Tilde{\mathcal A}_G(s_\alpha,\lambda_0)$ is well defined.
On the other hand, if $\mathcal A_G(s_\alpha,\lambda_0)$ makes sense
then $\Tilde{\mathcal A}_G(s_\alpha,\lambda_0)$
also makes sense since $\mathbf c_\alpha(-\lambda)^{-1}$ is regular at $\lambda=\lambda_0$.

Suppose now $\mathbf e_\alpha(-\lambda_0)=0$.
Then $\mathcal P_G^{\lambda_0} : B_{G}(\lambda_0)_\Kf\to \mathscr A(G/K, \lambda_0)_\Kf$ is bijective
since $\mathbf e_\alpha(\lambda_0)\ne0$
while $\mathcal P_G^{s_\alpha\lambda_0} : B_{G}(s_\alpha\lambda_0)_\Kf\to \mathscr A(G/K, s_\alpha\lambda_0)_\Kf$ is not bijective
since
$\mathbf e_\alpha(s_\alpha\lambda_0)=\mathbf e_\alpha(-\lambda_0)=0$.
The latter means $\mathcal P_G^{s_\alpha\lambda_0}\bigr|_{B_G(s_\alpha\lambda_0)_\Kf}$ is not surjective.
Now assume $\Tilde{\mathcal A}_G(s_\alpha,\lambda_0)$ makes sense.
Then $\mathcal P_G^{s_\alpha\lambda_0}\circ\Tilde{\mathcal A}_G(s_\alpha,\lambda_0)\,\mathbf 1^{\lambda_0}_G=\phi_{\lambda_0}$ by \eqref{eq:image1G}
and it follows from Proposition \ref{prop:PoissonUniq} that
\[
\mathcal P_G^{\lambda_0}\bigr|_{B_G(\lambda_0)_\Kf}
=\mathcal P_G^{s_\alpha\lambda_0}\circ\Tilde{\mathcal A}_G(s_\alpha,\lambda_0)\bigr|_{B_G(\lambda_0)_\Kf}.
\]
This is a contradiction since
the left-hand side is bijective while the right-hand side is not.
Hence $\Tilde{\mathcal A}_G(s_\alpha,\lambda_0)$ cannot make sense.
Thus we get (i).

Let us return to the general case where $G$ may have higher rank.
For a generic $\lambda$ the right-hand side of the formula of (ii)
equals a scalar multiple of $\mathcal A_G(w,\lambda)$  by \eqref{eq:AchainG}.
This scalar does not depend on the decomposition
since $\Tilde{\mathcal A}_G(w,\lambda)\mathbf 1_G^{\lambda}
=\mathbf 1_G^{w\lambda}$ by (i).

To prove (iii) suppose $w_1,w_2\in W$.
Since it is well known that $B_G(\lambda)$ is irreducible when $\lambda\in i\Liea^*$,
\[
\Tilde{\mathcal A}_G(w_2,\lambda)^{-1}\Tilde{\mathcal A}_G(w_1,w_2\lambda)^{-1}\Tilde{\mathcal A}_G(w_1w_2,\lambda)
\in\End_G(B_G(\lambda))
\]
is a well-defined scalar operator.
This sends $\mathbf 1_G^{\lambda}$ to itself
and hence must be the identity.
Thus (iii) is valid for $\lambda\in i\Liea^*$,
from which the general case follows.
\end{proof}

Let us develop an analogous story for $\mathbf H$.
For $\alpha\in\Pi$ put
\[
\tau_\alpha= \alpha^\vee s_\alpha+\mathbf m_1(\alpha) \in \mathbf H.
\]
Then by \eqref{eq:Hrel} we have
\[
\tau_\alpha\,\xi=s_\alpha(\xi)\,\tau_\alpha\quad\text{for any }\xi\in\Liea_\CC^*.
\]
Suppose
$w\in W$ and let $w=s_{\alpha_1}\cdots s_{\alpha_k}$ be a reduced expression.
Then one can prove that
\[
\tau_w=\tau_{\alpha_1}\cdots\tau_{\alpha_k}\in \mathbf H
\]
is defined independently of the expression (see \cite[Theorem 4.2]{Op:book}).
Hence we can define the intertwining operator
$\mathcal A_{\mathbf H}(w,\lambda):
B_{\mathbf H}(\lambda)\to B_{\mathbf H}(w\lambda)$
by
\[
\mathcal A_{\mathbf H}(w,\lambda)F\,(h)=F(h\,\trans{\tau_{w^{-1}}}).
\]
If $w=w_1w_2$ is a minimal decomposition then it clearly holds that
\begin{equation*}
\mathcal A_{\mathbf H}(w,\lambda)
=\mathcal A_{\mathbf H}(w_1,w_2\lambda)\,\mathcal A_{\mathbf H}(w_2,\lambda).
\end{equation*}
\begin{defn}
For $\alpha\in\Pi$ define
\[
\Tilde{\mathcal A}_{\mathbf H}(s_\alpha,\lambda)
=\frac1{\mathbf m_1(\alpha)-\lambda(\alpha^\vee)}\,
\mathcal A_{\mathbf H}(s_\alpha,\lambda).
\]
\end{defn}
\begin{prop}\label{prop:AH}
{\normalfont (i)}
For any $\alpha\in\Pi$
the intertwining operator $\Tilde{\mathcal A}_{\mathbf H}(s_\alpha,\lambda):
B_{\mathbf H}(\lambda)\to B_{\mathbf H}(s_\alpha\lambda)$ makes sense
if and only if $\mathbf m_1(\alpha)\ne\lambda(\alpha^\vee)$.
For such $\lambda$ it holds that
\begin{equation}\label{eq:image1H}
\Tilde{\mathcal A}_{\mathbf H}(s_\alpha,\lambda)\mathbf 1_{\mathbf H}^{\lambda}
=\mathbf 1_{\mathbf H}^{s_\alpha\lambda}
\end{equation}
where $\mathbf 1_{\mathbf H}^{\lambda}\in B_{\mathbf H}(\lambda)$ and
$\mathbf 1_{\mathbf H}^{s_\alpha\lambda}\in B_{\mathbf H}(s_\alpha\lambda)$ are
functions taking the constant value $1$ on $W$.

\noindent
{\normalfont (ii)}
Suppose $w\in W$
and
let $w=s_{\alpha_1}\cdots s_{\alpha_k}$ be 
a reduced expression.
Then the intertwining operator
\[
\Tilde{\mathcal A}_{\mathbf H}(w,\lambda)
=
\Tilde{\mathcal A}_{\mathbf H}(s_{\alpha_1},s_{\alpha_2}\cdots s_{\alpha_k}\lambda)
\Tilde{\mathcal A}_{\mathbf H}(s_{\alpha_2},s_{\alpha_3}\cdots s_{\alpha_k}\lambda)
\cdots
\Tilde{\mathcal A}_{\mathbf H}(s_{\alpha_k},\lambda)
\]
is defined independently of the expression. 

\noindent
{\normalfont (iii)}
For any $w_1,w_2\in W$ it holds that
\[
\Tilde{\mathcal A}_{\mathbf H}(w_1w_2,\lambda)
=
\Tilde{\mathcal A}_{\mathbf H}(w_1,w_2\lambda)
\Tilde{\mathcal A}_{\mathbf H}(w_2,\lambda).
\]
\end{prop}
\begin{proof}
Suppose $\alpha\in\Pi$ and $\lambda\in\Liea_\CC^*$.
For $F\in B_{\mathbf H}(\lambda)$ one easily calculates
\begin{equation}\label{eq:explicitAH}
\mathcal A_{\mathbf H}(s_\alpha,\lambda)F\,(w)
=\mathbf m_1(\alpha)F(w)-\lambda(\alpha^\vee)F(ws_\alpha)
\quad\text{for }w\in W.
\end{equation}
Hence we readily have $\mathcal A_{\mathbf H}(s_\alpha,\lambda)\ne 0$
and
$\Tilde{\mathcal A}_{\mathbf H}(s_\alpha,\lambda)
\mathbf 1_{\mathbf H}^{\lambda}
=(\mathbf m_1(\alpha)-\lambda(\alpha^\vee))\mathbf 1_{\mathbf H}^{s_\alpha\lambda}$.
These facts prove (i).

We can prove (ii) in the same way as in the proof of Proposition \ref{prop:AG} (ii).

It is well known that $B_{\mathbf H}(\lambda)$ is irreducible if and only if
\begin{equation}\label{eq:BHirr}
\lambda(\alpha^\vee)\ne\mathbf m_1(\alpha)
\quad\text{for any }\alpha\in R_1.
\end{equation}
Indeed \eqref{eq:BHirr} is equivalent to the condition
that both $\mathcal P_{\mathbf H}^\lambda$ and
$\mathcal P_{\mathbf H}^{-\bar\lambda}$
are bijective by Proposition \ref{prop:PoiBij}.
The former bijectivity implies
any non-zero submodule of $B_{\mathbf H}(\lambda)$
contains $\mathbf 1_{\mathbf H}^\lambda$ by Theorem \ref{thm:modstr} (iii)
and the latter implies
$B_{\mathbf H}(\lambda)=\mathbf H\mathbf 1_{\mathbf H}^\lambda$
by Proposition \ref{prop:sesquiHinv} and Theorem \ref{thm:modstr} (ii).
Thus \eqref{eq:BHirr} implies the irreducibility of $B_{\mathbf H}(\lambda)$.
Conversely, if $B_{\mathbf H}(\lambda)$ is irreducible
then both $\mathcal P_{\mathbf H}^\lambda$ and
$\mathcal P_{\mathbf H}^{-\bar\lambda}$
are clearly bijective and \eqref{eq:BHirr} holds.
In particular $B_{\mathbf H}(\lambda)$ is irreducible when $\lambda\in i\Liea_\CC^*$
and we can prove (iii) in the same way
as in the proof of Proposition \ref{prop:AG} (iii).
\end{proof}
The following is the main result of this section:
\begin{thm}\label{thm:it}
Suppose $\alpha\in\Pi$ and $\lambda\in\Liea_\CC^*$ satisfy
$\mathbf e_\alpha(-\lambda)\ne0$.
(Hence both $\Tilde{\mathcal A}_{G}(s_\alpha,\lambda)$ and
$\Tilde{\mathcal A}_{\mathbf H}(s_\alpha,\lambda)$ are well defined.)
Then
\[
(\Tilde{\mathcal A}_G(s_\alpha,\lambda),\Tilde{\mathcal A}_{\mathbf H}(s_\alpha,\lambda)):\,
\bigl(B_G(\lambda)_\Kf, B_{\mathbf H}(\lambda)\bigr)
\longrightarrow
\bigl(B_G(s_\alpha\lambda)_\Kf, B_{\mathbf H}(s_\alpha\lambda)\bigr)
\]
is a morphism of $\Crad$.
That is, if\/ $V\in\Km$ then

\noindent{\normalfont (i)}
for any $\Phi\in\Hom_K(V,C^\infty(K/M))$ it holds that
\begin{equation}\label{eq:it1}
\Tilde{\mathcal A}_G(s_\alpha,\lambda)(\Phi[v])(1)
=\frac{\mathbf m_1(\alpha)\,\Phi[v](1)
-\lambda(\alpha^\vee)\,\Phi[s_\alpha v](1)}%
{\mathbf m_1(\alpha)-\lambda(\alpha^\vee)}\,
\quad\forall v\in V^M_\single
\end{equation}
(cf.~\eqref{eq:explicitAH}) and

\noindent{\normalfont (ii)}
for any $\Phi\in\Hom_K(V,C^\infty(K/M))$ such that
\[
\Phi[v](1)=0\quad\forall v\in V^M_\double
\]
it holds that
\begin{equation}\label{eq:it2}
\Tilde{\mathcal A}_G(s_\alpha,\lambda)(\Phi[v])(1)=0
\quad\forall v\in V^M_\double.
\end{equation}
If\/ $\Real\lambda(\alpha^\vee)<0$ then \eqref{eq:it1} is written more explicitly as
\[
\int_{\bar s_\alpha^{-1}N\bar s_\alpha\,\cap\, \theta N}
\Phi[s_\alpha v](\bar n)\,d\bar n
=
\frac{\mathbf c_\alpha(-\lambda)}{\mathbf c_\alpha(\rho)}\cdot
\frac{\mathbf m_1(\alpha)\,\Phi[v](1)
-\lambda(\alpha^\vee)\,\Phi[s_\alpha v](1)}%
{\mathbf m_1(\alpha)-\lambda(\alpha^\vee)}.
\]
\end{thm}
\begin{proof}
Since each side of \eqref{eq:it1} and \eqref{eq:it2}
is holomorphic in $\lambda$ for any fixed $\Phi$ and $v$,
it suffices to prove the theorem when $\lambda\in i\Liea_\CC^*$.
In this case both $\mathcal P_G^{s_\alpha\lambda}\bigr|_{B_G(s_\alpha\lambda)_\Kf}$
and $\mathcal P_{\mathbf H}^{s_\alpha\lambda}$ are bijective.
Therefore
\[
\mathcal P^{s_\alpha\lambda}=
(\mathcal P_G^{s_\alpha\lambda},\mathcal P_{\mathbf H}^{s_\alpha\lambda}):\,\bigl(B_G(s_\alpha\lambda)_\Kf, B_{\mathbf H}(s_\alpha\lambda)\bigr)
\longrightarrow\bigl(\mathscr A(G/K,s_\alpha\lambda)_\Kf,\mathscr A(A,s_\alpha\lambda)\bigr)
\]
is an isomorphism in $\Crad$ and its inverse is given by
\[
(\mathcal P^{s_\alpha\lambda})^{-1}=
\biggl(
\bigl(\mathcal P_G^{s_\alpha\lambda}\bigr|_{B_G(s_\alpha\lambda)_\Kf}
\bigr)^{-1},\,
\bigl(
\mathcal P_{\mathbf H}^{s_\alpha\lambda}\bigr)^{-1}
\biggr).
\]
Since it follows from Proposition \ref{prop:PoissonUniq} that
\begin{align*}
\mathcal P_G^{\lambda}\bigr|_{B_G(\lambda)_\Kf}
&=\mathcal P_G^{s_\alpha\lambda}\bigr|_{B_G(s_\alpha\lambda)_\Kf}
\circ\Tilde{\mathcal A}_G(s_\alpha,\lambda)\bigr|_{B_G(\lambda)_\Kf},\\
\mathcal P_{\mathbf H}^{\lambda}
&=\mathcal P_{\mathbf H}^{s_\alpha\lambda}
\circ\Tilde{\mathcal A}_{\mathbf H}(s_\alpha,\lambda),
\end{align*}
we conclude
\[
(\Tilde{\mathcal A}_G(s_\alpha,\lambda),\Tilde{\mathcal A}_{\mathbf H}(s_\alpha,\lambda))
=(\mathcal P^{s_\alpha\lambda})^{-1}\circ \mathcal P^{\lambda}
\]
is a morphism in $\Crad$.
\end{proof}
 
\section{The Helgason-Fourier transform and the Opdam-Cherednik transform}\label{sec:F}
Helgason introduces the Fourier transform on $G/K$ in \cite{Hel1}
as a non-invariant version of the spherical transform
while Opdam defines an analogous transform on $A$ in \cite{Op:Cherednik}.
We shall show they constitute a morphism of
radial pairs and study some related topics from this view point.

Recall the symbol $C_\cpt^\infty$ stands for the class of compactly supported $C^\infty$-functions.
By Theorem \ref{thm:Ch} for $\mathscr F=C^\infty_\cpt$ and Proposition \ref{prop:R3GK},
$\bigl(C^\infty_\cpt(G/K)_\Kf,C^\infty_\cpt(A)\bigr)$ is a radial pair with
radial restriction $\gamma_0$.
Put $\ell=\dim\Liea$ and let
$dH$ be $|W|^{-1}(2\pi)^{-\ell/2}$ times of the Euclidean measure on $\Liea$ relative to the metric given by the Killing form
$B(\cdot,\cdot)$.
Let $da$ be the corresponding Haar measure on $A$.
We normalize a $G$-invariant non-zero volume element $dx$ on $G/K$
so that
\begin{equation}\label{eq:HermiteRes}
\int_{G/K} f(x)\,dx
=\int_A f(a)
\prod_{\alpha\in\Sigma^+}\bigl|2\sinh \alpha(\log a)\bigr|^{\dim\Lieg_\alpha}\,da\quad
\text{for }f\in C_\cpt^\infty(G/K)^{\ell(K)}
\end{equation}
(see \cite[Ch.\,I,~Theorem~5.8]{Hel4}).
The sesquilinear form
\[
\sang{f_1}{f_2}_{G/K}=\int_{G/K}f_1(x)\overline{f_2(x)}\,dx
\]
on $C_\cpt^\infty(G/K)\times C^\infty(G/K)$ is invariant and non-degenerate
and of course the restriction of this form to $C_\cpt^\infty(G/K)\times C_\cpt^\infty(G/K)$
is a Hermitian inner product.
Define a sesquilinear form
\[
\sang{f_1}{f_2}_A
=\int_A f_1(a)\overline{f_2(a)}\prod_{\alpha\in\Sigma^+}\bigl|2\sinh \alpha(\log a)\bigr|^{\dim\Lieg_\alpha}\,da
\]
on $C_\cpt^\infty(A)\times C^\infty(A)$.
It has the following invariance property:
\begin{prop}[{\cite[Lemma~7.8]{Op:Cherednik}}]\label{prop:Hstar}
For any $f_1\in C_\cpt^\infty(A)$, $f_2\in C^\infty(A)$ and $h \in\mathbf H$ it holds that
\[
\sang{\mathscr T(h)f_1}{f_2}_A
=\sang{f_1}{\mathscr T(h^\star)f_2}_A.
\]
\end{prop}
As in the previous sections, we consider $C^\infty(A)$ (or $C^\infty_\cpt(A)$)
as an $\mathbf H$-module by
$hf=\mathscr T(\theta_{\mathbf H}h)f$.
Since $(\theta_{\mathbf H}h)^\star=\theta_{\mathbf H}(h^\star)$ for $h\in\mathbf H$, $(\cdot,\cdot)_A$ is a non-degenerate invariant sesquilinear form
on $C_\cpt^\infty(A)\times C^\infty(A)$ which restricts to
a Hermitian inner product on $C_\cpt^\infty(A)\times C^\infty_\cpt(A)$.
The next proposition is an easy corollary of \eqref{eq:HermiteRes}:
\begin{prop}
Suppose $V\in\Km$ and take bases $\{v_1,\ldots,v_{m},\ldots,v_n\}\subset V$
and $\{v_i^\star\}\subset V^\star$ as in
{\normalfont Proposition \ref{prop:Blcomp}}.
Then for any $\Phi_1\in\Hom_K(V,C^\infty_\cpt(G/K)_\Kf)$
and $\Phi_2\in\Hom_K(V^\star,C^\infty(G/K)_\Kf)$
\[
\sum_{i=1}^n
\bigl(\Phi_1[v_i],\Phi_2[v_i^\star]\bigr)_{G/K}
=\sum_{i=1}^{m}
\bigl(
\Gamma^{V}_0(\Phi_1)[v_i],
\Gamma^{V^\star}_0(\Phi_2)[v_i^\star]
\bigr)_A.
\]
In particular,
for $\bigl(C^\infty_\cpt(G/K)_\Kf,C^\infty_\cpt(A)\bigr)$
and $\bigl(C^\infty(G/K)_\Kf,C^\infty(A)\bigr)$,
the pair of $(\cdot,\cdot)_{G/K}$ and $(\cdot,\cdot)_{A}$
is compatible with restriction in the sense of\/ {\normalfont Definition \ref{defn:compRM}}.
\end{prop}

The transforms we study in this section are the following:
\begin{defn}[the Helgason-Fourier transform {\cite[Ch.\,III, \S1, No.\,1]{Hel5}}]\label{defn:HelF}
For $g\in G$
let $A(g)$ be as in \S\ref{sec:Poisson}.
Suppose $f(x)\in C^\infty_\cpt(G/K)$.
For $\lambda\in\Liea^*$ and $k\in K$ we put
\[
\mathcal F_G f\,(\lambda,k)=\int_{G/K} f(x)\, e^{(-i\lambda+\rho)(A(k^{-1}x))}\,dx.
\]
Since
$\mathcal F_G f\,(\lambda,km)=\mathcal F_G f\,(\lambda,k)$ for any $m\in M$,
$\mathcal F_G f\,(\lambda, k)$ is a function on $\Liea^*\times K/M$.
Since the integral converges for any $\lambda\in\Liea^*_\CC$,
$\mathcal F_G f\,(\lambda, k)$ extends to an analytic function on $\Liea^*_\CC\times K/M$
which is holomorphic in $\lambda\in\Liea^*_\CC$.
\end{defn}

\begin{defn}[the Opdam-Cherednik transform {\cite[Definition 7.9]{Op:Cherednik}}]
Suppose $f(a)\in C^\infty_\cpt(A)$.
For $\lambda\in\Liea^*$ and $w\in W$ we put
\[
\mathcal F_{\mathbf H} f\,(\lambda,w)=\int_{A} f(a)\, 
\mathbf G(-i\lambda,w^{-1}a)\prod_{\alpha\in\Sigma^+}\bigl|2\sinh \alpha(\log a)\bigr|^{\dim\Lieg_\alpha}\,da.
\]
One can see from Theorem \ref{thm:G} that for each fixed $w\in W$, $\mathcal F_{\mathbf H} f\,(\lambda, w)$ extends to an entire holomorphic function in $\lambda\in\Liea^*_\CC$.
\end{defn}
Let us pack the target spaces of these transforms into a radial pair.
For any $\lambda\in\Liea_\CC^*$ naturally
$B_G(i\lambda)\simeq C^\infty(K/M)$
and $B_{\mathbf H}(i\lambda)\simeq (\CC W)^*$ by restriction.
Using injections
\begin{align*}
C^\infty(\Liea^*\times K/M) \ni F(\lambda, k) &\longmapsto \bigl(k\mapsto F(\lambda, k)\bigr)_{\lambda\in\Liea^*}\in \prod_{\lambda\in\Liea^*} B_G(i\lambda),\\
C^\infty(\Liea^*\times W) \ni F(\lambda, w) &\longmapsto \bigl(w\mapsto F(\lambda, w)\bigr)_{\lambda\in\Liea^*}\in \prod_{\lambda\in\Liea^*} B_{\mathbf H}(i\lambda),
\end{align*}
we can consider $C^\infty(\Liea^*\times K/M)$ and $C^\infty(\Liea^*\times W)$
as $G$- and $\mathbf H$-modules respectively.
In fact one easily checks that these two spaces are
closed under the action of $G$ or $\mathbf H$
and that $G$ acts continuously on $C^\infty(\Liea^*\times K/M)$
equipped with the topology of compact convergence in all derivatives.
By a similar argument to \S\ref{sec:Poisson},
for each fixed $\lambda\in\Liea^*$ the map
\begin{align*}
C^\infty_\cpt(G/K)\ni f(x)\longmapsto
\biggl(G\ni g\longmapsto&
\int_{G/K} f(x)\, e^{(-i\lambda+\rho)(A(g^{-1}x))}\,dx\\
&\qquad =\bigl(f(x),\, e^{(i\lambda+\rho)(A(g^{-1}x))}\bigr)_{G/K}
\biggr)
\in B_G(i\lambda)
\end{align*}
and the map
\begin{align*}
C^\infty_\cpt(A)\ni f(a)\mapsto\biggl(
\mathbf H\ni h\mapsto&
\int_Af(a)\bigl(
\mathscr T(\theta_{\mathbf H}h)\,\mathbf G(-i\lambda,a)\bigr)
\!\!\prod_{\alpha\in\Sigma^+}\bigl|2\sinh \alpha(\log a)\bigr|^{\dim\Lieg_\alpha}da\\
&\qquad\qquad=\bigl(f(a),\, \mathscr T\bigl(\overline{\theta_{\mathbf H}h}\bigr)\,\mathbf G(i\lambda,a)\bigr)_{A}\biggr)\in B_{\mathbf H}(i\lambda)
\end{align*}
are respectively $G$- and $\mathbf H$-homomorphisms.
This shows
$\mathcal F_G$ and $\mathcal F_{\mathbf H}$ are homomorphisms
(the continuity of $\mathcal F_G$ is clear from the definition).
Define the restriction map
$\gamma_B: C^\infty(\Liea^*\times K/M) \to C^\infty(\Liea^*\times W)$ by
\[F(\lambda,k)\longmapsto \bigl(\Liea^*\times W\ni (\lambda,w)\longmapsto F(\lambda,\bar w)\in\CC \,\bigr).
\]
Note that $\gamma_B(F)(\lambda,\cdot)=\gamma_{B(i\lambda)}(F(\lambda,\cdot))$
for any $\lambda\in\Liea^*$.
Suppose $V\in\Km$.
If $\Phi\in\Hom_K(V,C^\infty(\Liea^*\times K/M))$ then
\[
\Gamma^V_B(\Phi):=\gamma_B\circ\Phi\bigr|_{V^M}
\]
belongs to $\Hom_W(V^M, C^\infty(\Liea^*\times W))$.
Conversely if $\varphi\in \Hom_W(V^M, C^\infty(\Liea^*\times W))$
then $\Phi\in \Hom_K(V, C^\infty(\Liea^*\times K/M))$ defined by
\begin{equation}\label{eq:Blift}
\Phi[v](\lambda,k)=\varphi\bigl[p^V(k^{-1}v)\bigr](\lambda,1),
\end{equation}
with $p^V$ as in Theorem \ref{thm:Ch} (ii),
is a unique lift satisfying $\Gamma^V_B(\Phi)=\varphi$.
Hence it readily follows from Theorem \ref{thm:R3B}
that $\bigl(C^\infty(\Liea^*\times K/M)_\Kf, C^\infty(\Liea^*\times W)\bigr)$
is a radial pair with radial restriction $\gamma_B$.
\begin{rem}
For $w\in W$ let $\delta_w\ (w\in W)$
be the element in $B_{\mathbf H}(\lambda)\simeq(\CC W)^*$ such that
\[\delta_w(t)=\begin{cases}
1 &(t=w),\\
0 &(t\ne w).
\end{cases}
\]
Then $B_{\mathbf H}(\lambda)=\mathbf H\,\delta_{w_0}$
and by \eqref{eq:wxiw} one easily sees
\[
\xi\, \delta_{w_0} = -(w_0\lambda)(\xi)\,\delta_{w_0}
\quad\text{for }\xi\in\Liea_\CC.
\]
Hence $B_{\mathbf H}(\lambda)\simeq I_{-w_0\lambda}$ by $\delta_{w_0}\leftrightarrow1\otimes1$
where we put $I_\lambda=\mathbf H\otimes_{S(\Liea_\CC)}\CC_\lambda$
for any $\lambda\in\Liea_\CC^*$ according to \cite{Op:Cherednik}.
On the other hand, if $I_\lambda^{\theta_{\mathbf H}}$ denotes
$I_\lambda$ endowed with the twisted $\mathbf H$-module structure by $\theta_{\mathbf H}$,
then $I_\lambda^{\theta_{\mathbf H}}\simeq I_{-w_0\lambda}$
by $w_0\otimes1 \leftrightarrow 1\otimes1$.
Hence $I_\lambda^{\theta_{\mathbf H}}\simeq B_{\mathbf H}(\lambda)$
by $w\otimes 1\leftrightarrow \delta_w$ ($\forall w\in W$).

Now let $\mathcal F'_{\mathbf H}$ be exactly the same as
Opdam's Cherednik transform  `$\mathcal F$'
defined in \cite[Definition 7.9]{Op:Cherednik}
and let $(\cdot,\cdot)$ be the inner product on $I_{i\lambda}$ $(\lambda\in\Liea^*)$
used in his definition.
Then for $\lambda\in\Liea^*$,
$\mathcal F'_{\mathbf H}(i\lambda)\in I_{i\lambda}$
is defined by
\begin{align*}
(\mathcal F'_{\mathbf H}f\,(i\lambda),w\otimes1)
&=(2\pi)^{\ell/2}|W|\int_{A}f(a)\, \mathbf G(-i\lambda,w^{-1}a)\,da\\
&=(2\pi)^{\ell/2}|W|\,\mathcal F_{\mathbf H}f\,(\lambda,w)
\qquad\text{for }\forall w\in W
\end{align*}
(note Opdam employs $(2\pi)^{\ell/2}|W|da$ as a Haar measure on $A$)
and under the identification $I_{i\lambda}
=I_{i\lambda}^{\theta_{\mathbf H}}\simeq B_{\mathbf H}(i\lambda)\simeq (\CC W)^*$
it holds that
\[
\mathcal F'_{\mathbf H}(i\lambda)
=
(2\pi)^{\ell/2}|W|\,\mathcal F_{\mathbf H}f\,(\lambda,\cdot).
\]
\end{rem}
We prepare some function spaces.
\begin{defn}
Suppose $\eta>0$.
Let $\PW_\eta(\Liea^*)$ be the space of holomorphic functions $\psi(\lambda)$
on $\Liea_\CC^*$ such that
\begin{equation}\label{eq:PWcond}
\sup_{\lambda\in\Liea_\CC^*}
e^{-\eta|\Imaginary \lambda|}(1+|\lambda|)^N|\psi(\lambda)|<\infty
\quad\text{for each }N\in\ZZ_{\ge0}.
\end{equation}
Let $\PW_\eta(\Liea^*\times W)$ be the space of 
holomorphic functions $F(\lambda,w)$ on $\Liea^*_\CC\times W$
such that $F(\cdot, w)\in \PW_\eta(\Liea^*)$ for each $w\in W$.
We naturally consider $\PW_\eta(\Liea^*\times W)\subset C^\infty(\Liea^*\times W)$.
One easily observes  by use of \eqref{eq:wxiw}
that this is an $\mathbf H$-submodule.
Let $\PW_\eta(\Liea^*\times K/M)$ be the space of those continuous functions on
$\Liea_\CC^*\times K/M$ which are holomorphic in $\lambda$ and satisfying
\begin{equation}\label{eq:PWcondG}
\sup_{(\lambda,k)\in\Liea_\CC^*\times K/M}
e^{-\eta|\Imaginary \lambda|}(1+|\lambda|)^N|F(\lambda,k)|<\infty
\quad\text{for each }N\in\ZZ_{\ge0}.
\end{equation}
This is a  Fr\'echet space
by the system of seminorms
\begin{equation*}
||F||_N
=\sup_{(\lambda,k)\in\Liea_\CC^*\times K/M}
e^{-\eta|\Imaginary \lambda|}(1+|\lambda|)^N|F(\lambda,k)|
\qquad(N\in\ZZ_{\ge0}).
\end{equation*}
Moreover we put
\[
\PW(\Liea^*\times K/M)=\bigcup_{\eta>0}\PW_\eta(\Liea^*\times K/M),\quad
\PW(\Liea^*\times W)=\bigcup_{\eta>0}\PW_\eta(\Liea^*\times W).\]
\end{defn}
\begin{defn}\label{defn:WinvGH}
Let $\tPW_\eta(\Liea^*\times K/M)$ be the closed subspace of 
$\PW_\eta(\Liea^*\times K/M)$ consisting of those functions $F$ satisfying
\begin{equation*}%\label{eq:WinvG}
\begin{aligned}
\int_K F(t\lambda,k)\,e^{(it\lambda+\rho)(A(k^{-1}x))}\,dk
&=\int_K F(\lambda,k)\,e^{(i\lambda+\rho)(A(k^{-1}x))}\,dk\\
&\qquad\qquad\text{for any }t\in W,\lambda\in\Liea^*\text{ and }x\in G.
\end{aligned}
\end{equation*}
Let $\tPW_\eta(\Liea^*\times W)$ be the subspace of 
$\PW_\eta(\Liea^*\times W)$ consisting of those functions $F$ satisfying
\begin{equation*}\label{eq:WinvH}
\begin{aligned}
\sum_{w\in W} F(t\lambda,w)\,\mathbf G(it\lambda,w^{-1}a)
&=\sum_{w\in W} F(\lambda,w)\,\mathbf G(i\lambda,w^{-1}a)\\
&\qquad\qquad\text{for any }t\in W,\lambda\in\Liea^*\text{ and }a\in A.
\end{aligned}
\end{equation*}
Moreover we put
\[
\tPW(\Liea^*\times K/M)=\bigcup_{\eta>0}\tPW_\eta(\Liea^*\times K/M),\quad
\tPW(\Liea^*\times W)=\bigcup_{\eta>0}\tPW_\eta(\Liea^*\times W).\]
We equip $\tPW(\Liea^*\times K/M)$
with the topology of the inductive limit.
\end{defn}
\begin{defn}
Let $d$ denote the distance function on $G/K$ or $A$
given by the Riemannian metric corresponding to the Killing form. 
Put for each $\eta>0$
\begin{align*}
C^\infty_\eta(G/K)&=\{
f(x)\in C^\infty(G/K);\,f(x)=0\text{ whenever }d(x,1K)\ge \eta
\},\\
C^\infty_\eta(A)&=\{
f(a)\in C^\infty(A);\,f(a)=0\text{ whenever }d(a,1)\ge \eta
\}.
\end{align*}
They are Fr\'echet spaces with the topology of uniform convergence
in all derivatives.
(The topology of $C^\infty_\eta(A)$ in not used in the paper.)
\end{defn}
Now we describe Helgason's results on $\mathcal F_G$
with some subsidiary information.
Let $d\lambda$ is $|W|^{-1}(2\pi)^{-\ell/2}$ times of
the Euclidean measure on $\Liea^*$ induced by the Killing form.
\begin{prop}\label{prop:FourierG}
{\normalfont (i)}
If $F\in\PW_\eta(\Liea^*\times K/M)$ then
\begin{equation}\label{eq:adjG}
\mathcal J_GF\,(x)=\int_{\Liea^*}\left(\int_K
e^{(i\lambda+\rho)(A(k^{-1}x))}
F(\lambda, k)\,dk\right)\,|\mathbf c(\lambda)|^{-2}\,d\lambda
\end{equation}
converges for all $x\in G$ and defines a $C^\infty$ function on $G/K$. 
Here $\mathbf c(\lambda)$ is Harish-Chandra's $\mathbf c$-function
defined by
\[
\mathbf c(\lambda)=\prod_{\alpha\in R_1^+}
\frac{\mathbf c_\alpha(i\lambda)}{\mathbf c_\alpha(\rho)}.\]
The linear map $\mathcal J_G:\PW_\eta(\Liea^*\times K/M)\to C^\infty(G/K)$
is continuous (\/$C^\infty(G/K)$ has the topology of compact convergence
in all derivatives). 

\noindent {\normalfont (ii)}
The transform $\mathcal F_G$ is a bijection of $C^\infty_\cpt(G/K)$
onto $\tPW(\Liea^*\times K/M)$.
More precisely, for each $\eta>0$, $C^\infty_\eta(G/K)$ is isomorphic to $\tPW_\eta(\Liea^*\times K/M)$
by $\mathcal F_G$ as a topological vector space.
The inverse map is given by $\mathcal J_G$.

\noindent {\normalfont (iii)}
For any $f_1, f_2\in C^\infty_\cpt(G/K)$
\[
(f_1,f_2)_{G/K}=\int_{\Liea^*}\int_K \mathcal F_Gf_1\,(\lambda,k)\,\overline{\mathcal F_Gf_2\,(\lambda,k)}\,
dk\,|\mathbf c(\lambda)|^{-2}\,d\lambda.
\]
\end{prop}
\begin{cor}\label{cor:PWG}
Any function in $\tPW(\Liea^*\times K/M)=\mathcal F_G\bigl(C^\infty_\cpt(G/K)\bigr)$
is necessarily analytic in all variables.
Hence an embedding\/ $\tPW(\Liea^*\times K/M)\hookrightarrow C^\infty(\Liea^*\times  K/M)$ 
is naturally defined and this is continuos
since $\mathcal F_G:C^\infty_\cpt(G/K)\to C^\infty(\Liea^*\times  K/M)$ is continuous.
Consider\/ $\tPW(\Liea^*\times K/M)$
as a $G$-submodule of\/ $C^\infty(\Liea^*\times  K/M)$.
Since any function in $C^\infty_\cpt(G/K)$ is a \/ $C^\infty$ vector,
the same thing holds for\/ $\tPW(\Liea^*\times K/M)$.
Thus $U(\Lieg_\CC)$ acts
on $\tPW(\Liea^*\times K/M)$ in a compatible way with the embedding.
For each $\eta>0$,
$C^\infty_\eta(\Liea^*\times  K/M)$ and\/ $\tPW_\eta(\Liea^*\times K/M)$
are stable under the actions
of $K$ and $U(\Lieg_\CC)$.
Hence we can consider $(\Lieg_\CC,K)$-modules
$C^\infty_\eta(\Liea^*\times  K/M)_\Kf$ and\/ $\tPW_\eta(\Liea^*\times K/M)_\Kf$.
They are isomorphic via $\mathcal F_G$ or $\mathcal J_G$.
\end{cor}
\begin{proof}[Proof of {\normalfont Proposition \ref{prop:FourierG}}]
In the proof of (i) we regard $C^\infty(G/K)$
as a closed subspace of $C^\infty(G)$.
Suppose $D\in U(\Lieg_\CC)$.
Let $E$ be a finite-dimensional $\Ad(K)$-stable subspace of $U(\Lieg_\CC)$ containing $D$
and take a basis $\{D_j\}$ of $E$.
Then there exist analytic functions $\pi_{j}(k)$ on $K$ such that
\[
\Ad(k)D=\sum_j\pi_{j}(k)D_j
\quad\text{for any }k\in K.
\]
For $g\in G$ let $\kappa(g)$ denote the unique element in $K$ such that
$g=n\, \exp A(g)\, \kappa(g)$ for some $n\in N$.
For $k\in K$ and $x\in G$ it holds that
\begin{align*}
r_x(D)\,e^{(i\lambda+\rho)(A(k^{-1}x))}
&=e^{(i\lambda+\rho)(A(k^{-1}x))}
\bigl(r_y(\Ad(\kappa(k^{-1}x))D)\,e^{(i\lambda+\rho)(A(y))}\bigr)\bigr|_{y=1}\\
&=e^{(i\lambda+\rho)(A(k^{-1}x))}
\sum_j\pi_{j}(\kappa(k^{-1}x))
\bigl(r_y(D_j)e^{(i\lambda+\rho)(A(y))}\bigr)\bigr|_{y=1}\\
&=e^{(i\lambda+\rho)(A(k^{-1}x))}
\sum_j\pi_{j}(\kappa(k^{-1}x))
\gamma(D_j)(i\lambda).
\end{align*}
Hence by the inequality $|A(k^{-1}x)|\le d(xK,1K)$
(cf.~\cite[Ch.\,IV, \S10, (14)]{Hel4})
we get for any $F\in\PW_\eta(\Liea^*\times K/M)$ 
and any $N\in\ZZ_{\ge0}$
\begin{multline*}
\sup_{\lambda\in\Liea_\CC}
(1+|\lambda|)^N
\left|
r_x(D) \int_K \,e^{(i\lambda+\rho)(A(k^{-1}x))} 
F(\lambda, k)\,dk
\right|\\
\le
\biggl(\sup_{\substack{H\in \Liea;\\ |H|\le d(xK,1K)}}\!\!\!e^{\rho(H)}\biggr)
\sum_j ||\pi_{j}(\cdot)||_{L^\infty(K)}
\sup_{\substack{\lambda\in\Liea_\CC \\ k\in K}}
\Bigl((1+|\lambda|)^N |\gamma(D_j)(i\lambda)\,
F(\lambda,k)|\Bigr).
\end{multline*}
Since $|\mathbf c(\lambda)|^{-2}$ has at most polynomial growth
(cf.~\cite[Ch.\,IV, Proposition 7.2]{Hel4}),
we can differentiate \eqref{eq:adjG} as a function in $x\in G$
repeatedly under the outer integral.
The first assertion is thus proved.
The continuity of $\mathcal J_G$ also easily
follows from the above estimate.

To prove (ii) let $F\in\tPW_\eta(\Liea^*\times K/M)$ be given.
We assert there exists $f\in C^\infty_\eta(G/K)$ such that $F=\mathcal F_Gf$.
In fact, if $F$ has $C^\infty$ regularity then the proof for this assertion
is given in \cite[pp.268--271]{Hel5}. But the proof there works for
any continuous $F$ with the help of (i).
Thus the bijectivity of
\begin{equation}\label{eq:JGRbij}
\mathcal J_G:\tPW_\eta(\Liea^*\times K/M)\to C^\infty_\eta(G/K)
\end{equation}
follows from \cite[Ch.\,III, Theorem 5.1]{Hel5}.
Since the subspace $C^\infty_\eta(G/K)\subset C^\infty(G/K)$
is a Fr\'echet space,
from (i) and the open mapping theorem one sees \eqref{eq:JGRbij} is an isomorphism of topological vector spaces.

For (iii) we refer the reader to \cite[Ch.\,III, \S1, No.\,2]{Hel5}.
\end{proof}

Let us return to the argument on the target spaces.
\begin{lem}\label{lem:targetSp}
For $F\in \PW_\eta(\Liea^*\times K/M)$
\begin{align*}
F\in \tPW_\eta(\Liea^*\times K/M)
&\Leftrightarrow
\left\{
\begin{aligned}
&F\in C^\infty(\Liea^*\times K/M),\\
&\mathcal P_G^{it\lambda} \bigl(F(t\lambda,\cdot)\bigr)
=\mathcal P_G^{i\lambda} \bigl(F(\lambda,\cdot)\bigr)\text{ for any }t\in W
\text{ and }\lambda\in\Liea^*
\end{aligned}
\right.\\
&\Leftrightarrow
\left\{
\begin{aligned}
&F\in C^\infty(\Liea^*\times K/M),\\
&\Tilde{\mathcal A}_G(t,i\lambda) \bigl(F(\lambda,\cdot)\bigr)
=F(t\lambda,\cdot)\text{ for any }t\in W
\text{ and }\lambda\in\Liea^*.\end{aligned}
\right.
\end{align*}
For $F\in \PW_\eta(\Liea^*\times W)$
\begin{align*}
F\in \tPW_\eta(\Liea^*\times W)
&\Leftrightarrow
\mathcal P_{\mathbf H}^{it\lambda} \bigl(F(t\lambda,\cdot)\bigr)
=\mathcal P_{\mathbf H}^{i\lambda} \bigl(F(\lambda,\cdot)\bigr)\text{ for any }t\in W
\text{ and }\lambda\in\Liea^*\\
&\Leftrightarrow
\Tilde{\mathcal A}_{\mathbf H}(t,i\lambda) \bigl(F(\lambda,\cdot)\bigr)
=F(t\lambda,\cdot)\text{ for any }t\in W
\text{ and }\lambda\in\Liea^*.
\end{align*}
In particular\/ $\tPW_\eta(\Liea^*\times W)$
is an $\mathbf H$-submodule of\/ $\PW_\eta(\Liea^*\times W)$.
\end{lem}
\begin{proof}
Suppose $F\in \tPW_\eta(\Liea^*\times W)$.
Then $F\in C^\infty(\Liea^*\times K/M)$ by Corollary \ref{cor:PWG}
and for any $\lambda\in\Liea_\CC^*$
we can consider $F(\lambda,\cdot)\in B_G(i\lambda)$
to which $\mathcal P_G^{i\lambda}$ and $\Tilde{\mathcal A}_G(t,i\lambda)$
can be applied.
Hence the lemma is immediate from
Definition \ref{defn:WinvGH},
the definition of the Poisson transforms, and the identities
\[
\mathcal P_G^{i\lambda}
=\mathcal P_G^{it\lambda}\circ
\Tilde{\mathcal A}_G(t,i\lambda),
\quad
\mathcal P_{\mathbf H}^{i\lambda}
=\mathcal P_{\mathbf H}^{it\lambda}\circ
\Tilde{\mathcal A}_{\mathbf H}(t,i\lambda).\qedhere
\]
\end{proof}
\begin{prop}\label{prop:Fpairs}
For any $\eta>0$ the pair
$\bigl(C^\infty_\eta(G/K)_\Kf,\, C^\infty_\eta(A)\bigr)$
is a radial pair with radial restriction $\gamma_0$
and is a subobject of\/
$\bigl(C^\infty_\cpt(G/K)_\Kf,\, C^\infty_\cpt(A)\bigr)$.
Likewise $\bigl(\tPW_\eta(\Liea^*\times K/M)_\Kf,\, \tPW_\eta(\Liea^*\times W)\bigr)$
is a subobject of\/ $\bigl(C^\infty(\Liea^*\times K/M)_\Kf, C^\infty(\Liea^*\times W)\bigr)\in\Crad$.
\end{prop}
\begin{proof}
First, it is clear that $\gamma_0\bigl(C_\eta^\infty(G/K)\bigr)\subset C_\eta^\infty(A)$.
Suppose $V\in\Km$.
For any $\varphi\in \Hom_W(V^M_\single,C^\infty_\eta(A))$
let $\Phi$ be its lift in $\Hom_K^\ttt(V,C^\infty(G/K))$.
Extending $\varphi$ to an element of
$\Hom_W(V^M,C^\infty_\eta(A))$ by $\varphi\bigr|_{V^M_\double}=0$,
we have
\[
\Phi[v](kak^{-1})= \varphi\bigl[p^V(k^{-1}v)\bigr](a) 
\quad\text{for any }v\in V, k\in K\text{ and }a\in A.
\]
From this we easily see $\Phi\in\Hom_K(V,C^\infty_\eta(G/K))$.
These facts prove
the first assertion.

Let us prove the second assertion.
If $F\in \tPW_\eta(\Liea^*\times K/M)$ then
$\gamma_B(F)\in \PW_\eta(\Liea^*\times W)$ since
Condition \eqref{eq:PWcondG} for $F$ implies
Condition \eqref{eq:PWcond} for each $\gamma_B(F)(\cdot,w)$ ($w\in W$).
Suppose $V\in\Km$.
For any $\Phi\in \Hom_K(V,\, \tPW_\eta(\Liea^*\times K/M))$
put
\[
\tilde\Gamma^V_B(\Phi)=\gamma_B\circ \Phi\bigr|_{V^M_\single}
\in \Hom_W(V^M_\single,\, \PW_\eta(\Liea^*\times W)).
\]
Then for any $\lambda\in\Liea^*$ we have
\begin{align}
\Phi_\lambda
:=\bigl(V\ni v\longmapsto \Phi[v](\lambda,\cdot)\bigr)
&\in \Hom_K(V,B_G(i\lambda)),\label{eq:Phisplam}\\
\tilde\Gamma^V_B(\Phi)_\lambda
:=\bigl(V^M_\single\ni v\longmapsto
\tilde\Gamma^V_B(\Phi)[v](\lambda,\cdot)\bigr)
&\in \Hom_W(V^M_\single,B_{\mathbf H}(i\lambda)),\notag\\
\tilde\Gamma^V_{B(i\lambda)}(\Phi_\lambda)
:=\gamma_{B(i\lambda)}\circ \Phi_\lambda \bigr|_{V^M_\single}
&=\tilde\Gamma^V_B(\Phi)_\lambda.\label{eq:spGammaT}
\end{align}
Now for any $t\in W$
\begin{align*}
\Tilde{\mathcal A}_{\mathbf H}(t,i\lambda)\circ \tilde\Gamma^V_B(\Phi)_\lambda
&=\Tilde{\mathcal A}_{\mathbf H}(t,i\lambda)\circ \tilde\Gamma^V_{B(i\lambda)}(\Phi_\lambda)&&(\because \eqref{eq:spGammaT})\\
&=\tilde\Gamma^V_{B(it\lambda)}
\bigl(
\Tilde{\mathcal A}_{G}(t,i\lambda)\circ
\Phi_\lambda
\bigr)&&(\because \text{Theorem \ref{thm:it}})\\
&=\tilde\Gamma^V_{B(it\lambda)}
\bigl(
\Phi_{t\lambda}
\bigr)&&(\because \text{Lemma \ref{lem:targetSp}})\\
&=\tilde\Gamma^V_B(\Phi)_{t\lambda}.
&&(\because \eqref{eq:spGammaT})
\end{align*}
Hence by Lemma \ref{lem:targetSp} we conclude
$\tilde\Gamma^V_B(\Phi)
\in \Hom_W(V^M_\single,\, \tPW_\eta(\Liea^*\times W))$, namely
\begin{equation}\label{eq:GVB}
\tilde\Gamma^V_B\bigl(
\Hom_K(V,\, \tPW_\eta(\Liea^*\times K/M))
\bigr)\subset \Hom_W(V^M_\single,\, \tPW_\eta(\Liea^*\times W)).
\end{equation}
Conversely, suppose $\varphi\in \Hom_W(V^M_\single,\, \tPW_\eta(\Liea^*\times W))$
and let $\Phi$ be its unique lift in $\Hom_K^\ttt(V,\,C^\infty(\Liea^*\times K/M))$.
We extend $\varphi$ to an element of
$\Hom_W(V^M,\, \tPW_\eta(\Liea^*\times W))$ by $\varphi\bigr|_{V^M_\double}=0$.
Then \eqref{eq:Blift} holds for any $v\in V$, $k\in K$ and $\lambda\in\Liea^*$.
This means for each $v\in V$,
$\Phi[v](\lambda,k)$ extends to an analytic function
on $\Liea_\CC^*\times K/M$ which is holomorphic in $\lambda$.
Let $|\cdot|$ be the norm of $V$
induced from a $K$-invariant inner product.
Then for each $N\in \ZZ_{\ge0}$ there exists $C_N>0$ such that
\[
\sup_{\lambda\in\Liea_\CC^*}
e^{-\eta|\Imaginary \lambda|}(1+|\lambda|)^N|\varphi[v](\lambda,1)|<C_N
\quad\text{ for any }v\in V^M\text{ with }|v|\le1.
\]
Since $|p^V(k^{-1}v)|\le|v|$ for any $k\in K$ and $v\in V$,
$F:=\Phi[v]$ satisfies \eqref{eq:PWcondG}
for $v\in V$ with $|v|\le1$.
Thus $\Phi[v]\in \PW(\Liea^*\times K/M)$ for all $v\in V$.
Now for any $\lambda\in\Liea^*$ let $\Phi_\lambda$ be as in \eqref{eq:Phisplam}
and put
\[
\varphi_\lambda
:=\bigl(V^M_\single\ni v\longmapsto
\varphi[v](\lambda,\cdot)\bigr)
\in \Hom_W(V^M_\single,B_{\mathbf H}(i\lambda)).
\]
Then $\Phi_\lambda\in \Hom_K^\ttt(V,B_G(i\lambda))$
and for any $t\in W$ we have
\begin{align*}
\Phi_{t\lambda}&\in \Hom_K^\ttt(V,B_G(it\lambda)),\\
\Tilde{\mathcal A}_{G}(t,i\lambda)\circ
\Phi_\lambda&\in
\Hom_K^\ttt(V,B_G(it\lambda)),
&&(\because \text{Theorem \ref{thm:it}})\\
\tilde\Gamma^V_{B(it\lambda)}
\bigl(
\Tilde{\mathcal A}_{G}(t,i\lambda)\circ
\Phi_\lambda
\bigr)
&=\Tilde{\mathcal A}_{\mathbf H}(t,i\lambda)\circ \tilde\Gamma^V_{B(i\lambda)}(\Phi_\lambda)
&&(\because \text{Theorem \ref{thm:it}})\\
&=\Tilde{\mathcal A}_{\mathbf H}(t,i\lambda)\circ \varphi_\lambda\\
&=\varphi_{t\lambda}&&(\because \text{Lemma \ref{lem:targetSp}})\\
&=\tilde\Gamma^V_{B(it\lambda)}
\bigl(
\Phi_{t\lambda}
\bigr)
\end{align*}
and hence $\Tilde{\mathcal A}_{G}(t,i\lambda)\circ
\Phi_\lambda=\Phi_{t\lambda}$.
Hence 
$\Phi \in \Hom_K(V,\, \tPW_\eta(\Liea^*\times K/M))$
by Lemma \ref{lem:targetSp}.
Thus $\tilde\Gamma^V_B$ induces a linear bijection
\begin{multline*}
\tilde\Gamma^V_B:
\Hom_K(V,\, \tPW_\eta(\Liea^*\times K/M))
\cap \Hom_K^\ttt(V,\, C^\infty(\Liea^*\times K/M))\\
\simarrow
\Hom_W(V^M_\single,\, \tPW_\eta(\Liea^*\times W)).
\end{multline*}
This and \eqref{eq:GVB} prove that
$\bigl(\tPW_\eta(\Liea^*\times K/M)_\Kf,\, \tPW_\eta(\Liea^*\times W)\bigr)$
is a subobject of $\bigl(C^\infty(\Liea^*\times K/M)_\Kf, C^\infty(\Liea^*\times W)\bigr)$ in $\CCh$ (and hence in $\Crad$).
\end{proof}
To import Opdam's result on the characterization of the image of $\mathcal F_{\mathbf H}$ we shall prepare another description of $\tPW_\eta(\Liea^*\times W)$.
Put $\mathscr M=S(\Liea_\CC)\Liea_\CC$ and let 
$\widehat S(\Liea_\CC)$ be the $\mathscr M$-adic completion of $S(\Liea_\CC)$,
namely the algebra of formal power series at $0\in\Liea^*_\CC$.
Then there exists uniquely an algebra $\widehat{\mathbf H}$ over $\CC$
with the following properties:\smallskip

\noindent{\normalfont (i)}
$\widehat{\mathbf H}\simeq \widehat S(\Liea_\CC)\otimes\CC W$ as a $\CC$-linear space;

\noindent{\normalfont (ii)}
The maps $\widehat S(\Liea_\CC)\rightarrow \widehat{\mathbf H}, \psi\mapsto \psi\otimes1$ and
 $\CC W\rightarrow \widehat{\mathbf H}, w\mapsto1\otimes w$ are algebra homomorphisms;

\noindent{\normalfont (iii)}
$(\psi\otimes1)\cdot(1\otimes w)=\psi\otimes w$ for any $\psi\in \widehat S(\Liea_\CC)$ and $w\in W$;

\noindent{\normalfont (iv)}
$(1\otimes s_\alpha)\cdot(\psi(\lambda)\otimes1)
 =\psi(s_\alpha\lambda)\otimes s_\alpha-\mathbf k(\alpha)\,\dfrac{\psi(\lambda)-\psi(s_\alpha\lambda)}{\alpha^\vee}$
for any $\alpha\in\Pi$ and $\psi\in \widehat S(\Liea_\CC)$.
\smallskip

\noindent
We identify $\mathbf H$, $\widehat S(\Liea_\CC)$ and $\CC W$
with subalgebras of $\widehat{\mathbf H}$ in obvious ways.
By (iv) we have the following:
\begin{lem}\label{lem:psitw}
For any $\psi\in\widehat S(\Liea_\CC)$ and $w\in W$
define $\psi^w_t\in\widehat S(\Liea_\CC)$ (\/$t\in W$) by the identity
\[
w^{-1}\psi=\sum_{t\in W} \psi_t^w\otimes t^{-1}
\]
in $\widehat{\mathbf H}$.
Then for each $w,t\in W$ the correspondence $\psi\mapsto\psi_t^w$ is continuous
with respect to the $\mathscr M$-adic topology.
\end{lem}
For each $\eta>0$ we identify the function space
\[
\PW_\eta(-i\Liea^*):=\bigl\{\psi(i\lambda);\,\psi\in \PW_\eta(\Liea^*)\bigr\}
\]
with a subspace of $\widehat S(\Liea_\CC)$.
Then it holds that
\begin{equation}\label{eq:PWHc}
\PW_\eta(-i\Liea^*)\cdot W= W\cdot \PW_\eta(-i\Liea^*)=\mathbf H\cdot \PW_\eta(-i\Liea^*)
\end{equation}
in $\widehat{\mathbf H}$ (cf.~\cite[\S8]{Op:Cherednik}).
Let
\[
\mathbf 1_{\mathbf H}(\lambda,w) \in C^\infty(\Liea^*\times W) \subset \prod_{\lambda\in\Liea^*} B_{\mathbf H}(i\lambda)
\]
be the constant function with value $1$.
Let $\pi_{\mathbf H}$ denote the action of $\mathbf H$ on $\prod_{\lambda\in\Liea^*} B_{\mathbf H}(i\lambda)$.
For any $\psi\in \widehat S(\Liea_\CC)$ take a sequence $\{\psi_n\}\subset S(\Liea_\CC)$
converging to $\psi$ with respect to the $\mathscr M$-adic topology.
Then for each $w\in W$
$\{(\pi_{\mathbf H}(\psi_n)\mathbf 1_{\mathbf H})(\cdot,w)\}\subset S(\Liea_\CC)$
converges to an element of $\widehat S(\Liea_\CC)$,
the limit being independent of the choice of sequences.
More precisely, if we take $\psi_t^w \in \widehat S(\Liea_\CC)$ ($t\in W$) for $\psi$ as in Lemma \ref{lem:psitw},
then
\begin{equation}\label{eq:LHimage}
\Bigl(\lim_{n\to\infty} (\pi_{\mathbf H}(\psi_n)\mathbf 1_{\mathbf H})(\cdot,w)\Bigr)(\lambda)
=\sum_{t\in W} \psi_t^w(-i\lambda).
\end{equation}
Hence because of \eqref{eq:PWHc},
if $\psi\in \PW_\eta(-i\Liea^*)$ then we can define
a function $\pi_{\mathbf H}(\psi)\mathbf 1_{\mathbf H}\in \PW_\eta(\Liea^*\times W)$ 
by
\[
(\pi_{\mathbf H}(\psi)\mathbf 1_{\mathbf H})(\lambda,w)
=
\Bigl(\lim_{n\to\infty} (\pi_{\mathbf H}(\psi_n)\mathbf 1_{\mathbf H})(\cdot,w)\Bigr)(\lambda).
\]
Using \eqref{eq:PWHc} again,
we see $\PW_\eta(-i\Liea^*)\otimes \CC_\triv$
is an $\mathbf H$-submodule
of $\widehat S(\Liea_\CC)\otimes\CC_\triv\simeq\widehat{\mathbf H}\otimes_{\CC W}\CC_\triv$.
Now fix a $v_\triv\in \CC_\triv\setminus \{0\}$.
It is then easy to see that the linear map
\[
\mathcal L_{\mathbf H} : \PW_\eta(-i\Liea^*)\otimes \CC_\triv\longrightarrow 
\PW_\eta(\Liea^*\times W)\, ;
\quad \psi\otimes v_\triv \longmapsto \pi_{\mathbf H}(\psi)\mathbf 1_{\mathbf H}
\]
is an $\mathbf H$-homomorphism.
\begin{prop}\label{prop:LHbij}
The homomorphism $\mathcal L_{\mathbf H}$ is a bijection of $\PW_\eta(-i\Liea^*)\otimes \CC_\triv$
onto $\tPW_\eta(\Liea^*\times W)$. If $F(\lambda,w)\in \tPW_\eta(\Liea^*\times W)$ then
$\mathcal L_{\mathbf H}^{-1}F\,(\lambda)=F(i\lambda,1)\otimes v_\triv$.
\end{prop}
\begin{proof}
Suppose $\psi\in \PW_\eta(-i\Liea^*)$
and let us prove $\mathcal L_{\mathbf H}(\psi\otimes v_\triv) \in \tPW_\eta(\Liea^*\times W)$.
Let $\{\psi_n\}\subset S(\Liea_\CC)$ be the sequence converging to $\psi$
and put $F_n(\lambda,w)=(\pi_{\mathbf H}(\psi_n)\mathbf 1_{\mathbf H})(\lambda,w)
\in C^\infty(\Liea^* \times W)$.
Then it follows from \eqref{eq:explicitAH} that
for each $\alpha\in\Pi$, $w\in W$ and $\lambda\in\Liea^*$
\[
\Tilde{\mathcal A}_{\mathbf H}(s_\alpha,i\lambda)
\bigl(F_n(\lambda,\cdot)\bigr)(w)
=
\frac{\mathbf m_1(\alpha)F_n(\lambda,w)
-i\lambda(\alpha^\vee)F_n(\lambda,ws_\alpha)}{\mathbf m_1(\alpha)-i\lambda(\alpha^\vee)}.
\]
For fixed $\alpha$ and $w$, the right-hand side converges to
\begin{multline*}
\frac{\mathbf m_1(\alpha)\,(\pi_{\mathbf H}(\psi)\mathbf 1_{\mathbf H})(\lambda,w)
-i\lambda(\alpha^\vee)\,(\pi_{\mathbf H}(\psi)\mathbf 1_{\mathbf H})(\lambda,ws_\alpha)}{\mathbf m_1(\alpha)-i\lambda(\alpha^\vee)}\\
=\Tilde{\mathcal A}_{\mathbf H}(s_\alpha,i\lambda)
\bigl((\pi_{\mathbf H}(\psi)\mathbf 1_{\mathbf H})(\lambda,\cdot)\bigr)(w)
=\Tilde{\mathcal A}_{\mathbf H}(s_\alpha,i\lambda)
\bigl(\mathcal L_{\mathbf H}(\psi\otimes v_\triv)(\lambda,\cdot)\bigr)(w)
\end{multline*}
in $\widehat S(\Liea_\CC)$ while the left-hand side equals
\begin{align*}
\Tilde{\mathcal A}_{\mathbf H}(s_\alpha,i\lambda)
\bigl((\pi_{\mathbf H}(\psi_n)\mathbf 1_{\mathbf H})(\lambda,\cdot)\bigr)(w)
&=
\Tilde{\mathcal A}_{\mathbf H}(s_\alpha,i\lambda)
\bigl(\psi_n\, \mathbf 1_{\mathbf H}^{i\lambda}\bigr)(w)\\
&=
\bigl(\psi_n\, \Tilde{\mathcal A}_{\mathbf H}(s_\alpha,i\lambda)
\mathbf 1_{\mathbf H}^{i\lambda}\bigr)(w)\\
&=
\bigl(\psi_n\, \mathbf 1_{\mathbf H}^{is_\alpha\lambda}\bigr)(w)
&&(\because \eqref{eq:image1H})\\
&=(\pi_{\mathbf H}(\psi_n)\mathbf 1_{\mathbf H})(s_\alpha\lambda,w)
\end{align*}
and converges to $\mathcal L_{\mathbf H}(\psi\otimes v_\triv)(s_\alpha\lambda,w)$ in $\widehat S(\Liea_\CC)$.
This shows
\[
\Tilde{\mathcal A}_{\mathbf H}(s_\alpha,i\lambda)
\bigl(\mathcal L_{\mathbf H}(\psi\otimes v_\triv)(\lambda,\cdot)\bigr)
=\mathcal L_{\mathbf H}(\psi\otimes v_\triv)(s_\alpha\lambda,\cdot)
\quad\text{for }\alpha\in\Pi\text{ and }\lambda\in\Liea^*.
\]
Hence from Lemma \ref{lem:targetSp} we conclude
$\mathcal L_{\mathbf H}(\psi\otimes v_\triv) \in \tPW_\eta(\Liea^*\times W)$.

Secondly, 
for a given $\psi\in\PW_\eta(-i\Liea^*)$
let $\psi_t^1 \in \widehat S(\Liea_\CC)$ ($t\in W$) be as in Lemma \ref{lem:psitw}
for $w=1$.
Then
\[
\psi_t^1=
\begin{cases}
\psi & \text{if }t=1,\\
0 & \text{otherwise.}
\end{cases}
\]
Hence by \eqref{eq:LHimage} we have 
$\mathcal L_{\mathbf H}(\psi\otimes v_\triv)(\lambda,1)=\psi(-i\lambda)$.
This proves the injectivity of $\mathcal L_{\mathbf H}$.

In order to show the surjectivity, let $F\in\tPW_\eta(\Liea^*\times W)$ be given.
We put $\psi(\lambda)=F(i\lambda,1)$.
Then $\psi\in\PW_\eta(-i\Liea^*)$ and
$F':=F-\mathcal L_{\mathbf H}(\psi\otimes v_\triv)$ satisfies
$F'(\lambda,1)=0$.
We assert $F'=0$.
Indeed, since $F'\in \tPW_\eta(\Liea^*\times W)$,
it follows from Lemma \ref{lem:targetSp} and \eqref{eq:explicitAH}
that for any $\alpha\in\Pi$ and $w\in W$
\[
F'(s_\alpha\lambda,w)
=
\Tilde{\mathcal A}_{\mathbf H}(s_\alpha,i\lambda)
\bigl(F'(\lambda,\cdot)\bigr)(w)
=
\frac{\mathbf m_1(\alpha)F'(\lambda,w)-i\lambda(\alpha^\vee)F'(\lambda,ws_\alpha)}%
{\mathbf m_1(\alpha)-i\lambda(\alpha^\vee)}.
\]
This means $F'(\cdot,w)=0$ implies $F'(\cdot,ws_\alpha)=0$.
Hence we can prove $F'(\cdot,w)=0$ for all $w\in W$
by the induction in the length of $w$.
Thus $F=\mathcal L_{\mathbf H}(\psi\otimes v_\triv)$.
\end{proof}
Now we are in the position of stating an analogue of Proposition \ref{prop:FourierG}
for $\mathbf H$.
\begin{prop}\label{prop:FourierH}
{\normalfont (i)}
If $F\in\PW_\eta(\Liea^*\times W)$ then
\begin{equation}\label{eq:adjH}
\mathcal J_{\mathbf H}F\,(a)=
\int_{\Liea^*}
\frac1{|W|}\sum_{w\in W} \mathbf G(i\lambda,w^{-1}a)\,
F(\lambda,w)\,|\mathbf c(\lambda)|^{-2}\,d\lambda.
\end{equation}
absolutely converges for all $a\in A$ and defines a $C^\infty$ function on $A$. 
The linear map $\mathcal J_{\mathbf H}:\PW_\eta(\Liea^*\times W)\to C^\infty(A)$
is an $\mathbf H$-homomorphism. 

\noindent {\normalfont (ii)}
The transform $\mathcal F_{\mathbf H}$ is a bijection of $C^\infty_\cpt(A)$
onto $\tPW(\Liea^*\times W)$.
More precisely, for each $\eta>0$, $C^\infty_\eta(A)$ is isomorphic to $\tPW_\eta(\Liea^*\times W)$
by $\mathcal F_{\mathbf H}$ as an $\mathbf H$-module.
The inverse map is given by $\mathcal J_{\mathbf H}$.

\noindent {\normalfont (iii)}
For any $f_1, f_2\in C^\infty_\cpt(A)$
\[
(f_1,f_2)_A=\int_{\Liea^*}
\frac1{|W|}\sum_{w\in W} \mathcal F_{\mathbf H}f_1\,(\lambda,w)\,\overline{\mathcal F_{\mathbf H}f_2\,(\lambda,w)}\,
|\mathbf c(\lambda)|^{-2}\,d\lambda.
\]
\end{prop}
\begin{proof}
Thanks to the estimate of $\mathbf G$ in \cite[Corollary 6.2]{Op:Cherednik},
the integral \eqref{eq:adjH} converges 
and can be differentiated repeatedly in $a$ under the integral sign.
Hence for any $h\in\mathbf H$ and any regular $a\in A$
\begin{align*}
\mathscr T(\theta_{\mathbf H}h)\mathcal J_{\mathbf H}F\,(a)
&=\int_{\Liea^*}
\mathscr T(\theta_{\mathbf H}h)
\biggl(\frac1{|W|}\sum_{w\in W} \mathbf G(i\lambda,w^{-1}\cdot)\,
F(\lambda,w)\biggr)(a)\,|\mathbf c(\lambda)|^{-2}\,d\lambda\\
&=\int_{\Liea^*}
\mathscr T(\theta_{\mathbf H}h)
\bigl(
\mathcal P_{\mathbf H}^{i\lambda}\bigl(F(\lambda,\cdot)\bigr)
\bigr)(a)\,|\mathbf c(\lambda)|^{-2}\,d\lambda\\
&=\int_{\Liea^*}
\mathcal P_{\mathbf H}^{i\lambda}\bigl(h F(\lambda,\cdot)\bigr)
(a)\,|\mathbf c(\lambda)|^{-2}\,d\lambda\\
&=\int_{\Liea^*}
\mathcal P_{\mathbf H}^{i\lambda}\bigl((\pi_{\mathbf H}(h) F)(\lambda,\cdot)\bigr)
(a)\,|\mathbf c(\lambda)|^{-2}\,d\lambda\\
&=\mathcal J_{\mathbf H}(\pi_{\mathbf H}(h) F)\,(a).
\end{align*}
Thus we get (i).

Opdam shows in \cite[{\S\S8, 9}]{Op:Cherednik} that
$\mathcal F_{\mathbf H}$ gives a linear bijection
of $C^\infty_\eta(A)$ onto $\mathcal L_{\mathbf H}\bigl(\PW_\eta(-i\Liea^*)\otimes \CC_\triv\bigr)$ and the inverse map is given by
\[
\mathcal J'_{\mathbf H}F\,(a):=|W|^2
\Bigl(\prod_{\alpha\in R_1^+}\frac{\rho(\alpha^\vee)}{\rho(\alpha^\vee)+\mathbf m_1(\alpha)}\Bigr)^2
\int_{\Liea^*_+}\sum_{w\in W}\mathbf G(i\lambda,w^{-1}a)\,
F(\lambda,w)\,|\mathbf c(\lambda)|^{-2}\,d\lambda
\]
where $\Liea^*_+=\{\lambda\in\Liea^*;\,\lambda(\alpha^\vee)>0\text{ for all }\alpha\in\Pi\}$.
(Although he uses another kind of support conditions, his proof works in our setting.)
By Proposition \ref{prop:LHbij},
(ii) follows if we can prove 
$\mathcal J'_{\mathbf H}F=\mathcal J_{\mathbf H}F$ for $F\in \tPW(\Liea^*\times W)$.
First, from Definition \ref{defn:WinvGH}
and the $W$-invariance of $|\mathbf c(\lambda)|^{-2}$ we have
\[
\int_{\Liea^*_+}\sum_{w\in W}\mathbf G(i\lambda,w^{-1}a)
F(\lambda,w)|\mathbf c(\lambda)|^{-2}\,d\lambda
=\frac1{|W|}\int_{\Liea^*}\sum_{w\in W}\mathbf G(i\lambda,w^{-1}a)
F(\lambda,w)|\mathbf c(\lambda)|^{-2}\,d\lambda.
\]
Secondly, 
by \cite[Proposition 1.4 (3)]{Op:Cherednik}
it holds that for any
$\lambda\in\Liea^*_\CC$ and $a\in A$
\[
\Bigl(
\prod_{\alpha\in R_1^+}
\lambda(\alpha^\vee)
\Bigr)
\sum_{w\in W}\mathbf G(\lambda,wa)=
\sum_{w\in W}(\sgn w)
\Bigl(
\prod_{\alpha\in R_1^+}
((w\lambda)(\alpha^\vee)-\mathbf m_1(\alpha))\Bigr)
\mathbf G(w\lambda,a).
\]
Specializing this to $(\lambda,a)=(-\rho,1)$, we get
\[
|W|\prod_{\alpha\in R_1^+}
\rho(\alpha^\vee)
=
\prod_{\alpha\in R_1^+}
(\rho(\alpha^\vee)+\mathbf m_1(\alpha))
\]
and hence
\[
|W|^2
\Bigl(\prod_{\alpha\in R_1^+}\frac{\rho(\alpha^\vee)}{\rho(\alpha^\vee)+\mathbf m_1(\alpha)}\Bigr)^2=1.
\]
Thus $\mathcal J'_{\mathbf H}F=\mathcal J_{\mathbf H}F$.

Finally (iii) is an easy corollary of (i) and (ii).
\end{proof}
The main result of this section is the following:
\begin{thm}\label{thm:F}
Suppose $\eta>0$.
Then the pair of
$\mathcal F_G:C^\infty_\eta(G/K)_\Kf\to \tPW_\eta (\Liea^*\times K/M)_\Kf$
and
$\mathcal F_{\mathbf H}:C^\infty_\eta(A)\to \tPW_\eta (\Liea^*\times W)$
is an isomorphism of radial pairs:
\[
(\mathcal F_G,\mathcal F_{\mathbf H}):\,
\bigl(C^\infty_\eta(G/K)_\Kf,C^\infty_\eta(A)\bigr)
\longrightarrow
\bigl(\tPW_\eta (\Liea^*\times K/M)_\Kf, \tPW_\eta (\Liea^*\times W)\bigr).
\]
The inverse morphism is $(\mathcal J_G,\mathcal J_{\mathbf H})$.
Therefore, if\/ $V\in\Km$ then

\noindent{\normalfont (i)}
for any $\Phi\in\Hom_K(V,C^\infty_\cpt(G/K))$ it holds that
\begin{multline*}
\int_{G/K}\Phi[v](x)\,e^{(-i\lambda+\rho)(A(\bar w^{-1}x))}\,dx\\
=\int_A \Phi[v](a)\,\mathbf G(-i\lambda,w^{-1}a)
\prod_{\alpha\in\Sigma^+}\bigl|2\sinh \alpha(\log a)\bigr|^{\dim\Lieg_\alpha}\,da
\\
\forall v\in V^M_\single, \forall \lambda\in \Liea^*\text{ and\/ }\forall w\in W;
\end{multline*}

\noindent{\normalfont (ii)}
for any $\Phi\in\Hom_K(V,C^\infty_\cpt(G/K))$ such that
\[
\Phi[v](a)=0\quad\forall v\in V^M_\double\text{ and\/ }\forall a\in A
\]
it holds that
\[
\int_{G/K}\Phi[v](x)\,e^{(-i\lambda+\rho)(A(\bar w^{-1}x))}\,dx
=0
\quad
\forall v\in V^M_\double, \forall \lambda\in \Liea^*\text{ and\/ }\forall w\in W.
\]
\end{thm}
\begin{proof}
It suffices to show the inverse
\[
(\mathcal J_G,\mathcal J_{\mathbf H}):\,
\bigl(\tPW_\eta (\Liea^*\times K/M)_\Kf, \tPW_\eta (\Liea^*\times W)\bigr)
\longrightarrow
\bigl(C^\infty_\eta(G/K)_\Kf,C^\infty_\eta(A)\bigr)
\]
is a morphism in $\CCh$.
Suppose $V\in \Km$ and
let $\Phi\in\Hom_K(V,\tPW_\eta (\Liea^*\times K/M))$.
Let $\tilde\Gamma^V_B$ be as in the proof of Proposition \ref{prop:Fpairs}. For each $\lambda\in\Liea^*$ define $\Phi_\lambda\in\Hom_K(V,B_G(i\lambda))$
as in \eqref{eq:Phisplam}.
Then for $v\in V^M_\single$ and $a\in A$
we have
\begin{align*}
(\mathcal J_G&\circ\Phi)[v](a)\\
&=\int_{\Liea^*}\left(\int_K
e^{(i\lambda+\rho)(A(k^{-1}a))}
\Phi[v](\lambda, k)\,dk\right)|\mathbf c(\lambda)|^{-2}\,d\lambda\\
&=\int_{\Liea^*}\left(\int_K
e^{(i\lambda+\rho)(A(k^{-1}a))}
\Phi_\lambda[v](k)\,dk\right)|\mathbf c(\lambda)|^{-2}\,d\lambda\\
&=\int_{\Liea^*}\left(
\frac1{|W|}\sum_{w\in W} \Phi_\lambda[v](\bar w)\mathbf G(i\lambda,w^{-1}a)
\right)|\mathbf c(\lambda)|^{-2}\,d\lambda
&&(\because \text{Theorem \ref{thm:Poisson} (i)'})\\
&=\int_{\Liea^*}\left(
\frac1{|W|}\sum_{w\in W} \tilde\Gamma^V_B(\Phi)[v](\lambda, w)\mathbf G(i\lambda,w^{-1}a)
\right)|\mathbf c(\lambda)|^{-2}\,d\lambda\\
&=(\mathcal J_{\mathbf H}\circ\tilde\Gamma^V_B(\Phi))[v](a).
\end{align*}
Thus $\tilde\Gamma^V_0(\mathcal J_G\circ\Phi)=\mathcal J_{\mathbf H}\circ\tilde\Gamma^V_B(\Phi)$.
Next, suppose $\Phi\in\Hom_K^\ttt(V,\tPW_\eta (\Liea^*\times K/M))$.
Then for $v\in V^M_\double$ and $\lambda\in\Liea^*$
it holds that
\[
\Phi_\lambda[v](1)=\Phi[v](\lambda,1)=0.
\]
Hence from Theorem \ref{thm:Poisson} (ii)' we have
for $v\in V^M_\double$ and $a\in A$
\[
(\mathcal J_G\circ\Phi)[v](a)
=\int_{\Liea^*}\left(\int_K
e^{(i\lambda+\rho)(A(k^{-1}a))}
\Phi_\lambda[v](k)\,dk\right)|\mathbf c(\lambda)|^{-2}\,d\lambda
=0.
\]
This shows $\mathcal J_G\circ\Phi\in \Hom_K^\ttt(V,C^\infty_\eta(G/K))$.
Thus we have proved $(\mathcal J_G,\mathcal J_{\mathbf H})$
satisfies Conditions \eqref{cond:Ch1} and \eqref{cond:Ch2} in Definition \ref{defn:CCh}.
\end{proof}
To relate two Plancherel formulas given in Proposition \ref{prop:FourierG} (iii)
and Proposition \ref{prop:FourierH} (iii),
we should mention that
the inner product on $\tPW_\eta (\Liea^*\times K/M)_\Kf$ defined by
\[
\int_{\Liea^*}\int_K F_1\,(\lambda,k)\,\overline{F_2\,(\lambda,k)}\,
dk\,|\mathbf c(\lambda)|^{-2}\,d\lambda
\]
and the inner product on $\tPW_\eta (\Liea^*\times W)$ defined by
\[
\int_{\Liea^*}
\frac1{|W|}\sum_{w\in W} F_1\,(\lambda,w)\,\overline{F_2\,(\lambda,w)}\,
|\mathbf c(\lambda)|^{-2}\,d\lambda
\]
are compatible with restriction in the sense of Definition \ref{defn:compRM}.
This is indeed immediate from Proposition \ref{prop:Blcomp}.

\section{The Chevalley restriction theorem, II}\label{sec:Ch2}
In this section, we shall generalize
Theorem \ref{thm:Ch} to the case where $\mathscr F=\mathscr A$ ,
the class of analytic functions.
We start with generalization of the invariant case.
\begin{thm}\label{thm:ChAnal}
The restriction map $\gamma_0:C^\infty(G/K)\to C^\infty(A)$
induces the bijection
\begin{equation*}
\gamma_0: \mathscr A(G/K)^K\simarrow \mathscr A(A)^W.
\end{equation*}
\end{thm}
As far as the author knows, the theorem had been open.
Since this can be considered
as a local version of Helgason's result \cite[Lemma 2.2]{Hel3},
the first half of our proof follows his idea.
\begin{proof}[Proof of {\normalfont Theorem \ref{thm:ChAnal}}]
In view of \eqref{eq:K-isom} and \eqref{eq:W-isom}, $\gamma_0$ is identified with
the restriction map $C^\infty(\Lies)\to C^\infty(\Liea)$.
Hence we show that the latter interpretation of $\gamma_0$
induces the bijection
\begin{equation*}
\gamma_0: \mathscr A(\Lies)^K\simarrow \mathscr A(\Liea)^W.
\end{equation*}
It is trivial that $\gamma_0\bigl(\mathscr A(\Lies)^K\bigr)\subset \mathscr A(\Liea)^W$
and we already know any $\varphi\in \mathscr A(\Liea)^W$ is lifted to a unique $f\in
C^\infty(\Lies)^K$ such that $\gamma_0(f)=\varphi$.
Hence the only non-trivial point is the analyticity of $f$.
Because of the $K$-invariance of $f$ it is enough to show
that $f$ is analytic in a neighborhood of each $x\in \Liea$ such that
$\alpha(x)\ge0$ for $\alpha\in\Pi$.
Put $\Theta=\{\alpha\in\Pi;\,\alpha(x)=0\}$,
$\Liea(\Theta)=\sum_{\alpha\in\Theta}\RR\alpha^\vee\subset\Liea$ and 
$\alpha^\Theta=\{H\in\Liea;\,\alpha(H)=0\text{ for any }\alpha\in\Theta\}$.
Let $W(\Theta)$ be the subgroup of $W$ generated by $\{s_\alpha;\,\alpha\in\Theta\}$.
Let $\ell=|\Pi|$, $k=|\Theta|$ and $\Theta=\{\alpha_1,\ldots,\alpha_k\}$.
Take linearly independent $\ell-k$ elements
$\varpi_{k+1},\ldots \varpi_\ell\in\Liea^*$
so that $\varpi_i|_{\Liea(\Theta)}=0$ for $i=k+1,\ldots,\ell$.
Since $\{\alpha_1,\ldots,\alpha_k,\varpi_{k+1},\ldots,\varpi_\ell\}$
is a coordinate system of $\Liea$, we can expand $\varphi$ into
a power series of the form
\[
\sum_{\nu=(\nu_1,\ldots,\nu_\ell)} c_\nu\, \alpha_1^{\nu_1} \cdots
 \alpha_k^{\nu_k} (\varpi_{k+1}-\varpi_{k+1}(x))^{\nu_{k+1}} \cdots 
(\varpi_{\ell}-\varpi_{\ell}(x))^{\nu_{\ell}}
\]
which converges absolutely and uniformly
on a complex open neighborhood $U\subset\Liea_\CC$ of $x$ and coincides with $\varphi$ on $U\cap\Liea$.
Take
a set $\{j_1,\ldots,j_k\}$ of algebraically independent homogeneous generators of
$\mathscr P(\Liea(\Theta))^{W(\Theta)}$.
Then the holomorphic map
\[
j:\Liea_\CC=\Liea(\Theta)_\CC\times\Liea_\CC^\Theta
\ni (H_1,H_2)\longmapsto
\bigl((j_1(H_1),\ldots,j_k(H_1)),
H_2\bigr)\in \CC^k\times\Liea_\CC^\Theta
\]
is proper (the inverse image of any compact set is compact) by \cite[Lemma 7]{Ko1}.
From this fact and the well-known surjectivity of $j$ it follows that
the usual topology of $\CC^k\times\Liea_\CC^\Theta$
coincides with the quotient topology by $j$.
Furthermore, since a fiber of $j$ is a $W(\Theta)$-orbit
one easily sees $j$ is an open map.
In particular $j(U)$ is open.
Now we put for $m=0,1,2,\ldots$
\[
\varphi_m=\sum_{\nu_1+\cdots+\nu_\ell=m} c_\nu\, \alpha_1^{\nu_1} \cdots
 \alpha_k^{\nu_k} (\varpi_{k+1}-\varpi_{k+1}(x))^{\nu_{k+1}} \cdots 
(\varpi_{\ell}-\varpi_{\ell}(x))^{\nu_{\ell}}.
\]
Since each $\varphi_m$ is a $W(\Theta)$-invariant polynomial, there exists
a polynomial $\psi_m$ on $\CC^k\times\Liea_\CC^\Theta$
such that $\varphi_m=\psi_m\circ j$.
Then as a limit of uniform convergence,
$\psi:=\sum_m \psi_m$ is a holomorphic function on $j(U)$
such that $\varphi=\psi\circ j$.

Let $\Lieg(\Theta)$ be the Lie subalgebra of $\Lieg$ generated by
$\Liem$ and $\Lieg_\alpha$ for $\alpha\in\Sigma$ with $\alpha(x)=0$.
This is reductive and $\theta$-stable.
Put $\Liek(\Theta)=\Liek\cap\Lieg(\Theta)$,
$\Lies(\Theta)=\Lies\cap\Lieg(\Theta)$ and let $K(\Theta)\subset G$
be the analytic subgroup of $\Liek(\Theta)$.
Note $\Lies(\Theta)\cap\Liea=\Liea(\Theta)$ is a maximal Abelian subspace of $\Lies(\Theta)$ and the Weyl group for $(\Lieg(\Theta), \Liea(\Theta))$
is naturally identified with $W(\Theta)$.
Hence by the classical Chevalley restriction theorem,
for each $i=1,\ldots,k$ there exists $J_i\in \mathscr P(\Lies(\Theta))^{K(\Theta)}$ such that $J_i|_{\Liea(\Theta)}=j_i$.
Define the holomorphic map
\[
J:\Liea_\CC=\Lies(\Theta)_\CC\times\Liea_\CC^\Theta
\ni (X,H_2)\longmapsto
\bigl((J_1(X),\ldots,J_k(X)),
H_2\bigr)\in \CC^k\times\Liea_\CC^\Theta.
\]
Then $\psi\circ J$ is holomorphic on $J^{-1}(j(U))$
and for any $k\in K(\Theta)$ and $(H_1,H_2)\in U\cap\Liea$ it holds that
\[
(\psi\circ J)(k(H_1,H_2))=
(\psi\circ J)(kH_1,H_2)=(\psi\circ j)(H_1,H_2)=\varphi(H_1,H_2)=f(k(H_1,H_2)).
\]
Since $\Lies(\Theta)$ has 
a $K(\Theta)$-invariant metric and since each $K(\Theta)$-orbit in $\Lies(\Theta)$
intersects with $\Liea(\Theta)$,
one has $\mathcal B=\Ad(K(\Theta))(\mathcal B\cap \Liea(\Theta))$
for any open ball $\mathcal B\subset \Lies(\Theta)$ with center $0$.
Hence we can take an open neighborhood $U'$ of $x$
in $\Lies(\Theta)\times\Liea^\Theta$ (\,$\simeq\Lies(\Theta)\oplus\Liea^\Theta\subset\Lies$\,)
so that $U'\subset \Ad(K(\Theta))(U\cap\Liea)$.
The above calculation shows $f|_{U'}$ is analytic.

Now put
\[
\Liek^\Theta:=\sum_{\alpha\in \Sigma;\,\alpha(x)>0}
(\Lieg_\alpha+\Lieg_{-\alpha})\cap\Liek,
\qquad
\Lies^\Theta:=\sum_{\alpha\in \Sigma;\,\alpha(x)>0}
(\Lieg_\alpha+\Lieg_{-\alpha})\cap\Lies
\]
and consider the analytic map
\[
L:\Liek^\Theta \times (\Lies(\Theta)\times\Liea^\Theta)
\ni (Y,X)\longmapsto \Ad(\exp Y)(X)
\in\Lies.
\]
We assert $L$ is a local diffeomorphism at $(0,x)$.
In fact, the tangent spaces of both sides are naturally identified with themselves
at each point,
and $\Lies=\Lies^\Theta\oplus (\Lies(\Theta)\times \Liea^\Theta)$.
Hence the assertion follows from
\begin{align*}
dL_{(0,x)}(X_\alpha+\theta X_\alpha,0)
&=\frac{d}{dt}\Ad(\exp t(X_\alpha+\theta X_\alpha))(x)\Bigr|_{t=0}
=-\alpha(x)(X_\alpha-\theta X_{-\alpha})\\
&\quad\qquad\qquad\qquad\qquad\qquad
\text{ for }\alpha\in\Sigma\text{ with }\alpha(x)>0\text{ and }X_\alpha\in\Lieg_\alpha,\\
dL_{(0,x)}(0,X)&=\frac{d}{dt}(x+tX)\Bigr|_{t=0}=X
\quad\,\text{for }X\in\Lies(\Theta)\times\Liea^\Theta.
\end{align*}
Take an open neighborhood $U''$ of $(0,x)\in\Liek^\Theta \times (\Lies(\Theta)\times\Liea^\Theta)$ so that $U''\subset \Liek^\Theta \times U'$ and
$U''$ is analytically diffeomorphic to $L(U'')$ by $L$.
Since $(f\circ L)(Y,X)=f(X)$ for $(Y,X)\in U''$,
$f\circ L$ is analytic on $U''$.
Hence $f=f\circ L\circ (L|_{U''})^{-1}$ is analytic on
the open neighborhood $L(U'')$ of $x\in\Lies$.
\end{proof}
Let $(\sigma,U)$ be a finite-dimensional representation of $W$ and $\{u_1,\ldots,u_m\}$ a basis of $U$ ($m=\dim U$).
Let $H_W(\Liea_\CC^*)\subset \mathscr P(\Liea)$ be the space of $W$-harmonic polynomials on $\Liea$.
It is well known that $\dim \Hom_W(U,H_W(\Liea_\CC^*))=m$.
Let $\{\varphi_1,\ldots,\varphi_m\}$ be a basis of $\Hom_W(U,H_W(\Liea_\CC^*))$.
\begin{lem}\label{lem:psidet}
Put $\mathbf k_\sigma(\alpha)=(m-\Trace \sigma(s_\alpha))/2$
for each $\alpha\in R_1$.
Then there exists a non-zero constant $C$ such that
\[
\det(\varphi_j[u_i])_{1\le i,j\le m}=C\prod_{\alpha\in R_1^+}\alpha^{\mathbf k_\sigma(\alpha)}.
\]
\end{lem}
\begin{proof}
The determinant is non-zero by \cite[\S2]{HC}
(cf.~\cite[Ch.\,4, Exercise 70 (d)]{V} and \cite[Lemma 5.2.1]{Ma}) and has degree
$\sum_{\alpha\in R_1^+}\mathbf k_\sigma(\alpha)$ by \cite[Formula (1)]{Br}.
Hence it suffices to prove the determinant is divided by $\alpha^{\mathbf k_\sigma(\alpha)}$ for each $\alpha\in R_1^+$.
If $U^{-s_\alpha}\subset U$
denotes the $-1$-eigenspace of $\sigma(s_\alpha)$
then $\dim U^{- s_\alpha}=\mathbf k_\sigma(\alpha)$.
We may assume $\{u_1,\ldots,u_{\mathbf k_\sigma(\alpha)}\}$
is a basis of $U^{-s_\alpha}$.
Since $\varphi_j[u_i]$ is divided by $\alpha$ for each
$i=1,\ldots,\mathbf k_\sigma(\alpha)$ and $j=1,\ldots,m$,
$\det(\varphi_j[u_i])_{1\le i,j\le m}$ is divided by $\alpha^{\mathbf k_\sigma(\alpha)}$.
\end{proof}
\begin{lem}\label{lem:sepvar}
For any $\varphi\in\Hom_W(U,\mathscr A(\Liea))$
there exist $c_1,\ldots,c_m\in \mathscr A(\Liea)^W$ such that
\begin{equation}\label{eq:psilinrel}
\varphi=c_1\varphi_1+\cdots+c_m\varphi_m.
\end{equation}
\end{lem}
\begin{proof}
Let $Q=(q_{ij})\in\Mat(m,m;\mathscr P(\Liea))$ be
the cofactor matrix of $(\varphi_j[u_i])_{1\le i,j\le m}$
and denote the determinant of Lemma \ref{lem:psidet} by $D$.
In general if $\varphi\in\Hom_W(U,\mathscr A(\Liea))$
and $c_1,\ldots,c_m\in\mathscr A(\Liea)$ satisfy
\eqref{eq:psilinrel} then
\begin{equation}\label{eq:psilinrel2}
c_i=\frac1{D}\sum_{j=1}^m q_{ij}\varphi[u_j]
\quad
\text{for }i=1,\ldots,m
\end{equation}
and hence
\begin{equation}\label{eq:psilinrel3}
\sum_{j=1}^m q_{ij}\varphi[u_j]
\text{ is divided by $D$ for }i=1,\ldots,m.
\end{equation}
Conversely if $\varphi\in\Hom_W(U,\mathscr A(\Liea))$ satisfies \eqref{eq:psilinrel3}
then \eqref{eq:psilinrel} holds for $\{c_i\}$ defined by \eqref{eq:psilinrel2}.
(In this case $\{c_i\}\subset \mathscr A(\Liea)^W$
since for any $w\in W$, $c'_i:=wc_i$ ($i=1,\ldots,m$) also satisfy \eqref{eq:psilinrel}
and hence \eqref{eq:psilinrel2}.)
By Lemma \ref{lem:psidet}, \eqref{eq:psilinrel3} is still equivalent to
\begin{equation}\label{eq:psilinrel4}
\partial(\alpha^\vee)^k\Bigl(\sum_{j=1}^m q_{ij}\varphi[u_j]\Bigr)\biggr|_{\alpha=0}=0
\quad\text{for }
\left\{\begin{aligned}
&\alpha\in R_1^+,\, k=0,\ldots,\mathbf k_\sigma(\alpha)-1,\\
&\text{and }i=1,\ldots,m.
\end{aligned}\right.
\end{equation}
Now for any $\varphi\in\Hom_W(U,\mathscr A(\Liea))$ 
and $u\in U$ expand $\varphi[u]$
into the Taylor series at $0$
and let $\varphi^{(d)}[u]$ be the homogeneous part of the series with degree $d$
($d=0,1,2,\ldots$).
Thus there is a complex neighborhood of $0\in\Liea_\CC$
on which
$\sum_{d=0}^\infty \varphi^{(d)}[u]$
converges absolutely and uniformly for all $u\in U$.
Clearly the map $\varphi^{(d)} : U\ni u\mapsto \varphi^{(d)}[u] \in \mathscr P(\Liea)$
is a $W$-homomorphism.
Since $\Hom_W(U, \mathscr P(\Liea))=\bigoplus_{j=1}^m\mathscr P(\Liea)^W\varphi_j$,
each $\varphi^{(d)}$ satisfies \eqref{eq:psilinrel4}.
Taking the sum over $d$ we conclude $\varphi$ also satisfies \eqref{eq:psilinrel4}
around $0$, namely for $\alpha\in R_1^+$,
$k=0,\ldots,\mathbf k_\sigma(\alpha)-1$,
and $i=1,\ldots,m$
there exists an open neighborhood $Y$ of $0$ in the hyperplane $\alpha=0$
such that
\[
\partial(\alpha^\vee)^k\Bigl(\sum_{j=1}^m q_{ij}\varphi[u_j]\Bigr)\biggr|_Y=0.
\]
This implies \eqref{eq:psilinrel4} for $\varphi$ since
the left-hand side of \eqref{eq:psilinrel4} is analytic on $\alpha=0$.
\end{proof}
\begin{thm}\label{thm:Chanal}
{\normalfont Theorem \ref{thm:Ch}} is valid for $\mathscr F=\mathscr A$.
\end{thm}
\begin{proof}
All except the second assertion of (iv) are trivial.
But the ``only if" part of the assertion reduces to the polynomial case
(see the last part of the proof of \cite[Theorem 3.5]{{Oda:HC}}).
Hence we have only to show for any $V\in \Km$ and
$\varphi\in\Hom_W(V^M,\mathscr A(A))\simeq \Hom_W(V^M,\mathscr A(\Liea))$
with $\varphi\bigl[V^M_\double\bigr]=\{0\}$
there exists $\Phi\in \Hom_W(V^M,\mathscr A(\Lies))$
such that $\Gamma^V_0(\Phi)=\varphi$.
Take a basis $\{\varphi_1,\ldots,\varphi_{m}\}$ of
$\Hom_W(V^M_\single,\mathscr P(\Liea))$
and extend each $\varphi_i$ to an element of $\Hom_W(V^M,\mathscr P(\Liea))$
by letting $\varphi\bigl[V^M_\double\bigr]=\{0\}$.
Applying Lemma \ref{lem:sepvar} to the case 
of $U=V^M_\single$,
we obtain $c_1,\ldots,c_m\in\mathscr A(\Liea)^W$ such that
$\varphi=c_1\varphi_1+\cdots+c_m\varphi_m$.
Now for each $i=1,\ldots,m$
we have
$C_i\in\mathscr A(\Lies)^K$ such that $\gamma_0(C_i)=c_i$
by Theorem \ref{thm:ChAnal} and
$\Phi_i\in \Hom_K(V,\mathscr P(\Lies))$ such that $\Gamma^V_0(\Phi_i)=\varphi_i$
by Theorem \ref{thm:Ch} for $\mathscr F=\mathscr P$.
Hence $\Phi:=\sum_iC_i\Phi_i\in\Hom_K(V,\mathscr A(\Lies))$
satisfies $\Gamma^V_0(\Phi)=\varphi$.
\end{proof}
\begin{cor}\label{cor:AApair}
$\bigl(\mathscr A(G/K), \mathscr A(A)\bigr)$
is a radial pair with radial restriction $\gamma_0$.
\end{cor}
\begin{proof}
By the theorem the pair is an object of $\CCh$
and $\gamma_0$ is its radial restriction.
Furthermore the pair
is a subobject of $\bigl(C^\infty(G/K),C^\infty(A)\bigr)\in\Crad$.
Hence it belongs to $\Crad$ by Proposition \ref{prop:catpro} (ii).
\end{proof}
In particular the correspondence
\[
\Ximin_0
: \{\mathbf H\text{-submodules of }C^\infty(A)\}\to\{(\Lieg_\CC,K)\text{-submodules of }C^\infty(G/K)_\Kf\}
\]
defined in Definition \ref{defn:Ximin}
restricts to
\begin{equation}\label{eq:Acor}
\Ximin_0
: \{\mathbf H\text{-submodules of }\mathscr A(A)\}\to\{(\Lieg_\CC,K)\text{-submodules of }\mathscr A(G/K)_\Kf\}.
\end{equation}
In \S\ref{sec:Xi} the latter correspondence will be extended to
a functor $\Xi$
sending any $\mathbf H$-module to a $(\Lieg_\CC,K)$-module.

\section{The functor $\Xirad$}\label{sec:Xirad}
In Definition \ref{defn:Xiwrad} we introduced the functor $\Xiwrad$
which sends an $\mathbf H$-module $\mathscr X$
to a $(\Lieg_\CC,K)$-module
\begin{equation*}
\Xiwrad(\mathscr X)
=\bigoplus_{V\in\Km} \bar P_G(V)\otimes \Hom_{\mathbf H}(P_{\mathbf H}(V^M_\single),\mathscr X).
\end{equation*}
The pair $(\Xiwrad(\mathscr X), \mathscr X)$
is a weak radial pair
with the universal property stated in Proposition \ref{prop:Xiwrad}.
In this section we study a functor $\Xirad$
which has similar properties for the category
$\Crad$ of radial pairs.
\begin{defn}[the functor $\Xirad$]
Suppose $\mathscr X\in\mathbf H\Mod$.
Let $\mathscr N_\rad(\mathscr X)$ be the $\CC$-linear subspace
of $\Xiwrad(\mathscr X)$ spanned by
\begin{equation}\label{eq:Nelm}
\bar\Psi(D)\otimes\varphi-D\otimes (\varphi\circ \tilde\Gamma(\Psi))
\quad
\text{with }\begin{cases}
F, V\in \Km,\, 
\Psi\in\Hom_{\Lieg_\CC,K}^\ttt(P_G(V),P_G(F)),\\
D\in \bar P_G(V),\, \varphi\in \Hom_{\mathbf H}(P_{\mathbf H}(F^M_\single),\mathscr X),
\end{cases}
\end{equation}
where $\bar\Psi$ is the image of $\Psi$ under \eqref{eq:gkPbar}.
Note  $\mathscr N_\rad(\mathscr X)$ is stable under the $(\Lieg_\CC,K)$-action.
We put 
\[
\Xirad(\mathscr X)=\Xiwrad(\mathscr X)/\mathscr N_\rad(\mathscr X)
\ \in(\Lieg_\CC,K)\Mod.
\] 
\end{defn}
We give $(\Xirad(\mathscr X),\mathscr X)$
a structure of a weak radial pair first.
Recall the linear maps 
$\tilde\gamma^{\bar P_G(V)} : \bar P_G(V)\to P_{\mathbf H}(V^M_\single)$ ($V\in\Km$)
and $\gamma_\wrad:\Xiwrad(\mathscr X)\to\mathscr X$
defined in Lemma \ref{lem:PGbar} (iv) and before Lemma \ref{lem:gammawrad} respectively.
Suppose an element in $\mathscr N_\rad(\mathscr X)$
is given by \eqref{eq:Nelm}.
Since $\Psi\in\Hom^\ttt$, it follows from the commutativity of \eqref{cd:HCfunct}
that
$\tilde\gamma^{\bar P_G(F)}\circ \bar\Psi = \tilde\Gamma(\Psi)\circ \tilde\gamma^{\bar P_G(V)}$.
Hence we have
\[
\gamma_\wrad\bigl(
\bar\Psi(D)\otimes\varphi-D\otimes (\varphi\circ \tilde\Gamma(\Psi))
\bigr)
=\varphi\bigl(\tilde\gamma^{\bar P_G(F)}(\bar\Psi(D))\bigr)
-(\varphi\circ \tilde\Gamma(\Psi))\bigl(\tilde\gamma^{P_G(V)}(D)\bigr)=0.
\]
Combining this with Lemma \ref{lem:gammawrad},
we can apply Proposition \ref{prop:XinaturalGen} (ii)
to the case where
$(\Xiwrad(\mathscr X),\mathscr X, \gamma_\wrad, \mathscr N_\rad(\mathscr X),\{0\})$ is $(\mathcal M_G,\mathcal M_{\mathbf H},\gamma_{\mathcal M},\mathscr Y,\mathscr X)$
in the proposition.
Hence $(\Xirad(\mathscr X),\mathscr X)=(\Xiwrad(\mathscr X)/\mathscr N_\rad(\mathscr X),\mathscr X/\{0\})\in\Cwrad$
and the induced linear map $\gamma_\rad : \Xirad(\mathscr X)\to\mathscr X$
satisfies Conditions \eqref{cond:rest1'}, \eqref{cond:rest2} and \eqref{cond:rest3}.
In particular, the structure of $(\Xirad(\mathscr X),\mathscr X)$
as a weak radial pair is given by the linear map
\begin{align*}
\tilde\Gamma^V_\rad :\,&
 \Hom_K(V,\Xirad(\mathscr X))
\longrightarrow
\Hom_W(V^M_\single,\mathscr X)\,;\\
&\Phi\longmapsto
\bigl(
V^M_\single \hookrightarrow
V
\xrightarrow{\Phi} \Xirad(\mathscr X)
\xrightarrow{\gamma_\rad} \mathscr X
\bigr)
\end{align*}
and the natural surjective map
\begin{equation}\label{eq:surj22}
\Hom_K^\ttt(V, \Xiwrad(\mathscr X))
\twoheadrightarrow
\Hom_K^\ttt(V, \Xirad(\mathscr X))
\end{equation}
(actually this is a bijective map) defined for each $V\in\Km$.
\begin{thm}
$(\Xirad(\mathscr X), \mathscr X)\in\Crad$.
\end{thm}
\begin{proof}
By Lemma \ref{lem:rest3}
$(\Xirad(\mathscr X), \mathscr X)$ satisfies \eqref{cond:rad1}.
Let us check \eqref{cond:rad2}.
Suppose $F,V\in\Km$. Recall
\begin{align*}
\Hom_K(V,\Xiwrad(\mathscr X))
&\simeq\bigoplus_{F\in\Km} \Hom_K(V,\bar P_G(F))\otimes
\Hom_{\mathbf H}(P_{\mathbf H}(F^M_\single),\mathscr X),
\\
\Hom_K^\ttt(V,\Xiwrad(\mathscr X))&
=I_V\otimes \Hom_{\mathbf H}(P_{\mathbf H}(V^M_\single),\mathscr X),
\end{align*}
where
$I_V\in\Hom_K(V,\bar P_G(V))$ denotes the map $V\ni v\mapsto 1\otimes v \in \bar P_G(V)$.
By \eqref{eq:surj22}
any $\Phi \in \Hom_K^\ttt(V,\Xirad(\mathscr X))$
is written as $I_V\otimes\varphi$
with $\varphi\in \Hom_{\mathbf H}(P_{\mathbf H}(V^M_\single),\mathscr X)$
(we omit ``$\bmod{\mathscr N_\rad(\mathscr X)}$").
Now for any $\Psi\in\Hom_K^\ttt(E,P_G(V))$ and $e\in E$
\[
(\Phi\circ\Psi)[e]=\bar\Psi[e]\otimes \varphi
=(1\otimes e)\otimes (\varphi\circ\tilde\Gamma(\Psi))
\]
by \eqref{eq:Nelm}.
Thus $\Phi\circ\Psi=I_E\otimes (\varphi\circ\tilde\Gamma(\Psi))
\in \Hom_K^\ttt(E,\Xirad(\mathscr X))$.
This proves \eqref{cond:rad2}.
\end{proof}
\begin{lem}\label{lem:Xirad}
Suppose $\mathcal M=(\mathcal M_G, \mathcal M_{\mathbf H})\in\Crad$
and $(\mathcal I_G,\mathcal I_{\mathbf H}): (\Xiwrad(\mathscr X),\mathscr X)\to \mathcal M$ is a morphism of $\Cwrad$.
Then $\mathscr N_\rad(\mathscr X)\subset\Ker \mathcal I_G$ and
hence the morphism $(\Xirad(\mathscr X),\mathscr X)\to \mathcal M$ of $\Crad$ is naturally induced.
\end{lem}
\begin{proof}
For $F,V\in \Km$, $\Psi\in\Hom_{\Lieg_\CC,K}^\ttt(P_G(V),P_G(F))$
and $\varphi\in \Hom_{\mathbf H}(P_{\mathbf H}(F^M_\single),\mathscr X)$
define $\Phi\in \Hom_K(V,\mathscr N_\rad(\mathscr X))$ by
\[
V\ni v\longmapsto \bar\Psi[v]\otimes\varphi- (1\otimes v)\otimes (\varphi\circ\tilde\Gamma(\Psi))\in \mathscr N_\rad(\mathscr X).
\]
We assert $\mathcal I_G\circ \Phi=0$.
Since $\Phi=\bar\Psi \otimes \varphi
-I_V\otimes (\varphi\circ\tilde\Gamma(\Psi))$ we have
\begin{align*}
\tilde\Gamma_{\mathcal M}^V(\mathcal I_G\circ \Phi)
&=\mathcal I_{\mathbf H}\circ\tilde\Gamma_\wrad^V(\Phi)
&&(\because \eqref{cond:Ch1}\text{ for }\mathcal I_G)\\
&=\mathcal I_{\mathbf H}\circ\bigl(\varphi\circ\tilde\Gamma(\Psi)-\varphi\circ\tilde\Gamma(\Psi)\bigr)
&&(\because \eqref{eq:Gwraddef})\\
&=0.
\end{align*}
Since $I_F\otimes \varphi\in\Hom_K^\ttt(F,\Xiwrad(\mathscr X))$ and $I_V\otimes (\varphi\circ\tilde\Gamma(\Psi))\in\Hom_K^\ttt(V,\Xiwrad(\mathscr X))$ we also have
\begin{align*}
&\mathcal I_G\circ (I_F\otimes \varphi) \in \Hom_K^\ttt(F,\mathcal M_G),
&&(\because \eqref{cond:Ch2}\text{ for }\mathcal I_G)\\
&\mathcal I_G\circ (I_F\otimes \varphi) \circ \Psi \in \Hom_K^\ttt(V,\mathcal M_G),
&&(\because \eqref{cond:rad2}\text{ for }\mathcal M_G)\\
&\mathcal I_G\circ (I_V\otimes (\varphi\circ\tilde\Gamma(\Psi)))\in \Hom_K^\ttt(V,\mathcal M_G),
&&(\because \eqref{cond:Ch2}\text{ for }\mathcal I_G)\\
\therefore\quad&\mathcal I_G\circ \Phi=\mathcal I_G\circ (I_F\otimes \varphi) \circ \Psi
-\mathcal I_G\circ (I_V\otimes (\varphi\circ\tilde\Gamma(\Psi)))\\
&\qquad\qquad\qquad\in \Hom_K^\ttt(V,\mathcal M_G).
\end{align*}
These facts imply $\mathcal I_G\circ \Phi=0$
since $\tilde\Gamma_{\mathcal M}^V$ is injective on $\Hom_K^\ttt(V,\mathcal M_G)$.
Thus we get our assertion, showing $\mathcal I_G$ maps any element given by \eqref{eq:Nelm} to zero.
\end{proof}
From the lemma we can see in particular that
the correspondence $\mathbf H\Mod\ni\mathscr X\to (\Xirad(\mathscr X),\mathscr X)\in\Crad$ is a functor.
The following is an easy corollary of Proposition \ref{prop:Xiwrad}
and Lemma \ref{lem:Xirad}.
\begin{thm}\label{thm:Xirad}
The functor $\mathbf H\Mod\ni\mathscr X\mapsto (\Xirad(\mathscr X),\mathscr X)\in\Crad$
is left adjoint to the functor $\Crad\ni (\mathcal M_G, \mathcal M_{\mathbf H})
\mapsto \mathcal M_{\mathbf H}\in \mathbf H\Mod$.
More precisely, if $\mathcal M=(\mathcal M_G, \mathcal M_{\mathbf H})\in\Crad$
and an $\mathbf H$-homomorphism $\mathcal I_{\mathbf H}:\mathscr X\to \mathcal M_{\mathbf H}$ are given,
then
there exists a unique $(\Lieg_\CC,K)$-homomorphism
$\mathcal I_G:\Xirad(\mathscr X)\to\mathcal M_G$ such that
$(\mathcal I_G,\mathcal I_{\mathbf H}): (\Xirad(\mathscr X),\mathscr X)
\to (\mathcal M_G, \mathcal M_{\mathbf H})$ is a morphism of $\Crad$.
\end{thm}
\begin{defn}\label{defn:ccic}
Suppose $\lambda\in\Liea_\CC^*$.
We say $\mathscr X\in\mathbf H\Mod$ has a central character $[\lambda]$
if
\[
(\Delta-\Delta(\lambda))x=0
\quad\text{for any }\Delta\in\ S(\Liea_\CC)^W
\text{ and }x\in\mathscr X.
\]
We say $\mathscr X\in\mathbf H\Mod$ has a generalized central character $[\lambda]$
if for any $x\in\mathscr X$ and $\Delta\in\ S(\Liea_\CC)^W$
there exists a positive integer $n$ such that
\[
(\Delta-\Delta(\lambda))^nx=0.
\]

Fix a maximal Abelian subalgebra $\Lieb$ of $\Liem$.
Then $\Lieh=\Lieb+\Liea$ is a Cartan subalgebra of $\Lieg$.
Let $W(\Lieg_\CC, \Lieh_\CC)$ be the Weyl group for $(\Lieg_\CC, \Lieh_\CC)$
and $\Upsilon:U(\Lieg_\CC)^G\simarrow S(\Lieh_\CC)^{W(\Lieg_\CC, \Lieh_\CC)}$
the Harish-Chandra isomorphism for the complex Lie algebra $\Lieg_\CC$.
Suppose $\mu\in\Lieh_\CC^*$.
We say $\mathscr Y\in(\Lieg_\CC,K)\Mod$ has an infinitesimal character $[\mu]$ if
\[
(\Delta-\Upsilon(\Delta)(\mu))y=0
\quad\text{for any }\Delta\in U(\Lieg_\CC)^G
\text{ and }y\in\mathscr Y.
\]
We say $\mathscr Y\in(\Lieg_\CC,K)\Mod$ has a generalized infinitesimal character $[\mu]$
if for any $y\in\mathscr Y$ and $\Delta\in\ U(\Lieg_\CC)^G$
there exists a positive integer $n$ such that
\[
(\Delta-\Upsilon(\Delta)(\mu))^nx=0.
\]
\end{defn}
\begin{thm}\label{thm:charcorr}
Fix a positive system of the root system for $(\Liem_\CC,\Lieb_\CC)$
and let $\rho_\Liem \in \Lieb_\CC^*$ be half the sum of positive roots.
Suppose $\lambda\in\Liea_\CC^*$.
Let $(\rho_\Liem,\lambda)$ denote the element of $\Lieh^*_\CC$
which equals $\rho_\Liem$ on $\Lieb$ and $\lambda$ on $\Liea$.
If $\mathscr X\in\mathbf H\Mod$ has (resp., generalized) central character $[\lambda]$
then $\Xirad(\mathscr X)$ has (resp., generalized) infinitesimal character $[(\rho_\Liem,\lambda)]$.
\end{thm}
\begin{proof}
First we assert
\begin{equation}\label{eq:Upsilongamma}
\Upsilon(\Delta)((\rho_\Liem,\lambda))=\gamma(\Delta)(\lambda)
\quad\text{for any }\Delta\in U(\Lieg_\CC)^G.
\end{equation}
Since $U(\Lieg_\CC)^{MA}\subset \Lien_\CC U(\Lieg_\CC)\oplus U(\Liem_\CC+\Liea_\CC)^{MA}$,
the projection to the second summand defines a linear map (actually an algebra homomorphism)
\[
\Upsilon':U(\Lieg_\CC)^{MA}\to 
U(\Liem_\CC+\Liea_\CC)^{MA} \simeq U(\Liem_\CC)^M\otimes S(\Liea_\CC).
\]
If we define two more algebra homomorphisms
\begin{align*}
\Upsilon'': U(\Liem_\CC)^M\otimes S(\Liea_\CC)
&\xrightarrow{\text{H.-C.isom.for $\Liem_\CC$ }\otimes\text{ shift by }-\rho} S(\Lieb_\CC)^{W(\Liem_\CC, \Lieb_\CC)} \otimes S(\Liea_\CC)\hookrightarrow S(\Lieh_\CC),\\
\Upsilon''': U(\Liem_\CC)^M\otimes S(\Liea_\CC) 
&=\bigl((U(\Liem_\CC)\Liem_\CC)^M\otimes S(\Liea_\CC)\bigr)
\oplus S(\Liea_\CC)\\
&\qquad\qquad\xrightarrow{\text{projection to the $2$nd summand \& shift by }-\rho} S(\Liea_\CC),
\end{align*}
then one easily sees $\Upsilon=\Upsilon''\circ\Upsilon'|_{U(\Lieg_\CC)^G}$ and
$\gamma|_{U(\Lieg_\CC)^{MA}}=\Upsilon'''\circ\Upsilon'$.
Since the maximal ideal $(U(\Liem_\CC)\Liem_\CC)^M$ of $U(\Liem_\CC)^M$
corresponds to $[\rho_\Liem]$
by the Harish-Chandra isomorphism for $U(\Liem_\CC)^M$,
it holds that
$\Upsilon''(D)((\rho_\Liem,\lambda))=\Upsilon'''(D)(\lambda)$
for any $D\in U(\Liem_\CC)^M\otimes S(\Liea_\CC)$.
Hence we have \eqref{eq:Upsilongamma}.

Secondly
let $V\in\Kqsp$,
$D\in \bar P_G(V)$ and
$\varphi\in\Hom_{\mathbf H}(P_{\mathbf H}(V^M_\single),\mathscr X)$.
Then $\Xirad(\mathscr X)$ is spanned by elements like $D\otimes \varphi$.
For any $\Delta\in U(\Lieg_\CC)^G$ 
define $\Psi_\Delta\in \End_{\Lieg_\CC,K}(P_G(V))\simeq\Hom_K(V,P_G(V))$ by
$\Psi_{\Delta}[v]=\Delta\otimes v$ for $v\in V$.
Then $\tilde\Gamma(\Psi_\Delta) \in\End_{\mathbf H}(P_{\mathbf H}(V^M_\single))$
equals the multiplication by $(\Upsilon'''\circ\Upsilon')(\Delta)=\gamma(\Delta)\in S(\Liea_\CC)^W$.
Thus
we have from \eqref{eq:Nelm}
\[
\Delta D\otimes\varphi=
\bar\Psi_\Delta(D)\otimes\varphi
= D \otimes (\varphi \circ \tilde\Gamma(\Psi_\Delta))
= D \otimes(\gamma(\Delta)\varphi).
\]
Therefore if
there exists a positive integer $n$ such that
\[
(\gamma(\Delta)-\gamma(\Delta)(\lambda))^n\varphi[v]=0
\quad\text{for any }v\in V^M_\single
\]
then
\begin{align*}
\bigl(\Delta-\Upsilon(\Delta)((\rho_\Liem,\lambda))\bigr)^n D\otimes\varphi
&=
\bigl(\Delta-\gamma(\Delta)(\lambda)\bigr)^n D\otimes\varphi\\
&=D\otimes\bigl(\bigl(
\gamma(\Delta)-\gamma(\Delta)(\lambda)
\bigr)^n \varphi \bigr)=0.\qedhere
\end{align*}
\end{proof}
\begin{thm}\label{thm:FD2HC}
If $\mathscr X\in\mathbf H\Mod$ has finite dimension
then\/ $\Xirad(\mathscr X)\in(\Lieg_\CC,K)\Mod$ has finite length.
\end{thm}
\begin{proof}
It suffices to show that $\Xirad(\mathscr X)$ is finitely generated
and locally $U(\Lieg_\CC)^G$-finite (see \cite[Proposition 3.7.1 and Theorem 4.2.6]{Wa}).
Since 
\begin{align*}
\Xirad(\mathscr X)
&=\sum_{V\in\Kqsp} \bar P_G(V)\otimes \Hom_{\mathbf H}(P_{\mathbf H}(V^M)_\single,\mathscr X)\\
&=U(\Lieg_\CC)\!\sum_{V\in\Kqsp} (1\otimes V)\otimes \Hom_{\mathbf H}(P_{\mathbf H}(V^M)_\single,\mathscr X)
\end{align*}
and since there
are only finitely many quasi-single-petaled $K$-types \cite[Proposition 3.19]{Oda:HC},
$\sum_{V\in\Kqsp} (1\otimes V)\otimes \Hom_{\mathbf H}(P_{\mathbf H}(V^M)_\single,\mathscr X)$ is
a finite-dimensional subspace generating $\Xirad(\mathscr X)$.
The local finiteness is immediate from Theorem \ref{thm:charcorr}.
\end{proof}

The functor $\Xirad$ can be constructed in a more conceptual way.
Put 
\[
\mathbf P_G=\bigoplus_{V\in\Km} P_G(V),\ 
\bar{\mathbf P}_G=\bigoplus_{V\in\Km} \bar P_G(V),\ 
\mathbf P_{\mathbf H}=\bigoplus_{V\in\Km} P_{\mathbf H}(V^M_\single)
=\bigoplus_{V\in\Kqsp} P_{\mathbf H}(V^M_\single).
\]
The algebra
\begin{equation*}
\End_{\Lieg_\CC,K}(\mathbf P_G)
=\prod_{E\in\Km}\bigoplus_{V\in\Km} \Hom_{\Lieg_\CC,K}(P_G(E),P_G(V))
\end{equation*}
contains
\[
\mathbf A^\ttt:=
\prod_{E\in\Km}\bigoplus_{V\in\Km} \Hom^\ttt_{\Lieg_\CC,K}(P_G(E),P_G(V))
\]
as a subalgebra.
By Lemma \ref{lem:PGbar} (iii),
$\bar{\,\cdot\,}$ induces the algebra homomorphism
\[
\bar{\,\cdot\,} : \mathbf A^\ttt  \longrightarrow \End_{\Lieg_\CC,K}(\bar{\mathbf P}_G).
\]
By Proposition \ref{prop:chain},
$\tilde\Gamma$ induces the algebra homomorphism
\[
\tilde \Gamma : \mathbf A^\ttt \longrightarrow \End_{\mathbf H}(\mathbf P_{\mathbf H})
=\!\!\bigoplus_{E,V\in\Kqsp} \Hom_{\mathbf H}(P_{\mathbf H}(E^M_\single),P_{\mathbf H}(V^M_\single)).
\]
Let $\mathbf A^\ttt_\op$ denote the opposite algebra of $\mathbf A^\ttt$.
Then $\bar{\mathbf P}_G$ is a right $\mathbf A^\ttt_\op$-module via $\bar{\,\cdot\,}$
and for any $\mathscr X\in\mathbf H\Mod$,
$\Hom_{\mathbf H}(\mathbf P_{\mathbf H},\mathscr X)$
is naturally a left $\mathbf A^\ttt_\op$-module via $\tilde\Gamma$.
\begin{prop}\label{prop:tensorRad}
For any $\mathscr X\in\mathbf H\Mod$,
\begin{equation}\label{eq:XiradC}
\Xirad(\mathscr X)=\bar{\mathbf P}_G \otimes_{\mathbf A^\ttt_\op}
\Hom_{\mathbf H}(\mathbf P_{\mathbf H},\mathscr X).
\end{equation}
In particular, 
$\Xirad$ is right exact thanks to
the projectivity of $\mathbf P_{\mathbf H}$ in $\mathbf H\Mod$.
\end{prop}
\begin{proof}
We denote the right-hand side of \eqref{eq:XiradC} by $\mathscr Y$.
First we note
\[
\mathscr Y=\sum_{E\in\Km,V\in\Kqsp}\bar P_G(E) \otimes
\Hom_{\mathbf H}(P_{\mathbf H}(V^M_\single),\mathscr X).
\]
If $D\in \bar P_G(E)$ and $\varphi\in \Hom_{\mathbf H}(P_{\mathbf H}(V^M_\single),\mathscr X)$ then
\[
D\otimes\varphi=D\otimes(\varphi\circ\tilde\Gamma(I_V))
=\bar I_V(D) \otimes \varphi=0
\]
unless $E=V$.
Here $I_V\in\End_{\Lieg_\CC,K}(P_G(V))\subset \mathbf A^\ttt$ denotes the identity.
Thus 
\[
\mathscr Y=\sum_{V\in\Kqsp}\bar P_G(V) \otimes
\Hom_{\mathbf H}(P_{\mathbf H}(V^M_\single),\mathscr X)
\]
and by this expression
a surjective $(\Lieg_\CC,K)$-homomorphism
$\pi : \Xiwrad(\mathscr X)\to\mathscr Y$
is naturally defined.
But it is easy to observe that $\Ker\pi=\mathscr N_\rad(\mathscr X)$.
\end{proof}

\section{The functor $\Ximin$}\label{sec:Ximin}
In this section, we shall define a functor $\Ximin: \mathbf H\Mod \to (\Lieg_\CC,K)\Mod$
so that $(\Ximin(\mathscr X),\mathscr X)$ becomes a radial pair with some canonical
radial restriction $\gammamin$.
This functor will turn out to extend the correspondence
\[
\Ximin : \{\mathbf H\text{-submodules of }P_{\mathbf H}(\CC_\triv)\}\to\{(\Lieg_\CC,K)\text{-submodules of }P_G(\CC_\triv)\}
\]
defined in Definition \ref{defn:Ximin}.
It will also be shown that
when restricted to the category $\FD$ of the finite-dimensional $\mathbf H$-modules,
$\Ximin$ lifts any sesquilinear pairing.

Suppose $\mathscr X\in\mathbf H\Mod$.
Recall 
the linear map $\gamma_\rad:\Xirad(\mathscr X)\to\mathscr X$
used in the beginning of the last section,
which satisfies Conditions \eqref{cond:rest1'}, \eqref{cond:rest2} and \eqref{cond:rest3} for $(\Xirad(\mathscr X),\mathscr X)\in\Crad$.
Hence we can apply the correspondence $\Xi^\natural_{(\Xirad(\mathscr X),\mathscr X)}$ of Proposition \ref{prop:XinaturalGen} to the $\mathbf H$-submodule $\{0\}\subset\mathscr X$.
Put
$\Nmin(\mathscr X)=\Xi^\natural_{(\Xirad(\mathscr X),\mathscr X)}(\{0\})$,
namely
\[
\Nmin(\mathscr X)=\sum\bigl\{
\mathscr V\subset \Xirad(\mathscr X);\, \text{a $K$-stable $\CC$-subspace with }
\gamma_\rad(\mathscr V)=\{0\} 
\bigr\}.
\]
Then from Proposition \ref{prop:XinaturalGen} (iii) we have
\begin{thm}[the functor $\Ximin$]
We put
\[
\Ximin(\mathscr X)=\Xirad(\mathscr X)/\Nmin(\mathscr X)\ \in(\Lieg_\CC,K)\Mod.
\]
Then $(\Ximin(\mathscr X), \mathscr X)\in\Crad$.
A linear map $\gammamin:\Ximin(\mathscr X)\to\mathscr X$ naturally induced from
$\gamma_\rad$ is a radial restriction of\/ $(\Ximin(\mathscr X), \mathscr X)\in\Crad$
in the sense of\/ {\normalfont Definition \ref{defn:RadRest}}.
Furthermore $\gammamin$ satisfies {\normalfont Condition \eqref{cond:rest3}}
in {\normalfont Lemma \ref{lem:rest3}}.
\end{thm}
\begin{lem}\label{lem:Ximin}
Suppose $\mathcal M=(\mathcal M_G, \mathcal M_{\mathbf H})\in\Crad$
has a radial restriction $\gamma_{\mathcal M}$ satisfying \eqref{cond:rest3}.
Suppose moreover $(\mathcal I_G,\mathcal I_{\mathbf H}): (\Xirad(\mathscr X),\mathscr X)\to \mathcal M$ is a morphism of $\Crad$.
Then it holds that
\begin{equation}\label{eq:naturalKer}
\Ker \mathcal I_G=\Xi^\natural_{(\Xirad(\mathscr X),\mathscr X)}(\Ker \mathcal I_{\mathbf H}).
\end{equation}
In particular, $\Nmin(\mathscr X)\subset\Ker \mathcal I_G$ and
hence a morphism $(\Ximin(\mathscr X),\mathscr X)\to \mathcal M$ of\/ $\Crad$ is naturally induced.
\end{lem}
\begin{proof}
We first prove
\begin{equation}\label{eq:gammaIIgamma}
\gamma_{\mathcal M}\circ\mathcal I_G=\mathcal I_{\mathbf H}\circ\gamma_\rad.
\end{equation}
For $V\in\Km$ and $\Phi\in\Hom_K^\ttt(V,\,\Xirad(\mathscr X))$ we have
$\tilde\Gamma_{\mathcal M}^V(\mathcal I_G\circ\Phi)=\mathcal I_{\mathbf H}\circ\tilde\Gamma_\rad^V(\Phi)$ by \eqref{cond:Ch1}.
Hence $(\gamma_{\mathcal M}\circ \mathcal I_G\circ\Phi)\bigr|_{V^M_\single}=
(\mathcal I_{\mathbf H}\circ\gamma_\rad\circ\Phi)\bigr|_{V^M_\single}$.
Since $\mathcal I_G\circ\Phi \in \Hom_K^\ttt(V,\mathcal M_G)$ by \eqref{cond:Ch2},
Condition \eqref{cond:rest1} for $\gamma_{\mathcal M}$ and Condition \eqref{cond:rest1'} for $\gamma_\rad$
imply $(\gamma_{\mathcal M}\circ \mathcal I_G\circ\Phi)\bigr|_{V^M_\double}=
(\mathcal I_{\mathbf H}\circ\gamma_\rad\circ\Phi)\bigr|_{V^M_\double}=0$.
Thus $(\gamma_{\mathcal M}\circ \mathcal I_G\circ\Phi)\bigr|_{V^M}=
(\mathcal I_{\mathbf H}\circ\gamma_\rad\circ\Phi)\bigr|_{V^M}$.
But this implies
$\gamma_{\mathcal M}\circ \mathcal I_G\circ\Phi=\mathcal I_{\mathbf H}\circ\gamma_\rad\circ\Phi$
because of \eqref{eq:gammam}
for $\gamma_{\mathcal M}$ and $\gamma_\rad$.
Furthermore from \eqref{eq:gammaH} and \eqref{eq:gammaX} for $\gamma_{\mathcal M}$ and $\gamma_\rad$,
we have for $D\in U(\Lien_\CC+\Liea_\CC)$ and $v\in V$
\begin{align*}
(\gamma_{\mathcal M}\circ \mathcal I_G)\bigl(D\Phi[v]\bigr)
&=\gamma(D)(\gamma_{\mathcal M}\circ \mathcal I_G)\bigl(\Phi[v]\bigr)\\
&=\gamma(D)(\mathcal I_{\mathbf H}\circ\gamma_\rad)\bigl(\Phi[v]\bigr)
=(\mathcal I_{\mathbf H}\circ\gamma_\rad)\bigl(D\Phi[v]\bigr).
\end{align*}
Now \eqref{eq:gammaIIgamma} follows since
\begin{equation}\label{eq:radgen22}
\begin{aligned}
\Xirad(\mathscr X)
&=\sum_{V\in\Km} \bar P_G(V)\otimes \Hom_{\mathbf H}(P_{\mathbf H}(V^M_\single),\mathscr X)\\
&=\sum_{V\in\Km} U(\Lien_\CC+\Liea_\CC)
\bigl((1\otimes V)\otimes \Hom_{\mathbf H}(P_{\mathbf H}(V^M_\single),\mathscr X)\bigr)\\
&=\sum_{V\in\Km} 
\,\sum_{\varphi\in \Hom_{\mathbf H}(P_{\mathbf H}(V^M_\single),\mathscr X)}
U(\Lien_\CC+\Liea_\CC)
\bigl((I_V\otimes \varphi)[V]\bigr)\\
&=\sum_{V\in\Km} 
\,\sum_{\Phi\in \Hom_{K}^\ttt(V,\,\Xirad(\mathscr X))}
U(\Lien_\CC+\Liea_\CC)\Phi[V].
\end{aligned}
\end{equation}
Here $I_V\in\Hom_K(V,\bar P_G(V))$ is the map $V\ni v\mapsto 1\otimes v \in \bar P_G(V)$.

Secondly for any $V\in\Km$ and $\Phi\in\Hom_K(V,\,\Xirad(\mathscr X))$
we have
\begin{align*}
\Phi\in\Hom_K(V,\Ker \mathcal I_G)
&\Longleftrightarrow
\mathcal I_G\circ \Phi=0\ \bigl(\,\in \Hom_K(V,\mathcal M_G) \,\bigr)\\
&\Longleftrightarrow
(\gamma_{\mathcal M}\circ \mathcal I_G\circ \Phi)[V^M]=\{0\}
&&(\because \text{Remark \ref{rem:RadR}})\\
&\Longleftrightarrow
(\mathcal I_{\mathbf H}\circ \gamma_{\rad}\circ \Phi)[V^M]=\{0\}
&&(\because \eqref{eq:gammaIIgamma})\\
&\Longleftrightarrow
\gamma_{\rad}\bigl( \Phi[V^M]\bigr)\subset \Ker \mathcal I_{\mathbf H}\\
&\Longleftrightarrow
\gamma_{\rad}\bigl( \Phi[V]\bigr)\subset \Ker \mathcal I_{\mathbf H}
&&(\because \eqref{eq:gammam}\text{ for }\gamma_\rad)\\
&\Longleftrightarrow
\Phi[V]\subset \Xi^\natural_{(\Xirad(\mathscr X),\mathscr X)}(\Ker \mathcal I_{\mathbf H})\\
&\Longleftrightarrow
\Phi\in\Hom_K\bigl(V,\Xi^\natural_{(\Xirad(\mathscr X),\mathscr X)}(\Ker \mathcal I_{\mathbf H})\bigr).
\end{align*}
This shows \eqref{eq:naturalKer}.
\end{proof}
\begin{thm}\label{thm:Ximin}
Both $\mathbf H\Mod\ni\mathscr X\mapsto \Ximin(\mathscr X)\in(\Lieg_\CC,K)\Mod$
and
$\mathbf H\Mod\ni\mathscr X\mapsto (\Ximin(\mathscr X),\mathscr X)\in\Crad$
are functors.
If $\mathcal M=(\mathcal M_G, \mathcal M_{\mathbf H})\in\Crad$
has a radial restriction $\gamma_{\mathcal M}$ satisfying \eqref{cond:rest3}
and if $\mathcal I_{\mathbf H}:\mathscr X\to \mathcal M_{\mathbf H}$ is an $\mathbf H$-homomorphism,
then there exists a unique $(\Lieg_\CC,K)$-homomorphism
$\mathcal I_G:\Ximin(\mathscr X)\to\mathcal M_G$ such that
$(\mathcal I_G,\mathcal I_{\mathbf H}): (\Ximin(\mathscr X),\mathscr X)
\to (\mathcal M_G, \mathcal M_{\mathbf H})$ is a morphism of $\Crad$.
Here it holds that
\[
\gamma_{\mathcal M}\circ\mathcal I_G =\mathcal I_{\mathbf H}\circ\gammamin.
\]
Furthermore, if\/ $\mathcal I_{\mathbf H}$ is injective then\/
$\Ximin(\mathscr X)$ is isomorphic to\/ $\Ximin_{\mathcal M}(\Image \mathcal I_{\mathbf H})$
by\/ $\mathcal I_G$ and hence $(\Ximin(\mathscr X),\mathscr X)\simeq(\Ximin_{\mathcal M}(\Image \mathcal I_{\mathbf H}),\Image \mathcal I_{\mathbf H})$
as a radial pair.
\end{thm}
\begin{proof}
All except the last statement easily follow from Theorem \ref{thm:Xirad}, Lemma \ref{lem:Ximin} and \eqref{eq:gammaIIgamma}.
Assume $\mathcal I_{\mathbf H}$ is injective.
Then $\mathcal I_{G}$ is also injective by Lemma \ref{lem:Ximin}.
On the other hand, from \eqref{eq:radgen22} and \eqref{eq:R3Mimage}
we can see $\Ximin_{(\Ximin(\mathscr X),\mathscr X)}(\mathscr X)=\Ximin(\mathscr X)$.
Using \eqref{eq:R3Mimage} again we get
\[
\Image \mathcal I_{G}=\mathcal I_{G}(\Ximin_{(\Ximin(\mathscr X),\mathscr X)}(\mathscr X))
=\Ximin_{\mathcal M}(\Image \mathcal I_{\mathbf H}).
\]
Thus the theorem is proved.
\end{proof}
\begin{cor}\label{cor:Ximinext}
If $\mathcal M=(\mathcal M_G, \mathcal M_{\mathbf H})\in\Crad$
has a radial restriction $\gamma_{\mathcal M}$ satisfying \eqref{cond:rest3}
then the functor $\Ximin : \mathbf H\Mod\to(\Lieg_\CC,K)\Mod$ extends the correspondence
\[
\Ximin_{\mathcal M} : \{\mathbf H\text{-submodules of }\mathcal M_{\mathbf H}\}
\to\{(\Lieg_\CC,K)\text{-submodules of }\mathcal M_G\}
\]
defined in {\normalfont Definition \ref{defn:Ximin}}.
\end{cor}

For example $\bigl(P_G(\CC_\triv), P_{\mathbf H}(\CC_\triv) \bigr)$
and $\bigl(P_{G}(\CC_\triv,\bar\lambda), P_{\mathbf H}(\CC_\triv,\bar\lambda)\bigr)$ with $\lambda\in\Liea^*_\CC$
satisfy the assumption of the corollary
(cf.~Example \ref{exmp:rest3} and Proposition \ref{prop:Aduals}).
Since $\Ximin(P_{\mathbf H}(\CC_\triv))=P_G(\CC_\triv)$,
we have $\Ximin_{(P_{G}(\CC_\triv,\bar\lambda), P_{\mathbf H}(\CC_\triv,\bar\lambda))}(P_{\mathbf H}(\CC_\triv,\bar\lambda))
=P_G(\CC_\triv,\bar\lambda)$ by \eqref{eq:R3Mimage}.
Hence
\begin{equation}\label{eq:XiPCl}
\Ximin(P_{\mathbf H}(\CC_\triv,\bar\lambda))=P_G(\CC_\triv,\bar\lambda)\ \,(\lambda\in\Liea^*_\CC).
\end{equation}

From Proposition \ref{prop:tensorRad} and Lemma \ref{lem:Ximin}
we also have the following:
\begin{cor}\label{cor:partialexact}
Suppose
\[0\longrightarrow\mathscr X_1\longrightarrow\mathscr X_2\longrightarrow\mathscr X_3\longrightarrow0\]
is an exact sequence in $\mathbf H\Mod$.
Then in $(\Lieg_\CC,K)\Mod$ both
\[
0\longrightarrow\Ximin(\mathscr X_1)\longrightarrow\Ximin(\mathscr X_2)
\quad\text{and}\quad
\Ximin(\mathscr X_2)\longrightarrow\Ximin(\mathscr X_3)
\longrightarrow0
\]
are exact.
\end{cor}

The functor $\Ximin$ has some better properties when restricted to $\FD$.
Suppose $\mathscr X\in \FD$.
Then $\Ximin(\mathscr X)\in\HC$ by Theorem \ref{thm:FD2HC}.
First we shall give another realization of $\Ximin(\mathscr X)$
which is useful to deduce some properties.
Let $\sigma$ denote the action of\/ $\mathbf H$ on $\mathscr X$.
Since $\mathscr X$ has finite dimension, $\sigma$ can be integrated to
the action of the simply connected Lie group $A$:
\[
A \ni a\longmapsto a^\sigma:=\exp \sigma(\log a)\in\End_\CC\mathscr X.
\]
Let $\Ind_{MAN}^G(\mathscr X)$ be as in Definition \ref{defn:BG}.
Let $\gammaind:\Ind_{MAN}^G(\mathscr X)\to\mathscr X$ be the linear map defined by $F(g)\mapsto F(1)$.
(Note this is different from
the map $\gamma_{\Ind(\mathscr X)}:\Ind_{MAN}^G(\mathscr X)\to\Ind_{S(\Liea_\CC)}^{\mathbf H}(\mathscr X)$ in Definition \ref{defn:restInd}.)
One easily observes this satisfies \eqref{cond:rest3}.
For each $V\in\Km$ we have the natural identification
\begin{equation}\label{eq:Indev}
\begin{gathered}
\Hom_K(V, \Ind_{MAN}^G(\mathscr X))
\simarrow
\Hom_\CC(V^M,\mathscr X);\\
\Phi\longmapsto
\bigl(\varphi:
V^M\hookrightarrow
V
\xrightarrow{\Phi} \Ind_{MAN}^G(\mathscr X)
\xrightarrow{\gammaind} \mathscr X
\bigr)
\end{gathered}
\end{equation}
(cf.~the proof of Theorem \ref{thm:R3B}).
Identify
$V\otimes \Hom_K(V, \Ind_{MAN}^G(\mathscr X))$
with the $V$-isotypic component of $\Ind_{MAN}^G(\mathscr X)$.
We also identify the linear subspace
\[
\bigl\{
\varphi\in \Hom_\CC(V^M,\mathscr X);\,
\varphi\bigl[V^M_\double\bigr]=\{0\}
\bigr\}\subset\Hom_\CC(V^M,\mathscr X)
\]
with $\Hom_\CC(V^M_\single,\mathscr X)$.
Hence we can consider
\begin{align*}
V\otimes\Hom_W(V^M_\single,\mathscr X)
&\subset 
V\otimes\Hom_\CC(V^M_\single,\mathscr X)
\subset 
V\otimes\Hom_\CC(V^M,\mathscr X)
\\
&
\simeq
V\otimes \Hom_K(V, \Ind_{MAN}^G(\mathscr X))\subset
\Ind_{MAN}^G(\mathscr X)_\Kf.
\end{align*}
This induces a natural $(\Lieg_\CC,K)$-homomorphism
\begin{align*}
\mathcal I_G:
\Xiwrad(\mathscr X)
=&\bigoplus_{V\in\Km} \bar P_G(V)\otimes \Hom_W(V^M_\single,\mathscr X)
\longrightarrow
\Ind_{MAN}^G(\mathscr X)_\Kf\,;\\
&D\otimes v\otimes\varphi\longmapsto
D\Phi[v]
\quad\text{for }\left\{\begin{aligned}
&D\in U(\Lieg_\CC),v\in V \text{ and}\\
&\varphi\in\Hom_W(V^M_\single,\mathscr X)\\
&\quad\text{corresponding to }\\
&\qquad\Phi\in \Hom_K(V, \Ind_{MAN}^G(\mathscr X)).
\end{aligned}\right.
\end{align*}
We put $\Xiind(\mathscr X)=\Image \mathcal I_G$.
Note that $\Xiind(\mathscr X)$ is the submodule of $\Ind_{MAN}^G(\mathscr X)$
generated by
\[
\bigcup_{V\in\Km }V\otimes\Hom_W(V^M_\single,\mathscr X).
\]
\begin{thm}\label{thm:minind}
For any $\mathscr X\in\FD$,
$(\Xiind(\mathscr X),\mathscr X)$ is a radial pair and the linear map
$\gammaind$ is its radial restriction satisfying \eqref{cond:rest3}.
Furthermore $(\Xiind(\mathscr X),\mathscr X) \simeq (\Ximin(\mathscr X),\mathscr X)$
as a radial pair.
\end{thm}
\begin{proof}
First we assert $\gamma_\wrad : \Xiwrad(\mathscr X)\to \mathscr X$
coincides with $\gammaind\circ\mathcal I_G$.
Indeed $\Xiwrad(\mathscr X)$ is spanned by
\[
\bigl\{Df(\lambda)\otimes v\otimes \varphi ;\, D\in U(\Lien_\CC+\Liea_\CC),
v\in V, \varphi\in \Hom_W(V^M_\single,\mathscr X)\bigr\}
\]
over $\CC$ and we have
\begin{align*}
\gammaind\bigl(\mathcal I_G(D\otimes v\otimes \varphi)\bigr)
&=\gammaind(D\Phi[v])
\quad\text{with }\Phi\in \Hom_K(V, \Ind_{MAN}^G(\mathscr X))\\
&\phantom{=\gammaind((Df(\lambda)\Phi[v])
\quad\text{with}}\qquad
\text{corresponding to }\varphi\\
&=\gamma(D)(\Phi[v](1))\\
&=\gamma(D)\varphi\bigl[p^V(v)\bigr]\\
&=\gamma_\wrad(D\otimes v\otimes \varphi)
\end{align*}
where $p^V$ denotes the orthogonal projection $V\to V^M$.

Now for $V\in\Km$ put
\begin{equation}\label{cond:indrest1}
\Hom_K^{\ttt}(V, \Xiind(\mathscr X))
=\bigl\{
\Phi\in \Hom_K(V, \Xiind(\mathscr X));\,
(\gammaind\circ\Phi)\bigl[ V^M_\double\bigr]=\{0\}
\bigr\}
\end{equation}
and define the linear map 
$\tGammaind^V : 
\Hom_K(V,\Xiind(\mathscr X))\to
\Hom_\CC(V^M_\single,\mathscr X)$ by
\begin{equation}\label{cond:indrest2}
\Phi\longmapsto
\bigl(
V^M_\single\hookrightarrow
V
\xrightarrow{\Phi} \Xiind(\mathscr X)
\xrightarrow{\gammaind} \mathscr X
\bigr).
\end{equation}
Let us prove $(\Xiind(\mathscr X),\mathscr X)\in\CCh$ by these data.
First it follows from the first assertion proved above
and \eqref{cond:rest2} for $\gamma_\wrad$
that
$\tilde\Gamma_\wrad^V=\tGammaind^V\circ(\mathcal I_G\circ\cdot)$.
Hence by the surjectivity of\/ $\mathcal I_G\circ\cdot : \Hom_K(V,\Xiwrad(\mathscr X)) \to \Hom_K(V,\Xiind(\mathscr X))$ we have
\[
\Image \tGammaind^V=
\Image \tilde\Gamma_\wrad^V=\Hom_W(V^M_\single,\mathscr X).
\]
Hence the restriction of \eqref{eq:Indev}
to $\Hom_K^{\ttt}(V, \Xiind(\mathscr X))$ reduces to
\begin{equation}\label{eq:Gamma22}
\tGammaind^V : \Hom_K^{\ttt}(V, \Xiind(\mathscr X))
\simarrow
\Hom_W(V^M_\single,\mathscr X).
\end{equation}
Thus $(\Xiind(\mathscr X),\mathscr X)\in\CCh$.

Now by \eqref{cond:indrest1} and \eqref{cond:indrest2},
$\gammaind$ satisfies \eqref{cond:rest1} and \eqref{cond:rest2}
in addition to \eqref{cond:rest3}.
Hence it follows from Lemma \ref{lem:rest3} that
$(\Xiind(\mathscr X),\mathscr X)\in\Crad$.
Finally since it is clear by \eqref{eq:Gamma22} and the definition of $\Xiind(\mathscr X)$ that
$\Ximin_{(\Xiind(\mathscr X),\mathscr X)}(\mathscr X)=\Xiind(\mathscr X)$,
it follows from Theorem \ref{thm:Ximin} that $(\Xiind(\mathscr X),\mathscr X) \simeq (\Ximin(\mathscr X),\mathscr X)$.
\end{proof}
From this realization we can deduce a double induction type property
of $\Ximin$, which however we hope to discuss in a subsequent paper.
Another application is lifting of a sesquilinear pairing in $\FD$
by $\Ximin$.
Suppose $\mathscr X\in\FD$.
Let $\mathscr X^\star$
be the linear space of antilinear functionals on $\mathscr X$
and $(\cdot,\cdot)_{\mathscr X}$
the canonical sesquilinear form 
on $\mathscr X^\star\times \mathscr X$.
Then $\mathscr X^\star$ is naturally an $\mathbf H$-module by
\[
(hx^\star, x)_{\mathscr X}
=(x^\star, h^\star x)_{\mathscr X}
\quad\text{for }h\in \mathbf H,
x^\star\in\mathscr X^\star\text{ and }x\in\mathscr X.
\]
In particular
$\Liea_\CC\subset\mathbf H$ acts on $\mathscr X^\star$.
But we strongly remark that this action is different from
the $\Liea_\CC$-action defined in the beginning of \S\ref{sec:sps}.
To distinguish them, we use the symbol $\mathscr X^\star_{\Liea_\CC}$
which stands for the linear space $\mathscr X^\star$ with the $\Liea_\CC$-module structure given by
\[
(\xi x^\star, x)_{\mathscr X}
=-(x^\star, \xi x)_{\mathscr X}
\quad\text{for }\xi\in \Liea,
x^\star\in\mathscr X^\star_{\Liea_\CC}\text{ and }x\in\mathscr X.
\]
Let $\threeset{\mathscr X}{\cdot,\cdot}{\mathbf H}{\mathscr X^\star_{\Liea_\CC}}$
be the sesquilinear form on $\Ind_{S(\Liea_\CC)}^{\mathbf H}(\mathscr X)\times \Ind_{S(\Liea_\CC)}^{\mathbf H}(\mathscr X^\star_{\Liea_\CC})$ defined in Definition \ref{defn:BH}
and let us consider the injective $\mathbf H$-homomorphism
\begin{equation*}
\iota:\,
\mathscr X\ni x\longmapsto
\bigl(
\mathbf H\ni h \longmapsto
\trans h x \in \mathscr X
\bigr) \in \Ind_{S(\Liea_\CC)}^{\mathbf H}(\mathscr X).
\end{equation*}
Then by Proposition \ref{prop:sesquiHinv}
there exists a surjective $\mathbf H$-homomorphism
$\iota^\star: \Ind_{S(\Liea_\CC)}^{\mathbf H}(\mathscr X^\star_{\Liea_\CC}) \to \mathscr X^\star$ such that
\begin{equation}\label{eq:iotaadj}
\threeset{\mathscr X}{\iota(x),F}{\mathbf H}{\mathscr X^\star_{\Liea_\CC}}
=
(x,\iota^\star(F))_{\mathscr X^\star}
\quad
\text{for }x\in\mathscr X\text{ and }F\in\Ind_{S(\Liea_\CC)}^{\mathbf H}(\mathscr X^\star_{\Liea_\CC})
\end{equation}
where $(x,x^\star)_{\mathscr X^\star}=\overline{(x^\star,x)}_{\mathscr X}$.
Applying $\Xiind$ to these morphisms we obtain
\begin{gather*}
\Xiind(\iota):\,\Xiind(\mathscr X)\hookrightarrow
\Xiind\bigl(\Ind_{S(\Liea_\CC)}^{\mathbf H}(\mathscr X)\bigr),\\
\Xiind(\iota^\star):\,\Xiind\bigl(
\Ind_{S(\Liea_\CC)}^{\mathbf H}(\mathscr X^\star_{\Liea_\CC})
\bigr)\twoheadrightarrow \Xiind(\mathscr X^\star).
\end{gather*}
Note $\Xiind(\iota^\star)$ is surjective by Corollary \ref{cor:partialexact}.
On the other hand, the $\Liea_\CC$-homomorphisms
\begin{gather*}
\ev:\,
\Ind_{S(\Liea_\CC)}^{\mathbf H}(\mathscr X)\ni F(h)\longmapsto F(1)\in\mathscr X,\\
\ev':\,
\Ind_{S(\Liea_\CC)}^{\mathbf H}(\mathscr X^\star_{\Liea_\CC})\ni F(h)\longmapsto F(1)\in\mathscr X^\star_{\Liea_\CC}
\end{gather*}
induce $(\Lieg_\CC,K)$-homomorphisms
\begin{gather*}
\beta:\,
\Xiind\bigl(
\Ind_{S(\Liea_\CC)}^{\mathbf H}(\mathscr X)
\bigr)
\hookrightarrow
\Ind_{MAN}^G\bigl(
\Ind_{S(\Liea_\CC)}^{\mathbf H}(\mathscr X)
\bigr)_\Kf
\longrightarrow
\Ind_{MAN}^G(\mathscr X)_\Kf,\\
\beta':\,
\Xiind\bigl(
\Ind_{S(\Liea_\CC)}^{\mathbf H}(\mathscr X^\star_{\Liea_\CC})
\bigr)
\hookrightarrow
\Ind_{MAN}^G\bigl(
\Ind_{S(\Liea_\CC)}^{\mathbf H}(\mathscr X^\star_{\Liea_\CC})
\bigr)_\Kf
\longrightarrow
\Ind_{MAN}^G\bigl(
\mathscr X^\star_{\Liea_\CC}
\bigr)_\Kf.
\end{gather*}
Let $\threeset{\mathscr X}{\cdot,\cdot}{G}{\mathscr X^\star_{\Liea_\CC}}$
be the sesquilinear form on $\Ind_{MAN}^{G}(\mathscr X)\times \Ind_{MAN}^{G}(\mathscr X^\star_{\Liea_\CC})$ defined in Definition \ref{defn:BG}.
Now we define the invariant sesquilinear form
$(\cdot,\cdot)'$
on $\Xiind(\mathscr X) \times \Xiind\bigl(
\Ind_{S(\Liea_\CC)}^{\mathbf H}(\mathscr X^\star_{\Liea_\CC})
\bigr)$ by
\[
(F_1,F_2)'
=
\bthreeset{\mathscr X}{
(\beta \circ \Xiind(\iota))(F_1)
,\,
\beta'(F_2)
}{G}{\mathscr X^\star_{\Liea_\CC}}.
\]
\begin{lem}\label{lem:sesquilift}
It holds that
\begin{equation}\label{eq:lemmaperp}
\begin{aligned}
\bigl\{F_2\in\Xiind\bigl(
\Ind_{S(\Liea_\CC)}^{\mathbf H}(\mathscr X^\star_{\Liea_\CC})
\bigr)&;\,
(F_1,F_2)'=0\text{ for any }F_1\in \Xiind(\mathscr X)
\bigr\}\\
&=\Xiind(\iota^\star)^{-1}
\bigl(\Ximax_{(\Xiind(\mathscr X^\star),\mathscr X^\star)}(\{0\})\bigr)
\supset \Ker \Xiind(\iota^\star).
\end{aligned}
\end{equation}
Hence $(\cdot,\cdot)'$ induces an invariant sesquilinear form $(\cdot,\cdot)$
on $\Xiind(\mathscr X) \times\Xiind(\mathscr X^\star)$.
Furthermore the pair of $(\cdot,\cdot)$ and $(\cdot,\cdot)_{\mathscr X^\star}$
is compatible with restriction in the sense of
{\normalfont Definition \ref{defn:compRM}}.
\end{lem}
\begin{proof}
Suppose $V\in\Km$
and 
let $\{v_1,\ldots,v_{m'}\}$, $\{v_{m'+1},\ldots,v_m\}$ and $\{v_{m+1},\ldots,v_n\}$
be bases of\/ $V^M_\single$, $V^M_\double$ and $(V^M)^\perp$.
Let $\{v_i^\star\}\subset V^\star$ be as in Definition \ref{defn:compRM}.
Let $\gammaind$, $\gammaind'$, $\gammaind''$ and $\gammaind'''$ be the
canonical restriction maps for
$(\Xiind(\mathscr X),\mathscr X)$,
$(\Xiind(\mathscr X^\star),\mathscr X^\star)$,
$\bigl(\Xiind\bigl(
\Ind_{S(\Liea_\CC)}^{\mathbf H}(\mathscr X)
\bigr),
\Ind_{S(\Liea_\CC)}^{\mathbf H}(\mathscr X)
\bigr)$
and
$\bigl(\Xiind\bigl(
\Ind_{S(\Liea_\CC)}^{\mathbf H}(\mathscr X^\star_{\Liea_\CC})
\bigr),
\Ind_{S(\Liea_\CC)}^{\mathbf H}(\mathscr X^\star_{\Liea_\CC})
\bigr)$ respectively.
Suppose
$\Phi_1\in\Hom_K(V,\Xiind(\mathscr X))$,
$\Phi_2\in\Hom_K\bigl(V^\star,
\Xiind\bigl(
\Ind_{S(\Liea_\CC)}^{\mathbf H}(\mathscr X^\star_{\Liea_\CC})
\bigr)\bigr)$.
Since the diagram
\[
\xymatrix{
\Xiind(\mathscr X)
\ar ^-{\Xiind(\iota)} [r]
\ar _-{\gammaind} [d]
&
\Xiind\bigl(\Ind_{S(\Liea_\CC)}^{\mathbf H}(\mathscr X)\bigr)
\ar ^-{\beta} [r]
\ar _-{\gammaind''} [d]
&
\Ind_{MAN}^G(\mathscr X)_\Kf
\ar ^-{\text{evaluation at }1} [d]\\
\mathscr X
\ar _-{\iota} [r]
&
\Ind_{S(\Liea_\CC)}^{\mathbf H}(\mathscr X)
\ar _-{\ev}[r]
&
\mathscr X
}
\]
commutes,
\[
(\beta \circ \Xiind(\iota)\circ\Phi_1)[v](1)
=(\ev \circ \iota \circ\gammaind\circ\Phi_1)[v]
=(\iota \circ\gammaind\circ\Phi_1)[v](1)
\quad
\text{for }v\in V.
\]
Likewise we have
\[
(\beta' \circ\Phi_2)[v^\star](1)
=(\ev' \circ \gammaind'''\circ\Phi_2)[v^\star]
=(\gammaind'''\circ\Phi_2)[v^\star](1)
\quad
\text{for }v^\star\in V^\star.
\]
Hence 
in a similar way to the proof of Proposition \ref{prop:Blcomp}
we calculate
\begin{align*}
\sum_{i=1}^n\bigl( \Phi_1[v_i],\, \Phi_2[v_i^\star] \bigr)'
&=
\sum_{i=1}^n
\bthreeset{\mathscr X}{
(\beta \circ \Xiind(\iota)\circ\Phi_1)[v_i]
,\,
(\beta'\circ\Phi_2)[v_i^\star]
}{G}{\mathscr X^\star_{\Liea_\CC}}\\
&=
\sum_{i=1}^n
\int_K
\bigl(
(\beta \circ \Xiind(\iota)\circ\Phi_1)[v_i](k)
,\,
(\beta'\circ\Phi_2)[v_i^\star](k)
\bigr)
_{\mathscr X^\star}
\,dk\\
&=
\sum_{i=1}^n
\bigl(
(\beta \circ \Xiind(\iota)\circ\Phi_1)[v_i](1)
,\,
(\beta'\circ\Phi_2)[v_i^\star](1)
\bigr)
_{\mathscr X^\star}\\
&=
\sum_{i=1}^m
\bigl(
(\iota \circ\gammaind\circ\Phi_1)[v_i](1)
,\,
(\gammaind'''\circ\Phi_2)[v^\star_i](1)
\bigr)
_{\mathscr X^\star}\\
&=
\sum_{i=1}^m\frac1{|W|}\sum_{w\in W}
\bigl(
(\iota \circ\gammaind\circ\Phi_1)[w^{-1}v_i](1)
,\,
(\gammaind'''\circ\Phi_2)[w^{-1}v^\star_i](1)
\bigr)
_{\mathscr X^\star}.
\end{align*}
Now suppose $\Phi_1\in\Hom^\ttt_K$ or $\Phi_2\in\Hom^\ttt_K$.
Then
$(\gammaind\circ\Phi_1)\bigr|_{V^M_\double}=0$ or
$(\gammaind'''\circ\Phi_2)\bigr|_{V^M_\double}=0$.
Since $\tGammaind^V(\Phi_1)=(\gammaind\circ\Phi_1)\bigr|_{V^M_\single}\in\Hom_W$
and $(\gammaind'''\circ\Phi_2)\bigr|_{(V^\star)^M_\single}\in\Hom_W$,
the last expression above still equals
\begin{align*}
\sum_{i=1}^{m'}\frac1{|W|}&\sum_{w\in W}
\bigl(
(\iota \circ\gammaind\circ\Phi_1)[w^{-1}v_i](1)
,\,
(\gammaind'''\circ\Phi_2)[w^{-1}v^\star_i](1)
\bigr)
_{\mathscr X^\star}\\
&=
\sum_{i=1}^{m'}\frac1{|W|}\sum_{w\in W}
\bigl(
(\iota \circ\tGammaind^V(\Phi_1))[v_i](w)
,\,
(\gammaind'''\circ\Phi_2)[v^\star_i](w)
\bigr)
_{\mathscr X^\star}\\
&=
\sum_{i=1}^{m'}
\bthreeset{\mathscr X}{
(\iota \circ\tGammaind^V(\Phi_1))[v_i]
,\,
(\gammaind'''\circ\Phi_2)[v^\star_i]
}{\mathbf H}{\mathscr X^\star_{\Liea_\CC}}\\
&=
\sum_{i=1}^{m'}
\bigl(
\tGammaind^V(\Phi_1)[v_i]
,\,
(\iota^\star \circ\gammaind'''\circ\Phi_2)[v^\star_i]
\bigr)
_{\mathscr X^\star}
&(\because \eqref{eq:iotaadj})
\\
&=
\sum_{i=1}^{m'}
\bigl(
\tGammaind^V(\Phi_1)[v_i]
,\,
(\gammaind'\circ\Xiind(\iota^\star) \circ\Phi_2)[v^\star_i]
\bigr)
_{\mathscr X^\star}
\\
&=
\sum_{i=1}^{m'}
\bigl(
\tGammaind^V(\Phi_1)[v_i]
,\,
\tGammaind'^{V^\star}
\bigl(\Xiind(\iota^\star) \circ\Phi_2\bigr)[v^\star_i]
\bigr)
_{\mathscr X^\star}
\end{align*}
where $\tGammaind'^{V^\star}(\Phi):=(\gammaind'\circ\Phi)\bigr|_{(V^\star)^M_\single}$
for $\Phi\in\Hom_K(V^\star, \Xiind(\mathscr X^\star))$.
Thus if $\Phi_1\in\Hom^\ttt_K$ or $\Phi_2\in\Hom^\ttt_K$ then
\begin{equation}\label{eq:compatiprime}
\sum_{i=1}^n\bigl( \Phi_1[v_i],\, \Phi_2[v_i^\star] \bigr)'
=
\sum_{i=1}^{m'}
\bigl(
\tGammaind^V(\Phi_1)[v_i]
,\,
\tGammaind'^{V^\star}
\bigl(\Xiind(\iota^\star) \circ\Phi_2\bigr)[v^\star_i]
\bigr)
_{\mathscr X^\star}.
\end{equation}

Now denote the leftmost part of \eqref{eq:lemmaperp} by $\Xiind(\mathscr X)^\perp$.
Suppose $E\in\Km$
and take bases $\{e_1,\ldots,e_\nu\}\subset E$ and $\{e_j^\star\}\subset E^\star$
as in Proposition \ref{prop:starMor}.
Since $\Hom_K(E,\,\Xiind(\mathscr X))$ is spanned over $\CC$ by elements of the form
\[
\Phi_1\circ\Psi\quad\text{with }
V\in\Km,\Phi_1\in\Hom_{\Lieg_\CC,K}^\ttt(P_G(V),\, \Xiind(\mathscr X))\text{ and }
\Psi\in\Hom_K(E,P_G(V)),
\]
we have for $\Phi\in\Hom_K\bigl(E^\star,
\Xiind\bigl(
\Ind_{S(\Liea_\CC)}^{\mathbf H}(\mathscr X^\star_{\Liea_\CC})
\bigr)\bigr)$
\begin{align*}
&\Phi
\in\Hom_K\bigl(E^\star,\,\Xiind(\mathscr X)^\perp\bigr)\\
&\Longleftrightarrow
\sum_{j=1}^\nu\bigl(
(\Phi_1\circ\Psi)[e_j],\,\Phi[e_j^\star]
\bigr)'=0\quad \forall V, \forall\Phi_1, \forall\Psi\\
&\Longleftrightarrow
\sum_{i=1}^n\bigl(
\Phi_1[v_i],\,(\Phi\circ\Psi^\star)[v_i^\star]
\bigr)'=0\quad \forall V, \forall\Phi_1, \forall\Psi
\qquad(\because \text{Corollary \ref{cor:sesqui}})\\
&\Longleftrightarrow
\sum_{i=1}^{m'}
\bigl(
\tGammaind^V(\Phi_1)[v_i]
,\,
\tGammaind'^{V^\star}
\bigl(\Xiind(\iota^\star) \circ\Phi\circ\Psi^\star\bigr)[v^\star_i]
\bigr)_{\mathscr X^\star}=0\quad \forall V, \forall\Phi_1, \forall\Psi
\quad(\because \eqref{eq:compatiprime})\\
&\Longleftrightarrow
\sum_{i=1}^{m'}
\bigl(
\varphi_1[v_i]
,\,
\tGammaind'^{V^\star}
\bigl(\Xiind(\iota^\star) \circ\Phi\circ\Psi^\star\bigr)[v^\star_i]
\bigr)_{\mathscr X^\star}=0\quad \forall V, \forall\Psi, \forall\varphi_1\in 
\Hom_W(V^M_\single, \mathscr X)\\
&\Longleftrightarrow
\tGammaind'^{V^\star}
\bigl(\Xiind(\iota^\star) \circ\Phi\circ\Psi^\star\bigr)=0\quad \forall V, \forall\Psi\\
&\Longleftrightarrow
\tGammaind'^{V^\star}(\Phi_2)=0\quad
\forall V, \forall \Phi_2\in
\Hom_K\bigl(V^\star,\, U(\Lieg_\CC)(\Xiind(\iota^\star)\circ\Phi)[E^\star]\bigr)\\
&\Longleftrightarrow
U(\Lieg_\CC)(\Xiind(\iota^\star)\circ\Phi)[E^\star] \subset 
\Ximax_{(\Xiind(\mathscr X^\star),\mathscr X^\star)}(\{0\})\\
&\Longleftrightarrow
\Xiind(\iota^\star)\circ \Phi
\in\Hom_K\bigl(E^\star,\,
\Ximax_{(\Xiind(\mathscr X^\star),\mathscr X^\star)}(\{0\})
\bigr).
\end{align*}
Thus we get \eqref{eq:lemmaperp} and the induced form $(\cdot,\cdot)$
on $\Xiind(\mathscr X) \times\Xiind(\mathscr X^\star)$.

The compatibility with restriction of the pair of $(\cdot,\cdot)$
and $(\cdot,\cdot)_{\mathscr X^\star}$ follows from
\eqref{eq:compatiprime}
and the surjectivity of the following maps:
\begin{gather*}
\Hom_K\bigl(V^\star,
\Xiind\bigl(
\Ind_{S(\Liea_\CC)}^{\mathbf H}(\mathscr X^\star_{\Liea_\CC})
\bigr)\bigr)
\xrightarrow{\Xiind(\iota^\star)\circ\cdot}
\Hom_K(V^\star,
\Xiind(\mathscr X^\star)),\\
\Hom_K^\ttt\bigl(V^\star,
\Xiind\bigl(
\Ind_{S(\Liea_\CC)}^{\mathbf H}(\mathscr X^\star_{\Liea_\CC})
\bigr)\bigr)
\xrightarrow{\Xiind(\iota^\star)\circ\cdot}
\Hom_K^\ttt(V^\star,
\Xiind(\mathscr X^\star)).\qedhere
\end{gather*}
\end{proof}
\begin{thm}\label{thm:sesquilift}
Suppose $\mathscr X_1,\mathscr X_2\in\FD$ and 
let $(\cdot,\cdot)^{\mathbf H}$ be an invariant sesquilinear form
on $\mathscr X_1\times\mathscr X_2$.
Then there exists a unique invariant sesquilinear form $(\cdot,\cdot)^{G}$
on $\Ximin(\mathscr X_1)\times\Ximin(\mathscr X_2)$
such that the pair $(\cdot,\cdot)^{G}$ and $(\cdot,\cdot)^{\mathbf H}$
is compatible with restriction in the sense of
{\normalfont Definition \ref{defn:compRM}}.
\end{thm}
\begin{proof}
Let $\mathscr X_1^\star\in\FD$ be as before.
Then the canonical sesquilinear form $(\cdot,\cdot)_{\mathscr X_1^\star}$
on $\mathscr X_1\times\mathscr X^\star_1$ can be lifted to
a sesquilinear form $(\cdot,\cdot)$ on $\Ximin(\mathscr X_1)\times \Ximin(\mathscr X_1^\star)$
by Lemma \ref{lem:sesquilift}.
Now there exists a unique $\mathbf H$-homomorphism
$\mathcal I_{\mathbf H}:\mathscr X_2\to \mathscr X_1^\star$ such that
\[
(x_1,x_2)^{\mathbf H}=(x_1,\mathcal I_{\mathbf H}(x_2))_{\mathscr X_1^\star}
\quad\text{for }x_1\in\mathscr X_1\text{ and }x_2\in\mathscr X_2.
\]
Using $\Ximin(\mathcal I_{\mathbf H}):\Ximin(\mathscr X_2)\to\Ximin(\mathscr X_1^\star)$ we define
\[
(y_1,y_2)^{G}=(y_1,\,\Ximin(\mathcal I_{\mathbf H})(y_2))
\quad\text{for }y_1\in\Ximin(\mathscr X_1)\text{ and }y_2\in\Ximin(\mathscr X_2).
\]
This is clearly an invariant sesquilinear form on $\Ximin(\mathscr X_1)\times\Ximin(\mathscr X_2)$
which, together with $(\cdot,\cdot)^{\mathbf H}$,
is compatible with restriction.
Such a sesquilinear form is unique by Corollary \ref{cor:compRM}.
\end{proof}

\section{The functor $\Xi$}\label{sec:Xi}
If $\mathcal M=(\mathcal M_G,\mathcal M_{\mathbf H})\in\Crad$ and
$\mathscr Y\in(\Lieg_\CC,K)\Mod$ are given then we can make a new radial pair
$\mathcal M'=(\mathcal M_G\times\mathscr Y,\mathcal M_{\mathbf H})$ 
by putting for each $V\in\Km$
\begin{gather*}
\Hom_K^\ttt(V,\mathcal M_G\times\mathscr Y)=\Hom_K^\ttt(V,\mathcal M_G)\times\{0\},\\
\tilde\Gamma^V_{\mathcal M'}=\tilde\Gamma^V_{\mathcal M}\circ
\bigl(\Hom_K(V,\mathcal M_G\times\mathscr Y)\xrightarrow{\text{projection}}
\Hom_K(V,\mathcal M_G)\bigr).
\end{gather*}
Here we note $\Ximin_{\mathcal M'}(\mathcal M_{\mathbf H})\subset \mathcal M_G\times\{0\}$
and $\Ximax_{\mathcal M'}(\{0\})\supset\{0\}\times\mathscr Y$.
This example shows a radial pair $(\mathcal M_G,\mathcal M_{\mathbf H})$
may contain a redundant part which gives no link between $\mathcal M_G$
and $\mathcal M_{\mathbf H}$.
\begin{defn}
We say a radial pair $\mathcal M=(\mathcal M_G,\mathcal M_{\mathbf H})$ is \emph{reduced}
if $\Ximin_{\mathcal M}(\mathcal M_{\mathbf H})=\mathcal M_G$
and $\Ximax_{\mathcal M}(\{0\})=\{0\}$.
\end{defn}
\begin{prop}
For any $\mathbf H$-submodule $\mathscr X$ of $\mathscr A(A)$,
$(\Ximin_0(\mathscr X), \mathscr X)$ is a reduced radial pair.
Hence by {\normalfont Theorem \ref{thm:Xpair}}
$\bigl(X_G(\lambda), X_{\mathbf H}(\lambda)\bigr)$
is reduced for any $\lambda\in\Liea_\CC^*$.
\end{prop}

\begin{proof}
First, let $(\mathcal I_G,\mathcal I_{\mathbf H}): (\Ximin_0(\mathscr X), \mathscr X)\to\bigl(\mathscr A(G/K)_\Kf, \mathscr A(A)\bigr)$
be the pair of inclusions (cf.~Corollary \ref{cor:AApair}).
Then by \eqref{eq:R3Mimage} we have
\[
\Ximin_{(\Ximin_0(\mathscr X), \mathscr X)}(\mathscr X)
=
\mathcal I_G\bigl(
\Ximin_{(\Ximin_0(\mathscr X), \mathscr X)}(\mathscr X)
\bigr)
=
\Ximin_0(\mathcal I_{\mathbf H}(\mathscr X))
=
\Ximin_0(\mathscr X).
\]

Secondly, put 
$\mathscr N=\Ximax_{(\Ximin_0(\mathscr X), \mathscr X)}(\{0\})$.
Then from Theorem \ref{thm:multcor} (ii) we have
\begin{equation}\label{eq:CN0}
\tilde\Gamma^{\CC_\triv}_0
\bigl(\Hom_K(\CC_\triv,\mathscr N)\bigr)=
\Hom_W(\CC_\triv,\{0\})=\{0\}.
\end{equation}
Assume now $\mathscr N\ni f\ne0$.
Since the sesquilinear from $(\cdot,\cdot)^G_{r}$
defined by \eqref{eq:sesquiGr}
is non-degenerate on $\mathscr A(G/K)_\Kf\times P_G(\CC_\triv)$,
there exists some $D\in U(\Lieg_\CC)$ such that
$(f,D\otimes v_\triv)^G_{r}\ne0$.
Thus $(D^\star f,1\otimes v_\triv)^G_{r}\ne0$.
This means $\mathscr N$ contains a non-zero $K$-invariant element,
contrary to \eqref{eq:CN0}.
Hence $\mathscr N=\{0\}$.
\end{proof}
Suppose $\mathcal M=(\mathcal M_G,\mathcal M_{\mathbf H})\in\Crad$
and put
$\mathcal M'=\bigl(\Ximin_{\mathcal M}(\mathcal M_{\mathbf H}),\,
\mathcal M_{\mathbf H}\bigr)$.
Then it follows from Theorem \ref{thm:multcor} (ii)
that $\mathcal M'$ and $\bigl(\Ximax_{\mathcal M'}(\{0\}), \{0\}\bigr)$
are radial pairs.
Moreover as the cokernel of
$\bigl(\Ximax_{\mathcal M'}(\{0\}), \{0\}\bigr)\hookrightarrow
\mathcal M'$ we have
$\bigl(\Ximin_{\mathcal M}(\mathcal M_{\mathbf H})\!
\bigm/\! \Ximax_{\mathcal M'}(\{0\})
,\,\mathcal M_{\mathbf H}\bigr)\in\Crad
$ by Proposition \ref{prop:catpro}
and this is reduced by Theorem \ref{thm:liftincl}.
Furthermore for each $V\in\Km$ it naturally holds that
\[
\Hom_K^\ttt(V, \mathcal M_G)
\simeq
\Hom_K^\ttt\bigl(V,\, \Ximin_{\mathcal M}(\mathcal M_{\mathbf H})\bigr)
\simeq
\Hom_K^\ttt\bigl(V,\, \Ximin_{\mathcal M}(\mathcal M_{\mathbf H})
\!\bigm/\! \Ximax_{\mathcal M'}(\{0\})).
\]
In this way we can always extract the reduced part from any radial pair.

Let us now construct a functor $\Xi:\mathbf H\Mod\to(\Lieg_\CC,K)\Mod$
such that $\mathcal M_G=\Xi(\mathcal M_{\mathbf H})$ 
for any reduced radial pair $(\mathcal M_G,\mathcal M_{\mathbf H})$.
Throughout the paper we have given many examples of radial pairs.
If we extract the reduced part $(\mathcal M_G,\mathcal M_{\mathbf H})$
from any such pair by the above method,
then $\mathcal M_G=\Xi(\mathcal M_{\mathbf H})$.
\begin{defn}[the functor $\Xi$]
For any $\mathscr X\in\mathbf H\Mod$ we put
\[
\Xi(\mathscr X)=\Xirad(\mathscr X)\!\bigm/\!\Ximax_{(\Xirad(\mathscr X),\mathscr X)}(\{0\}).
\]
Then $(\Xi(\mathscr X), \mathscr X)$ is a reduced radial pair.
\end{defn}
\begin{thm}\label{thm:Xiunivprop}
{\normalfont (i)}
If a reduced radial pair $\mathcal M=(\mathcal M_G,\mathcal M_{\mathbf H})$ and 
an $\mathbf H$-homomorphism $\mathcal I_{\mathbf H}: \mathscr X\to \mathcal M_{\mathbf H}$ are given, then
there exists a unique $(\Lieg_\CC,K)$-homomorphism
$\mathcal I_G:\Xi(\mathscr X)\to\mathcal M_G$ such that
$(\mathcal I_G,\mathcal I_{\mathbf H}): (\Xi(\mathscr X),\mathscr X)
\to (\mathcal M_G, \mathcal M_{\mathbf H})$ is a morphism of $\Crad$.
If $\mathcal I_{\mathbf H}$ is injective, so is $\mathcal I_G$.
If $\mathcal I_{\mathbf H}$ is surjective, so is $\mathcal I_G$.

\noindent
{\normalfont (ii)}
Suppose a radial pair $\mathcal M=(\mathcal M_G,\mathcal M_{\mathbf H})$
satisfies $\Ximin_{\mathcal M}(\mathcal M_{\mathbf H})=\mathcal M_G$
and \/
$\mathcal I_{\mathbf H}: \mathcal M_{\mathbf H} \to \mathscr X$ is 
an $\mathbf H$-homomorphism.
Then there exists a unique $(\Lieg_\CC,K)$-homomorphism
$\mathcal I_G:\mathcal M_G\to\Xi(\mathscr X)$ such that
$(\mathcal I_G,\mathcal I_{\mathbf H}): (\mathcal M_G, \mathcal M_{\mathbf H})
\to(\Xi(\mathscr X),\mathscr X)$ is a morphism of $\Crad$.
If $\mathcal I_{\mathbf H}$ is surjective, so is $\mathcal I_G$.
\end{thm}
\begin{proof}
Let $\mathcal M$ and $\mathcal I_{\mathbf H}$ be as in (i).
The existence and uniqueness of $\mathcal I_G:\Xi(\mathscr X)\to\mathcal M_G$ follow from Theorem \ref{thm:Xirad} and \eqref{eq:R3Mpreimage}.
The injectivity of $\mathcal I_{\mathbf H}$ implies that of $\mathcal I_G$ by \eqref{eq:R3Mpreimage}.
The surjectivity of $\mathcal I_{\mathbf H}$ implies that of $\mathcal I_G$ by \eqref{eq:R3Mimage}.

Secondly let $\mathcal M$ and $\mathcal I_{\mathbf H}$ be as in (ii).
Then we have the exact sequence
\[
0 \to 
\bigl(\Ximax_{\mathcal M}(\Ker\mathcal I_{\mathbf H}),
\Ker\mathcal I_{\mathbf H}\bigr)
\xrightarrow{\iota}
(\mathcal M_G,\mathcal M_{\mathbf H} )
\xrightarrow{\pi}
\bigl(
\mathcal M_G\!\bigm/\! \Ximax_{\mathcal M}(\Ker\mathcal I_{\mathbf H}),\,
\Coim\mathcal I_{\mathbf H}
\bigr) \to 0.
\]
We assert $\mathcal M':=\bigl(
\mathcal M_G\!\bigm/\! \Ximax_{\mathcal M}(\Ker\mathcal I_{\mathbf H}),\,
\Coim\mathcal I_{\mathbf H}
\bigr)$ is reduced.
Indeed $\Ximin_{\mathcal M'}\bigl(
\Coim\mathcal I_{\mathbf H}
\bigr)=\mathcal M_G\!\bigm/\! \Ximax_{\mathcal M}(\Ker\mathcal I_{\mathbf H})$
because of \eqref{eq:R3Mimage}.
In addition, since $\pi=(\pi_G,\pi_{\mathbf H})$ is epi we have
\begin{align*}
\Ximax_{\mathcal M'}(\{0\})
&=\pi_G\bigl( \pi_G^{-1} \bigl(\Ximax_{\mathcal M'}(\{0\})\bigr)\bigr)\\
&=\pi_G\bigl( \Ximax_{\mathcal M}(\Ker\pi_{\mathbf H})\bigr)
&(\because \eqref{eq:R3Mpreimage})\\
&=\pi_G\bigl( \Ximax_{\mathcal M}(\Ker\mathcal I_{\mathbf H})\bigr)=\{0\}.
\end{align*}
Hence $\mathcal M'\simeq (\Xi(\Coim\mathcal I_{\mathbf H}),\Coim\mathcal I_{\mathbf H})$
as a radial pair by (i).
Using (i) again we can uniquely lift
the $\mathbf H$-homomorphism
$\mathcal I_{\mathbf H}':\Coim\mathcal I_{\mathbf H}\to\mathscr X$
with $\mathcal I_{\mathbf H}=\mathcal I_{\mathbf H}'\circ\pi_{\mathbf H}$
to a morphism
$\mathcal I'=(\mathcal I_{\mathbf G}',\mathcal I_{\mathbf H}'):
\mathcal M'\to (\Xi(\mathscr X),\mathscr X)$ of $\Crad$.
Hence if we put $\mathcal I_{\mathbf G}=\mathcal I_{\mathbf G}'\circ\pi_G$
then $(\mathcal I_G,\mathcal I_{\mathbf H}): (\mathcal M_G, \mathcal M_{\mathbf H})
\to(\Xi(\mathscr X),\mathscr X)$ is a morphism.
We note $\mathcal I_{\mathbf G}'$ is surjective if $\mathcal I_{\mathbf H}'$
is surjective.
What remains to be shown is the uniqueness of $\mathcal I_G$.
Assume $(\mathcal I_G'',\mathcal I_{\mathbf H}): (\mathcal M_G, \mathcal M_{\mathbf H})
\to(\Xi(\mathscr X),\mathscr X)$ is also a morphism.
Then by \eqref{eq:R3Mpreimage} we have
\[
\Ker \mathcal I_G''
=\mathcal I_G''^{-1}(\{0\})
=\mathcal I_G''^{-1}(\Ximax_{(\Xi(\mathscr X),\mathscr X)}(\{0\}))
=\Ximax_{\mathcal M}(\mathcal I_{\mathbf H}^{-1}(\{0\}))
=\Ximax_{\mathcal M}(\Ker\mathcal I_{\mathbf H}).
\]
Hence $(\mathcal I_G'',\mathcal I_{\mathbf H})$
factors through $\pi$. From the uniqueness of $\mathcal I'_G$ we conclude
$\mathcal I_G''=\mathcal I'_G\circ\pi_G=\mathcal I_G$.
\end{proof}
\begin{cor}\label{cor:Xi}
{\normalfont (i)}
For any $\mathscr X\in\mathbf H\Mod$
\[
\Xi(\mathscr X)=\Ximin(\mathscr X)\!\bigm/\!\Ximax_{(\Ximin(\mathscr X),\mathscr X)}(\{0\}).
\]

\noindent
{\normalfont (ii)}
Suppose a radial pair $\mathcal M=(\mathcal M_G,\mathcal M_{\mathbf H})$
satisfies $\Ximin_{\mathcal M}(\mathcal M_{\mathbf H})=\mathcal M_G$.
Then the identity morphism on $\mathcal M_{\mathbf H}$ naturally induces
two consecutive epimorphisms
\[
(\Xirad(\mathcal M_{\mathbf H}),\mathcal M_{\mathbf H})\twoheadrightarrow
(\mathcal M_G,\mathcal M_{\mathbf H})\twoheadrightarrow
(\Xi(\mathcal M_{\mathbf H}),\mathcal M_{\mathbf H})
\]
in $\Crad$.

\noindent
{\normalfont (iii)}
The functor $\Xi$ extends the correspondence \eqref{eq:Acor}.
In particular, $\Xi(X_{\mathbf H}(\lambda))=X_{G}(\lambda)$
for any $\lambda\in\Liea_\CC^*$.

\end{cor}
In the below we shall see the functor $\Xi$ commutes with conjugate dual operations.
For any $\mathscr X\in\FD$ let $\mathscr X^\star$ and $(\cdot,\cdot)_{\mathscr X^\star}$ be as in the last section.
In general for $\mathscr Y\in\HC$ we define $\mathscr Y^\star\in\HC$ as follows:
For each $V\in\widehat K$
the $V$-isotypic component $\mathscr Y_V$ of $\mathscr Y$ has finite dimension;
Put $\mathscr Y^\star=\bigoplus_{V\in\widehat K}\mathscr Y_V^\star$
where $\mathscr Y_V^\star$ is the space of antilinear functionals on $\mathscr Y_V$;
Using a natural non-degenerate sesquilinear form $(\cdot,\cdot)_{\mathscr Y}$
on $\mathscr Y^\star\times \mathscr Y$, we define the $(\Lieg_\CC,K)$-module structure of $\mathscr Y^\star$ by
\[
(Dy^\star,y)_{\mathscr Y}=(y^\star,D^\star y)_{\mathscr Y},\quad
(ky^\star,y)_{\mathscr Y}=(y^\star,k^{-1} y)_{\mathscr Y}
\]
for $D\in U(\Lieg_\CC)$, $k\in K$,
$y^\star\in\mathscr Y^\star$ and  $y\in\mathscr Y$;
Then it is easy to see $\mathscr Y^\star\in\HC$.
Since $\Ximin(\mathscr X)\in\HC$ by Theorem \ref{thm:FD2HC},
$\Ximin(\mathscr X)^\star\in\HC$.
\begin{thm}\label{thm:Xiindform}
In the setting above
there exists a unique invariant sesquilinear form $(\cdot,\cdot)^{G}$
on $\Xi(\mathscr X)\times\Xi(\mathscr X^\star)$
such that the pair $(\cdot,\cdot)^{G}$ and $(\cdot,\cdot)_{\mathscr X^\star}$
is compatible with restriction in the sense of
{\normalfont Definition \ref{defn:compRM}}.
The form $(\cdot,\cdot)^{G}$ is non-degenerate.
In particular $\Xi(\mathscr X^\star)\simeq\Xi(\mathscr X)^\star$ as a $(\Lieg_\CC,K)$-module.
\end{thm}
\begin{proof}
Let $(\cdot,\cdot)$
be the sesquilinear form on $\Xiind(\mathscr X) \times \Xiind(\mathscr X^\star)$
in Lemma \ref{lem:sesquilift}.
We note \eqref{eq:lemmaperp} can be rewritten as
\[
\bigl\{F_2\in\Xiind(\mathscr X^\star);\,
(F_1,F_2)=0\text{ for any }F_1\in \Xiind(\mathscr X)
\bigr\}
=\Ximax_{(\Xiind(\mathscr X^\star),\mathscr X^\star)}(\{0\}).
\]
We can interchange the roles of $\mathscr X$ and $\mathscr X^\star$
by Corollary \ref{cor:compRM} to deduce
\[
\bigl\{F_1\in\Xiind(\mathscr X);\,
(F_1,F_2)=0\text{ for any }F_2\in \Xiind(\mathscr X^\star)
\bigr\}
=\Ximax_{(\Xiind(\mathscr X),\mathscr X)}(\{0\}).
\]
Since 
\[
\Xi(\mathscr X)=\Xiind(\mathscr X)\!\bigm/\!\Ximax_{(\Xiind(\mathscr X),\mathscr X)}(\{0\}),\quad
\Xi(\mathscr X^\star)=\Xiind(\mathscr X^\star)\!\bigm/\!\Ximax_{(\Xiind(\mathscr X^\star),\mathscr X^\star)}(\{0\}),
\]
$(\cdot,\cdot)$ induces an invariant non-degenerate sesquilinear form $(\cdot,\cdot)^{G}$
on $\Xi(\mathscr X)\times\Xi(\mathscr X^\star)$
which, together with $(\cdot,\cdot)_{\mathscr X^\star}$,
is compatible with restriction.
The uniqueness follows from Corollary \ref{cor:compRM}.
\end{proof}
From this theorem 
one can
deduce the following in the same way as Theorem \ref{thm:sesquilift}:
\begin{cor}
If $\mathscr X\in\FD$
has a non-degenerate invariant Hermitian form $(\cdot,\cdot)^{\mathbf H}$
then there exists a unique non-degenerate invariant Hermitian form $(\cdot,\cdot)^G$
on $\Xi(\mathscr X)$ such that the pair of
$(\cdot,\cdot)^G$ and $(\cdot,\cdot)^{\mathbf H}$ is compatible with restriction.
\end{cor}

\section{Examples for $G=SL(2,\RR)$}\label{sec:SL2R}
In this section we assume
\[
G=SL(2,\RR),\quad
K=\left\{\begin{pmatrix}
\cos\varphi & -\sin\varphi\\
\sin\varphi & \cos\varphi
\end{pmatrix};\,\varphi\in\RR
\right\},\quad
\Lies=\left\{\begin{pmatrix}
a & b\\
b & -a
\end{pmatrix};\,a,b\in\RR
\right\}
\]
and put
\[
e_0=\begin{pmatrix}
0 & i \\ -i & 0
\end{pmatrix},\quad
e_\pm=\frac12\begin{pmatrix}
1 & \mp i \\ \mp i & -1\end{pmatrix}.
\]
Then the $K$-module $(\Ad,\Lies_\CC)$
has a unique irreducible decomposition 
\[
\Lies_\CC=\Lies_+\oplus\Lies_-
\quad\text{with}\quad
\Lies_\pm=\CC e_\pm.
\]
If we identify $\widehat K$ with $\ZZ$
by 
\[
\left(K\ni \begin{pmatrix}
\cos\varphi & -\sin\varphi\\
\sin\varphi & \cos\varphi
\end{pmatrix}
\mapsto e^{in\varphi}\in\CC^\times \right)
\longleftrightarrow
n
\]
then $\Km=2\ZZ$ and $\Lies_\pm\leftrightarrow\pm2$.
Furthermore, since $G$ has real rank $1$,
it follows from \cite[Corollary 2.9]{Oda:HC}
that $\Kqsp=\Ksp=\{\CC_\triv,\Lies_+,\Lies_-\}=\{0,\pm2\}$.
Let $\alpha\in\Liea^*$ be such that $\Sigma^+=\{\alpha\}$
(and hence $R^+=\{2\alpha\}$).

The classification of the irreducible $(\Lieg_\CC,K)$-modules
is classical (cf.~\cite{Kn}).
Here we fix some notation.
For $n=1,2,\ldots$ we put
\begin{align*}
E_G^n&=\text{the irreducible representation with dimension }n,\\
D_G^{n,+}&=\text{the discrete series representation with $K$-types }
\{n+1,n+3,\ldots\},\\
D_G^{n,-}&=\text{the discrete series representation with $K$-types }
\{-n-1,-n-3,\ldots\}.
\end{align*}
If all the $K$-types of a given irreducible $(\Lieg_\CC,K)$-module
belong to $\Km$, then
this module is equivalent to exactly one of the following:
\[
\left\{
\begin{aligned}
&B_G(\lambda)_\Kf\simeq B_G(-\lambda)_\Kf \text{ with }\lambda(\alpha^\vee)\notin
\{\pm1,\pm3,\ldots\},\\
&E_G^n \text{ with }n=1,3,\ldots,\\
&D_G^{n,+} \text{ with }n=1,3,\ldots,\\
&D_G^{n,-} \text{ with }n=1,3,\ldots.
\end{aligned}\right.
\]
It is also well known that for $n=1,3,\ldots$, $B_G(\pm n\rho)_\Kf$ are indecomposable and
\begin{equation}\label{SL2RGM0}
\begin{aligned}
&E_G^n\subset B_G(n\rho)_\Kf,
&B_G(n\rho)_\Kf/E_G^n\simeq D_G^{n,+}\oplus D_G^{n,-},\\
&D_G^{n,+}\oplus D_G^{n,-}\subset B_G(-n\rho)_\Kf,
&B_G(-n\rho)_\Kf/(D_G^{n,+}\oplus D_G^{n,-})\simeq E_G^n.
\end{aligned}
\end{equation}
Hence from the bijectivity condition of $\mathcal P_G^{\lambda}$ 
stated in Remark \ref{rem:PoissonBij} (ii) we have
\begin{equation}\label{SL2RGM1}
X_G(\lambda)\simeq
\begin{cases}
B_G(\lambda)_\Kf\simeq\mathscr A(G/K,\lambda)_\Kf
& \text{if }\lambda(\alpha^\vee)\notin\{\pm1,\pm3,\ldots\},\\
E_G^{|\lambda(\alpha^\vee)|}
& \text{if }\lambda(\alpha^\vee)\in\{\pm1,\pm3,\ldots\}.
\end{cases}
\end{equation}
In addition,
from Proposition \ref{prop:sesquiGinv}
and Proposition \ref{prop:modstrAGl} (ii) we have
\begin{equation}\label{SL2RGM2}
P_G(\CC_\triv,\lambda)\simeq B_G(-\lambda)_\Kf
\quad
\text{for }\lambda(\alpha^\vee)\notin\{-1,-3,\ldots\}.
\end{equation}

Now let us look over the $\mathbf H$ side.
First $W=\{1,s_\alpha\}$ has only two irreducible modules,
namely $\CC_\triv=\CC v_\triv$ and $\CC_\sign=\CC v_\sign$.
If we define the $S(\Liea_\CC)$-action on them by
\[
\xi v_\triv=-\rho(\xi) v_\triv,\quad
\xi v_\sign=\rho(\xi) v_\sign\quad
\text{for }\xi\in\Liea_\CC,
\]
then by \eqref{eq:Hrel}
they become one-dimensional $\mathbf H$-modules,
which are respectively called the \emph{trivial} module
and the \emph{Steinberg} module.
We denote them by $E_{\mathbf H}^1$ and $D_{\mathbf H}^1$.
\begin{prop}\label{prop:HIrrep}
{\normalfont (i)}
An irreducible $\mathbf H$-module is equivalent to exactly one of the following:
\begin{equation}
\left\{
\begin{aligned}\label{eq:sl2Hlist}
&B_{\mathbf H}(\lambda)\simeq B_{\mathbf H}(-\lambda)
\text{ with }\lambda\ne \pm\rho,\\
&E_{\mathbf H}^1,\\
&D_{\mathbf H}^1.
\end{aligned}
\right.
\end{equation}

\noindent
{\normalfont (ii)}
$B_{\mathbf H}(\pm\rho)$ are indecomposable and
\[
\begin{aligned}
&E_{\mathbf H}^1\subset B_{\mathbf H}(\rho),
&B_{\mathbf H}(\rho)/E_{\mathbf H}^1\simeq D_{\mathbf H}^1,\\
&D_{\mathbf H}^1\subset B_{\mathbf H}(-\rho),
&B_{\mathbf H}(-\rho)/D_{\mathbf H}^1\simeq E_{\mathbf H}^1.
\end{aligned}
\]

\noindent
{\normalfont (iii)}
\[
X_{\mathbf H}(\lambda)\simeq
\begin{cases}
B_{\mathbf H}(\lambda)\simeq\mathscr A(A,\lambda)
& \text{if }\lambda\ne\pm\rho,\\
E_{\mathbf H}^1 & \text{if }\lambda=\pm\rho.
\end{cases}
\]

\noindent
{\normalfont (iv)}
For $\lambda\ne-\rho$
\[
P_{\mathbf H}(\CC_\triv,\lambda)\simeq B_{\mathbf H}(-\lambda).
\]

\end{prop}
\begin{proof}
The bijectivity condition for $\mathcal P_{\mathbf H}^\lambda$ in 
Proposition \ref{prop:PoiBij}
reduces to
\[
\lambda\ne -\rho
\]
in the current case.
Hence (iv) follows from Proposition \ref{prop:sesquiHinv}
and Theorem \ref{thm:modstr} (ii).

Since the irreducibility condition \eqref{eq:BHirr} for $B_{\mathbf H}(\lambda)$ reduces to $\lambda \ne \pm\rho$,
we have
\begin{equation}\label{eq:sl2H1}
B_{\mathbf H}(\lambda)\simeq \mathscr A(A,\lambda)=X_{\mathbf H}(\lambda)
=\mathscr A(A,-\lambda)\simeq B_{\mathbf H}(-\lambda)
\quad\text{ for }\lambda \ne \pm\rho.
\end{equation}
Observe  $B_{\mathbf H}(\pm\lambda)$ has central character
$[\lambda]=[-\lambda]$.
Hence $B_{\mathbf H}(\lambda)\not\simeq B_{\mathbf H}(\mu)$
if $\lambda\ne\pm\mu$.
Since $B_{\mathbf H}(\lambda)\simeq \CC W$ as a $W$-module,
one sees \eqref{eq:sl2Hlist} is a list of inequivalent irreducible $\mathbf H$-modules.
In order to show this list is complete,
suppose $\mathscr X$ is any irreducible $\mathbf H$-module.
Then $\mathscr X$ has a central character,
say $[\lambda]$ $(\lambda\in\Liea_\CC^*)$.
Furthermore $\mathscr X$ must contain an irreducible $W$-submodule $F$
which is equivalent to $\CC_\triv$ or $\CC_\sign$.
First, we assume $F\simeq \CC_\triv$.
Then there exists a surjective $\mathbf H$-homomorphism
\[
P_{\mathbf H}(\CC_\triv,\lambda)=P_{\mathbf H}(\CC_\triv,-\lambda)\twoheadrightarrow \mathscr X.
\]
If $\lambda\ne\pm\rho$ then by (iv)
we have $\mathscr X \simeq B_{\mathbf H}(-\lambda)$,
which is listed in \eqref{eq:sl2Hlist}.
If $\lambda=\pm\rho$ then by (iv) again
$\mathscr X$ is a quotient of $B_{\mathbf H}(-\rho)$.
On the other hand, 
since
$\mathcal P_{\mathbf H}^{-\rho} \mathbf 1_{\mathbf H}^{-\rho}=\gamma(\phi_{-\rho})
=\gamma(\phi_{\rho})\in X_{\mathbf H}(\rho)$,
we have the exact sequence
\begin{equation}\label{eq:sl2H2}
0\to \Ker \mathcal P_{\mathbf H}^{-\rho}
\to B_{\mathbf H}(-\rho) \to
X_{\mathbf H}(\rho) \to 0
\end{equation}
where $\Ker \mathcal P_{\mathbf H}^{-\rho}\simeq \CC_\sign$
and $X_{\mathbf H}(\rho)\simeq \CC_\triv$ as $W$-modules.
This implies $\mathscr X\simeq X_{\mathbf H}(\rho)$.
By considering the special case of $\mathscr X=E_{\mathbf H}^1$
we get
\begin{equation}\label{eq:sl2H3}
E_{\mathbf H}^1\simeq X_{\mathbf H}(\rho).
\end{equation}
Secondly, assume $F\simeq\CC_\sign$.
Then there exists a surjective $\mathbf H$-homomorphism
\[
P_{\mathbf H}(\CC_\sign,\lambda)=P_{\mathbf H}(\CC_\sign,-\lambda)\twoheadrightarrow \mathscr X.
\]
If $\lambda\ne\pm\rho$ then 
from \eqref{eq:sl2H1} and Theorem \ref{thm:modstr} (v)
we have $B_{\mathbf H}(\lambda)\simeq \mathscr X$.
If $\lambda=\pm\rho$ then
$\mathscr X$ is a quotient of $\mathscr A(A,\rho)$
by Theorem \ref{thm:modstr} (v).
Since $\mathscr A(A,\rho)$ contains $X_{\mathbf H}(\rho)$ as
a unique irreducible subspace by Theorem \ref{thm:modstr} (iii) 
and $\dim X_{\mathbf H}(\rho)=1$,
we have $\mathscr X\simeq \mathscr A(A,\rho)/X_{\mathbf H}(\rho)$.
By considering two special cases of $\mathscr X=D_{\mathbf H}^1$
and $\mathscr X=\Ker \mathcal P_{\mathbf H}^{-\rho}$
we get
\begin{equation}\label{eq:sl2H4}
D_{\mathbf H}^1\simeq \Ker \mathcal P_{\mathbf H}^{-\rho}
\simeq \mathscr A(A,\rho)/X_{\mathbf H}(\rho).
\end{equation}
Hence in either case $\mathscr X$ is equivalent to one of \eqref{eq:sl2Hlist}.

Now $B_{\mathbf H}(\rho)\simeq\mathscr A(A,\rho)$ is clearly indecomposable and
we have from \eqref{eq:sl2H3} and \eqref{eq:sl2H4} the exact sequence
\[
0\to E_{\mathbf H}^1 \to B_{\mathbf H}(\rho) \to  D_{\mathbf H}^1 \to 0.
\]
Since $B_{\mathbf H}(-\rho)\simeq P_{\mathbf H}(\CC_\triv,\rho)$
is generated by the unique one-dimensional $W$-invariant subspace,
$B_{\mathbf H}(-\rho)$
is also indecomposable.
In addition, from \eqref{eq:sl2H2}, \eqref{eq:sl2H3}
and \eqref{eq:sl2H4}  we have the exact sequence
\[
0\to D_{\mathbf H}^1 \to B_{\mathbf H}(-\rho) \to  E_{\mathbf H}^1 \to 0.
\]
Thus (ii) is proved.

Finally
(iii) follows from \eqref{eq:sl2H1} and \eqref{eq:sl2H3}.
\end{proof}

Since $G=SL(2,\RR)$ is split,
$\Lieb=\{0\}$ and $\rho_\Liem=0$ in Theorem \ref{thm:charcorr}.
Let $(\Lieg_\CC,K)\Mod^\fl_{M}$ be the full subcategory of
$(\Lieg_\CC,K)\Mod^\fl$ which consists of the 
$(\Lieg_\CC,K)$-modules of finite length
having all $K$-types in $\Km$.
For $\lambda\in\Liea_\CC^*$
let $(\Lieg_\CC,K)\Mod^\fl_{M,[\lambda]}$ denote
the full subcategory of $(\Lieg_\CC,K)\Mod^\fl_M$ consisting of
the objects with generalized infinitesimal character $[\lambda]$
and
let $\mathbf H\Mod^\fd_{[\lambda]}$
denote the full subcategory of $\FD$ consisting of
the finite-dimensional $\mathbf H$-modules with generalized central character $[\lambda]$.
Then we have
\[
(\Lieg_\CC,K)\Mod^\fl_{M}
=\bigoplus_{[\lambda]} (\Lieg_\CC,K)\Mod^\fl_{M,[\lambda]},
\qquad
\FD=\bigoplus_{[\lambda]} \mathbf H\Mod^\fd_{[\lambda]}
\]
and
it follows from Theorems \ref{thm:charcorr} and \ref{thm:FD2HC}
that for each $[\lambda]$ three functors $\Xirad$, $\Ximin$ and $\Xi$
send an object in $\mathbf H\Mod^\fd_{[\lambda]}$
into $(\Lieg_\CC,K)\Mod^\fl_{M,[\lambda]}$.

\begin{thm}\label{thm:SL2Rfunctors}
{\normalfont (i)}
On $\FD$
two functors $\Xirad$ and $\Ximin$ are exact 
and coincide with each other.
For $\lambda$ with $\lambda(\alpha^\vee)\notin\{\pm3,\pm5,\ldots\}$,
three functors $\Xirad$, $\Ximin$ and $\Xi$ coincide
on $\mathbf H\Mod^\fd_{[\lambda]}$ (and hence are exact there).

\noindent
{\normalfont (ii)}
\begin{align*}
\Xirad(B_{\mathbf H}(\lambda))=\Ximin(B_{\mathbf H}(\lambda))=\Xi(B_{\mathbf H}(\lambda))&=B_{G}(\lambda)_\Kf\quad
\text{if }\lambda(\alpha^\vee)\notin\{\pm3,\pm5,\ldots\},\\
\Xirad(E_{\mathbf H}^1)=\Ximin(E_{\mathbf H}^1)=\Xi(E_{\mathbf H}^1)&=E_{G}^1,\\
\Xirad(D_{\mathbf H}^1)=\Ximin(D_{\mathbf H}^1)=\Xi(D_{\mathbf H}^1)
&=D_{G}^{1,+}\oplus D_{G}^{1,-}.
\end{align*}

\noindent
{\normalfont (iii)}
For $n=3,5,\ldots$
\begin{align*}
\Xirad(B_{\mathbf H}(-n\rho))=\Ximin(B_{\mathbf H}(-n\rho))&=B_G(-n\rho)_\Kf,\\
\Xi(B_{\mathbf H}(-n\rho))&=E_G^n.
\end{align*}
(Note $B_{\mathbf H}(n\rho)=B_{\mathbf H}(-n\rho)$ in this case.)

\noindent
{\normalfont (iv)}
For $n=3,5,\ldots$
the functor $\Xi$ is not exact on $\mathbf H\Mod^\fd_{[n\rho]}$.
\end{thm}
\begin{proof}

Suppose $\lambda(\alpha^\vee)\notin\{\pm3,\pm5,\ldots\}$.
Then each irreducible object in $(\Lieg_\CC,K)\Mod^\fl_{M,[\lambda]}$
contains at least one single-petaled $K$-type
and conversely any single-petaled $K$-type is
contained in exactly one irreducible object with multiplicity $1$:
\begin{align*}
\lambda(\alpha^\vee)\notin\{\pm1,\pm3,\ldots\} &:\ 
B_G(\lambda)_\Kf \supset \CC_\triv,\, \Lies_\pm,\\
\lambda=\pm\rho &:\ 
E_G \supset \CC_\triv,\, D_G^{1,+}\supset \Lies_+,\, D_G^{1,-}\supset \Lies_-. 
\end{align*}
Now for $\mathscr X\in \mathbf H\Mod^\fd_{[\lambda]}$ we have
consecutive surjective $(\Lieg_\CC,K)$-homomorphisms
\begin{equation}\label{eq:cons2surj}
\Xirad(\mathscr X)
\twoheadrightarrow
\Ximin(\mathscr X)
\twoheadrightarrow
\Xi(\mathscr X).
\end{equation}
They are in fact natural transforms.
We assert 
if $\mathscr I\in(\Lieg_\CC,K)\Mod^\fl_{M,[\lambda]}$ is an irreducible object
then it appears with the same multiplicity
in the three composition series for 
$\Xirad(\mathscr X)$, $\Ximin(\mathscr X)$ and $\Xi(\mathscr X)$.
Indeed, if $\mathscr I$ contains $V\in \Ksp$ then we have from \eqref{eq:genspcor}
\[
\Hom_K(V,\Xirad(\mathscr X))
\simeq
\Hom_K(V,\Ximin(\mathscr X))
\simeq
\Hom_K(V,\Xi(\mathscr X))
\simeq
\Hom_W(V^M,\mathscr X).
\]
Hence the multiplicity of $\mathscr I$ in each series
equals $\dim \Hom_W(V^M,\mathscr X)$.
Thus two homomorphisms in
\eqref{eq:cons2surj} are bijective
and $\Xirad=\Ximin=\Xi$ on $\mathbf H\Mod^\fd_{[\lambda]}$.
The exactness of these functors follows from Proposition \ref{prop:tensorRad}
and Corollary \ref{cor:partialexact}.

If $\lambda(\alpha^\vee)\notin\{1,3,\ldots\}$
then 
by Proposition \ref{prop:HIrrep} (iv), \eqref{eq:XiPCl} and \eqref{SL2RGM2}
\begin{equation}\label{eq:BPPB}
\Ximin(B_{\mathbf H}(\lambda))=\Ximin(P_{\mathbf H}(\CC_\triv,-\lambda))=
P_G(\CC_\triv,-\lambda)=B_{G}(\lambda)_\Kf.
\end{equation}
Using this result for $\lambda=-\rho$,
Proposition \ref{prop:sesquiHinv},
Theorem \ref{thm:Xiindform} and
Proposition \ref{prop:sesquiGinv},
we obtain
\[
\Xi(B_{\mathbf H}(\rho))=\Xi(B_{\mathbf H}(-\rho)^\star)
=\Xi(B_{\mathbf H}(-\rho))^\star=(B_{G}(-\rho)_\Kf)^\star=B_{G}(\rho)_\Kf.
\]
Thus the first assertion of (ii) is proved.

Now it follows from Proposition \ref{prop:HIrrep} (iii), Corollary \ref{cor:Xi} (iii)
and \eqref{SL2RGM1} that
\begin{align*}
\Xi(E_{\mathbf H}^1)&=\Xi(X_{\mathbf H}(\rho))=X_{G}(\rho)=E_{G}^1,\\
\Xi(B_{\mathbf H}(-n\rho))&=\Xi(X_{\mathbf H}(-n\rho))=X_{G}(-n\rho)=E_{G}^n
\quad\text{for }n=3,5,\ldots.
\end{align*}
This proves the second assertion of (ii) and the second assertion of (iii).
The third assertion of (ii) follows from \eqref{SL2RGM0}, Proposition \ref{prop:HIrrep} (ii)
and the exactness of the three functors on $\mathbf H\Mod^\fd_{[\rho]}$.
Because of \eqref{eq:BPPB}
the first assertion of (iii) follows if we can show the coincidence of $\Xirad$ and $\Ximin$
on $\mathbf H\Mod^\fd_{[n\rho]}$ for $n=3,5,\ldots$.

Suppose $n=3,5,\ldots$. Let
$C=(e_0^2+2e_+e_-+2e_-e_+)/8\in U(\Lieg_\CC)^G$ be the Casimir element.
Then $\gamma(C)=((\alpha^\vee)^2-1)/8$ and $\CC[\gamma(C)]=S(\Liea_\CC)^W$.
Hence
$\CC[\gamma(C)]\otimes \CC W$ is a subalgebra of $\mathbf H=S(\Liea_\CC)\otimes \CC W=\CC[\alpha^\vee]\otimes\CC W$.
Now, for any $\mathscr X\in \mathbf H\Mod^\fd_{[n\rho]}$
the subspace $\mathscr X^W$ of $W$-fixed elements is stable under
the action of $W$ and $\gamma(C)$.
Hence an $\mathbf H$-homomorphism
\begin{equation}\label{anotherXmin}
\mathbf H \otimes_{\CC[\gamma(C)]\otimes\CC W} \mathscr X^W
\longrightarrow
\mathscr X.
\end{equation}
is naturally defined.
From the decomposition
\[
\mathbf H=\bigl(\CC \oplus \CC(\alpha^\vee+1)\bigr) \otimes \CC[\gamma(C)]\otimes \CC W
\]
we have
\begin{equation}\label{eq:HCXW}
\mathbf H \otimes_{\CC[\gamma(C)]\otimes\CC W} \mathscr X^W = 
1\otimes \mathscr X^W\,
\oplus\, (\alpha^\vee+1)\otimes \mathscr X^W.
\end{equation}
It is easy to check by \eqref{eq:Hrel}
this is the decomposition into
the $\CC_\triv$- and $\CC_\sign$-isotypic components as a $W$-module.
Note also that $\mathbf H \otimes_{\CC[\gamma(C)]\otimes\CC W} \mathscr X^W\in \mathbf H\Mod^\fd_{[n\rho]}$.
We assert \eqref{anotherXmin} is bijective.
In fact, 
since the first summand of \eqref{eq:HCXW}
is bijectively mapped to $\mathscr X^W$ by \eqref{anotherXmin},
the kernel and the cokernel of \eqref{anotherXmin}
do not have any non-zero $W$-fixed vector.
But since
the unique irreducible object $B_{\mathbf H}(-n\rho)$ in $\mathbf H\Mod^\fd_{[n\rho]}$
has a non-zero $W$-fixed vector,
both the kernel and the cokernel must be zero.
Now by \cite{KR} the following decomposition holds:
\begin{equation}\label{eq:Usl2rdcp}
U(\Lieg_\CC)=
\bigl(\CC\oplus \bigoplus_{\nu\ge1}(\CC e_+^\nu \oplus \CC e_-^\nu)\bigr)\otimes
\CC[C]\otimes U(\Liek_\CC).
\end{equation}
In particular, $\CC[C]\otimes U(\Liek_\CC)$ is a subalgebra of $U(\Lieg_\CC)$.
Regarding $\mathscr X^W$ as a $(\CC[C]\otimes U(\Liek_\CC), K)$-module
by the trivial $K$-action and the $\CC[C]$-action given by
$Cx=\gamma(C)x$,
we define the induced $(\Lieg_\CC,K)$-modules
\[
\Tilde{\mathscr Y}:=U(\Lieg_\CC)\otimes_{U(\Liek_\CC)} \mathscr X^W,
\qquad
\mathscr Y:=U(\Lieg_\CC)\otimes_{\CC[C]\otimes U(\Liek_\CC)} \mathscr X^W.
\]
If we put
\[
\kappa: \Tilde{\mathscr Y} \ni D\otimes x
\longmapsto DC\otimes x-D\otimes \gamma(C)x \in \Tilde{\mathscr Y}
\]
then we have the exact sequence
\[
\Tilde{\mathscr Y} \xrightarrow{\kappa} \Tilde{\mathscr Y}\to\mathscr Y \to 0.
\]
Applying the right exact functor $\Gamma$ in \S\ref{sec:HC} to this,
we get
\begin{align*}
\Gamma(\Tilde{\mathscr Y})
 &\xrightarrow{\Gamma(\kappa)} \Gamma(\Tilde{\mathscr Y})\to
 \Gamma(\mathscr Y) \to 0,\quad\text{(exact)}\\
\Gamma(\Tilde{\mathscr Y})
&=\mathbf H\otimes_{\CC W} \mathscr X^W,\\
\Gamma(\kappa): \mathbf H\otimes_{\CC W} \mathscr X^W \ni h\otimes x
&\longmapsto h\gamma(C)\otimes x-h\otimes \gamma(C)x
\in \mathbf H\otimes_{\CC W} \mathscr X^W,\\
\therefore\quad\Gamma(\mathscr Y)&=\mathbf H \otimes_{\CC[\gamma(C)]\otimes\CC W} \mathscr X^W
=\mathscr X.
\end{align*}
Note the linear map $\gamma^{\mathscr Y}: \mathscr Y\to\mathscr X$ reduces to
$D\otimes x\mapsto \gamma(D)\otimes x$.
Let us prove $(\mathscr Y,\mathscr X)$ is a radial pair 
and $\gamma^{\mathscr Y}$ is its radial restriction
satisfying \eqref{cond:rest3}.
First, by \ref{eq:Usl2rdcp} we can
decompose $\mathscr Y$ into
the $K$-isotypic components:
\[
\mathscr Y =
(1\otimes \mathscr X^W)
\oplus
(e_+ \otimes \mathscr X^W)
\oplus
(e_- \otimes \mathscr X^W)
\oplus \bigoplus_{\nu\ge2}
\bigl((e_+^\nu\otimes \mathscr X^W)
\oplus(e_-^\nu\otimes \mathscr X^W)\bigr).
\]
Thus all $K$-types of $\mathscr Y$ belong to $\Km$.
The first three summands in the decomposition
correspond to $\CC_\triv$, $\Lies_+$ and $\Lies_-$,
respectively.
Take $m\in\ZZ_{\ge0}$ so that $(\gamma(C)-\gamma(C)(n\rho))^m \mathscr X^W=\{0\}$.
Suppose $\nu\in\ZZ_{\ge0}$.
Then one can directly calculate
\[
\gamma(e_\pm^\nu)=\frac1{2^\nu}(\alpha^\vee+1)(\alpha^\vee+3)\cdots(\alpha^\vee+2\nu-1).
\]
Hence there exists some $f_\nu\in S(\Liea_\CC)=\CC[\alpha^\vee]$ such that
\[
f_\nu \gamma(e_\pm^\nu) -(\alpha^\vee+n) \in S(\Liea_\CC)(\gamma(C)-\gamma(C)(n\rho))^m.
\]
The restriction of $\gamma^{\mathscr Y}$ to 
each $K$-isotypic component $e_\pm^\nu\otimes \mathscr X^W$ is injective
since the composition of this map with multiplication by $f_\nu$ reduces to 
\[\xymatrix{
e_\pm^\nu\otimes \mathscr X^W \ni 
e_\pm^\nu\otimes x \ar @{|->} ^-{\gamma^{\mathscr Y}} [r]
&\gamma(e_\pm^\nu) \otimes x
\ar @{|->} ^-{f_\nu\cdot} [r]
&f_\nu \gamma(e_\pm^\nu) \otimes x=(\alpha^\vee+n) \otimes x
\in \mathscr X.
}\]
If $\nu=0,\pm1$ we also have
\[
\gamma^{\mathscr Y}(1\otimes \mathscr X^W)=1\otimes \mathscr X^W,
\quad
\gamma^{\mathscr Y}(e_+ \otimes \mathscr X^W)=
\gamma^{\mathscr Y}(e_- \otimes \mathscr X^W)=
(\alpha^\vee+1)\otimes \mathscr X^W.
\]
Thus $\gamma^{\mathscr Y}$ is a radial restriction defining
a structure of $(\mathscr Y,\mathscr X)$
as an object of $\CCh$.
Since $\gamma^{\mathscr Y}$ satisfies \eqref{cond:rest3} by Remark \ref{eq:goodRR},
Lemma \ref{lem:rest3} implies $(\mathscr Y,\mathscr X)\in\Crad$.
Now $\Ximin_{((\mathscr Y,\mathscr X))}(\mathscr X)=\mathscr Y$
since $\mathscr Y$ is generated by $1\otimes \mathscr X^W$.
Hence by Theorem \ref{thm:Ximin} we have
$\Ximin(\mathscr X)=\mathscr Y$ and
the functor $\Ximin$ restricted to
$\mathbf H\Mod^\fd_{[n\rho]}$ coincides with the functor
\[
\mathscr X \longmapsto
U(\Lieg_\CC)\otimes_{\CC[C]\otimes U(\Liek_\CC)} \mathscr X^W.
\]
This is exact by \eqref{eq:Usl2rdcp}.
We must prove this is also equal to $\Xirad$.
To do so recall the linear map 
$\gamma_\rad : \Xirad(\mathscr X)\to\mathscr X$ used in \S\ref{sec:Xirad}
satisfies Conditions
\eqref{cond:rest2} and \eqref{cond:rest3}.
Thus $\gamma_\rad$ induces the linear bijection
\[
\gamma_\rad : \Xirad(\mathscr X)^K \simarrow \mathscr X^W
\]
by \eqref{cond:rest2} and satisfies
\[
\gamma_\rad(Cy)=\gamma(C)\gamma_\rad(y) \quad\text{for }y\in \Xirad(\mathscr X)^K
\]
by \eqref{cond:rest3}.
Hence we can define a $(\Lieg_\CC,K)$-homomorphism
$\mathcal I: U(\Lieg_\CC)\otimes_{\CC[C]\otimes U(\Liek_\CC)} \mathscr X^W \to \Xirad(\mathscr X)$ by
\[
D\otimes x \longmapsto D y\text{ with }y\in \Xirad(\mathscr X)^K
\text{ such that }\gamma_\rad(y)=x.
\]
If $\mathscr I\in (\Lieg_\CC,K)\Mod^\fl_{M,[n\rho]}$ is
a composition factor of $\Coker\mathcal I$,
then it does not contain the trivial $K$-type.
This means $\mathscr I=D_G^{n,+}$ or $D_G^{n,-}$.
But since each of $D_G^{n,\pm}$ contains neither $\Lies_+$ nor $\Lies_-$
and since $\Xirad(\mathscr X)$ is generated by the sum of
those isotypic components for $\CC_\triv$, $\Lies_+$ and $\Lies_-$,
we conclude $\Coker\mathcal I=\{0\}$ and $\mathcal I$ is surjective.
Hence the natural surjective homomorphism $\Xirad(\mathscr X)\to\Ximin(\mathscr X)$
must be a bijection,
proving $\Xirad=\Ximin$ on $\mathbf H\Mod^\fd_{[n\rho]}$.

Finally in order to prove (iv), suppose $n=3,5,\ldots$ as in the last paragraph
and put
\[
\mathscr U_j=\CC[\gamma(C)]\bigm/
\CC[\gamma(C)](\gamma(C)-\gamma(C)(n\rho))^j
\quad(j=1,2).
\]
This is a $\CC[\gamma(C)]\otimes \CC W$-module
by the trivial $W$-action and is also
a $(\CC[C]\otimes U(\Liek_\CC),K)$-module
by the trivial $K$-action and the $\CC[C]$-action given by $Cu=\gamma(C)u$.
Put
\[
\mathscr X_j=\mathbf H \otimes_{\CC[\gamma(C)]\otimes\CC W} \mathscr U_j,
\quad
\mathscr Y_j=U(\Lieg_\CC) \otimes_{\CC[C]\otimes U(\Liek_\CC)} \mathscr U_j
\quad(j=1,2).
\]
Then by the above argument we have
$\Xirad(\mathscr X_j)=\mathscr Y_j$
and the exact sequence
\[
0\longrightarrow \mathscr U_1 \xrightarrow{\text{multiplication by }\gamma(C)-\gamma(C)(n\rho)}
\mathscr U_2 \xrightarrow{\text{quotient map}}
\mathscr U_1\longrightarrow 0
\]
induces exact sequences
\[
0\to \mathscr X_1 \to \mathscr X_2 \to \mathscr X_1 \to 0,\quad
0\to \mathscr Y_1 \to \mathscr Y_2 \to \mathscr Y_1 \to 0.
\]
We shall prove
\begin{equation}\label{eq:nonexact}
0\to \Xi(\mathscr X_1) \to \Xi(\mathscr X_2) \to \Xi(\mathscr X_1) \to 0
\end{equation}
is not exact.
Since $\dim\mathscr X_1=2$, we see
$\mathscr X_1=B_{\mathbf H}(-n\rho)$ by Proposition \ref{prop:HIrrep} (i).
Hence $\Xi(\mathscr X_1)=E^n_G$ by (iii) of the theorem.
Since $E^n_G$ does not contain the $K$-type corresponding to $n+1$,
it suffices to show $\Xi(\mathscr X_2)$ has this $K$-type.
By definition we have
$\Xi(\mathscr X_2) = \mathscr Y_2/\mathscr N$ with
\[
\mathscr N = \sum\bigl\{\mathscr Y\subset \mathscr Y_2;\,
\text{a }(\Lieg_\CC,K)\text{-submodule containing no single-petaled $K$-type}
\bigr\}.
\]
We assert $e_+^{\frac{n+1}2}\otimes 1\in\mathscr Y_2$ does not belong to $\mathscr N$.
Indeed, if it does, then since
\[
e_-e_+^{\frac{n+1}2}=e_+^{\frac{n-1}2}\biggl(
2(C-\gamma(C)(n\rho))-\frac{n}2e_0-\frac14e_0^2
\biggr)
\quad\text{in }U(\Lieg_\CC),
\]
$\mathscr N$ must contain
\[
\CC e_-\bigl(e_+^{\frac{n+1}2}\otimes 1\bigr)
=
\CC e_+^{\frac{n-1}2}\otimes 2(\gamma(C)-\gamma(C)(n\rho)),
\]
a $K$-type corresponding to $n-1$;
Hence in the composition series of $\mathscr N$
there appears the irreducible object $E^n_G$, whose $K$-types are
\[
n-1,n-3,\ldots,2,0,-2,\ldots,-n+3,-n+1.
\]
This contradicts the definition of $\mathscr N$.
Thus $\Xi(\mathscr X_2)$ contains
\[
\CC \bigl(e_+^{\frac{n+1}2}\otimes 1\bmod \mathscr N\bigr),
\]
a $K$-type corresponding to $n+1$.
Hence \eqref{eq:nonexact} is not exact.
\end{proof}

\appendix

\section{Non-symmetric hypergeometric functions}\label{apndx:NSHG}
Let $\mathbf k : W\backslash R\to \CC$ and $\mathscr T_{\mathbf k}$
be as in Definition \ref{defn:Chered}.
Let $\lambda\in\Liea_\CC^*$.
Opdam's non-symmetric hypergeometric function $\mathbf G(\lambda,\mathbf k,a)$
is an analytic function on $A$ satisfying
\begin{equation*}
\left\{\begin{aligned}
&\mathscr T_{\mathbf k}(\xi)\mathbf G(\lambda,\mathbf k,a)=\lambda(\xi)\mathbf G(\lambda,\mathbf k,a)\quad\text{for }\xi\in\Liea_\CC,\\
&\mathbf G(\lambda,\mathbf k,1)=1.\end{aligned}\right.
\end{equation*}
Opdam shows in \cite[\S3]{Op:Cherednik} that
there uniquely exists such a function for a generic $\mathbf k$.
This is the case when $\mathbf k=\mathbf m$
and $\mathbf G(\lambda, a):=\mathbf G(\lambda,\mathbf m,a)$
plays central roles in \S\S\ref{sec:modstr}--\ref{sec:F}.
The purpose of the appendix is to prove the following:
\begin{thm}\label{thm:Ganal}
Suppose $\varphi(a)\in C^\infty(A)$ satisfies
\begin{equation}\label{eq:Tkl}
\mathscr T_{\mathbf k}(\xi)\varphi(a)=\lambda(\xi)\varphi(a)\quad\text{for }\xi\in\Liea_\CC.
\end{equation}
Then $\varphi(a)\in\mathscr A(A)$.
\end{thm}
Recall Lemma \ref{lem:Glem} is founded on this result.
Let us start with an elementary lemma.

\begin{lem}\label{lem:fregu}
Suppose $\epsilon>0$ and a $C^\infty$ function $F(x)$ on $(-\epsilon,\epsilon)$
satisfies the following conditions:\smallskip

\noindent{\normalfont (i)}
$F(x)$ is analytic on $(-\epsilon,0)\sqcup(0,\epsilon)$;

\noindent{\normalfont (ii)}
there exist $f_{\lambda,j} \in\mathscr A((-\epsilon,\epsilon))$
with a finite index set $\Lambda\subset \CC\times \ZZ_{\ge0}$
such that
\[
F(x)=\sum_{(\lambda,j)\in\Lambda} x^\lambda (\log x)^j f_{\lambda,j}(x)
\quad\text{on }(0,\epsilon);
\]

\noindent{\normalfont (iii)}
there exist $g_{\lambda,j} \in\mathscr A((-\epsilon,\epsilon))$
with a finite index set $\Lambda'\subset \CC\times \ZZ_{\ge0}$
such that
\[
F(x)=\sum_{(\lambda,j)\in\Lambda'} (-x)^\lambda (\log (-x))^j g_{\lambda,j}(x)
\quad\text{on }(-\epsilon,0)
\]
Then $F(x)$ is analytic on $(-\epsilon,\epsilon)$.
\end{lem}

\begin{proof}
Without loss of generality we may assume in Condition (ii)
\begin{equation*}
f_{\lambda,j}(0)\ne 0
\quad\text{for all }(\lambda,j)\in\Lambda.
\end{equation*}
We then assert $\lambda\in\ZZ_{\ge0}$ for any $(\lambda,j)\in\Lambda$.
To prove this, suppose $\lambda\notin\ZZ_{\ge0}$ for some $(\lambda,j)\in\Lambda$.
Then the derivative $F^{(k)}$ of $F$ with a sufficiently high order $k$
has an expression
\[
F^{(k)}(x)
=\sum_{(\lambda,j)\in\tilde\Lambda} x^\lambda (\log x)^j h_{\lambda,j}(x)
\quad\text{on }(0,\epsilon)
\]
with
$h_{\lambda,j}(0)\ne0$
for all $(\lambda,j)\in\tilde\Lambda$
and $\Real\lambda<0$ for some $(\lambda,j)\in\tilde\Lambda$.
Putting
\begin{align*}
\lambda^\circ&=\min\{\Real\lambda;\,(\lambda,j)\in\tilde\Lambda\},\\
j^\circ&=\max\{j;\,(\lambda,j)\in\tilde\Lambda\text{ with }\Real\lambda=\lambda^\circ\},\\
\tilde\Lambda^\circ&=\{(\lambda,j)\in\tilde\Lambda;\,\Real\lambda=\lambda^\circ,j=j^\circ\},
\end{align*}
we have for $x\in(0,\min\{\epsilon,1\})$
\begin{multline*}
|F^{(k)}(x)|=x^{\lambda^\circ}|\log x|^{j^\circ}
\Biggl|
\sum_{(\lambda,j)\in \tilde\Lambda^\circ} x^{\lambda-\lambda^\circ}h_{\lambda,j}(0)
+
\sum_{(\lambda,j)\in \tilde\Lambda^\circ} x^{\lambda-\lambda^\circ+1}
\frac{h_{\lambda,j}(x)-h_{\lambda,j}(0)}x\\
+\sum_{(\lambda,j)\in \tilde\Lambda\setminus \tilde\Lambda^\circ}
x^{\lambda-\lambda^\circ}(\log x)^{j-j^\circ}h_{\lambda,j}(x)
\Biggr|
\end{multline*}
But since
\begin{gather*}
\limsup_{x\downarrow 0}\Biggl|
\sum_{(\lambda,j)\in \tilde\Lambda^\circ} x^{\lambda-\lambda^\circ}h_{\lambda,j}(0)
\Biggr|\ge
\Biggl(
\sum_{(\lambda,j)\in \tilde\Lambda^\circ} |h_{\lambda,j}(0)|^2
\Biggr)^{1/2},
\quad(\because\text{\cite[Ch.\,I, Exercise D5]{Hel4}})\\
\lim_{x\downarrow 0}\Biggl|
\sum_{(\lambda,j)\in \tilde\Lambda^\circ} x^{\lambda-\lambda^\circ+1}
\frac{h_{\lambda,j}(x)-h_{\lambda,j}(0)}x
\Biggr|=
\lim_{x\downarrow 0}\Biggl|
\sum_{(\lambda,j)\in \tilde\Lambda\setminus \tilde\Lambda^\circ}
x^{\lambda-\lambda^\circ}(\log x)^{j-j^\circ}h_{\lambda,j}(x)
\Biggr|=0,
\end{gather*}
$\limsup_{x\downarrow 0}|F^{(k)}(x)|=\infty$, contradicting the fact
$F(x)\in C^\infty((-\epsilon,\epsilon))$.
Thus our assertion is proved and 
we may assume $F$ has an expression
\[
F(x)=\sum_{j\in J}x^{n_j}(\log x)^jf_j(x)\qquad\text{on }(0,\epsilon)
\]
where
$J\subset \ZZ_{\ge0}$ is a finite set
and for each $j\in J$
\[
n_j\in\ZZ_{\ge0},
\quad
f_j\text{ is analytic on }(-\epsilon,\epsilon)\text{ with }
f_j(0)\ne0.
\]
Now let us assume $J$ contains some $j>0$ and lead a contradiction.
Put
\[
n'=\min\{n_j;\,j\in J\setminus\{0\}\},\qquad
j'=\max\{j;\,j\in J\setminus\{0\}\text{ with }n_j=n'\}.
\]
Then there exist $h_j\in\mathscr A((-\epsilon,\epsilon))$ $(j=0,1,\ldots,\max J)$ such that
\[
F^{(n'+1)}(x)=\sum_{j} x^{-1} (\log x)^j h_j(x),\quad
h_{j'-1}(0)\ne0,\quad
h_{j}(0)=0\text{ for }j\ge j'.
\]
This implies
\[
\lim_{x\downarrow0}\frac{F^{(n'+1)}(x)}{x^{-1} (\log x)^{j'-1}}=h_{j'-1}(0)\ne0,
\]
a contradiction.
Hence $J=\emptyset$ or $\{0\}$
and $F|_{(0,\epsilon)}$ extends to an analytic function $F_1$ on $(-\epsilon,\epsilon)$.
Similarly one can prove $F|_{(-\epsilon,0)}$ extends to an analytic function $F_2$ on $(-\epsilon,\epsilon)$.
Since $F^{(k)}_1(0)=F^{(k)}_2(0)=F^{(k)}(0)$ for any $k$,
we conclude $F_1=F_2=F$.
\end{proof}
Put
\begin{align*}
\Liea_\reg&=\{H\in\Liea;\,\alpha(H)\ne 0\text{ for any }\alpha\in R\},\\
\Liea_+&=\{H\in\Liea;\,\alpha(H)>0\text{ for any }\alpha\in\Pi\},\\
U&=\{H\in\Liea;\,|\alpha(H)|<2\pi \text{ for any }\alpha\in R\},\\
U_+&=U\cap\Liea_+,\\
(\Liea+iU)_\reg&=\{H\in \Liea+iU;\, \alpha(H)\ne 0\text{ for any }\alpha\in R\}.
\end{align*}

\begin{defn}[the Knizhnik-Zamolodchikov connection \cite{Ma}]
Let $E=(\Liea+iU)\times \CC W$ be the trivial vector bundle with fiber $\CC W$
over the complex manifold $\Liea+iU$.
The \emph{Knizhnik-Zamolodchikov connection} $\nabla=\nabla(\lambda,\mathbf k)$
is a connection on $(\Liea+iU)_\reg\times \CC W\subset E$
whose covariant derivative along $\xi\in\Liea_\CC$ is given by
\begin{equation}\label{eq:KZ}
\begin{aligned}
\nabla_\xi\Psi
=\sum_{w\in W}
\biggl(
\biggl(
\partial(\xi)
-
\lambda(w^{-1}\xi)
+
\frac12\sum_{\alpha\in R^+}\mathbf k(\alpha)&\alpha(\xi)
\frac{1+e^{-\alpha}}{1-e^{-\alpha}}
\biggr)\Psi_w\\
-
&\sum_{\alpha \in wR^+}\mathbf k(\alpha)\alpha(\xi)
\frac{e^{-\alpha}}{1-e^{-\alpha}}
\Psi_{s_\alpha w}
\biggr)w
\end{aligned}
\end{equation}
for any section $\Psi:H\mapsto \Psi(H)=\sum_{w\in W}\Psi_w(H)w\in\CC W$.
\end{defn}
\begin{rem}
It is known that the KZ connection is integrable (cf.~\cite[Proposition 3.3.1]{Ma}), although we do not need this fact here.
\end{rem}
Now suppose $\varphi(a)\in C^\infty(A)$ satisfies \eqref{eq:Tkl}
and define a $\CC W$-valued $C^\infty$ function
\[\Phi(H)=\sum_{w\in W}\varphi(\exp (w^{-1}H))w\]
on $\Liea$.
Then it follows from \cite[Lemma 3.2]{Op:Cherednik} that
\begin{equation}\label{eq:KZeq}
\nabla_\xi \Phi(H)=0
\quad\text{for any }\xi\in\Liea\text{ and }H\in\Liea_\reg.
\end{equation}
Hence it suffices to deduce from \eqref{eq:KZeq}
the analyticity of $\Phi$ on $\Liea$.
If $\{\xi_1,\cdots,\xi_\ell\}$ is a basis of $\Liea$, then
the system $\nabla_\xi\Psi=0$ $(\xi\in\Liea)$ 
for a holomorphic section $\Psi$ of $(\Liea+iU)_\reg\times \CC W\subset E$ 
can be written as
\begin{equation}\label{eq:dpsiA}
\partial(\xi_j)\Psi=A_j\Psi\quad(j=1,\ldots,\ell)
\end{equation}
where $A_j: (\Liea+iU)_\reg\to\End_\CC(\CC W)$ are holomorphic functions.
Hence $\Phi$ is analytic on $\Liea_\reg=\bigsqcup_{w\in W}w\Liea_+$ and
$\Phi|_{\Liea_+}$ extends to a global (possibly multi-valued) holomorphic solution
$\tilde\Phi$ of \eqref{eq:dpsiA} on the whole $(\Liea+iU)_\reg$
(cf.~\cite[Appendix B, \S2]{Kn}).

\begin{lem}\label{lem:singval}
The global solution $\tilde\Phi$ is single-valued on $(\Liea+iU)_\reg$
and $\tilde\Phi|_{\Liea_\reg}=\Phi|_{\Liea_\reg}$.
\end{lem}

\begin{proof}
Note
\[
(\Liea+iU)_\reg=\bigsqcup_{w\in W} (w\Liea_++iU)\,\cup\,
\bigsqcup_{t\in W} (\Liea+itU_+)
\]
and for any $w,t\in W$
\[
w\Liea_++iU,\quad
\Liea+itU_+,\quad
(w\Liea_++iU)\cap(\Liea+itU_+)=w\Liea_++itU_+
\]
are all simply connected.
Hence it suffices to prove that
for any $w_1,w_2,t\in W$
\[
\Phi|_{w_1\Liea_+} \xrightarrow{\text{extension}}
\Phi_1\text{ on }w_1\Liea_++iU \xrightarrow{\text{restriction}}
\Phi_1|_{w_1\Liea_++itU_+} \xrightarrow{\text{extension}}
\tilde\Phi_1\text{ on }\Liea+itU_+
\]
and
\[
\Phi|_{w_2\Liea_+} \xrightarrow{\text{extension}}
\Phi_2\text{ on }w_2\Liea_++iU \xrightarrow{\text{restriction}}
\Phi_2|_{w_2\Liea_++itU_+} \xrightarrow{\text{extension}}
\tilde\Phi_2\text{ on }\Liea+itU_+
\]
give the same result.
Clearly we may assume $w_2=w_1 s_\alpha $ for some $\alpha\in \Pi$.
Fix an arbitrary $\xi_0\in \Liea_+$
and put $H_0=w_1\xi_0+w_2\xi_0$.
Then
\[
(w_1\alpha)(H_0)=(w_2\alpha)(H_0)=0,\quad
(w_1\beta)(H_0),(w_2\beta)(H_0)> 0
\text{ for any }\beta\in R_1^+\setminus\{\alpha\}.
\]
Take a sufficiently small $\epsilon>0$
so that
\[
\{H_0+zt\xi_0;\,z\in\CC\text{ with }0<|z|<\epsilon\}\subset (\Liea+iU)_\reg.
\]
Then from \eqref{eq:KZ}
one sees the covariant derivative $\nabla_{t\xi_0}$ 
on $\{H=H_0+zt\xi_0;\,0<|z|<\epsilon\}$
is written as 
\[
\nabla_{t\xi_0}\Psi
=
\Bigl(\frac{d}{dz}-\frac{B(z)}z\Bigr)\Psi
\]
where $B$ is an $\End_\CC(\CC W)$-valued holomorphic function on $\{z\in\CC;\,|z|<\epsilon\}$.
Now let us consider the $\CC W$-valued $C^\infty$ function
\[
\Psi(x)=\sum_{w\in W}\Psi_w(x)w:=\Phi(H_0+xt\xi_0)
\]
on $(-\epsilon,\epsilon)$.
Since $\Psi$ is a solution of the first-order ordinary linear system
\[\frac{d}{dx}\Psi=\frac{B(x)}x\Psi\]
with a regular singular point $x=0$, each $\Psi_w$ $(w\in W)$
satisfies the assumption of Lemma \ref{lem:fregu}.
Hence $\Psi$ is analytic at $x=0$ and extends to a holomorphic function
$\tilde\Psi$ on $\{z\in\CC;\,|z|<\epsilon\}$.
Now suppose $(w_1\alpha)(t\xi_0)>0$.
(If $(w_1\alpha)(t\xi_0)<0$ we swap $w_1$ and $w_2$.)
Identifying $\{z\in\CC;\,|z|<\epsilon\}$ with a subset of $\Liea+iU$
by $H=H_0+zt\xi_0$
we have
\begin{align*}
D_+:=\{z\in\CC;\,|z|<\epsilon,\Imaginary z>0\}&\subset \Liea+itU_+,\qquad\\
\{z\in\CC;\,|z|<\epsilon,\Real z>0\}&\subset w_1\Liea_++iU,\qquad\\
\{z\in\CC;\,|z|<\epsilon,\Real z<0\}&\subset w_2\Liea_++iU.
\end{align*}
Thus $\tilde\Phi_1|_{D_+}=\tilde\Phi_2|_{D_+}=\tilde\Psi|_{D_+}$.
Since a solution of \eqref{eq:dpsiA} on $\Liea+itU_+$
is determined by the value at any one point,
we conclude $\tilde\Phi_1=\tilde\Phi_2$.
\end{proof}

\begin{lem}\label{lem:rnvsing}
Suppose $\alpha \in R_1^+$ and $H_1\in \Liea+iU$ satisfy
\[
\alpha(H_1)=0,\quad
\beta(H_1)\ne 0\text{ for any }\beta\in R_1^+\setminus\{\alpha\}.
\]
Then there exists an open neighborhood $\Omega$ of $H_1$
such that $\tilde\Phi$ extends to 
a single-valued holomorphic function on $(\Liea+iU)_\reg\cup\Omega$.
\end{lem}

\begin{proof}
We may assume $\Pi=\{\alpha_1,\ldots,\alpha_\ell\}$
and $\alpha=w\alpha_1$ for some $w\in W$.
Let $\{\xi_1,\ldots,\xi_\ell\}\subset\Liea$ be the dual basis of
$\{w\alpha_1,\ldots,w\alpha_\ell\}\subset\Liea^*$.
Take $\xi_0\in\Liea_+$ and put $H_0=w\xi_0+ws_{\alpha_1}\xi_0$.
Then in the same way as the proof of the previous lemma
one can prove for a sufficiently small $\epsilon>0$
\[
\{H_0+z\xi_1;\,z\in\CC\text{ with }0<|z|<\epsilon\}\subset (\Liea+iU)_\reg,
\]
the function $\Phi(H_0+x\xi_1)$ of $x\in(-\epsilon,\epsilon)$
extends to a holomorphic function $\tilde\Psi(z)$ on
$D:=\{z\in\CC;\,|z|<\epsilon\}$, and
$\tilde\Phi(H_0+z\xi_1)=\tilde\Psi(z)$ for any $z\in D\setminus\{0\}$.

Now put
\[
Z=\{H\in \Liea+iU;\,
\alpha(H)=0,\
\beta(H)\ne 0\text{ for any }\beta\in R_1^+\setminus\{\alpha\}\}.
\]
Then $Z$ is a complex submanifold of $\Liea+iU$
containing $H_0$ and $H_1$.
Since $Z$ is pathwise connected,
by shrinking $D$ if necessary,
we can find a connected open subset $\omega\subset Z$ 
containing $H_0$ and $H_1$ so that
\[
\Omega:=\{H+z\xi_1;\,z\in D, H\in\omega \} \subset
\{H\in \Liea+iU;\,
\beta(H)\ne 0\text{ for any }\beta\in R_1^+\setminus\{\alpha\}\}.
\]
Then
\[
\Omega \cap (\Liea+iU)_\reg=
\{H+z\xi_1;\,z\in D\setminus\{0\}, H\in\omega\}.
\]
Observe from \eqref{eq:KZ} that
$A_j$ in \eqref{eq:dpsiA} for the current $\xi_j$ 
($j=2,\ldots,\ell$)
is holomorphic on 
$\Omega$.
Now for each fixed $z\in D$
let us consider the system
\begin{equation}\label{eq:restsys}
\partial(\xi_j)\Psi(z,H)=A_j(H+z\xi_1)\Psi(z,H)
\quad(j=2,\ldots,\ell)
\end{equation}
for a $\CC W$-valued holomorphic function
$\Psi(z,\cdot)$ on $\omega$
with the initial condition
\begin{equation}\label{eq:inicon}
\Psi(z,H_0)=\tilde\Psi(z).
\end{equation}
It is clear that
$\Psi(z,H):=\tilde\Phi(H+z\xi_1)$ 
is the solution when $z\ne0$.
This $\Psi(z,H)$ is holomorphic on $(D\setminus\{0\})\times \omega$
as a function in $z$ and $H$.
Using \eqref{eq:restsys} and \eqref{eq:inicon}
we can extend $\Psi(z,H)$ to a holomorphic function on
$D \times \omega$ (this is possible without the integrability of the system).
Thus $\tilde\Phi$ extends to a holomorphic function on $(\Liea+iU)_\reg\cup\Omega$
by letting
\[
\tilde\Phi(H+z\xi_1)=\Psi(z,H)\quad\text{for }(z,H)\in D \times \omega.
\qedhere
\]
\end{proof}
Now put
\[
X=\{H\in \Liea_\CC;\,\alpha(H)=\beta(H)=0
\text{ for some two distinct }\alpha,\beta\in R_1^+\}.
\]
By Lemma \ref{lem:rnvsing}, $\tilde\Phi$ extends to a holomorphic function on
$(\Liea+iU)\setminus X$.
But since $X$ is a finite union of linear subspaces of $\Liea_\CC$
with codimension $\ge2$,
$\tilde\Phi$ still extends to a holomorphic function on the whole $\Liea+iU$.
Hence by Lemma \ref{lem:singval} we have $\Phi=\tilde\Phi|_\Liea$
and the analyticity of $\Phi$ on $\Liea$.
This completes the proof of Theorem \ref{thm:Ganal}.

\end{document}